\documentclass{article}
\usepackage{graphicx} 

\usepackage{amsmath}
\usepackage{comment}
\usepackage{leftidx}
\usepackage{multicol}
\usepackage{amssymb}
\usepackage{stmaryrd}
\usepackage{accents}
\usepackage{mathrsfs}
\usepackage{xfrac}
\usepackage{amsthm}
\usepackage{xcolor}
\usepackage{sectsty}
\usepackage{relsize}
\usepackage{float}

\usepackage{Mypackage}
\usepackage{indentfirst}

\usepackage{geometry}
\usepackage{tikz}
\tikzset{shorten <>/.style={shorten >=#1,shorten <=#1}}

\usepackage{tikz-cd}
\newcounter{nodemaker}
\tikzset{Rightarrow/.style={double equal sign distance,>={Implies},->},
triple/.style={-,preaction={draw,Rightarrow}},
quadruple/.style={preaction={draw,Rightarrow,shorten >=0pt},shorten >=1pt,-,double,double
distance=0.2pt}}

\tikzset{%
    symbol/.style={%
        draw=none,
        every to/.append style={%
            edge node={node [sloped, allow upside down, auto=false]{$#1$}}}
    }
}

\usepackage{quiver}

\newcommand{\dbtilde}[1]{\accentset{\approx}{#1}}

\usepackage{hyperref}
\hypersetup{
    colorlinks = true,
    linkbordercolor = {red},
   	linkcolor ={teal},
	anchorcolor = {pink},
	citecolor =  {orange},
	filecolor = {teal},
	menucolor = {teal},
	runcolor =  {teal},
	urlcolor = {teal},
}

\usepackage{cleveref}
\newtheorem{theorem}{Theorem}[subsection]
\newtheorem*{theorem*}{Theorem}
\newtheorem{proposition}[theorem]{Proposition}
\newtheorem{corollary}[theorem]{Corollary}
\newtheorem{corollary'}[theorem]{Corollary}
\newtheorem{lemma}[theorem]{Lemma}
\theoremstyle{definition}
\newtheorem{definition}[theorem]{Definition}

\theoremstyle{definition}
\newtheorem{remark}[theorem]{Remark}
\theoremstyle{definition}
\newtheorem{division}[theorem]{}

\usepackage[backend=biber,
sorting=nty
]{biblatex}
\title{The spectral site revisited}
\author{Axel Osmond\thanks{Istituto Grothendieck, Mondovì, Italy}}
\date{March 2023}

\begin{document}

\maketitle

\begin{abstract}

We give the site-theoretic account of the spectral construction as first introduced by Coste. We provide a detailed examination of the geometric properties of the spectrum, in particular what classes of topoi it produces when applied to the different classes of objects and maps in a geometry. We also give a new proof of the spectral adjunction for set-valued models and the classifying property of the structure sheaf. We also discuss the opposition between ``gros and petit" spectra and relation with some canonical classifying sites. We then describe the general case of a modelled topos. We also prove that the spectrum of a locally modelled topos is local over its base and deduce a new proof of the spectral adjunction for the general case. We finally give several examples from the world of propositional algebras, and in particular recover classifying topoi of first order theories in various fragment of logic as spectra of some suitably modelled topoi. 

\end{abstract}

 {
   \hypersetup{linkcolor=black}
   \tableofcontents
  }

\newpage

\section*{Introduction}

This work, which derives directly from chapter 7 and 9 of the author's thesis \cite{osmond:tel-03609605}, provides an in-depth description of the geometric properties of the spectrum associated with a geometry. Several manner exists to construct spectra, either an abstract way through 2-dimensional universal properties as in \cite{Cole} or \cite{dubuc2000axiomatic}, focusing on its classifying purpose though without actually any insight on the spatial aspects of the spectrum, or more a more concrete pointset manner as in \cite{Diers}; this paper deals somewhat with a middle-way, the site-theoretic approach of the construction of the spectrum. Such an approach combine both the advantages of giving a spatial insight and encoding the universal properties one expect from the spectrum. Though we start from the recipe given in \cite{Coste} (and the akin approach in \cite{Anel}) when defining the spectral site, most of our the work is done in a new manner, developing many aspects that were left mostly implicit our sources. \\

A first section gives all required notions about geometries and a few generalities about factorization systems and geometric extensions as we are going to use throughout this paper. Most of this is treated in a more detailed way in a upcoming companion paper, also in the first three chapters of \cite{osmond:tel-03609605}.\\

In section 2 we give the spectral site of set-valued models, which gathers the finitely presented etale maps under it, and prove it possesses the desired universal property of the spectrum: in particular it induces a restricted spectral adjunction \cref{biadjunction for set-models}. Then we provide a pseudolimit decomposition \cref{pro-etale spectrum} for the spectrum of arbitrary etale maps, and give some geometric properties of the spectra, proving in particular that the spectrum turns etale maps into etale geometric morphisms (\cref{spectrum of fp-etale maps is etale}) and local objects into local topoi (\cref{Spec of local object is local}). We also give a careful examination of the universal property of the structure sheaf and its associated \emph{generic etale map} and \emph{generic local unit}, from which we deduce a restricted spectral adjunction for set-valued models. We also give presentation sites for topoi classifying either etale maps and (etale) maps toward a local object without fixing the domain, and discuss also the relation between gros and petit topos.  \\


In section 3, we provide a detailed account of the spectral site of an arbitrary modelled topos. 
We also use fibrational techniques to prove the structure sheaf of the spectrum to be a local object. We also describe the geometric and sheaf data associated to the canonical fibration, which will be proven to be the unit of the spectral adjunction. \\

Section 4 is devoted to proving the pseudofunctoriality of the spectrum: though this aspect was somewhat swept under the rug in most sources, this problem is actually highly non trivial and requires a large amount of work, involving some techniques of extensions of models along the equivalence between a sheaf topos and its own category of sheaves for its canonical topology, and splitting the problem between the \emph{horizontal and vertical} morphisms of modelled topoi as they were distinguished at \cref{factorization of morphisms of modelled topoi}. \\ 

Section 5 is devoted to proving the spectral adjunction in the general case, for arbitrary modelled topoi, combining the two previous sections. The main lemma is to prove the existence of a retraction of the unit for locally modelled topoi (see \cref{geometric part of the counit}), which generalizes also the localness of the spectrum of a locally modelled topos, but this time \emph{over its base topos} -- see \cref{spectrum of a locally modelled topoi is local over its base}. \\

Section 6 describes the functoriality of this method relative to transformations of geometries, and how we get a \emph{comparison} functor between the spectra associated with two different geometries related by a transformation. \\

Finally, section 7 gather examples from Stone-like duality, some of them were not until now very well known as instance of the spectral construction. We also exploit those examples to construct classifying topoi of theory in various fragment of first order logic as spectra of modelled topoi consisting of the classifyier for the theory of objects (or its multi-sorted analogs) together with some internal Lindenbaum-Tarksi algebra eligible to a spectral construction or another. Finally we present the subobject hyperdoctrines associated to syntactic categories of first order theories as instances of modelled topoi and describe how their spectra are related to the classifying topoi of those theories through local geometric morphisms. 

\section*{Ackonowledgements}

The author is very grateful to Dorette Pronk and Eduardo Dubuc for their reviewing efforts of the thesis from which most of this paper originates. He would like also to thanks Mathieu Anel for numerous discussions and observations that were crucial in our comprehension of this theory, and he ackowledges that this work is very indebted to \cite{Anel}. He should also thanks Olivia Caramello for discussions about various aspect of topos theory involved in this work. Finally he also thanks Morgan Rogers for helpful comments on several points of this topic.  
\newpage

\section{Geometries and (locally) modelled topoi}

In this section we recapitulate the essential about geometries (aka. admissibility structures), from which we are going to build spectral functors in the main part of the paper. It is primarily a concentrate of chapters 1, 3 and 6 of \cite{osmond:tel-03609605} treating respectively of left generated factorization systems, geometries proper and categories of modelled topoi.  

We first expose, as a preliminary definition, the following version of the notion of \emph{geometry}, which encapsulates the categorical and logical data from which will be extracted the geometric content of each spectral construction. Though several possible formulations exist, emphasizing either the logical, categorical or geometric aspect, we chose here to follow \cite{Coste} definition (though called here \emph{triple} or also \emph{admissibility structure}) though using \cite{lurie2009derived} terminology for its evocative virtue. 

\begin{definition}\label{geometry}
A \emph{geometry}\index{geometry} is the data of\begin{itemize}
    \item a finite limit theory $\mathbb{T}$ 
    \item a \emph{saturated class} $\mathcal{V}$ in $ \mathbb{T}[\mathcal{S}]$
    \item a Grothendieck pretopology $J$ on $\mathcal{C}_\mathbb{T}$ whose covers are made duals of maps in $\mathcal{V}$.
\end{itemize} 
\end{definition}

This definitions involves three interdependent data, which will be unfolded : \begin{itemize}
    \item the finite limit theory will code for a class of \emph{ambient objects}; in particular the category of its set-valued model will be a locally finitely presentable category, but we are going to consider more generally the category of all possible models regardless of their base topos;
    \item the saturated class (whose definition will be given in the next subsection) will be a class of finitely presented maps from which one will generate a factorization system in categories of models of the theory of ambient objects in any Grothendieck topoi;
    \item the Grothendieck pretopology can also be visualized as a geometric extension of the theory of ambient objects; its models will be called \emph{local objects}. The crucial condition of being generated in the saturated class will enable a special interaction between local objects and factorization.
\end{itemize}

Let us now recall briefly the content of the notions involved in this definition. First, the factorization data encoded in the saturated class:

\subsection{Left generated factorization systems}

\begin{definition}\label{saturated class}
A \emph{saturated class} in $\mathbb{T}[\mathcal{S}]$ is a set $ \mathcal{V} \subseteq \mathbb{T}[\mathcal{S}]^2_{\omega}$ of finitely presented maps such that:\begin{itemize}
    \item $\mathcal{V}$ contains isomorphisms and is stable under composition,
    \item $\mathcal{V}$ is right-cancellative
    \item $ \mathcal{V}$ is closed under finite colimits in $ { \mathbb{T}[\mathcal{S}]}^2$
    \item $\mathcal{V}$ is closed under pushouts along arbitrary maps between finitely presented objects
\end{itemize}
\end{definition} 

\begin{remark}
    From its definition, a saturated class is always small as consisting of finitely presented maps.
\end{remark}

In several sources were described process to induce from a saturated class a factorization system in a locally presentable category, as in \cite{Coste}, \cite{Anel} (also summed up in \cite{osmond:tel-03609605}[section 1.3]). This process involves the class of all left maps that are of finite presentation in the coslice under each object, which will generate arbitrary left maps out of this objects.

\begin{definition}\label{etale generator}
For any object $B$ in $\mathcal{B}$, define the \emph{etale generator at $B$}\index{etale generator} as full subcategory $ \mathcal{V}_B$ of $ B \downarrow \mathcal{B}$ consisting of morphisms $ n : B \rightarrow C$ such that there exists some $ l : K \rightarrow K'$ in $\mathcal{V}$ and $ a : K \rightarrow B$ exhibiting $ n$ as the pushout     
\[\begin{tikzcd}
	{K} & {K'} \\
	{B} & {C}
	\arrow["{a}"', from=1-1, to=2-1]
	\arrow["{l}", from=1-1, to=1-2]
	\arrow[from=1-2, to=2-2]
	\arrow["{n}"', from=2-1, to=2-2]
	\arrow["\lrcorner"{very near start, rotate=180}, from=2-2, to=1-1, phantom]
\end{tikzcd}\]
\end{definition}

\begin{division}
The etale generator $\mathcal{V}_B$ can be shown to consists of finitely presented objects of in $B \downarrow \mathcal{B}$ and to be closed under finite colimits. Hence we know the category $ \Ind(\mathcal{V}_B)$ to be locally finitely presentable: its objects consists of maps $ l : B \rightarrow C$ under $B$ that are filtered colimits of maps in $\mathcal{V}_B$.
\end{division}
 
\begin{proposition}[{\cite{osmond:tel-03609605}[Proposition 1.1.3.2}]\label{factorization from saturated class}
Let $ \mathbb{T}$ be a finite limit theory and $\mathcal{V}$ be a saturated class in $\mathbb{T}[\mathcal{S}]$: then any arrow $ f : B \rightarrow C$ in $ \mathcal{B}$ admits a factorization 
\[ 
\begin{tikzcd}
B \arrow[rr, "f"] \arrow[rd, "{\colim \, \mathcal{V}_B\downarrow f }"'] &                                                 & C \\
                                                        & {\underset{\mathcal{V}_B\downarrow f}{\colim} \,C} \arrow[ru, "r_f"'] &       
\end{tikzcd} \]
with $ l_f$ is in $\Ind(\mathcal{V}_B)$ and $ r_f$ is in $\mathcal{V}^\perp$.
\end{proposition}

The interest of left generated factorization systems is that, being axiomatizable by a finite limit theory encoded by the small lex category $ \mathcal{V}^{\op}$, they are inherited in the categories of models in Grothendieck topoi beyond $\mathcal{S}$; moreover not only the left maps, but also the right maps behave as nicely as possible: in particular we will see that both left and right maps are stable under inverse and direct images of geometric morphisms, and that the factorizations in sheaf topoi are computed essentially component-wise. Complete proofs of this section can be found at \cite{osmond:tel-03609605}[Section 3.3.2].

\begin{division}For a left generated factorization system $(\mathcal{L}, \mathcal{R})$ generated from a saturated class $ \mathcal{V}$, the associated factorization system, the class of finitely presented left maps $ \mathcal{V}$ in $\mathbb{T}[\mathcal{S}]_{\omega}$ is dual to a class of morphisms in the syntactic category $ \mathcal{C}_\mathbb{T}$. That is, an arrow $ n : K_\phi \rightarrow K_\psi$ (with $ \phi$, $\psi$ the presentation formulas of the domain and codomains) in $\mathcal{V}$ corresponds to an arrow $ [\theta_n(\overline{x}, \overline{y})]_\mathbb{T} : \{ \overline{y}, \psi \} \rightarrow  \{ \overline{x}, \phi \}$, which, as a $\mathbb{T}$-provably functional formula, should be seen as a function symbol coding for an operation. \end{division}

\begin{division}
The previous items were stated in the context of a locally finitely presentable category $\mathcal{B}$: we now abstract ourselves from this and work at the level of the underlying finite limit theory $\mathbb{T}$. From the definition of a saturated class, the category $ \mathcal{V}^{\op}$ has finite limits, hence codes itself for a finite limit theory which admits as classifying topos $ \widehat{\mathcal{V}^{\op}}= [ \mathcal{V}, \mathcal{S}]$ which we denote as $\mathcal{S}[\mathcal{L}] $. In particular this allows us to define for each Grothendieck topos $\mathcal{E}$ a class of arrows $ \mathcal{L}[\mathcal{E}]$ in $\mathbb{T}[\mathcal{E}]$ as 
\[   \mathcal{L}[\mathcal{E}]\simeq \Geom\big{[}\mathcal{E}, \mathcal{S}[\mathcal{L}] \big{]} \simeq \Lex[\mathcal{V}^{\op}, \mathcal{E}] \]
\end{division}

As well as it is sufficient to test the orthogonality property of right maps between set-valued models against only the finitely presented etale maps, right maps in category of models in arbitrary topoi also can be characterized in a way only involving the saturated class rather than all left maps:

\begin{proposition}
For a Grothendieck topos $\mathcal{E}$, a morphism in $\mathbb{T}[\mathcal{E}]$, seen as a natural transformation in $\Lex[\mathcal{C}_{\mathbb{T}},\mathcal{E}]$
\[\begin{tikzcd}
	{\mathcal{C}_{\mathbb{T}}} && {\mathcal{E}}
	\arrow["{F}"{name=0}, from=1-1, to=1-3, curve={height=-12pt}]
	\arrow["{E}"{name=1, swap}, from=1-1, to=1-3, curve={height=12pt}]
	\arrow[Rightarrow, "{u}", from=0, to=1, shorten <=2pt, shorten >=2pt]
\end{tikzcd}\]
is right orthogonal to $ \mathcal{V}$ if and only if its naturality square at a morphism $ [\theta_n(\overline{x}, \overline{y})]_\mathbb{T} : \{ \overline{y}, \psi \} \rightarrow  \{ \overline{x}, \phi \}$ dual of a morphism $n$ in $\mathcal{V}$ is a pullback in $\mathcal{E}$
\[\begin{tikzcd}
	{F(\{ \overline{y}, \psi \})} & {E(\{ \overline{y}, \psi \})} \\
	{F(\{ \overline{x}, \phi \})} & {E(\{ \overline{x}, \phi \})}
	\arrow["{F([\theta_n(\overline{x}, \overline{y})]_\mathbb{T})}"', from=1-1, to=2-1]
	\arrow["{u_{\{ \overline{y}, \psi \}}}", from=1-1, to=1-2]
	\arrow["{E([\theta_n(\overline{x}, \overline{y})]_\mathbb{T})}", from=1-2, to=2-2]
	\arrow["{u_{\{ \overline{x}, \phi \}}}"', from=2-1, to=2-2]
	\arrow["\lrcorner"{very near start, rotate=0}, from=1-1, to=2-2, phantom]
\end{tikzcd}\]
In the following we denote as $ \mathcal{R}[\mathcal{E}]$ the class of right maps in $\mathcal{E}$. 
\end{proposition}

\begin{proposition}[{\cite{osmond:tel-03609605}[Proposition 3.3.2.9]}]
For a left generated factorization system, right maps are stable under inverse image: any geometric morphism $ f : \mathcal{F} \rightarrow \mathcal{E}$ induces a functor $ f^* : \mathcal{R}[\mathcal{E}] \rightarrow \mathcal{R}[\mathcal{F}]$. 
\end{proposition}

Now, recall that models of finite limit theories in sheaf topoi are sheaves of set-valued models over the base site, and morphisms between them are natural transformation. In particular we have:

\begin{proposition}[ {\cite{osmond:tel-03609605}[Proposition 3.3.2.10]} ]\label{localness is a global property}
If $ \mathcal{E} \simeq \Sh (\mathcal{C}_\mathcal{E}, J_\mathcal{E})$, then a transformation $ u : F \rightarrow E$ in $ \mathcal{E}$ is a right map if and only if  for any object $c $ in $ \mathcal{C}$, the component $ u_c : F(c) \rightarrow E(c)$ is in $\mathcal{R}$. 
\end{proposition}

Now it appears that the factorization structure $ (\mathcal{L}, \mathcal{R})$ generated from $ \mathcal{V} $ in $\mathbb{T}[\mathcal{S}]$ is inherited in the category $\mathbb{T}[\mathcal{E}]$ in any Grothendieck topos, and moreover in a functorial and point-wise way. The point-wiseness of this factorization was first established at \cite{Coste}[Theorem 3.6.3]. 

\begin{proposition}\label{factorization in topoi}
For any Grothendieck topos $\mathcal{E}$ with a standard site of presentation $ (\mathcal{C}_\mathcal{E}, J_\mathcal{E})$, we have a factorization system $(\mathcal{L}[\mathcal{E}], \mathcal{R}[\mathcal{E}]) $ in $\mathbb{T}[\mathcal{E}]$. Moreover, for any $ f: \mathcal{F} \rightarrow \mathcal{E}$ we have adjunctions
\[\begin{tikzcd}
	{\mathcal{L}[\mathcal{E}]} && {\mathcal{L}[\mathcal{F}]}
	\arrow["{f_*}"{name=0, swap}, from=1-1, to=1-3, curve={height=12pt}]
	\arrow["{f^*}"{name=1, swap}, from=1-3, to=1-1, curve={height=12pt}]
	\arrow["\dashv"{rotate=-90}, from=1, to=0, phantom]
\end{tikzcd} \quad \quad \begin{tikzcd}
	{\mathcal{R}[\mathcal{E}]} && {\mathcal{R}[\mathcal{F}]}
	\arrow["{f_*}"{name=0, swap}, from=1-1, to=1-3, curve={height=12pt}]
	\arrow["{f^*}"{name=1, swap}, from=1-3, to=1-1, curve={height=12pt}]
	\arrow["\dashv"{rotate=-90}, from=1, to=0, phantom]
\end{tikzcd}  \]
\end{proposition}

\begin{division}
    In the remaining part of the paper, following a mixed terminology arising from different sources, we will give special name to the left and right classes induced from the saturated class $ \mathcal{V}$ in a geometry $ (\mathbb{T}, \mathcal{V}, J)$: \begin{itemize}
    \item the left maps between set-valued $\mathbb{T}$-models will be called \emph{etale maps} and we will denote their class as $\Et$. In particular the saturated class $\mathcal{V}$ will coincide the class of finitely presented etale maps; more generally, for an arbitrary Grothendieck topos $\mathcal{E}$, the class of etale maps in $\mathbb{T}[\mathcal{E}]$ will be denoted  $ \Et[\mathcal{E}]$;
    \item the right maps between set-valued $\mathbb{T}$-models will be called \emph{local maps} and we will denote their class as $\Loc$; more generally, for an arbitrary Grothendieck topos $\mathcal{E}$, the class of local maps in $\mathbb{T}[\mathcal{E}]$ will be denoted  $ \Loc[\mathcal{E}]$.
\end{itemize}
\end{division}

\subsection{Admissibility}

Now we turn to the second component of the notion of geometry, the choice of a Grothendieck topology coding for some geometric extension. 

\begin{division}Recall that a model of $\mathbb{T}_J$ is an object in $\mathbb{T}[\mathcal{S}]$ which is local relative to the dual in $\mathbb{T}[\mathcal{S}]_{\omega}$ of covering families in $J$. But the families dual to cover in $J$ can be extended to the whole category $ \mathbb{T} [\mathcal{S}]$ as follows:
\end{division}

\begin{definition}\label{generalized covers}
Define the \emph{generalized $J$-covers}\index{generalized $J$-covers} as consisting, for each $ B$ in $\mathbb{T}[\mathcal{S}]$, of the families $({n_i} : B {\rightarrow} B_i)_{i \in I}$ such that there exists some $ a : K \rightarrow B$ and some family $(l_i : K \rightarrow K_i)_{i \in I}$ dual to a $J$-cover such that for each $i \in I$ one has a pushout 
\[\begin{tikzcd}
	K & B \\
	{K_i} & {B_i}
	\arrow["{l_i}"', from=1-1, to=2-1]
	\arrow["a", from=1-1, to=1-2]
	\arrow["{n_i}", from=1-2, to=2-2]
	\arrow[from=2-1, to=2-2]
	\arrow["\lrcorner"{anchor=center, pos=0.125, rotate=180}, draw=none, from=2-2, to=1-1]
\end{tikzcd}\]
\end{definition}

Then the property of injectiveness of local objects relatively to $J$-covers extends automatically to those extended covers (which was observed also in \cite{Anel}[Lemma 22]):

\begin{division}
An object $ A$ in $\mathbb{T}[\mathcal{S}]$ is $J$-local if and only if for any object $ B$ in $\mathbb{T}[\mathcal{S}]$, any arrow $ f : B \rightarrow A$ and any generalized $J$-cover $ (n_i : B \rightarrow B_i)_{i \in I}$, there is a factorization of $f$ for some $i \in I$ 
\[ \begin{tikzcd}
       & B \arrow[]{ld}[swap]{n_i} \arrow[]{rd}{n_j} \arrow[]{rr}{f} & & A \\ B_i & ... & B_j \arrow[dashed]{ru}[swap]{\exists}
       \end{tikzcd} \]
In fact $A$ is $J$-local if and only if for any generalized cover under it $ (n_i : A \rightarrow B_i)_{i \in I}$, one has a retraction  
\[ \begin{tikzcd}
       & A \arrow[]{ld}[swap]{n_i} \arrow[]{rd}{n_j} \arrow[equal]{rr}{} & & A \\ B_i & ... & B_j \arrow[dashed]{ru}[swap]{\exists}
       \end{tikzcd} \]
\end{division}

\begin{definition}
For an object $ B$ in $\mathbb{T}[\mathcal{S}]$, a \emph{local unit}\index{local unit} of $B$ is an etale map $ n : B \rightarrow A$ toward a $J$-local object. 
\end{definition}

\begin{remark}
Beware that local units are not required to be finitely presented in general. Etale maps will play the role of saturated compacts of the spectral topology, while finitely presented etale maps will play the role of basic compact opens from which we are going to construct the spectral topology. While this is not apparent in $\mathbb{T}[\mathcal{S}]$ which is ``on the algebraic side", this is more intuitive in the opposite category $\mathbb{T}[\mathcal{S}]^{ \op}$, whose objects should now be thought of as spaces, where the etale morphisms could be seen as generalized inclusions, and the generalized covers induced from $J$ as covers over objects.
\end{remark}

\begin{remark}
Local objects are like \emph{focal spaces}\index{focal spaces}, that is, spaces with a least point in the specialization order. For instance, in a topological space $X$ and a point $x \in X$, the focal component of $X$ in $x$ is the intersection of all neighborhoods of $x$, and this is the upset $ \uparrow x$ in the specialization order. local units behave like inclusions of the form $ \uparrow_\sqsubseteq x \hookrightarrow X$ as such upsets are unreachable by open covering: indeed, in a cover of $\uparrow_\sqsubseteq x$, one open must contain $x$ itself. But as open are up-sets for the specialization order, this open is the whole $\uparrow_\sqsubseteq x$. Hence maximal points, as they do not admit non trivial local units, are alike those $ x $ such that $ \uparrow_\sqsubseteq x = \{ x \}$.\\

In particular, triangles between local units 
\[\begin{tikzcd}
	{B} & {A_1} \\
	& {A_2}
	\arrow["{x_1}", from=1-1, to=1-2]
	\arrow["{n}", from=1-2, to=2-2]
	\arrow["{x_2}"', from=1-1, to=2-2]
\end{tikzcd}\]
should be seen as coding for specialization order between the corresponding minimal point $ x_1 \leq x_2$. Then in $\mathbb{T}[\mathcal{S}]^{ \op}$ this will be turned into an inclusion of focal component $ \uparrow x_2 \subseteq \uparrow x_1$.
\end{remark}

The generation condition in the axioms of geometry ensures the following key property relating local objects and local maps:

\begin{lemma}[{\cite{osmond:tel-03609605}[Lemma 3.3.1.8]}]\label{Gliding for local objects along local maps}
Any object $A$ admitting a local morphism into a local object $ u : A \rightarrow A_0 $ is itself a local object.  
\end{lemma}
\begin{proof}
If $ (n_i :A {\longrightarrow} B_i)_{i \in I}$ is a generalized cover of $A$ then its pushout along $u$ is a generalized cover for $ A_0$, hence admits a lifting $r$ for some $i$, so we have a square that diagonalizes because $u $ is local and $ n_i$ is in $ \mathcal{V}_A$ :
\[ \begin{tikzcd}
A \arrow[equal]{r}{} \arrow[]{d}[swap]{n_i} & A \arrow[]{d}{u } \\ B_i \arrow[]{r}[swap]{r} \arrow[dashed]{ru}{\exists} & A_0 
\end{tikzcd} \]
\end{proof}

From this we deduce the following, which is a form of admissibility for set-valued models; a first occurence of this result is for instance \cite{Coste}[Theorem 3.4.1]

\begin{corollary}\label{local object glide along local maps}
For any arrow $f : B \rightarrow A$ in $\mathbb{T}[\mathcal{S}]$ with $ A$ a $J$-local object, the $ (\Et, \Loc)$ factorization of $f$ 
\[\begin{tikzcd}
	{B} && {A} \\
	& {A_f}
	\arrow["{f}", from=1-1, to=1-3]
	\arrow["{n_f}"', from=1-1, to=2-2]
	\arrow["{u_f}"', from=2-2, to=1-3]
\end{tikzcd}\]
returns a $J$-local object $A_f$
\end{corollary}

We have seen in the previous subsection that the factorization structure is inherited by the category of models in any Grothendieck topos. Similarly, let us see now how the admissibility structure itself is inherited, inducing the multireflectivity of the category of local object and local maps at any Grothendieck topos. 

\begin{proposition}[{\cite{osmond:tel-03609605}{Proposition 3.3.3.1}}]
Let $\mathcal{E}$ be a Grothendieck topos and $ u : F \rightarrow E$ in $\Loc[\mathcal{E}]$ with $ E$ in $\mathbb{T}_J[\mathcal{E}]$: then $F$ itself is in $\mathbb{T}_J[\mathcal{E}]$.
\end{proposition}

This says that for any Grothendieck topos $\mathcal{E}$, the category of $\mathbb{T}$-models in $ \mathcal{E}$ inherits the admissibility structure defined by the geometry $ (\mathbb{T},\mathcal{V}, J)$

\begin{corollary}\label{admissible fact in topos}
 Let $\mathcal{E}$ be a Grothendieck topos: then for any $ f: F \rightarrow E$ in $\mathbb{T}[\mathcal{S}]$ with $ E$ in $\mathbb{T}_J[\mathcal{E}]$, then in the $(\Et[\mathcal{E}], \Loc[\mathcal{E}])$-factorization 
\[\begin{tikzcd}
	{F} && {E} \\
	& {H_f}
	\arrow["{f}", from=1-1, to=1-3]
	\arrow["{n_f}"', from=1-1, to=2-2]
	\arrow["{u_f}"', from=2-2, to=1-3]
\end{tikzcd}\]
the middle term is in $\mathbb{T}_J[\mathcal{E}]$.
\end{corollary}

All of this justifies the terminology for local objects and local transformations as our objects of interest behave locally, at points, as such objects. We also have the following closure property under retracts:

\begin{proposition}[{\cite{osmond:tel-03609605}[Proposition 3.3.3.9]} ]\label{retract of local arrows are local}
In any topos $\mathcal{E}$, a retract of a local object still is a local object.  Similarly retracts of local maps between local objects still are local, as it can be tested point-wisely. 
\end{proposition}

\begin{division}
The main result of the construction of spectra is the existence of a certain adjunction at the level of categories of modelled topoi, as we are going to see in the next sections. If we restrict to the categories of models in fixed topoi, then no such adjunction exists (would they exist, there would be no need of constructing spectra); yet we meet an intermediate situation of \emph{multiadjunction}.
Recall briefly that a \emph{right multiadjoint} is a functor $ U : \mathcal{A} \rightarrow \mathcal{B}$ such that any $B$ in $\mathcal{B}$ admits a small multi-initial family in the comma $B \downarrow U$, that is, a cone of \emph{local units} $ (n_x : B \rightarrow U(A_x))_{x \in I_B}$ such that any $ f : B \rightarrow U(A)$ in $\mathcal{B}$ admits a unique factorization through a unique $x \in I_B$ of the form
\[\begin{tikzcd}
	B && {U(A)} \\
	& {U(A_x)}
	\arrow["f", from=1-1, to=1-3]
	\arrow["{n_x}"', from=1-1, to=2-2]
	\arrow["{U(u_f)}"', from=2-2, to=1-3]
\end{tikzcd}\]
The local units assume conjointly the universal property of the unit of an adjunction. See \cite{Diers} and \cite{osmond2020diersI} for all details on this notion: we would like to emphisize this precise application of the notion of multi-adjunction (established at \cite{osmond:tel-03609605}[theorem 3.3.3.6]):
\end{division}
\begin{theorem}\label{Accessible Mradj at arbitrary topos}
Any a geometry $(\mathbb{T}, \mathcal{V}, J)$ induces at each Grothendieck topos $ \mathcal{E}$ a right multi-adjoint 
\[\begin{tikzcd}
	{\mathbb{T}_J[\mathcal{E}]^{\Loc}} & {\mathbb{T}[\mathcal{E}]}
	\arrow["{\iota_{J,\Loc}[\mathcal{E}]}", hook, from=1-1, to=1-2]
\end{tikzcd}\]
\end{theorem}

    In general, geometric extensions lack a free construction: for instance there is no universal way to correct a commutative ring into a local ring, for the latter are models of a non-cartesian geometric theory. Geometries are intermediate situations where this failure is somewhat not to bad as one at least has a topos-wise ``multi-free" construction in presence of a suited factorization system. As we will see, the purpose of the spectrum is to provide a way to glue altogether those multi-adjunctions into a true global adjunction -- provided one is allowed to navigate across topoi to properly gather the cones of local units as the stalks of some sheaf behaving as the free local object. The details of the multi-adjunction aspect of this question are addressed in chapters 2, 3, 4 of \cite{osmond:tel-03609605}, and will be gathered in an upcomming companion paper.

\subsection{Morphisms of geometries}

Let us also give a few words on the 2-category of geometries and how they can be compared.

\begin{definition}\label{morphisms of geometries}
A \emph{transformation of geometries}\index{geometry!transformation of} $ (\mathbb{T}_1, \mathcal{V}_1, J_1) \rightarrow  (\mathbb{T}_2, \mathcal{V}_2, J_2)$ consists of a functor $ \Phi : \mathcal{C}_{\mathbb{T}_1} \rightarrow \mathcal{C}_{\mathbb{T}_2}$ such that\begin{itemize}
    \item $ \Phi$ is lex
    \item $ \Phi ( \mathcal{V}_1) \subseteq \mathcal{V}_2$
    \item $ \Phi$ induces a morphism of sites $ (\mathcal{C}_{\mathbb{T}_1}, J_1) \rightarrow (\mathcal{C}_{\mathbb{T}_2}, J_2) $
\end{itemize}
\end{definition}

\begin{definition}
In the following we denote $ \mathfrak{Geom}$ the 2-category whose \begin{itemize}
    \item 0-cells are geometries $ (\mathbb{T},  \mathcal{V}, J) $
    \item 1-cells are transformations of geometries $ (\mathbb{T}_1, \mathcal{V}_1, J_1) \rightarrow  (\mathbb{T}_2, \mathcal{V}_2, J_2)$
    \item and 2-cells between two transformations of geometries $ \Phi$, $ \Psi$ are natural transformation between the underlying lex functors $ \alpha : \Phi \Rightarrow \Psi$.
\end{itemize}
\end{definition}

\begin{remark}
Observe that $ \Phi : \mathcal{C}_{\mathbb{T}_1} \rightarrow \mathcal{C}_{\mathbb{T}_2}$ induces a geometric morphism $ \Phi^* \dashv \Phi_*$ (we also denote as $\Phi$) which moreover restricts to the sheaf subtopoi as follows:
\[\begin{tikzcd}
	{\Sh({\mathcal{C}_{\mathbb{T}_2}},J_2)} & {\widehat{\mathcal{C}_{\mathbb{T}_2}}} \\
	{\Sh({\mathcal{C}_{\mathbb{T}_1}},J_1)} & {\widehat{\mathcal{C}_{\mathbb{T}_1}}}
	\arrow["{\iota_2}", hook, from=1-1, to=1-2]
	\arrow["{\iota_1}"', hook, from=2-1, to=2-2]
	\arrow["{\Phi}", from=1-2, to=2-2]
	\arrow["{\Phi}"', from=1-1, to=2-1]
\end{tikzcd}\]

\end{remark}

\begin{division}
Moreover, since $ \Phi$ is a morphism of site, any $ J_2$-continuous lex functor $ F : \mathcal{C}_{\mathbb{T}_1} \rightarrow \mathcal{E}$ is sent by precomposition with $\Phi$ to a $ J_1$-continuous lex functor $\Phi[\mathcal{E}]_*F= F\Phi$. Hence the direct image part of the adjunction above restricts to local objects; the same is true regarding local maps: for any Grothendieck topos $ \mathcal{E}$, we have $ \Phi[\mathcal{E}]_*(\Loc_1[\mathcal{E}]) \subseteq \Loc_2[\mathcal{E}] $. This tel us that any transformation of geometries induces a pseudocommutative square intertwinning the associated right multi-adjoints obtained at \cref{Accessible Mradj at arbitrary topos}:
\[\begin{tikzcd}
	{\mathbb{T}_{J_2}[\mathcal{E}]} & {{\mathbb{T}_2}[\mathcal{E}]} \\
	{\mathbb{T}_{J_1}[\mathcal{E}]} & {{\mathbb{T}_1}[\mathcal{E}]}
	\arrow["{\Phi[\mathcal{E}]_*}", from=1-2, to=2-2]
	\arrow["{\iota_2[\mathcal{E}]}", hook, from=1-1, to=1-2]
	\arrow["{\iota_1[\mathcal{E}]}"', hook, from=2-1, to=2-2]
	\arrow[dashed, from=1-1, to=2-1]
\end{tikzcd}\]

In the following, notably the examples, we will prefer to visualize the transformations of geometries through such squares (in particular the version for set-valued models) rather than as lex functors satisfying the condition above, for the evocative virtue of such a representation. 
\end{division}

\subsection{$\mathbb{T}$-Modelled topoi}

\begin{division}
We have a pseudofunctor, which is representable by $\mathcal{S}[\mathbb{T}]$ 
\[\begin{tikzcd}
	{\GTop^{\op}} & {\Cat} & {}
	\arrow["{\mathbb{T}[-]}", from=1-1, to=1-2]
\end{tikzcd}\]
and is defined as follows:\begin{itemize}
    \item for 0-cells, it returns the category $\mathbb{T}[\mathcal{E}]$ of $\mathbb{T}$-models
    \item for a 1-cell $ f : \mathcal{F} \rightarrow \mathcal{E}$, it returns the inverse image functor 
    \[ \mathbb{T}[\mathcal{E}] \stackrel{f^*}{\longrightarrow} \mathbb{T}[\mathcal{F}] \]
    which is moreover lex and cocontinuous, being left adjoint to the direct image functor 
    \[ \mathbb{T}[\mathcal{F}] \stackrel{f_*}{\longrightarrow} \mathbb{T}[\mathcal{E}] \]
    \item on a 2-cell $ \alpha : f \Rightarrow g$ it returns the natural transformation also denoted as $\alpha$
\[\begin{tikzcd}
	{\mathbb{T}[\mathcal{E}] } && {\mathbb{T}[\mathcal{F}] }
	\arrow["{f^*}"{name=0}, from=1-1, to=1-3, curve={height=-12pt}]
	\arrow["{g^*}"{name=1, swap}, from=1-1, to=1-3, curve={height=12pt}]
	\arrow[Rightarrow, "{\alpha}", from=0, to=1, shorten <=2pt, shorten >=2pt]
\end{tikzcd}\]
\end{itemize}

This defines a 2-fibration with 1-truncated fibers
\[\begin{tikzcd}
	{\int\mathbb{T}[-]} & {\GTop}
	\arrow["{p_\mathbb{T}}", from=1-1, to=1-2]
\end{tikzcd}\]
This bicategory can be considered as a category of all models of $\mathbb{T}$ regardless of their base topos. An object in this category is a \emph{modelled topos}\index{modelled topos}\index{modelled topos!morphism of}, that is, a pair $( \mathcal{E},E)$ with $ \mathcal{E}$ a Grothendieck topos and $ E $ in $\mathbb{T}[E]$. However in the following, we choose to work with an \emph{algebraic convention} in the sense that we want morphisms between modelled topoi to have the orientation of the morphisms of models rather than the orientation of the underlying geometric morphism. To this end, we use the following, which is nothing but the \emph{direct fibration} associated to $ \mathbb{T}[-]$:\end{division}

\begin{definition}
The bicategory $ \mathbb{T}\hy\GTop$ of $\mathbb{T}$-modelled topoi has\begin{itemize}
    \item for 0-cells, modelled topoi $ (\mathcal{E}, E)$
    \item for 1-cells, $(f, \phi) : (\mathcal{F}, F) \rightarrow (\mathcal{E}, E)$ with $ f: \mathcal{E} \rightarrow \mathcal{F}$ a geometric morphism and $ \phi$ consisting of a pair $ (\phi^\flat, \phi^\sharp)$ with 
    \[ \phi^*F \stackrel{f^\flat}{\longrightarrow} E  \quad \textrm{ and } \quad F \stackrel{\phi^\sharp}{\longrightarrow } f_*E \]
    mates along the adjunction $f^* \dashv f_*$
    \item for 2-cells $ \alpha : (f,\phi) \rightarrow (g,\psi)$, 2-cell $ \alpha : f \rightarrow g$ in $\GTop$ satisfying
\[\begin{tikzcd}
	{f^*F} & E \\
	{g^*F}
	\arrow["{\alpha^\flat_F}"', from=1-1, to=2-1]
	\arrow["{\phi^\flat}", from=1-1, to=1-2]
	\arrow["{\psi^\flat}"', from=2-1, to=1-2]
\end{tikzcd} \hskip0.5cm 
\begin{tikzcd}
	F & {f_{*}E} \\
	& {g_{*}E}
	\arrow["{\phi^\sharp}", from=1-1, to=1-2]
	\arrow["{\psi^\sharp}"', from=1-1, to=2-2]
	\arrow["{\alpha^\sharp_E}"', from=2-2, to=1-2]
\end{tikzcd}\] 
\end{itemize}
\end{definition}
From its construction, this bicategory is equipped with a forgetful functor which is both a fibration and an opfibration $ p_\mathbb{T} : \mathbb{T}\hy\GTop \rightarrow \GTop^{\op}$. 

\begin{division}\label{factorization of morphisms of modelled topoi}
Before turning to local objects, it is worth giving some details about the morphisms in this category, which will be shown enlightening in Section 4.3. The 2-category $\mathbb{T}\hy\GTop $ inherits from its fibrational and opfibrational properties over $\GTop^{\op}$ two factorization systems
\[ \textrm{ (Vertical, Cartesian) and (Cocartesian, Vertical) } \]
as seen below in the two alternative factorizations of a same morphism 
\[\begin{tikzcd}
	{(\mathcal{F}_1,F_1)} && {(\mathcal{F}_2,f^*F_1)} \\
	{(\mathcal{F}_1,f_*F_2)} && {(\mathcal{F}_2,F_2)}
	\arrow["{(1_{F_1}, (\phi^\sharp, \phi^\sharp))}"', from=1-1, to=2-1]
	\arrow["{(f, (1_{f^*F_1}, \eta_{F_1}))}", from=1-1, to=1-3]
	\arrow["{(1_{F_2}, (\phi^\flat, \phi^\flat))}", from=1-3, to=2-3]
	\arrow["{(f, ( \epsilon_{F_2}, 1_{f_*F_2}))}"', from=2-1, to=2-3]
	\arrow["{(f,\phi)}"{description}, from=1-1, to=2-3]
\end{tikzcd}\]
provided by the respective unit-counit triangles for the $ f^*\dashv f_*$-adjunction
\[
\begin{tikzcd}
	& {f^*f_*F_2} \\
	{f^*F_1} & {F_2}
	\arrow["{f^*\phi^\sharp}", from=2-1, to=1-2]
	\arrow["{\epsilon_{F_2}}", from=1-2, to=2-2]
	\arrow["{\phi^\flat}"', from=2-1, to=2-2]
\end{tikzcd} \hskip 1cm \begin{tikzcd}
	{f_*f^*F_1} \\
	{F_1} & {f_*F_2}
	\arrow["{\eta_{F_1}}", from=2-1, to=1-1]
	\arrow["{f_*\phi^\flat}", from=1-1, to=2-2]
	\arrow["{\phi^\sharp}"', from=2-1, to=2-2]
\end{tikzcd}\]
 \end{division}

\subsection{$\mathbb{T}_J$-locally modelled topoi}

In the previous chapter, the category of locally modelled topoi was a non-full sub-bicategory of the oplax slice over the classifier of local objects: in fact this can be obtained as the Grothendieck construction associated to a sub-pseudofunctor of the representable. 

\begin{division}
The local data associated to $ (\mathcal{V},J)$ also define a pseudofunctor
\[\begin{tikzcd}
	{\GTop^{\op}} & {\Cat}
	\arrow["{\mathbb{T}_J[-]^{\Loc}}", from=1-1, to=1-2]
\end{tikzcd}\]
and again we can consider the associated direct 2-fibration; but for each topos $\mathcal{E}$, we have a non-full inclusion of category 
\[ \mathbb{T}_J[\mathcal{E}]^{\Loc[\mathcal{E}]} \hookrightarrow \mathbb{T}[\mathcal{E}] \]
inducing the following non-full sub-bicategory of $\mathbb{T}\hy\GTop$ defined:\end{division}

\begin{definition}
The bicategory $ \mathbb{T}_J\hy\GTop^\Loc$ of \emph{ $\mathbb{T}_J$-locally modelled topoi}\index{locally modelled topoi}\index{locally modelled topoi!morphism of} has \begin{itemize}
    \item for 0-cells pairs $ (\mathcal{E}, E)$ with $ E $ in $ \mathbb{T}_J[\mathcal{E}]$,
    \item for 1-cells pairs $( f,\phi) : (\mathcal{F}, F) \rightarrow (\mathcal{E},E)$ with $ f : \mathcal{E} \rightarrow \mathcal{F}$ and $ \phi $ is such that $ \phi^\flat : f^*F \rightarrow E$ is in $\Loc[\mathcal{E}]$,
    \item and the same 2-cells as $\mathbb{T}\hy\GTop$.
\end{itemize}
\end{definition}

In particular this inclusion is a strict morphism of opfibrations 
\[\begin{tikzcd}
	{\mathbb{T}_J\hy\GTop^\Loc} && {\mathbb{T}\hy\GTop} \\
	& {\GTop^{\op}}
	\arrow["{\iota_{J,\Loc}}"{name=0}, from=1-1, to=1-3, hook]
	\arrow[from=1-1, to=2-2, "p_{J,\Loc}"']
	\arrow[from=1-3, to=2-2, "p_\mathbb{T}"]
	\arrow[Rightarrow, "{=}" description, from=0, to=2-2, shorten <=2pt, shorten >=2pt, phantom, no head]
\end{tikzcd}\]
so that in the following, we may innocently write $\iota_{J,\Loc}(\mathcal{E},E)$ as $ (\mathcal{E}, wE)$ where we denote abusively $ w$ for the faithful (but non-full) inclusion $ w_\mathcal{E} : \mathbb{T}_J[\mathcal{E}]^{\Loc[\mathcal{E}]} \hookrightarrow \mathbb{T}[\mathcal{E}]$: in other words, $ \iota_{J,\Loc}$ just forget the localness of models and arrows without modifying the underlying topos. However beware that $ p_{J,\Loc}$ does not inherit the fibration structure of $p_\mathbb{T} $ as $ \mathbb{T}_J$ models are not stable under direct image of geometric morphisms. 

\begin{remark}\label{iota is conservative}
We should really emphasize that the inclusion is not full, for the restriction imposed on the inverse image part of morphism of locally modelled topoi. However it is faithful, and moreover it can be shown to lift equivalences -- see \cite{osmond:tel-03609605}[lemma 6.1.2.3].
\end{remark}

The problem of the spectrum is to construct a left bi-adjoint to this inclusion functor -- that is, to provide a ``free" locally modelled topos under any modelled topos. There are several manners to do so, and in this paper, we will focus on a concrete construction of the spectrum as a locally modelled topos, describing both a site for the underlying topos and the associated local object as a structure sheaf.\\

In the remaining of this paper, we fix a geometry $ ( \mathbb{T},\mathcal{V},J)$ and denote as $ \Et$ and $\Loc$ the associated factorization system in $\mathbb{T}[\mathcal{S}]$, as well as $\mathbb{T}_J$ the geometric extension of $\mathbb{T}$ encoded by $J$. 
\section{Spectral site of a set-valued model}

We process in two times in order to give a detailed insight of the process of the construction. We first emphasize the construction of the spectrum for a set valued model, as well as several geometric properties of the spectra of the different class of objects and maps involved in a geometry. We also want to give a very explicate description of how processes the universal property of the associated structure sheaf, and present a first version of the spectral adjunction for set-valued models which is more akin to what people encounter in algebraic geometry. We also discuss sheaf representability condition, which are an important topic of sheaf theoretic algebra. While the key results of relative to the construction and the adjunction are already present in \cite{Coste} and also \cite{Anel}, we give here quite different description

\subsection{Spectral topology on the etale generator}

Recall that for any set-valued $\mathbb{T}$-model $B$ we denote as $ \mathcal{V}_B$ the etale generator of $B$, which consists of arrows obtained as pushouts under $B$ of finitely presented etale maps in $\mathcal{V}$.

\begin{definition}
For any $ B$ in $ \mathbb{T}[\mathcal{S}]$, the opposite category of the etale generator $ \mathcal{V}_B^{\op}$ can be equipped with the \emph{spectral pretopology}\index{spectral!pretopology}, which is the Grothendieck pretopology $ J_B$ defined from the duals of the families
\[ \Bigg{(} \begin{tikzcd}
B \arrow[]{r}{n} \arrow[]{rd}[swap]{n_i} & C\arrow[]{d}{m_i} \\ & C_i 
\end{tikzcd} \Bigg{)}_{ i \in I} \]
such that there exists some arrow $b : K \rightarrow C$ and a covering family $ (l_i : K \rightarrow K_i)_{i \in I}$ dual of a covering family in $J$ such that for each $i \in I$ we have
\[\begin{tikzcd}
	K & C & B \\
	{K_i} & {C_i}
	\arrow["n"', from=1-3, to=1-2]
	\arrow["{m_i}"{description}, from=1-2, to=2-2]
	\arrow["{n_i}", from=1-3, to=2-2]
	\arrow["b", from=1-1, to=1-2]
	\arrow[from=2-1, to=2-2]
	\arrow["{l_i}"', from=1-1, to=2-1]
	\arrow["\lrcorner"{anchor=center, pos=0.125, rotate=180}, draw=none, from=2-2, to=1-1]
\end{tikzcd}\]
\end{definition}

\begin{remark}\label{spectral topology for finitely presented objects}
In particular, in the case of a finitely presented object $K$, we know $ \mathcal{V}_K$ to be made of finitely presented etale arrows in $\mathcal{V}$, that is, $ \mathcal{V}_K \hookrightarrow \mathcal{V}$ for $ \mathcal{V}$ is closed under pushouts between finitely presented objects. Then, in particular, any map $ n $ in $\mathcal{V}_K$ has a finitely presented codomain. But now, a covering family in $J_K(n)$ is induced by pushout from some $ (n_i : K_0 \rightarrow K'_i)_{i \in I}  $ corresponding to a covering family in $J$ as follows
\[\begin{tikzcd}
	K & {K'} & {K_0} \\
	& {K_i} & {K_i'}
	\arrow["n", from=1-1, to=1-2]
	\arrow["{m_i}"{description}, from=1-2, to=2-2]
	\arrow["{n_i}"', from=1-1, to=2-2]
	\arrow["k"', from=1-3, to=1-2]
	\arrow[from=2-3, to=2-2]
	\arrow["{l_i}", from=1-3, to=2-3]
	\arrow["\lrcorner"{anchor=center, pos=0.125, rotate=90}, draw=none, from=2-2, to=1-3]
\end{tikzcd}\]
But $J$, as a Grothendieck coverage, is closed under pullback, and hence the pushout family $ (m_i : K' \rightarrow K_i)_{i \in I}$ is dual to a covering family in $J(K')$. Conversely, any family $ (m_i : K' \rightarrow K_i)_{i \in I}$ dual to a cover of $ K'$ in $J$ induces trivially a covering family of $l$ in $J_K$ as it is a pushout of itsefl along the identity $1_{K'}$. Hence a familly $ (m_i : n \rightarrow n_i)_{i \in I}$ is covering in $J_K$ if and only if $ (m_i : \cod(n) \rightarrow \cod(n_i))_{i \in I}$ is dual to a covering family in $J(\cod(n))$.  
\end{remark}

\begin{definition}
For a set-valued model $ B$ in $ \mathbb{T}[\mathcal{S}]$, $(\mathcal{V}_B^{\op}, J_B) $ is called the \emph{spectral site}\index{spectral!site} of $B$. Then define the \emph{(Coste) spectrum}\index{spectrum!Coste} of $B$ as 
\[ \Spec (B) =  \Sh(\mathcal{V}_B^{\op}, J_B) \]
We have a geometric embedding into the presheaf topos over $\mathcal{V}_B^{\op}$
\[\begin{tikzcd}
	{\Spec(B)} && {\widehat{\mathcal{V}_B^{\op}}}
	\arrow[""{name=0, anchor=center, inner sep=0}, "{\iota_B }"', start anchor=-15, bend right=20, hook, from=1-1, to=1-3]
	\arrow[""{name=1, anchor=center, inner sep=0}, "{\mathfrak{a}_{J_B}}"', end anchor=15, bend right=20, from=1-3, to=1-1]
	\arrow["\dashv"{anchor=center, rotate=-90}, draw=none, from=1, to=0]
\end{tikzcd}\]
\end{definition}

\begin{remark}
Observe that the sheafification functor $\mathfrak{a}_{J_B}$ extends into a functor between categories of $\mathbb{T}$-models and $ \mathbb{T}_J$-models. In fact we have a pseudocommutative square 
\[\begin{tikzcd}
	{\mathbb{T}[\widehat{\mathcal{V}_B^{\op}}]} & {\mathbb{T}[\Spec(B)]} \\
	{\mathbb{T}_J[\widehat{\mathcal{V}_B^{\op}}]} & {\mathbb{T}_J[\Spec(B)]}
	\arrow["{\mathfrak{a}_{J_B}^*}", from=1-1, to=1-2]
	\arrow[hook, from=2-1, to=1-1]
	\arrow[hook, from=2-2, to=1-2]
	\arrow["{\mathfrak{a}_{J_B}^*}"', from=2-1, to=2-2]
	\arrow["\simeq"{description}, draw=none, from=1-1, to=2-2]
\end{tikzcd}\]
which is the pseudonaturality square of the natural transformation 
\[\begin{tikzcd}
	{\Geom[\widehat{\mathcal{V}_B^{\op}}, -]} && { \Geom[\Spec(B) ,-]}
	\arrow["{\Geom[\iota_B, -]}", Rightarrow, from=1-1, to=1-3]
\end{tikzcd}\]
at the inclusion $ \iota_J : \mathcal{S}[\mathbb{T}_J] \hookrightarrow \mathcal{S}[\mathbb{T}]$. 
\end{remark}

\begin{division}\label{Spec(f) in the set case}
Now we turn to the functoriality of the construction. For a morphism $ f: B_1 \rightarrow B_2$ in $\mathbb{T}[\mathcal{S}]$, the geometric morphism $\Spec(f)$ is computed from the pushout functor along $f$
\[\begin{tikzcd}[row sep=tiny]
	{\mathcal{V}_{B_1}} \arrow[]{r}{f_*}&  {\mathcal{V}_{B_2}}
\end{tikzcd}\]
sending a finitely presented etale arrow to its pushout
\[\begin{tikzcd}
	{B_1} & {B_2} \\
	{\cod(n)} & {\cod(f_*n)}
	\arrow["n"', from=1-1, to=2-1]
	\arrow["f", from=1-1, to=1-2]
	\arrow["{f_*n}", from=1-2, to=2-2]
	\arrow["{n_*f}"', from=2-1, to=2-2]
	\arrow["\lrcorner"{anchor=center, pos=0.125, rotate=180}, draw=none, from=2-2, to=1-1]
\end{tikzcd}\]
But now, observe that the pushout functor sends finite colimits of $\mathcal{V}_{B_1} $ to finite colimits in $\mathcal{V}_{B_2}$ hence defines a lex functor $ \mathcal{V}^{\op}_{B_1} \rightarrow \mathcal{V}^{\op}_{B_2}$. Moreover, this functor is $J_{B_1}$-continuous, by composition of pushouts. Hence Diaconescu applies and returns an extension  
\[\begin{tikzcd}[column sep=huge]
	{\mathcal{V}_{B_1}^{\op}} & {\mathcal{V}_{B_2}^{\op}} \\
	{\Spec(B_1)} & {\Spec(B_2)}
	\arrow["{f_*}", ""{name=0, anchor=center, inner sep=0}, from=1-1, to=1-2]
	\arrow["{\mathfrak{a}_{J_{B_2}}\katayo}", from=1-2, to=2-2]
	\arrow["{\mathfrak{a}_{J_{B_1}}\katayo}"', from=1-1, to=2-1]
	\arrow[""{name=1, anchor=center, inner sep=0}, "{\lan_{\mathfrak{a}_{J_{B_1}}\katayo}\mathfrak{a}_{J_{B_2}}\katayo f_*}"', from=2-1, to=2-2]
	\arrow["{\simeq}"{description}, shorten <=5pt, shorten >=5pt, draw=none, from=0, to=1]
\end{tikzcd}\]
which is the inverse image part of $\Spec(f)$. \end{division}

\subsection{Etale maps produce (pro-)etale geometric morphisms}

\begin{division}
$\mathcal{V}_B^{\op}$ is a lex site coding for ``basic compact open inclusions". Objects of the sheaf topos $ \Spec(B) \hookrightarrow [\mathcal{V}_B, \mathcal{S}]$ should be seen as generalized opens of the spectral topology, while objects of $\Ind(\mathcal{V}_B)$, which are arbitrary etale arrows under $B$, should be seen as saturated compacts of the spectral topology. In particular, the embedding $ \mathcal{V}_B^{\op} \hookrightarrow \Spec(B)$ exibits $ \mathcal{V}_B$ as a basis of \emph{basic compact sets that are open} -- and $\mathcal{V}^{\op}_B$ as a basis of \emph{open sets that are compacts}. 
\end{division}

The following observation (already present in \cite{Anel}[proposition 38] motivates the name for etale arrows:

\begin{proposition}\label{spectrum of fp-etale maps is etale}
Finitely presented etale arrows $ n : B \rightarrow C$ under $B$ correspond to etale geometric morphisms of the form:
\[ \Spec(C) \simeq \Spec(B) / \mathfrak{a}_{J_B}\katayo_n \stackrel{\Spec(n)}{\longrightarrow} \Spec(B) \]
\end{proposition}

\begin{proof}
For $n : B \rightarrow C $ in $\mathcal{V}_B$ we have an equivalence of categories 
\[ \mathcal{V}_C \simeq n\downarrow\mathcal{V}_B \]
sending $ m : C \rightarrow D$ to the triangle 
\[\begin{tikzcd}
	B \\
	C & D
	\arrow["n"', from=1-1, to=2-1]
	\arrow["m"', from=2-1, to=2-2]
	\arrow["mn", from=1-1, to=2-2]
\end{tikzcd}\]
and conversely any triangle $ l = mn $ in $ n \downarrow \mathcal{V}_B$ to the underlying arrow $ m$. In particular we have 
\[ \widehat{\mathcal{V}_C^{\op}} \simeq \widehat{(n\downarrow\mathcal{V}_B)^{\op}} \simeq \widehat{\mathcal{V}^{\op}_B}/\katayo_n \]

But also by the expression of slices in a sheaf category (see \cite{sga4}[III, Proposition 5.4]), we know that the topology induced on $ n \downarrow \mathcal{V}_B$ by $J_B$ is the same as $J_C$ -- this is the corresponding topology $ J'$ corresponding to $ J_C$ through the equivalence above -- and we have 
\[ \Spec(C) \simeq \Sh( n \downarrow \mathcal{V}^{\op}_B, J')  \simeq  \Spec(B) / \mathfrak{a}_{J_B}\katayo_n \]\end{proof}

\begin{remark}\label{etale as pb}
Observe that we have a 2-pullback square in the bicategory of Grothendieck topoi
\[\begin{tikzcd}
	{\Spec(B)/\mathfrak{a}_{J_B}\katayo_n} & {\widehat{\mathcal{V}_B^{\op}}/\katayo_n} \\
	{\Spec(B)} & {\widehat{\mathcal{V}_B^{\op}}}
	\arrow["{\Spec(n)}"', from=1-1, to=2-1]
	\arrow[""{name=0, anchor=center, inner sep=0}, "{\iota_B}", hook, from=2-1, to=2-2]
	\arrow["{\overline{\katayo_n}}", from=1-2, to=2-2]
	\arrow[from=1-1, to=1-2]
	\arrow["\lrcorner"{anchor=center, pos=0.125}, draw=none, from=1-1, to=0]
\end{tikzcd}\]
exhibiting the etale geometric morphism $\Spec(n)$ as the 2-pullback of the etale geometric morphism associated to ${\katayo_n} $.
\end{remark}

\begin{remark}
The further left adjoint of the inverse image will be induced from the postcomposition functor $ \mathcal{V}^{\op}_C \rightarrow \mathcal{V}^{\op}_B $ sending $ m : C \rightarrow D$ to the composite $ mn : B \rightarrow D$ which is in $\mathcal{V}_B$. This functor defines a left adjoint 
\[\begin{tikzcd}
	{\mathcal{V}_{B}^{\op}} && {\mathcal{V}_C^{\op}}
	\arrow["{n_*}"{name=0, swap}, from=1-1, to=1-3, curve={height=12pt}]
	\arrow["{n^*}"{name=1, swap}, from=1-3, to=1-1, curve={height=12pt}]
	\arrow["\dashv"{rotate=-90}, from=1, to=0, phantom]
\end{tikzcd}\] 
\end{remark}

The intuition that objects of $\mathcal{V}_B$ are compact can be formalized thanks to the following property. Recall that a geometric morphism is said to be \emph{tidy}\index{geometric morphism!tidy} if its direct image part preserves filtered colimits. From \cite{moerdijk2000proper}[Theorem 4.8] we know that tidy geometric morphisms are stable under 2-pullback.

\begin{proposition}\label{spectrum of finitely presented etale map is tidy}
For $ n : B \rightarrow C$ in $\mathcal{V}_B$, the geometric morphism $\Spec(n) : \Spec(C) \rightarrow \Spec(B)$ is tidy.
\end{proposition}

\begin{proof}
Recall we can express $\Spec(n)$ as the pullback of the etale geometric morphism $\widehat{\mathcal{V}_B^{\op}}/\katayo_n \rightarrow \widehat{\mathcal{V}^{\op}_B} $ along $\iota_B$. But we know that $\katayo_n$ is a finitely presented object in the presheaf topos $ \widehat{\mathcal{V}_B^{\op}}$, so that the associated internal hom functor $ (-)^{\katayo_n}$ preserves filtered colimits: hence its associated etale geometric morphism is tidy, and hence its pullback $ \Spec(n)$ also is. 
\end{proof}

\begin{remark}
Arbitrary etale arrows are not in the topos $\Spec(B)$, but rather from the side of points and saturated compacts. Hence they do not correspond to etale geometric morphisms in general. In fact observe that an arbitrary etale map $ l : B \rightarrow C$ is an object of $\Ind(\mathcal{V}_B)$, for the factorization system $ (\Et, \Loc)$ was left generated from $\mathcal{V}$; but we have $  \Ind(\mathcal{V}_B) \simeq \Pro(\mathcal{V}_B^{\op})^{\op} $, which is the pro-completion of $ \mathcal{V}_B^{\op}$, whose objects are those functors $ \mathcal{V}_{B}^{\op} \rightarrow \mathcal{S}$ which are cofiltered limits of representables: this mimics the fact that arbitrary etale maps are constructed as cofiltered intersection of basic open compact sets. For this reason, \cite{Anel} says \emph{pro-etale}\index{pro-etale map} for what we call arbitrary etale, reserving ``etale" for the basic ones. This can be formalized into the following result:\end{remark}

\begin{theorem}\label{pro-etale spectrum}
Let $ l : B \rightarrow C$ be an arbitrary etale arrow under $B$. Then $ \Spec(C)$ decomposes as a cofiltered pseudolimit
\[ \Spec(C) \simeq \underset{(n,a) \in \int \! \hirayo_l}{\bilim} \Spec(\cod(n)) \]
\end{theorem}

\begin{proof}
From \cite{osmond:tel-03609605}[Theorem 1.1.4.3], we know that \[ \mathcal{V}_C \simeq \underset{(n,a) \in \int \! \hirayo_l}{\pscolim} \mathcal{V}_{\cod(n)} \]
Now recall that for each $ (n,a)$ in $\int \!\hirayo_l$, the opposite category of the etale generator $ \mathcal{V}_{\cod(n)}$ can be equipped with a pretopology $J_{\cod(n)}$, and the opposite category of the pseudocolimit \[ ({\pscolim}_{(n,a) \in \int \! \hirayo_l}\mathcal{V}_{\cod(n)})^{\op}  \]
can be equipped with the coarsest topology \[ \langle \bigcup_{(n,a) \in \int \!\hirayo_l} q_{(n,a)}(J_{\cod(n)}) \rangle  \]
making the canonical inclusion continuous
\[\begin{tikzcd}
	{\mathcal{V}_{\cod(n)}^{\op} } & {({\pscolim}_{(n,a) \in \int \! \hirayo_l}\mathcal{V}_{\cod(n)})^{\op} }
	\arrow["{q_{(n,a)}^{\op}}", from=1-1, to=1-2]
\end{tikzcd}\]

From \cite{sga4} and \cite{dubuc2011construction}, and also the general version of \cref{cofiltered pseudolimit of topos} on cofiltered pseudolimits of Grothendieck topoi, we know that the corresponding sheaf topos is the pseudolimit 
\[ \Sh\big{(}({\pscolim}_{(n,a) \in \int \! \hirayo_l}\mathcal{V}_{\cod(n)})^{\op} , \langle \bigcup_{(n,a) \in \int \!\hirayo_l} q_{(n,a)}(J_{\cod(n)}) \rangle\big{)} \simeq \underset{(n,a) \in \int \! \hirayo_l}{\bilim} \Spec(\cod(n)) \]
and moreover, this topology is exactly the image of the induced topology on the pseudocolimit along the equivalence of categories above with $\mathcal{V}_C$.
Now we can also glue the image of the $ J_{\cod(n)}$-covering families along the pushout functors $a_*$ to generate a topology on $ \mathcal{V}^{\op}_C$
\[ \langle \bigcup_{(n,a) \in \int \!\hirayo_l} a_*(J_{\cod(n)}) \rangle  \]
We must prove that any covering family $ a_*(J_{\cod(n)})$ is covering in $J_C$, and conversely that any $J_C$ covering family is covering in the topology induced from the colimit. \\

First, let us prove that a cover in the jointly generated pretopology is covering in $J_C$. For $ (n,a) $ in $ \int \! \hirayo_l$, $ m$ an object of $ \mathcal{V}_{\cod(n)}$ and a covering family $ (l_i)_{i\in I}$ induced as 
\[\begin{tikzcd}
	K & D & {\cod(n)} \\
	{K_i} & {D_i}
	\arrow["{l_i}"{description}, from=1-2, to=2-2]
	\arrow["{k_i}"', from=1-1, to=2-1]
	\arrow[from=2-1, to=2-2]
	\arrow["b", from=1-1, to=1-2]
	\arrow["m"', from=1-3, to=1-2]
	\arrow["{m_i}", from=1-3, to=2-2]
	\arrow["\lrcorner"{anchor=center, pos=0.125, rotate=180}, draw=none, from=2-2, to=1-1]
\end{tikzcd}\]
in $J_{\cod(n)}(m)$, consider the composition of pushouts as in the diagram below
\[\begin{tikzcd}
	K && D && {\cod(n)} & B \\
	&&& {a_*D} && C \\
	{K_i} && {D_i} \\
	&&& {a_*D_i}
	\arrow["{l_i}"{description, pos=0.3}, from=1-3, to=3-3]
	\arrow["{k_i}"', from=1-1, to=3-1]
	\arrow[from=3-1, to=3-3]
	\arrow["b", from=1-1, to=1-3]
	\arrow["m"', from=1-5, to=1-3]
	\arrow["\lrcorner"{anchor=center, pos=0.125, rotate=180}, draw=none, from=3-3, to=1-1]
	\arrow["{m_*a}"{description}, from=1-3, to=2-4]
	\arrow["a"{description}, from=1-5, to=2-6]
	\arrow["n"', from=1-6, to=1-5]
	\arrow["l", from=1-6, to=2-6]
	\arrow["{a_*m_i}", from=2-6, to=4-4]
	\arrow["{m_*l_i}"{description}, from=2-4, to=4-4]
	\arrow[from=3-3, to=4-4]
	\arrow[""{name=0, anchor=center, inner sep=0}, from=2-6, to=2-4]
	\arrow["\lrcorner"{anchor=center, pos=0.125, rotate=180}, shift right=2, draw=none, from=4-4, to=3-3]
	\arrow[curve={height=12pt}, from=3-1, to=4-4]
	\arrow["\lrcorner"{anchor=center, pos=0.125, rotate=180}, shift left=3, draw=none, from=4-4, to=1-1]
	\arrow["\lrcorner"{anchor=center, pos=0.125, rotate=90}, draw=none, from=2-4, to=1-5]
	\arrow["{m_i}"{description}, from=2-4, to=3-3]
	\arrow[shorten >=7pt, no head, from=1-5, to=2-4]
	\arrow["\lrcorner"{anchor=center, pos=0.125, rotate=90}, draw=none, from=4-4, to=0]
	\arrow[curve={height=12pt}, from=1-1, to=2-4, crossing over]
\end{tikzcd}\]
Then we see that for each $ i \in I $ the arrow $ a_*l_i$ is also the pushout $ (m_*a \,b)k_i  $, and this exhibits the pushout family $(m_*l_i)_{i \in I}$ as a pushout of a family in $J$, and hence as a covering family of $J_C$. \\

Conversely let us prove that a $J_C$-cover is covering in the jointly generated pretopology. For a covering family of some $m$ in $\mathcal{V}_C$ for $J_C$
\[\begin{tikzcd}
	&& C \\
	K & D \\
	& {K_i} & {D_i}
	\arrow["m"', from=1-3, to=2-2]
	\arrow["{m_i}", from=1-3, to=3-3]
	\arrow["{l_i}", from=2-2, to=3-3]
	\arrow["b", from=2-1, to=2-2]
	\arrow[from=3-2, to=3-3]
	\arrow["{k_i}"', from=2-1, to=3-2]
	\arrow["\lrcorner"{anchor=center, pos=0.125, rotate=180}, draw=none, from=3-3, to=2-1]
\end{tikzcd}\]
we know from the essential surjectivity of the equivalence above that $ m$ is induced as some $ a_*m'$ for some $ (n,a) \in \int \! \hirayo_l$


\[\begin{tikzcd}
	& {\cod(n)} && C \\
	D' && D \\
	&&& {D_i}
	\arrow["m"{description}, from=1-4, to=2-3]
	\arrow["{m_i}", from=1-4, to=3-4]
	\arrow["{l_i}"{description}, from=2-3, to=3-4]
	\arrow[from=2-1, to=2-3]
	\arrow["{m'}"', from=1-2, to=2-1]
	\arrow["a", from=1-2, to=1-4]
	\arrow["\lrcorner"{anchor=center, pos=0.125, rotate=180}, draw=none, from=2-3, to=1-2]
\end{tikzcd}\]

Now observe that from the situation below
\[\begin{tikzcd}
	{\cod(n)} & C \\
	{D'} & D \\
	K
	\arrow["m"{description}, from=1-2, to=2-2]
	\arrow[from=2-1, to=2-2]
	\arrow["{m'}"', from=1-1, to=2-1]
	\arrow["a", from=1-1, to=1-2]
	\arrow["\lrcorner"{anchor=center, pos=0.125, rotate=180}, draw=none, from=2-2, to=1-1]
	\arrow["b"{description}, from=3-1, to=2-2]
\end{tikzcd}\]
there exists by general property of finitely presented objects some $(n_1,a_1)$ in $\int \!\hirayo_l$ and a factorization of $b$ through the intermediate pushout as below
\[\begin{tikzcd}
	{\cod(n)} & {\cod(n_1)} & C \\
	{D'} & {{a'}_*D'} & D \\
	& K
	\arrow["m"{description}, from=1-3, to=2-3]
	\arrow["{m'}"', from=1-1, to=2-1]
	\arrow["b"{description}, from=3-2, to=2-3]
	\arrow[from=2-1, to=2-2]
	\arrow["{a'}", from=1-1, to=1-2]
	\arrow[from=2-2, to=2-3]
	\arrow["\lrcorner"{anchor=center, pos=0.125, rotate=180}, draw=none, from=2-3, to=1-2]
	\arrow["{a_1}", from=1-2, to=1-3]
	\arrow[from=1-2, to=2-2, "a'_*m'"]
	\arrow["\lrcorner"{anchor=center, pos=0.125, rotate=180}, draw=none, from=2-2, to=1-1]
	\arrow["c", dashed, from=3-2, to=2-2]
\end{tikzcd}\]
and now we can pushout the family $(k_i)_{i \in I}$ of $J$ along $c$ to get a covering family $ (c_*k_i)_{i \in I}$ in $J_{\cod(n)_1}$ of the object $ a'_*m'$ in $\mathcal{V}_{\cod(n_1)}$

\[\begin{tikzcd}
	& {\cod(n)} && {\cod(n_1)} && C \\
	{D'} && {a'_*D'} & {} & D \\
	&& {} & {c_*K_i} && {D_i} \\
	&& K \\
	&&& {K_i}
	\arrow["m"{description}, from=1-6, to=2-5]
	\arrow["{m_i}", from=1-6, to=3-6]
	\arrow["{l_i}"{description}, from=2-5, to=3-6]
	\arrow[from=3-4, to=3-6]
	\arrow["{c_*k_i}"{description}, from=2-3, to=3-4]
	\arrow["{m'}"', from=1-4, to=2-3]
	\arrow[""{name=0, anchor=center, inner sep=0}, "{m_i'}"{description, pos=0.3}, from=1-4, to=3-4]
	\arrow["{a_1}", from=1-4, to=1-6]
	\arrow["\lrcorner"{anchor=center, pos=0.125, rotate=180}, draw=none, from=2-5, to=1-4]
	\arrow[shorten >=9pt, no head, from=4-3, to=3-4]
	\arrow["b"{description}, from=3-4, to=2-5]
	\arrow[""{name=1, anchor=center, inner sep=0}, "{k_i}"', from=4-3, to=5-4]
	\arrow[from=5-4, to=3-6]
	\arrow[from=5-4, to=3-4]
	\arrow[no head, from=2-3, to=2-4]
	\arrow[from=2-4, to=2-5]
	\arrow[from=2-1, to=2-3]
	\arrow["{m'}"', from=1-2, to=2-1]
	\arrow["{a'}", from=1-2, to=1-4]
	\arrow["c", from=4-3, to=2-3]
	\arrow["\lrcorner"{anchor=center, pos=0.125, rotate=-90}, draw=none, from=3-4, to=4-3]
	\arrow["\lrcorner"{anchor=center, pos=0.125, rotate=180}, draw=none, from=2-3, to=1-2]
	\arrow["\lrcorner"{anchor=center, pos=0.125, rotate=180}, draw=none, from=3-6, to=0]
	\arrow["\lrcorner"{anchor=center, pos=0.125, rotate=-90}, draw=none, from=3-6, to=1]
\end{tikzcd}\]
Moreover, the objects $ ((n,a),m')$ and $ ((n_1,a_1),a'_*m')$ are identified in the pseudocolimit for they are related through an opcartesian morphism, so that $ (c_*k_i)_{i \in I} $ is both covering for the class of $ ((n,a),m')$ in the induced topology on the pseudocolimit $ \colim_{(n,a) \in \int \! \hirayo_l} \mathcal{V}_{\cod(n)}$, and is sent to the covering family $(l_i)_{i \in I}$ of $J_C$ through the pushout functor $ {a_1}_*$. Hence the $J_C$-cover $ (l_i)_{i\in I}$ is in the induced topology $ {a_1}_*(J_{\cod(n_1)})$, and hence in the jointly generated topology. This proves that the equivalence of categories above induces an equivalence of sites
\[  \big{(}(\underset{(n,a) \in \int \! \hirayo_l}{\pscolim} \mathcal{V}_{\cod(n)})^{\op}, \langle \bigcup_{(n,a) \in \int \!\hirayo_l} q_{(n,a)}(J_{\cod(n)} \rangle \big{)}   \simeq  \big{(}(\mathcal{V}_C^{\op}, J_C \big{)}\]
so that they induce the same sheaf topos, which proves the desired limit decomposition of $\Spec(C)$.
\end{proof}

\begin{remark}
   In fact one could dispense form this tedious proof and rather announce that this result is an immediate consequence of \cref{biadjunction for set-models} where $\Spec$ is exhibited as a left adjoint, for which it has to send (filtered) colimits in $\mathbb{T}[S]$ to filtered colimits in $\mathbb{T}_J^{\Loc}\hy\GTop$, which can be proved to be computed in $ \mathbb{T}\hy\GTop$ and have as underlying topos a cofiltered bilimit topos (see \cite{osmond:tel-03609605}[Propositions 6.2.2.1 and 6.2.2.2]). However we think interesting to remain concrete and dispense to invoke those general statements about bicolimit of modelled topoi which belong to another work.
\end{remark}

\begin{remark}\label{sheafification is cartesian at etale}
Observe that the pseudocolimit decomposition of the etale generator induces first a bilimit decomposition of the corresponding presheaf topos:
 \[\widehat{\mathcal{V}^{\op}_C} \simeq \underset{(n,a) \in \int \! \hirayo_l}{\bilim} \widehat{\mathcal{V}^{\op}_B}/\katayo_n \]
 Since each $ \Spec(\cod(n)) $ for $ (n,a) \in \int \hirayo_l$ expresses as an etale geometric morphism obtained as a pullback as in \cref{etale as pb}, and pullback commutes with limits, we have 
\[\begin{tikzcd}
	{\underset{(n,m) \in \int \! \hirayo_l}{\bilim} \Spec(B)/\mathfrak{a}_J\katayo_n } & {\underset{(n,m) \in \int \! \hirayo_l}{\bilim} \widehat{\mathcal{V}^{\op}_B}/\katayo_n } \\
	{\Spec(B) } & {\widehat{\mathcal{V}^{\op}_B}}
	\arrow[from=1-1, to=1-2]
	\arrow[from=1-2, to=2-2]
	\arrow[from=1-1, to=2-1]
	\arrow[""{name=0, anchor=center, inner sep=0}, "{\iota_B}"', from=2-1, to=2-2]
	\arrow["\lrcorner"{anchor=center, pos=0.125}, draw=none, from=1-1, to=0]
\end{tikzcd}\]
so that we have a pullback in $\GTop$
\[\begin{tikzcd}
	{\Spec(C)} & {\widehat{\mathcal{V}_C^{\op}}} \\
	{\Spec(B)} & {\widehat{\mathcal{V}_B^{\op}}}
	\arrow["{\iota_C}", from=1-1, to=1-2, hook]
	\arrow[from=1-2, to=2-2]
	\arrow["{\Spec(n)}"', from=1-1, to=2-1]
	\arrow["{\iota_B}"', from=2-1, to=2-2, hook]
	\arrow["\lrcorner"{anchor=center, pos=0.125}, draw=none, from=1-1, to=2-2]
\end{tikzcd}\]
This means that the natural inclusion $ \Spec(-) \hookrightarrow \widehat{\mathcal{V}_{(-)}^{\op}}$ is cartesian at etale maps $ l \in \Et$. 
\end{remark}

\subsection{Spectrum and local data}

In particular we have the following, for local units are etale arrows (that are seldom finitely presented):

\begin{proposition}\label{points of the spectrum are local units}
Points of $\Spec(B)  $ correspond to local units $ x: B \rightarrow A$. 
\end{proposition}

\begin{proof}
First observe that a point of the spectrum, that is a $J_B$-continuous lex functor $ x $ in $ \Lex[ \mathcal{V}_B^{\op}, \mathcal{S}]$ is in particular an object of the ind-completion of ${\mathcal{V}_B}$, hence an object of $ B \downarrow \mathbb{T}[\mathcal{S}]$, so we can write $ x $ as an etale arrow $ x : B \rightarrow A$; now the condition of continuity says that for a covering $ (m_i^{\op} : n_i \rightarrow n )_ {i \in I} $ in $ \mathcal{V}_B^{\op}$ one has 
\[
\begin{tikzcd}
\displaystyle{\underset{i \in I}{\coprod}} \; x(n_i) \arrow[two heads]{rr}{\langle x(m_i) \rangle_{i \in I}} && x(n)
\end{tikzcd}
 \]
 But Yoneda tells us that
 \ \begin{align*}
     x(n) &\simeq \Ind(\mathcal{V}_B)[ \hirayo_n, x] \\
     &= \{ l : C \rightarrow A \mid \; ln=x \}
 \end{align*} and the surjectivity property above expresses the existence of the dashed arrow in the following diagram for some $i \in I$: 
 \[\begin{tikzcd}
 & B \arrow[bend right]{ldd}[swap]{n_i} \arrow[]{d}[description]{n} \arrow[bend left]{rdd}{x} & \\ & C \arrow[]{ld}[description]{m_i} \arrow[]{rd}[description]{l} & \\ C_i \arrow[dashed]{rr}[swap]{} & & A  
 \end{tikzcd}\]

We must prove that $A$ is a local object: let $ (n_i : K \rightarrow K_i)_{i \in I}$ be a $J$-cover and $a : K \rightarrow A$. Recall first that $x $, as an etale map, is the filtered colimit of the finitely presented etale maps under $B$ over it: $ x = \colim_{(n,l) \in V_\mathcal{B}\downarrow x}$, so is $A = \colim_{(n,l) \in V_\mathcal{B}\downarrow x} \cod(n)$. But then for $K$ is finitely presented, $a : K \rightarrow A$ lift through some $ l : C \rightarrow A$ as $ b : K \rightarrow C$ for some factorization of $x$ as above. Then we can consider the pushout family $ (b_*n_i)_{i \in I}$ which is a $ J_B$-cover of $n$ in $\mathcal{V}_B$: but then we see in the following diagram 
\[\begin{tikzcd}
	& C & B \\
	K && A \\
	& {b_*K_i} \\
	{K_i}
	\arrow["{n_i}", from=2-1, to=4-1]
	\arrow["b", from=2-1, to=1-2]
	\arrow["l"{description}, from=1-2, to=2-3]
	\arrow["n"', from=1-3, to=1-2]
	\arrow["x", from=1-3, to=2-3]
	\arrow["{b_*n_i}"{description, pos=0.7}, from=1-2, to=3-2]
	\arrow[from=4-1, to=3-2]
	\arrow["\lrcorner"{anchor=center, pos=0.125, rotate=180}, draw=none, from=3-2, to=2-1]
	\arrow["a"{description, pos=0.3}, from=2-1, to=2-3, crossing over]
	\arrow[dashed, from=3-2, to=2-3]
\end{tikzcd}\]
that there must be a dashed arrow as provided by $J_B$-injectiveness of $x$, which ensures existence of a lift for $A$. Hence $A$ is local and $x$ is a local unit. \end{proof}

\begin{remark}
More generally, geometric morphisms $ x: \mathcal{E} \rightarrow \Spec(F)$ will correspond to etale arrows $ x :\, !_{\mathcal{E}}^*B \rightarrow E$ in $\Et[\mathcal{E}]$ with $E$ in $\mathbb{T}_J[\mathcal{E}]$: however this correspondence can only be seen through the structure sheaf, which will be described later in this section. For now let us focus on the set-valued points.
\end{remark}

\begin{division}
At the level of points, an etale map of finite presentation $ n : B \rightarrow C $ produces a discrete fibration 
   \[ \pt(\Spec(C)) \simeq  \pt(\Spec(B)/\mathfrak{a}_{J_B}(\katayo_n)) \rightarrow \pt(\Spec(B)) \]
More generally, this remains true for an arbitrary etale map $ l : B \rightarrow C$; cartesian lifts are computed as follows: if $ x$ is a local unit under $B$ and $ x'$ a local unit under $C$, then a morphism $ x \rightarrow x'l$ is a square as below
\[\begin{tikzcd}
	B & C \\
	A & {A'}
	\arrow["l", from=1-1, to=1-2]
	\arrow["{x'}", from=1-2, to=2-2]
	\arrow["x"', from=1-1, to=2-1]
	\arrow["f"', from=2-1, to=2-2]
\end{tikzcd}\]
with $f$ an arbitrary arrow; but the etale-local factorization of this very map comes equiped with a diagonalization of the square below
\[\begin{tikzcd}
	B && C \\
	A & {A_f} & {A'}
	\arrow["l", from=1-1, to=1-3]
	\arrow["{x'}", from=1-3, to=2-3]
	\arrow["x"', from=1-1, to=2-1]
	\arrow["{n_f}"', from=2-1, to=2-2]
	\arrow["{u_f}"', from=2-2, to=2-3]
	\arrow["{x''}"{description}, from=1-3, to=2-2]
\end{tikzcd}\]
where the diagonal $x'' : C \rightarrow A_f$ is a local unit equiped with a lift $ x'' \rightarrow x'$ given by $u_f$. Moreover, as the functor of points $ \pt \simeq \Geom[\mathcal{S}, -]$ preserves pseudolimits, we have a pseudolimit of category 
\[ \pt(\Spec(C)) \simeq  \underset{(n,a) \in \int \! \hirayo_l}{\pslim} \pt(\Spec(\cod(n)) \]
\end{division} 
 
\begin{remark}
Observe that arbitrary etale maps under $B$ correspond to points of the presheaf topos $ \widehat{\mathcal{V}^{\op}_B}$ as $ \Lex[\mathcal{V}^{\op}_B, \mathcal{S}] \simeq\Ind(\mathcal{V}_B)$. 
\end{remark} 

\begin{division}As the term ``etale" was justified by the fact that finitely presented etale morphisms in $\mathcal{V}_B$ where sent to etale geometric morphisms by $\Spec$, the name of ``local" for objects is justified as follows. Recall that a geometric morphism $ f : \mathcal{F} \rightarrow \mathcal{E}$ is said to be \emph{local}\index{geometric morphism!local} if its inverse image part $f^*$ is full and faithful and moreover the direct image part $ f_*$ has a further right adjoint $ f^!$. In particular, a Grothendieck topos $\mathcal{E}$ is said to be \emph{local}\index{topos!local} if the global section functor $ \Gamma : \mathcal{E} \rightarrow \mathcal{S}$ has a further right adjoint -- the full-and-faithfulness condition being automatic in this context. For topoi are presentable, the adjoint functors tells us this amounts for $ \Gamma$ to preserving colimits, or equivalently, since $\Gamma = \mathcal{E}[1_{\mathcal{E}}, -]$, for $1_\mathcal{E}$ to be \emph{tiny}.  \\ 


Local topoi can be presented through a special class of sites, see \cite{elephant}[C3.6.3 (d)]. A site $(\mathcal{C},J)$ is said to be \emph{local} if it has a terminal object $1$ which is $J$-indecomposable (that is, the only covering sieve for $1$ is the maximal sieve). Then for such a site, $\Sh(\mathcal{C},J)$ is a local topos. 
\end{division}

\begin{division}
The prototypical example of local geometric morphism is the universal domain map $\partial_0 : \mathcal{E}^2 \rightarrow \mathcal{E}$ of a topos. Now for any point $ p : \mathcal{S} \rightarrow \mathcal{E}$, we can consider the \emph{Grothendieck Verdier localization at $p$}\index{Grothendieck-Verdier localization}, which is defined as the pseudopullback\[\begin{tikzcd}
	{\mathcal{E}_p} & {\mathcal{E}^2} \\
	{\mathcal{S}} & {\mathcal{E}}
	\arrow["{\partial_0}", from=1-2, to=2-2]
	\arrow["p"', from=2-1, to=2-2]
	\arrow["{p^*\partial_0}"', from=1-1, to=2-1]
	\arrow[from=1-1, to=1-2]
	\arrow["\lrcorner"{anchor=center, pos=0.125}, draw=none, from=1-1, to=2-2]
\end{tikzcd}\]
Its universal property is that for any Grotendieck topos $ \mathcal{F}$, we have an equivalence with the cocomma category
\[ \Geom[\mathcal{F}, \mathcal{E}_p] \simeq p \; !_\mathcal{F} \! \downarrow \Geom[\mathcal{F}, \mathcal{E}] \]
In particular, when $ \mathcal{F}$ is $\mathcal{S}$, we have an equivalence of categories
\[ \pt(\mathcal{E}_p) \simeq p\downarrow \pt(\mathcal{E})  \]
From \cite{localmap}[Theorem 3.7] we know that if $ \mathcal{E}$ has $ (\mathcal{C},J)$ as a lex site of definition, then $ \mathcal{E}_p$ can be expressed as a cofiltered pseudolimit of etale geometric morphisms 
\[ \mathcal{E}_p \simeq \underset{(C,a) \in \int p^*}{\bilim} \mathcal{E}/\mathfrak{a}_J\katayo_C  \]
where $ \int p^*$ is the cofiltered category of elements of the $J$-flat functor $ p^*: \mathcal{C} \rightarrow \mathcal{S}$. \end{division}

\begin{proposition}\label{Spec of local object is local}
Let $ A$ be a local object in $\mathbb{T}[\mathcal{S}]$: then $ \Spec(A)$ is a local topos. 
\end{proposition}

\begin{proof}
 We prove that the spectral site $ (\mathcal{V}^{\op}_A, J_A)$ is local: indeed, for $ A$ is a local object, for any $J$-cover $(n_i : A \rightarrow B_i)_{i \in I}$ the identity $ 1_A$ factorizes through some $ n_i$. Similarly, in $\mathcal{V}_A$, any $J_A$-cover $ (m_i : 1_A \rightarrow n_i)_{i \in I}$ of the initial object $ 1_A$ admits a retraction for some $i$ so that the corresponding $ m_i$ retracts on $1_A$. Hence the dual of $1_A$ is part of the sieve generated by the cover $ (m_i)_{i \in I}$: but as it is the initial object of $\mathcal{V}_A$, this means that any arrow is par of this generated sieve: hence any $J_A$-cover of $1_A$ generates the maximal sieve, which is the only covering sieve for $1_A$. Hence $(\mathcal{V}^{\op}_A, J_A)$ is a local site, and $\Spec(A)$ a local topos.
\end{proof}

We also have a partial converse of this proposition, allowing to test whether an object is local at the level of its spectrum -- though this requires strictly an additional assumption of subcanonicity. 

\begin{proposition}
    Suppose that for any object $B$, the spectral topology $ J_B$ is subcanonical. Then an object $ A$ in $\mathbb{T}[\mathcal{S}]$ is local if and only if $\Spec(A)$ is a local topos.
\end{proposition}

\begin{proof}
    The direct implication was done in the previous item. The converse relies on subcanonicity. If $ J_A$ is granted as subcanonical, then any corepresentable $ \katayo_n$ for $n $ in $\mathcal{V}_A$ is actually a sheaf for $J_A$. Now if $ A$ is such that $ \Spec(A)$ is local, then the terminal object $ 1_{\Spec(A)}$ is tiny, and moreover coincides with the representable $ \katayo_{1_A}$. As $ J_A$-covers are sent to colimits in $\Spec(A)$, any $J_A$-cover $(m_i : 1_A \rightarrow n_i)_{i \in I}$ yields a colimit decomposition $ \katayo_{1_A}  \simeq \colim_{i \in I} \katayo_{n_i}$ and this decomposition admits a section $  \katayo_{1_A} \rightarrow \katayo_{n_i}$ as $ \hirayo_{1_A}$ is tiny. But we supposed $J_A$ to be subcanonical, so that $ \mathcal{V}_A^{\op} \hookrightarrow \Spec(A)$ is full and faithful, so that this sections must come from a unique section $ r: n_i \rightarrow 1_A$ of $m_i$. This exactly says that $A$ lifts any of its extended $J_A$-covers, hence is a $J$-local object. 
\end{proof}

We also have this corollary from \cref{pro-etale spectrum}:

\begin{corollary}
For $ x : B \rightarrow A_x$ a local unit under $ B$, the geometric morphism $\Spec(x) : \Spec(A_x) \rightarrow \Spec(B)$ is the Grothendieck-Verdier localization of $\Spec(B)$ at the point $p_x : \mathcal{S} \rightarrow \Spec(B)$. 
\end{corollary}

To conclude this section, we characterize the class of geometric morphisms the spectrum sends local maps to. Let us first introduce the following notion, which is also related to \cite{caramello2020denseness}[Section 4.7]:

\begin{definition}
A geometric morphism $ f : \mathcal{E } \rightarrow \mathcal{F}$ is said to be \emph{terminally connected}\index{terminally connected! geometric morphism} if the inverse image part $ f^*$ lifts global elements, that is, if for any $ F$ in $\mathcal{F}$ one has 
\[  \mathcal{E}[1,f^*(F)] \simeq \mathcal{F}[1, F]  \]
\end{definition}

\begin{lemma}\label{testing terminal connectedness with local sites}
Let $ f: (\mathcal{C}, J) \rightarrow (\mathcal{D},K)$ be a morphism of sites between small lex sites with $ (\mathcal{D}, K)$ a local site. If $f$ lifts global elements then $\Sh(f) : \Sh(\mathcal{D},J) \rightarrow \Sh(\mathcal{C},J)$ is terminally connected. Moreover, if $ (\mathcal{C},J)$ and $ (\mathcal{D},K)$ both are subcanonical, the converse also holds.  
\end{lemma}

\begin{proof}
Suppose that $f$ lifts global elements, so that we have a natural isomorphism $ \mathcal{D}[1_\mathcal{D}, f] \simeq \mathcal{C}[1_\mathcal{C}, -]$. We prove this suffices to lift also global elements of $\Sh(f)^*$ that arise in $\Sh(\mathcal{D},K)$. Let be $a : 1_{\Sh(\mathcal{D},K)} \rightarrow \Sh(f)^*E$ with $E$ a object of $\Sh(\mathcal{C},J)$. Since $ \mathcal{C}$ generates $\mathcal{E}$ under colimits and $ 1_{\Sh(\mathcal{D},K)}$ is tiny by localness of ${\Sh(\mathcal{D},K)}$, there exists a factorization of the form
\[\begin{tikzcd}
	{1_{\Sh(\mathcal{D},K)}} & {\Sh(f)^*E} \\
	& {\Sh(f)^*\mathfrak{a}_J\hirayo_C}
	\arrow["a", from=1-1, to=1-2]
	\arrow["b"', from=1-1, to=2-2]
	\arrow["{\Sh(f)^*(x)}"', from=2-2, to=1-2]
\end{tikzcd}\]
which is moreover essentially unique in the sense that for any two such lifts are identified with $a$ in the colimit $ \colim_{(C,a) \in \mathfrak{a}_J\hirayo\downarrow E} \Sh(\mathcal{D},K)[1_{\Sh(\mathcal{D},K)}, \Sh(f)^*\hirayo_C] $.

As $ \Sh(f)^*$ restricts as $f$ from $\mathcal{C}$ to $\mathcal{D}$, one has $ \Sh(f)^*\mathfrak{a}_J\hirayo_C \simeq \mathfrak{a}_K \hirayo_{f(C)}$ 
It suffices then to prove one can lifts uniquely global elements of the form $b: 1_{\Sh(\mathcal{D},K)} \rightarrow \mathfrak{a}_K \hirayo_{f(C)} $. Beware that such maps do not necessarily comes from $ \mathcal{D}$ for the topology $ K$ is not supposed to be subcanonical so $\mathfrak{a}_K\hirayo$ may not be fully faithful. However, since $K$ covers are sent to canonical cover in $\Sh(\mathcal{D},K)$, and the terminal object comes from the terminal object of the site $ 1_{\Sh(\mathcal{D},K)} \simeq \mathfrak{a}_K\hirayo_{1_\mathcal{D}}$, we know that any such global element $b$ is induced through the universal property of colimits by the image of a $K$-cover $ (u_i : D_i \rightarrow 1_{\mathcal{D}})_{i \in I}$ together with a cocone $ (v_i : D_i \rightarrow f(C))_{i \in I}$ such that for any $i$ in $I$ one has
\[\begin{tikzcd}
	{\mathfrak{a}_K\hirayo_{D_i}} \\
	{1_{\Sh(\mathcal{D},K)}} & {\mathfrak{a}_K\hirayo_{f(C)}}
	\arrow["b"', from=2-1, to=2-2]
	\arrow["{\mathfrak{a}_K\hirayo_{u_i}}"', from=1-1, to=2-1]
	\arrow["{\mathfrak{a}_K\hirayo_{v_i}}", from=1-1, to=2-2]
\end{tikzcd}\]
But by localness of the site (equivalently, by localness of $\Sh(\mathcal{D},K)$), the identity of $1_{\mathcal{D}} $ factorizes through some $ u_i$ as $ 1_{\mathcal{D}} = u_iw $ in $\mathcal{D}$. This factorization is essentially unique in the sense that for any other factorization through $ 1_{\mathcal{D}}= u_{i'}w'$ then $ w,w'$ factorize through the pullback $D_i \times_D D_{i'} $ and are hence identified in the corresponding colimit in $\Sh(\mathcal{D},K)$. Hence we get an element $ v_i w : 1_{\mathcal{D}} \rightarrow f(C)$ in $\mathcal{D}$, which comes uniquely from an element $ \overline{c}:  1_{\mathcal{C}} \rightarrow C$ in $\mathcal{C}$. For $ w $ is a section of $ u_i$, we have that $f(c) = b$.  \\

In the case where both sites are subcanonical, that is, if both $ \mathfrak{a}_J\hirayo$ and $ \mathfrak{a}_K\hirayo$ are fully faithful, assuming $\Sh(f)$ to be terminally connected yields the following isomorphism in each $C$ of $\mathcal{C}$:
\begin{align*}
    \mathcal{D}[1_\mathcal{D}, f(C)] &\simeq \Sh(\mathcal{D},K)[ 1_{ \Sh(\mathcal{D},K)}, \Sh(f)^*(\hirayo_C)] \\
    &\simeq \Sh(\mathcal{C},J)[ 1_{ \Sh(\mathcal{C},J)},\hirayo_C]\\
    &\simeq \mathcal{C}[1_\mathcal{C}, C]
\end{align*}
\end{proof}

\begin{proposition}\label{spectrum of local map is terminally connected}
Let $u : A_1 \rightarrow A_2$ be a local map between local objects. Then $ \Spec(u) : \Spec(A_2) \rightarrow \Spec(A_1)$ is terminally connected. 
\end{proposition}

\begin{proof}
Recall that $\Spec(u)^*$ restricts as a site morphism $ u_* : (\mathcal{V}^{\op}_{A_1}, J_{A_1}) \rightarrow (\mathcal{V}^{\op}_{A_2}, J_{A_2})$ given by the pushout functor along $u$. From the proof of \cref{Spec of local object is local}, we know in particular that $(\mathcal{V}^{\op}_{A_2}, J_{A_2})$ is a local site. Hence from \cref{testing terminal connectedness with local sites}, it suffices to prove $u_*$ to lift global elements at the level of the spectral sites. Now, suppose one has, for some etale map $ n : A_1 \rightarrow \cod(n)$ in $ \mathcal{V}_{A_1}$, a global element of $ u_*n$, which is the same as a retraction as below
\[\begin{tikzcd}
	{A_1} & {A_2} \\
	{\cod(n)} & {\cod(u_*n)} & {A_2}
	\arrow[Rightarrow, no head, from=1-2, to=2-3]
	\arrow["u", from=1-1, to=1-2]
	\arrow["{u_*n}"{description}, from=1-2, to=2-2]
	\arrow["r"', from=2-2, to=2-3]
	\arrow["n"', from=1-1, to=2-1]
	\arrow["{n_*u}"', from=2-1, to=2-2]
	\arrow["\lrcorner"{anchor=center, pos=0.125, rotate=180}, draw=none, from=2-2, to=1-1]
\end{tikzcd}\]
so that we have a factorization $ u = r n_*u n $. But then, for $u$ is local and $n$ is etale, we know by general properties of orthogonal factorization systems that $n$ must retract on $ 1_{A_1}$ for the etale part of $u$ is invertible, which produces a retraction $ \overline{r}$ as below
\[\begin{tikzcd}
	{A_1} && {A_2} \\
	& {A_1} && {A_2} \\
	{\cod(n)} && {\cod(u_*n)}
	\arrow[Rightarrow, no head, from=1-3, to=2-4]
	\arrow["u", from=1-1, to=1-3]
	\arrow["r"', from=3-3, to=2-4]
	\arrow["n"', from=1-1, to=3-1]
	\arrow["{n_*u}"', from=3-1, to=3-3]
	\arrow["\lrcorner"{anchor=center, pos=0.125, rotate=180}, draw=none, from=3-3, to=1-1]
	\arrow[Rightarrow, no head, from=1-1, to=2-2]
	\arrow["u"{description, pos=0.2}, from=2-2, to=2-4]
	\arrow["{u_*n}"{description, pos=0.3}, crossing over, from=1-3, to=3-3]
	\arrow["{\overline{r}}"{description}, dashed, from=3-1, to=2-2]
\end{tikzcd}\]
and we have by cancellation of pushout that $ u_*(\overline{r}) = r$; moreover such an $\overline{r}$ has to be unique: for any other choice $r'$ of retraction with $ u_*(r') = r$, one would have 
$ u (r')= r n_*u = u \overline{r}$ and $ r'n = \overline{r}n$, so that $r'$ and $\overline{r}$ would be simultaneously equalized by $n$ and coequalized by $u$: but it is a general property of factorization systems that any two arrows that are simultaneously equalized by a left map and coequalized by a right map are actually equal, which enforces that $ r'= \overline{r}$. Hence $ u_*$ lifts global elements, and by the previous lemma, this prove $ \Spec(u)$ is terminally connected. 
\end{proof}

\begin{remark}\label{Terminal geometry}
This proves that, for any geometry $ (\mathbb{T}, \mathcal{V},J)$, the functor sending a set-valued $\mathbb{T}$-model in $ \mathbb{T}[\mathcal{S}]$ to the underlying topos of its spectrum (which we will abusively denote $ \Spec$ in this remark) restricts as below
\[\begin{tikzcd}
	{\mathbb{T}_J[\mathcal{S}]^{\Loc}} & {\mathbb{T}[\mathcal{S}]} \\
	(\textbf{LTop}^{\textbf{TCo}})^{\op} & \GTop^{\op}
	\arrow["{\iota_{\mathcal{V},J}}", hook, from=1-1, to=1-2]
	\arrow[dashed, from=1-1, to=2-1]
	\arrow["\Spec", from=1-2, to=2-2]
	\arrow[hook, from=2-1, to=2-2]
\end{tikzcd}\]
where $\textbf{LTop}^{\textbf{TCo}}$ denotes the bicategory of local Grothendieck topoi and terminally connected geometric morphisms between them. Moreover, one can show that terminally connected geometric morphisms are left orthogonal in a 2-categorical sense to etale geometric morphisms: this suggests that a kind of ``universal" 2-dimensional notion of geometry on the bicategory of Grothendieck topoi is involved in any instance of geometry. This will be investigated in a future work.
\end{remark}

\subsection{Gros versus petit spectrum}

In this subsection we describe standard sites associated to the different generic classifiers constructed in the previous sections. The techniques used here allow to construct such classifiers of arrows, arrows toward local objects, etale arrows and local forms, but \emph{before} fixing their domain; the specification of the domain is achieved as a variation of this construction and begets the spectrum as described in the previous section. Here the sites are constructed from the arrow category, the etale generator and convenient topologies on them derived from the Grothendieck pretopology provided by a choice of geometry.

\begin{division}
Before embarking into those constructions, let us give a word on the ambivalence of some limits in $\Lex$. It was observed first in \cite{Cole} that in $\Lex$ pseudopowers coincide with bitensors, so that in fact 
\[  (\mathbb{T}[\mathcal{S}]_{\omega}^{\op})^2 \simeq \mathbb{T}[\mathcal{S}]_{\omega}^{\op} \otimes 2 \]
This oddity comes from the fact that the power projections come equiped with retractions, provided by the identity arrow, which behave as tensor inclusions. Though we will not make use of it, let us also recall here that similarly, pseudoproducts in $\Lex$ are also bicoproducts, and that the pseudoterminal object $1$ is also bi-initial. Those properties generalize similar statements about the category of $\wedge$-semilattices. 
\end{division}

\begin{division}
Recall that $ (\mathbb{T}[\mathcal{S}]_{\omega}^{\op})^2$ is the syntactic site of the theory of morphisms of $ \mathbb{T}$-models. We have indeed that 
\begin{align*}
    \pt(\widehat{(\mathbb{T}[\mathcal{S}]_{\omega}^2)^{\op} }) &\simeq \Lex[(\mathbb{T}[\mathcal{S}]_{\omega}^2)^{\op}, \mathcal{S}] \\
    &\simeq \Ind(\mathbb{T}[\mathcal{S}]_{\omega}^2) \\
    &\simeq \mathbb{T}[\mathcal{S}]^2
\end{align*}

From the previous remark the equivalence above generalizes to arbitrary Grothendieck topoi as 
\begin{align*}
   \GTop[\mathcal{E},\widehat{(\mathbb{T}[\mathcal{S}]_{\omega}^{\op})^2] } &\simeq \Lex[(\mathbb{T}[\mathcal{S}]_{\omega}^{\op})^2, \mathcal{E}] \\
   &\simeq \Lex[\mathbb{T}[\mathcal{S}]_{\omega}^{\op} \otimes 2 , \mathcal{E}] \\
    &\simeq \mathbb{T}[\mathcal{E}]^2
\end{align*}
This exhibits $ \widehat{(\mathbb{T}[\mathcal{S}]_{\omega}^{\op})^2} $ as the bipower $ \mathcal{S}[\mathbb{T}]^2 $ in $ \GTop$, that is, as the classifier of morphisms between $ \mathbb{T}$-models.
\end{division}

\begin{division}\label{classifier of etale maps and factorization}
Similarly, from the factorization system $ (\Et, \Loc)$ was left generated from $ \mathcal{V}$, we have
\begin{align*}
    \Et &= \Ind(\mathcal{V}) \\ 
    &\simeq \Lex[\mathcal{V}^{\op}, \mathcal{S}] \\
    &\simeq \pt( \widehat{\mathcal{V}^{\op}})
\end{align*} 
so that $ \widehat{\mathcal{V}^{\op}}$ is the classifier of etale maps $ \mathcal{S}[\Et]$. \\

Moreover, the factorization can be constructed as follows from a site theoretic point of view. As $ \mathcal{V}$ is closed in $ \mathbb{T}[\mathcal{S}]_{\omega}^2$ under finite colimits, the inclusion $ \mathcal{V}^{\op} \hookrightarrow (\mathbb{T}[\mathcal{S}]_{\omega}^2)^{\op}$ is lex. Hence for a lex functor $ \ulcorner n \urcorner : \mathcal{V}^{\op} \rightarrow \mathcal{E} $ coding for an etale map in $ \mathbb{T}[\mathcal{E}]$, the left Kan extension 
\[\begin{tikzcd}
	{\mathcal{V}^{\op}} & {\mathcal{E}} \\
	{(\mathbb{T}[\mathcal{S}]_{\omega}^2)^{\op}}
	\arrow["{\iota_{\mathcal{V}}^{\op}}"', hook, from=1-1, to=2-1]
	\arrow["{\ulcorner n \urcorner}", from=1-1, to=1-2]
	\arrow[""{name=0, anchor=center, inner sep=0}, "{\lan_{\iota_{\mathcal{V}}^{\op}} \ulcorner n \urcorner}"', from=2-1, to=1-2]
	\arrow["\simeq"{description}, Rightarrow, draw=none, from=1-1, to=0]
\end{tikzcd}\]
is the name of $ \iota_{\Et} (n) $, that is, of the image of $n $ along the inclusion $ \Et \hookrightarrow \mathbb{T}[\mathcal{S}]^2$. On the other hand, for a lex functor $ f :(\mathbb{T}[\mathcal{S}]_{\omega}^2)^{\op} \rightarrow \mathcal{E}  $ coding for an arbitrary map, then its restriction along $ \iota_\mathcal{V}^{\op}$ codes for the etale part of its factorization, that is \[ \ulcorner f \urcorner\iota_{\mathcal{V}}^{\op} = \ulcorner n_f \urcorner  \]
while the counit of $\ulcorner f \urcorner$ 
\[\begin{tikzcd}
	{\mathcal{V}^{\op}} && {\mathcal{E}} \\
	{(\mathbb{T}[\mathcal{S}]_{\omega}^2)^{\op}}
	\arrow["{\iota_{\mathcal{V}}^{\op}}"', hook, from=1-1, to=2-1]
	\arrow["{\ulcorner f \urcorner\iota_{\mathcal{V}}^{\op}}", from=1-1, to=1-3]
	\arrow[""{name=0, anchor=center, inner sep=0}, "{ \ulcorner f \urcorner}"{description}, from=2-1, to=1-3]
	\arrow[""{name=1, anchor=center, inner sep=0}, "{\lan_{\iota_{\mathcal{V}}^{\op}} \ulcorner f \urcorner\iota_{\mathcal{V}}^{\op}}"', bend right=30, from=2-1, to=1-3, end anchor=260, start anchor=0]
	\arrow["{\epsilon_{\ulcorner f \urcorner}}", shorten <=2pt, shorten >=2pt, Rightarrow, from=1, to=0]
\end{tikzcd}\]
is the name of the local part, that is
\[ \epsilon_{\ulcorner f \urcorner} =  u_f  \]
which is obtained by whiskering $ \epsilon_{\ulcorner f \urcorner} $ with the morphism of site provided by the identity functor 
\[\begin{tikzcd}
	{(\mathbb{T}[\mathcal{S}]_{\omega}^{\op}, J)} & {((\mathbb{T}[\mathcal{S}]_{\omega}^2)^{\op}, J^2)}
	\arrow["{\textup{id}}", from=1-1, to=1-2]
\end{tikzcd}\]
\end{division}

\begin{remark}
Observe that the full faithfulness of $ \iota_{\mathcal{V}}^{\op}$ ensures that $ \ulcorner n \urcorner \simeq \lan_{\iota_{\mathcal{V}}^{\op}} \ulcorner n \urcorner \iota_ {\mathcal{V}}^{\op} $, that is, that an etale map is its own etale part. 
\end{remark}


\begin{division}
Now, we can equip  $ (\mathbb{T}[\mathcal{S}]_{\omega}^{\op})^2 $ with the following pretopology $ J^2$: for an arrow $ k : K \rightarrow K'$, a covering family for $ J^2(k)$ is a family of squares 
\[\begin{tikzcd}
	K & {K_i} \\
	{K'} & {K'_i}
	\arrow["k"', from=1-1, to=2-1]
	\arrow["{u_i}", from=1-1, to=1-2]
	\arrow["{k_i}", from=1-2, to=2-2]
	\arrow["{v_i}"', from=2-1, to=2-2]
\end{tikzcd}\]
such that $ (n_i : K \rightarrow K_i)_{i \in I}$ is in $J(K)$. Then observe that $ (n'_i : K' \rightarrow K'_i)_{i \in I}$ is automatically covering by the stability axiom of pretopologies. In particular one can choose the squares in a family in $J^2$ to be the pushouts squares along $k$ of a covering family $ (n_i : K \rightarrow K_i)_{i \in I}$ in $J(K)$.\\

Then the adjunction 
\[\begin{tikzcd}
	{\mathbb{T}[\mathcal{S}]_{\omega}^2} && {\mathbb{T}[\mathcal{S}]_{\omega}}
	\arrow[""{name=0, anchor=center, inner sep=0}, "\cod", curve={height=-12pt}, from=1-1, to=1-3]
	\arrow[""{name=1, anchor=center, inner sep=0}, "{\textup{id}}", curve={height=-12pt}, from=1-3, to=1-1]
	\arrow["\dashv"{anchor=center, rotate=-90}, draw=none, from=1, to=0]
\end{tikzcd}\]
defines an adjunction between morphisms of sites
\[\begin{tikzcd}
	{((\mathbb{T}[\mathcal{S}]_{\omega}^2)^{\op}, J^2)} && {((\mathbb{T}[\mathcal{S}]_{\omega})^{\op},J)}
	\arrow[""{name=0, anchor=center, inner sep=0}, "\cod"', bend right=20, from=1-1, to=1-3, start anchor = 340]
	\arrow[""{name=1, anchor=center, inner sep=0}, "{\textup{id}}"', hook, bend right=20, from=1-3, to=1-1, end anchor = 20]
	\arrow["\dashv"{anchor=center, rotate=-90}, draw=none, from=1, to=0]
\end{tikzcd}\]
as $ \textup{id}$ sends $ J$-covers to $ J^2$-covers. 
\end{division}

\begin{proposition}
The topos $ \Sh((\mathbb{T}[\mathcal{S}]_{\omega}^{\op})^2 , J^2) $ classifies morphisms between $\mathbb{T}_J$-models.
\end{proposition}

\begin{division}
Now we can also construct an intermediate topology in such a way that the associated sheaf topos classifies maps whose codomain is local -- and the domain arbitrary. Define the pretopology $ J_{\cod}$ on $ (\mathbb{T}[\mathcal{S}]_{\omega}^2)^{\op}$ whose families consist of those squares as above, but where only the codomain family $ (n'_i : K' \rightarrow K'_i )_{i \in I}$ are supposed to be $J$-covering. \\

Then in the adjunction 
\[\begin{tikzcd}
	{\mathbb{T}[\mathcal{S}]_{\omega}^2} && {\mathbb{T}[\mathcal{S}]_{\omega}}
	\arrow[""{name=0, anchor=center, inner sep=0}, "\cod"', curve={height=12pt}, from=1-1, to=1-3]
	\arrow[""{name=1, anchor=center, inner sep=0}, "{!_{(-)}}"', hook, curve={height=12pt}, from=1-3, to=1-1]
	\arrow["\dashv"{anchor=center, rotate=-90}, draw=none, from=1, to=0]
\end{tikzcd}\]
where $ !_{(-)}$ sends $ K$ to the initial map $ !_K : 0 \rightarrow K$ -- as 0 is always finitely presented -- defines an adjunction of morphisms of sites 
\[\begin{tikzcd}
	{((\mathbb{T}[\mathcal{S}]_{\omega}^2)^{\op}, J_{\cod})} && {(\mathbb{T}[\mathcal{S}]_{\omega}^{\op},J)}
	\arrow[""{name=0, anchor=center, inner sep=0}, "{\cod^{\op}}", bend left=20, from=1-1, to=1-3, start anchor=20]
	\arrow[""{name=1, anchor=center, inner sep=0}, hook', "{!_{(-)}^{\op}}", bend left=20, from=1-3, to=1-1, end anchor=340]
	\arrow["\dashv"{anchor=center, rotate=-90}, draw=none, from=1, to=0]
\end{tikzcd}\]
On one hand, the codomain functor sends $J_{\cod}$ families to $J$ families by construction; on the other hand, if $ (n_i : K \rightarrow K_i)_{i \in I}$ is in $ J(K)$, then the family consisting of those squares
\[\begin{tikzcd}
	0 & 0 \\
	K & {K_i}
	\arrow["{!_K}"', from=1-1, to=2-1]
	\arrow[Rightarrow, no head, from=1-1, to=1-2]
	\arrow["{!_{K_i}}", from=1-2, to=2-2]
	\arrow["{n_i}"', from=2-1, to=2-2]
\end{tikzcd}\]
is in $J_{\cod}(!_{K})$. 
\end{division}

\begin{proposition}
The sheaf topos $ \Sh((\mathbb{T}[\mathcal{S}]_{\omega}^2)^{\op}, J_{\cod})$ classifies morphisms toward $ \mathbb{T}_J$-models.
\end{proposition}

\begin{proof}
Let $ \ulcorner f \urcorner : (\mathbb{T}[\mathcal{S}]_{\omega}^2)^{\op} \rightarrow \mathcal{E}$ be the name of a morphism $ f : B \rightarrow A$ of $ \mathbb{T}$-models in $\mathcal{E}$. Then by Yoneda lemma, requiring it to be $J_{\cod}$-continuous amounts to saying that for any $ k : K \rightarrow K'$ and any family in $J_{\cod}(k)$, any square $ (a,b) : k \rightarrow f$ factorizes as below for some $i$
\[\begin{tikzcd}[row sep=small]
	K && B \\
	& {K_i} \\
	{K'} && A \\
	& {K_i'}
	\arrow["a", from=1-1, to=1-3]
	\arrow["k"', from=1-1, to=3-1]
	\arrow["b"'{pos=0.4, description}, from=3-1, to=3-3]
	\arrow["f", from=1-3, to=3-3]
	\arrow["{n_i}"{description}, from=1-1, to=2-2]
	\arrow["{n'_i}"', from=3-1, to=4-2]
	\arrow[dashed, from=2-2, to=1-3]
	\arrow[dashed, from=4-2, to=3-3]
	\arrow["{k_i}"{description, pos=0.3}, from=2-2, to=4-2, crossing over]
\end{tikzcd}\]
In particular, for any $ (n_i : K \rightarrow K_i)_{i \in I}$, we have a factorization 
\[\begin{tikzcd}[row sep=small]
	0 && B \\
	& 0 \\
	K && A \\
	& {K_i}
	\arrow["{!_B}", from=1-1, to=1-3]
	\arrow["k"', from=1-1, to=3-1]
	\arrow["b"'{pos=0.4, description}, from=3-1, to=3-3]
	\arrow["f", from=1-3, to=3-3]
	\arrow[Rightarrow, no head, from=1-1, to=2-2]
	\arrow["{!_B}"{description}, from=2-2, to=1-3]
	\arrow[dashed, from=4-2, to=3-3]
	\arrow["{k_i}"{description, pos=0.3}, from=2-2, to=4-2, crossing over]
	\arrow["{n_i}"', from=3-1, to=4-2]
\end{tikzcd}\]
so that $A$ is actually $J$-local. 
\end{proof}

\begin{division}
Now we would need to present the topos classifying local forms, that are, etale maps toward local objects. We saw that $ \mathcal{V}^{\op}$ was a site for the classifier of etale maps: then we claim it suffices to restrict the codomain topology to get the classifier of etale forms. As $ J$ is generated in $ \mathcal{V}$, we can consider in $ \mathcal{V}^{\op}$ the restricted pretopology $J_{\cod}\mid_{\mathcal{V}^{\op}} $ consisting of squares of finitely presented maps:
\[\begin{tikzcd}
	K & {K_i} \\
	{K'} & {K_i'}
	\arrow["m"', from=1-1, to=2-1]
	\arrow["{n_i}", from=1-1, to=1-2]
	\arrow["{m_i}", from=1-2, to=2-2]
	\arrow["{n_i'}"', from=2-1, to=2-2]
\end{tikzcd}\]
such that the bottom families $ (n'_i : K' \rightarrow K'_i)_{i \in I}$ are in $J(K')$. Beware that $ \mathcal{V}$ is a full subcategory of $\mathbb{T}[\mathcal{S}]^2_\omega$ so that it also contain squares with non-etale top and bottom arrows, but the generation condition makes we actually consider squares entirely made of $\mathcal{V}$ maps for the pretopology. \\

Then the inclusion of the basic etale maps defines a morphism of site 
\[\begin{tikzcd}
	{(\mathcal{V}^{\op}, J_{\cod}\mid_{\mathcal{V}^{\op}})} & {((\mathbb{T}[\mathcal{S}]_{\omega}^2 )^{\op}, J_{\cod})}
	\arrow["{\iota_\mathcal{V}^{\op}}", from=1-1, to=1-2, hook]
\end{tikzcd}\]
\end{division}

\begin{proposition}\label{local morphisms between classifiers}
The sheaf topos $ \Sh(\mathcal{V}^{\op}, J_{\cod}\mid_{\mathcal{V}^{\op}})$ is the classifier of local forms. Moreover we have a local geometric morphism 
\[\begin{tikzcd}
	{\Sh(\mathcal{V}^{\op}, J_{\cod}\mid_{\mathcal{V}^{\op}}) } & {\Sh((\mathbb{T}[\mathcal{S}]_{\omega}^2)^{\op}, J_{\cod})}
	\arrow["{}", from=1-1, to=1-2]
\end{tikzcd}\]
whose center is the geometric morphism $\Sh(\iota_\mathcal{V}^{\op})$. 
\end{proposition}

\begin{proof}
This is a combination of the previous remarks. A $J_{\cod}\mid_{\mathcal{V}^{\op}} $-continuous $ \ulcorner n \urcorner : \mathcal{V}^{\op} \rightarrow \mathcal{E}$ is the name of an etale map, and from \cref{classifier of etale maps and factorization}, we can consider its left Kan extension $ \lan_{\iota_\mathcal{V}^{\op}} \ulcorner n \urcorner$, which is $ J_{\cod}$-continuous as $ \iota_{\mathcal{V}}$ is a morphism of site: hence its codomain is a local object, so that $ \ulcorner n \urcorner$ is the name of a local form.\\

Moreover, the morphism of site $ \iota_{\mathcal{V}^{\op}}$ happens to be also a comorphism of site: indeed, for any $ n : K \rightarrow K'$ in $\mathcal{V}$ and $ ((u_i, n_i') : n \rightarrow k_i)_{i \in I}$ a cover in $ J_{\cod}(\iota_{\mathcal{V}}^{\op}(n)$ -- with, beware, $u_i$ and $ k_i$ arbitrary finitely presented maps -- we have that $ (n'_i)_{i \in I}$ is a family of etale maps so each composite $ n_i'n$ is an etale map and hence the family consisting of the squares
\[\begin{tikzcd}
	K & K \\
	{K'} & {K'_i}
	\arrow["n"', from=1-1, to=2-1]
	\arrow[Rightarrow, no head, from=1-1, to=1-2]
	\arrow["{n'_in}", from=1-2, to=2-2]
	\arrow["{n'_i}"', from=2-1, to=2-2]
\end{tikzcd}\]
is a $J_{\cod}\mid_{\mathcal{V}^{\op}}$-cover in $\mathcal{V}$, and we have a factorization at each $i$
\[\begin{tikzcd}
	&& {K_i} \\
	K & K & {K'_i} \\
	{K'} & {K'_i}
	\arrow["n"', from=2-1, to=3-1]
	\arrow[Rightarrow, no head, from=2-1, to=2-2]
	\arrow["{n'_in}"{description}, from=2-2, to=3-2]
	\arrow["{n'_i}"', from=3-1, to=3-2]
	\arrow["{u_i}", curve={height=-12pt}, from=2-1, to=1-3]
	\arrow["{u_i}"{description}, from=2-2, to=1-3]
	\arrow["{k_i}", from=1-3, to=2-3]
	\arrow[Rightarrow, no head, from=3-2, to=2-3]
\end{tikzcd}\]
ensuring that the image of this cover along $ \iota_{\mathcal{V}}^{\op}$ refines $ (u_i,n'_i)_{i \in I}$: hence $\iota_{\mathcal{V}}^{\op}$ has the cover lifting property, and is a comorphism of site. Hence, being also continuous, it defines from \cite{caramello2020denseness}[Theorem 7.20, (iii)] a local geometric morphism consisting of a triple of adjoints
\[\begin{tikzcd}
	{\Sh(\mathcal{V}^{\op}, J_{\cod}\mid_{\mathcal{V}^{\op}}) } && {\Sh((\mathbb{T}[\mathcal{S}]_{\omega}^2)^{\op}, J_{\cod})}
	\arrow["{\Sh(\iota_\mathcal{V}^{\op})^*}", curve={height=-25pt}, from=1-1, to=1-3]
	\arrow["{\Sh(\iota_\mathcal{V}^{\op})^!}"', curve={height=25pt}, from=1-1, to=1-3]
	\arrow["{\Sh(\iota_\mathcal{V}^{\op})_*}"{description}, from=1-3, to=1-1]
\end{tikzcd}\]
\end{proof}

\begin{remark}
Observe that admissibility can be detected at this level of description by combining the previous lemma with \cref{classifier of etale maps and factorization}. Indeed, suppose that $ \ulcorner f \urcorner : (\mathbb{T}[\mathcal{S}]_{\omega}^2)^{\op} \rightarrow \mathcal{E} $ is $ J_\cod$ continuous. Then its precomposition along the morphism of site $ \iota_\mathcal{V}^{\op}$, which codes for its etale maps, is also $ J_{\cod}\mid_{\mathcal{V}^{\op}}$-continuous: this means that the etale part of the factorization of a morphism with local codomain still has a local codomain.
\end{remark}

\begin{division}We now specify how those data specialize when we fix the domain of morphisms we want to classify: we recover the spectrum together with a coarser object classifying more generally arbitrary maps under the fixed object toward local objects.

In a similar manner of what was established for the category of etale maps under a fixed object, it is proven in \cite{osmond2021coslices}[theorem 2.12] that the category $ B\downarrow \mathcal{B}$ has as generator of finitely presented objects the class $\mathcal{G}_B$ of all maps under $B$ induced as pushouts 
\[\begin{tikzcd}
	K & {K'} \\
	B & C
	\arrow["a"', from=1-1, to=2-1]
	\arrow["l"', from=2-1, to=2-2]
	\arrow["k", from=1-1, to=1-2]
	\arrow[from=1-2, to=2-2]
\end{tikzcd}\]
Here we do not require the maps $k$ to be in any specific class beside being finitely presented. 
Then as like as the etale generator under $B$ is closed under finite colimits and generates left maps under $B$, the class $\mathcal{G}_B$ is closed under finite colimits in $B \downarrow \mathcal{B}$ and generates the coslice under $B$ as its ind-completion
\begin{align*}
    B \downarrow \mathcal{B} &\simeq \Ind(\mathcal{G}_B) \\
&\simeq \Lex[\mathcal{G}_B^{\op}, \mathcal{S}]
\end{align*}


There is a full inclusion $ \mathcal{V}_B \rightarrow \mathcal{G}_B$, where the fullness comes from right cancellation of etale maps. The spectral topology $J_B$ that was induced from $J$ on the etale generator actually extends to the coslice generator $ \mathcal{G}_B^{\op}$. We hence have an inclusion of sites 
\[\begin{tikzcd}
	{(\mathcal{V}_B^{\op}, J_B)} & {(\mathcal{G}_B^{\op}, J_B)}
	\arrow["{\iota_B}", hook, from=1-1, to=1-2]
\end{tikzcd}\]
which is obviously continous: but it has a stronger property induced from the axioms of a geometry, whose consequence is the following fact (already observed in \cite{Anel}[Proposition 39], also related to \cite{sga4}[4.10.6] and \cite{dubuc2000axiomatic}):
\end{division}
\begin{theorem}
 For any $\mathbb{T}$-model there is a local geometric morphism over the spectrum of $B$
\[\begin{tikzcd}
	{\Sh(\mathcal{G}_B^{\op}, J_B)} & {\Spec(B)}
	\arrow["{\iota_B}", from=1-1, to=1-2]
\end{tikzcd}\]
\end{theorem}

\begin{proof}
    The idea is very close to \cref{local morphisms between classifiers}: the idea is that $ \iota_B$ is not a mere morphisms of site but simultaneously a comorphism of site, and when working in coslices, this becomes almost automatic: indeed, if one take a $J_B$-cover of the image of a finitely presented etale map $ \iota_B(n)$ in $\mathcal{G}_B$, as the $J_B$-covering families in $\mathcal{G}_B$ still are induced as pushout of $J$-covers and hence are constituted of etale maps which are closed under composition: then each of the triangles 
\[\begin{tikzcd}
	B \\
	{B'} & {B_i}
	\arrow["{\iota_B(n)}"', from=1-1, to=2-1]
	\arrow["{m_i}"', from=2-1, to=2-2]
	\arrow["{l_i}", from=1-1, to=2-2]
\end{tikzcd}\]
actually lies in $\mathcal{V}_B$ and the cover-lifting property is immediate. 
\end{proof}

\begin{division}For this reason, certain sources as \cite{Anel} call the category $ \Sh(\mathcal{G}_B,J_B)$ the \emph{gros spectrum}. Beware however it lacks the universal property of the \emph{petit spectrum} $ \Spec(B) $ ! A more abstract reasoning involving bilimits of topoi can help the reader understanding that the gros spectrum, which can be presented as a comma object in $\GTop$, would only provide a left adjoint only for a restricted category of modelled topoi where the only allowed morphisms $(f,\phi) : (\mathcal{E}, E) \rightarrow (\mathcal{F},F)$ are those for which $\phi^\flat$ is an isomorphism. This is because the gros spectrum is a spectrum for a \emph{rigid geometry} where all maps are etale, and all local maps are isomorphisms. However, for those considerations are related to another way of presenting the spectrum we chose not to give further details about this fact and now turn to the structure sheaf of the petit spectrum which confers it its universal property. 
\end{division}
\subsection{The structure sheaf}

Now we turn to the structure sheaf, which is obtained from the codomain functor modulo sheafification. It will play the role of the free local object in the spectrum -- which will be visualizable at its stalks.

\begin{definition}
For a set-valued model $ B$ in $ \mathbb{T}[\mathcal{S}]$, the \emph{structure sheaf}\index{structure sheaf} of $B$ is the sheaf of set-valued $ \mathbb{T}$-models $\widetilde{B}$ obtained as 
\[\widetilde{B} = \mathfrak{a}_{J_B} \cod\]
\end{definition}

\begin{division}
The structure sheaf can also be described as follows: recall that any $\mathbb{T}$-model $B$ is in $\Ind(\mathbb{T}[\mathcal{S}]_{\omega}) \simeq \Lex[\mathcal{C}_\mathbb{T}, \mathcal{S}]$. Moreover, we can consider the \emph{conerve} of the codomain functor
\[\begin{tikzcd}[row sep=tiny]
	{\mathbb{T}[\mathcal{S}]^{\op}} & {\widehat{\mathcal{V}^{\op}_B}} \\
	{B} & {\mathbb{T}[\mathcal{S}] \big{[} B, \cod \big{]}}
	\arrow["{\cod_*}", from=1-1, to=1-2]
	\arrow[from=2-1, to=2-2, shorten <=2pt, shorten >=2pt, maps to]
\end{tikzcd}\]
which can be composed along the embedding $ \mathcal{C}_\mathbb{T} \hookrightarrow \mathbb{T}[\mathcal{S}]^{\op}$
\[\begin{tikzcd}
	{\mathcal{V}^{\op}_B} & {\mathbb{T}[\mathcal{S}]^{\op}} & {\mathcal{C}_\mathbb{T}} \\
	{\widehat{\mathcal{V}_B^{\op}}}
	\arrow["{\cod}", from=1-1, to=1-2]
	\arrow[""{name=0, inner sep=0}, from=1-1, to=2-1, hook]
	\arrow["{\katayo}"', from=1-3, to=1-2, hook']
	\arrow["{\cod_*}"{name=1}, from=1-2, to=2-1]
	\arrow[Rightarrow, "{\chi}"{near start}, from=0, to=1, shorten <=3pt, shorten >=5pt]
\end{tikzcd}\]
to produce a lex functor 
\[\begin{tikzcd}[row sep=tiny]
	{\mathcal{C}_\mathbb{T}} & {\widehat{\mathcal{V}^{\op}_B}} \\
	{\{ \overline{x}, \phi \}} & {\mathbb{T}[\mathcal{S}]\big{[} K_\phi, \cod \big{]}}
	\arrow["{\cod_*}", from=1-1, to=1-2]
	\arrow[from=2-1, to=2-2, maps to]
\end{tikzcd}\]
which we can now compose with the lex localization $\mathfrak{a}_{J_B} : {\widehat{\mathcal{V}^{\op}_B}} \rightarrow \Spec(B)$ to get a lex functor 

\[\begin{tikzcd}[row sep=tiny]
	{\mathcal{C}_{\mathbb{T}}} & {\Spec(B)} & {} \\
	{\{\overline{x}, \phi \}} & {\mathfrak{a}_{J_B}\mathbb{T}[\mathcal{S}] \big{[} K_\phi, \cod \big{]}}
	\arrow["{\widetilde{B}}", from=1-1, to=1-2]
	\arrow[shorten >=4pt, maps to, from=2-1, to=2-2]
\end{tikzcd}\]
\end{division}

\begin{proposition}\label{structure sheaf is local for sets}
The structure sheaf $\widetilde{B} $ is in $\mathbb{T}_J[\Spec(B)]$. In particular for any point $ x : \mathcal{S} \rightarrow \Spec(B)$ the stalk $ x^*\widetilde{B} $ is in $\mathbb{T}_J[\mathcal{S}]$.
\end{proposition}

\begin{proof}
We have to prove that for any $ J$-cover $ (\theta_i :\{\overline{x}_i, \phi_i \} \rightarrow \{\overline{x}, \phi \})_{i \in I}$ we have an epimorphism in the category of sheaves $ \Spec(B) = \Sh (\mathcal{V}^{\op}_B, J_B)$ 
\[\begin{tikzcd}[column sep=large]
	{\underset{i \in I}{\displaystyle\coprod}\mathfrak{a}_{J_B}\mathbb{T}[\mathcal{S}] \big{[} K_{\phi_i}, \cod \big{]}} && {\mathfrak{a}_{J_B}\mathbb{T}[\mathcal{S}] \big{[} K_\phi, \cod \big{]}}
	\arrow["{\langle \mathfrak{a}_{J_B}\mathbb{T}[\mathcal{S}] \big{[} f_{\theta_i}, \cod \big{]}\rangle_{i \in I}}", from=1-1, to=1-3, two heads]
\end{tikzcd}\]
But proving some morphism to be epic in a sheaf topos is a local condition, which however here would be made impossible to test because of the expression of the sheafification. We are in fact going to prove that before sheafification, the map $ \langle\mathbb{T}[\mathcal{S}] \big{[} f_{\theta_i}, \cod \big{]}\rangle_{i \in I} $ is a \emph{local epimorphism}, that is, a map that is sent to an epimorphism after sheafification -- though not yet being itself an epimorphism in the presheaf category. For sheafification is a left adjoint, it preserves coproducts, that is 
\[ \underset{i \in I}{\coprod}\mathfrak{a}_{J_B}\mathbb{T}[\mathcal{S}] \big{[} K_{\phi_i}, \cod \big{]} \simeq \mathfrak{a}_{J_B}\underset{i \in I}{\coprod}\mathbb{T}[\mathcal{S}] \big{[} K_{\phi_i}, \cod \big{]} \]
and moreover, the localness condition is preserved after sheafification, and makes \[\mathfrak{a}_{J_B}\langle \mathbb{T}[\mathcal{S}] \big{[} f_{\theta_i}, \cod \big{]}\rangle_{i \in I} \simeq \langle \mathfrak{a}_{J_B}\mathbb{T}[\mathcal{S}] \big{[} f_{\theta_i}, \cod \big{]}\rangle_{i \in I}\] an epimorphism in $\Spec(B)$. \\

Let us prove the localness condition for $ \langle \mathbb{T}[\mathcal{S}] \big{[} f_{\theta_i}, \cod \big{]}\rangle_{i \in I}$, which amounts to it to be a local surjection -- see for instance \cite{maclane&moerdijk}. Take some $b : K_\phi \rightarrow \cod (n)$; then one can push the $J$-cover $ ( f_{\theta_i}: K_\phi \rightarrow K_{\phi_i})_{i \in I}$ along $b$ so we get a cover of $n$ in $ \mathcal{V}^{\op}_B$  
\[\begin{tikzcd}
	{K_{\phi}} & {\cod(n)} & B \\
	{K_{\phi_i}} & {b_*K_{\phi_i}}
	\arrow["n"', from=1-3, to=1-2]
	\arrow["{b_*f_{\theta_i}}"{description}, from=1-2, to=2-2]
	\arrow["{nb_*f_{\theta_i}}", from=1-3, to=2-2]
	\arrow["b", from=1-1, to=1-2]
	\arrow[from=2-1, to=2-2]
	\arrow["{f_{\theta_i}}"', from=1-1, to=2-1]
	\arrow["\lrcorner"{anchor=center, pos=0.125, rotate=180}, draw=none, from=2-2, to=1-1]
\end{tikzcd}\]
Then for each $ i \in I$, we have
\begin{align*}
    \mathbb{T}[\mathcal{S}] \big{[} K_{\phi}, b_*f_{\theta_i} \big{]}(b) &= b_*f_{\theta_i} \, b \\
    &= {f_{\theta_i}}_*b \,f_{\theta_i} \\
    &= \mathbb{T}[\mathcal{S}] \big{[} f_{\theta_i}, K_{\phi_i} \big{]}( {f_{\theta_i}}_*b )
\end{align*}
which exactly says that the restriction of $ b$ along each member of the cover has an antecedent along $ \langle \mathbb{T}[\mathcal{S}] \big{[} f_{\theta_j}, b_*K_{\phi_i} \big{]}\rangle_{j \in I}$: hence the natural transformation $  \langle \mathbb{T}[\mathcal{S}] \big{[} f_{\theta_j}, \cod \big{]}\rangle_{j \in I}$ is a local epimorphism relative to the Grothendieck topology $ J_B$, hence its sheafification is an epimorphism in $\Spec(B)$. \end{proof}

\begin{remark}
Beware that the structure sheaf $ \widetilde{B} = \mathfrak{a}_{J_B}\cod$ needs not to return local objects as values; in particular, whenever the topology generated by $J_B$ is subcanonical, $\widetilde{B} = \cod$, but the codomains of basic etale arrows have no reason to be local objects. This is because local objects are models of a geometric extension of $\mathbb{T}$, being a model of which is a local notion that does not hold globally. Nevertheless, as we saw above, stalks of the structure sheaf are set-valued local objects; but in general having local objects at stalks may not be sufficient to ensure localness of a sheaf itself.
\end{remark}

\begin{division}
Now we address the functoriality of the spectrum at the level of the structure sheaves. For $ f: B_1 \rightarrow B_2$ in $\mathbb{T}[\mathcal{S}]$, we saw in \cref{Spec(f) in the set case} how $ \Spec(f)$ was obtained from the pushout functor $ f_* : \mathcal{V}_{B_1} \rightarrow \mathcal{V}_{B_2}$. Now to get the direct image part,  the bottom arrows at each $ n $ in $ \mathcal{V}_{B_1}$ of the pushout square $ n_*f : \cod(n) \rightarrow \cod(f_*n)$ define altogether a natural transformation  
\[\begin{tikzcd}
	{\cod_1} & {\Spec(f)_*\cod_2}
	\arrow["{\nu_f}", from=1-1, to=1-2]
\end{tikzcd}\]
which is sent after sheafification to a morphism in $ \mathbb{T}[\Spec(F_1)]$
\[\begin{tikzcd}
	{\widetilde{F_1}} & {\Spec(f)_*\widetilde{F_2}}
	\arrow["{\mathfrak{a}_{{J_{B_1}}}(\nu_f)}", from=1-1, to=1-2]
\end{tikzcd}\]
so we have to put 
\[ \widetilde{f}^\sharp = \mathfrak{a}_{{J_{B_1}}}(\nu_f) \]
while we get automatically from the adjunction $ \Spec(f)$ a mate 
\[\begin{tikzcd}
	{\Spec(f)^*\widetilde{F_1}} & {\widetilde{F_2}}
	\arrow["{\widetilde{\phi}^\flat}", from=1-1, to=1-2]
\end{tikzcd}\]
There is also a more concrete way to construct this morphism, which moreover expresses its connection to the (etale, locale factorization), exhibiting it as induced from the local part of precompositions with $f$: \end{division}

\begin{proposition}\label{f flat is local set case}
For each $ f: B_1 \rightarrow B_2$ in $ \mathbb{T}[\mathcal{S}]$, $ \widetilde{f}^\flat$ is in $ \Loc[\Spec(B_2)]$.
\end{proposition}

\begin{proof}
Recall that inverse images commute with sheafification, so we can first compute the inverse image of the codomain functor $ \Spec(f)^*\cod 
$ and then apply sheafification to get the inverse image $\Spec(f)^*\widetilde{F_1}$. But now the inverse image is computed as a left Kan extension $ \lan_{f_*}\cod$, which expresses at each $ n$ of $ \mathcal{V}_{B_2}$ as the filtered colimit
\[  \lan_{f_*}\cod(n) \simeq \underset{f_*\downarrow n}{\colim} \, \cod(m)  \]
But in fact, by the universal propert of the pushout -- which can be summed up in the adjunction $ f_* \dashv f^!$ where $ f^!$ is precomposition with $f$, we have an equivalence of categories $ f_* \downarrow n \simeq \mathcal{V}_{B_1}\downarrow nf $. This means that the colimit above can be seen equivalently as ranging over all factorizations 
\[\begin{tikzcd}
	{B_1} & {B_2} \\
	{\cod(m)} & {\cod(n)}
	\arrow["n"', from=1-1, to=2-1]
	\arrow["f", from=1-1, to=1-2]
	\arrow["{n}", from=1-2, to=2-2]
	\arrow["{g}"', from=2-1, to=2-2]
	\arrow["\lrcorner"{anchor=center, pos=0.125, rotate=180}, draw=none, from=2-2, to=1-1]
\end{tikzcd}\]
while the data of all the $g : \cod(m) \rightarrow \cod(n)$ induce a natural arrow $ (u_f)_n :  \lan_{f_*}\cod(n) \rightarrow \cod(n) $. But we saw in \cref{factorization from saturated class} that this is exactly how the (etale, locale) factorization of the composite $ nf : B_1 \rightarrow \cod(n)$ is computed as the colimit of all its factorizations through finitely presented etale maps under $B_1$, which are precisely the objects of the spectral site of $B_1$. Hence in the factorization
\[\begin{tikzcd}
	{B_1} & {B_2} & {\cod(n)} \\
	& { \lan_{f_*}\cod(n)}
	\arrow["f", from=1-1, to=1-2]
	\arrow["n", from=1-2, to=1-3]
	\arrow["{\colim \, \mathcal{V}_{B_1}\downarrow nf}"', from=1-1, to=2-2]
	\arrow["{(u_f)_n}"', from=2-2, to=1-3]
\end{tikzcd}\]
$\colim \, \mathcal{V}_{B_1}\downarrow nf $ is etale while $ (u_f)_n$ is local. Moreover from naturality of this process, this defines a natural transformation
\[\begin{tikzcd}
	{\Spec(f)^*\cod} & \cod
	\arrow["{u_f}", from=1-1, to=1-2]
\end{tikzcd}\]
whose sheafification for $ J_{F_2}$ is the desired inverse image part $ \widetilde{\phi}^\flat = \mathfrak{a}_{J_{B_2}}(u_f)$: since sheafification preserves localness, this exhibits $\widetilde{\phi}^\flat  $ as a local map. 
\end{proof}

\subsection{The generic etale map and classifying properties of the spectrum}\label{The generic etale map and classifying properties of the spectrum}

This section is actually the doorstep of the spectral adjunction: here are described the classifying property of the spectrum -- where the structure sheaf and a certain canonical map are crucially involved -- from which the later adjunction will appear as a corollary. 

\begin{division}
The structure sheaf comes associated with a canonical morphism of sheaves whose component indexes basic etale arrows under the object. For $B$ in $\mathbb{T}[\mathcal{S}]$, the spectrum of $B$ has a terminal geometric morphism $ !_{\Spec(B)}$ to $\mathcal{S}$ with the direct image part sending a sheaf $F : \mathcal{V}_B \rightarrow \mathcal{S}$ to its set of global sections $ \Gamma (F) = F(1_B)$ and the inverse image $ !_{\Spec(B)}^* $ sending a set $X$ to the $X$-indexed coproduct $ \coprod_X 1_{\Spec(B)}$. This adjunction lifts to the categories of $\mathbb{T}$-models 
\[\begin{tikzcd}
	{\mathbb{T}[\Spec(B)]} && {\mathbb{T}[\mathcal{S}]}
	\arrow[""{name=0, anchor=center, inner sep=0}, "\Gamma"', bend right=20, start anchor=-15, from=1-1, to=1-3]
	\arrow[""{name=1, anchor=center, inner sep=0}, "{!_{\Spec(B)}^*}"', bend right=20, end anchor=15, from=1-3, to=1-1]
	\arrow["\dashv"{anchor=center, rotate=-90}, draw=none, from=1, to=0]
\end{tikzcd}\]
Moreover this adjunction also exists at the level of the presheaf topos $ \widehat{\mathcal{V}^{\op}_B}$, and we have $ !_{\Spec(B)} = !_{\widehat{\mathcal{V}^{\op}_B}} \iota_B$. Now the identity of $ B$ defines a map $1_B : B \rightarrow B = \cod_*(1_B)$, and the latter object is actually the global section object of the codomain functor: by the version of the adjunction above relative to the presheaf topos, this defines a canonical map 
\[\begin{tikzcd}
	{!_{\widehat{\mathcal{V}^{\op}_B}}^*B} & {\cod_*}
	\arrow["{\nu_B}", from=1-1, to=1-2]
\end{tikzcd}\]
which is sent after sheafification to a composite map
\[\begin{tikzcd}
	{!_{\Spec(B)}^*B} & {w\widetilde{B}}
	\arrow["{\eta^\flat_B}", from=1-1, to=1-2]
\end{tikzcd}\]
This very map corresponds itself to a comparison map 
\[\begin{tikzcd}
	B & {\Gamma w\widetilde{B}}
	\arrow["{\eta^\sharp_B}", from=1-1, to=1-2]
\end{tikzcd}\]

Those two maps are going to be part of the unit of a restricted $ \Spec\dashv \Gamma $ adjunction below.
\end{division}

\begin{proposition}
The map $ \eta^\flat_B$ is etale in $\mathbb{T}[\Spec(B)]$
\end{proposition}

\begin{proof}
The inverse image presheaf $!_{\widehat{\mathcal{V}_B^{\op}}}^*B$ can also be described as the constant sheaf returning $B$ everywhere; hence the map $\nu_B$ can also be described as the natural transformation 
\[\begin{tikzcd}
	{\mathcal{V}_B} && {\mathbb{T}[\mathcal{S}]}
	\arrow[""{name=0, anchor=center, inner sep=0}, "{!_{\widehat{\mathcal{V}_B^{\op}}}^*B}", bend left= 25, from=1-1, to=1-3]
	\arrow[""{name=1, anchor=center, inner sep=0}, "{\cod}"', bend right= 25, from=1-1, to=1-3]
	\arrow["{\nu_B}", shorten <=3pt, shorten >=3pt, Rightarrow, from=0, to=1]
\end{tikzcd}\]
whose component at $ n$ is $n$ itself, which is etale: hence $ \nu_B$ is an etale map as it is pointwise etale, and so is its sheafification $ \eta_B^\flat$.
\end{proof}

\begin{definition}
In the following $ \nu^\flat_B$ will be called the \emph{generic etale map under $B$}\index{etale map!generic}, while $ \eta^\flat_B$ will be called the \emph{generic local unit under $B$}\index{local unit!generic}.
\end{definition}

The generic etale map gathers all the etale maps under $B$ you need to compute the etale-locale factorizations of maps under inverse images of $B$, as you can extract the etale part of any morphism of sheaves from the inverse image of this morphism. Similarly the generic local unit gathers all admissible factorizations under $B$, so that any morphism into a local object will now factorize through this very map: that is how admissible factorizations through different local units will be turned into an ordinary unit-like factorization. We split the process in two steps encoding first the factorization aspects and then the local data -- in the same way we did in the fifth chapter.  \\

\begin{division}Recall that the presheaf topos $\widehat{\mathcal{V}^{\op}_B} $ is endowed with the codomain presheaf $ \cod : \mathcal{V}_B \rightarrow \mathbb{T}[\mathcal{S}]$ together with the canonical etale map $ \nu_B : \; !_{\widehat{\mathcal{V}^{\op}_B}}^*B \rightarrow \cod_*$ whose component at $n$ is $n$ itself. Then any geometric morphism $ x :\mathcal{E} \rightarrow \widehat{\mathcal{V}^{\op}_B} $ defines an etale map in $\mathcal{E}$
\[\begin{tikzcd}
	{!_{\mathcal{E}}^*B} & {w x^*\cod}
	\arrow["{x^*\eta^{\flat}_B}", from=1-1, to=1-2]
\end{tikzcd}\]

Similarly, any geometric morphism $ x : \mathcal{E} \rightarrow \Spec(B)$ transfers the canonical map $ \eta^\flat_B$ in $\mathcal{E}$ to an etale map
\[\begin{tikzcd}
	{!_{\mathcal{E}}^*B} & {w x^*\widetilde{B}}
	\arrow["{x^*\eta^{\flat}_B}", from=1-1, to=1-2]
\end{tikzcd}\]
where $ x^*\widetilde{B}$ is a local object.
\end{division}

Now we want to describe the converse process: we want to prove that \begin{itemize}
    \item the presheaf topos $\widehat{\mathcal{V}^{\op}_B} $ classifies etale maps under $B$;
    \item the sheaf topos $ \Spec(B)$ classifies local units under $B$.
\end{itemize}


\begin{division}
 To any morphism of sheaves $ \phi^\flat : \; !_{\mathcal{E}}^*B \rightarrow E $ in $\mathbb{T}[\mathcal{E}]$ -- corresponding to a morphism $ \phi^\sharp : B \rightarrow \Gamma E $ in $ \mathbb{T}[\mathcal{S}]$, we want to associate a geometric morphism $ \mathcal{E} \rightarrow \Spec(B)$.  Suppose that $ \mathcal{E}$ has a standard site of definition $(\mathcal{C}_\mathcal{E},J_\mathcal{E})$. 
Define the functor $ x_\phi^*: \mathcal{V}_B^{\op} \rightarrow \widehat{\mathcal{C}_\mathcal{E}}$ as sending a finitely presented etale arrow $n : B \rightarrow \cod(n)$ to the presheaf $ x_\phi^*(n) : \mathcal{C}_\mathcal{E}^{\op} \rightarrow \mathcal{S}$ which associates to each $c$ of $\mathcal{C}_\mathcal{E}$ the set of all possible factorizations 
\[\begin{tikzcd}
	B && {\Gamma wE} \\
	{\cod(n)} && {E(c)}
	\arrow["n"', from=1-1, to=2-1]
	\arrow["{\phi^\sharp}", from=1-1, to=1-3]
	\arrow["a"', from=2-1, to=2-3]
	\arrow["{E(!_c)}", from=1-3, to=2-3]
\end{tikzcd}\]
of the composite $ E(!_c)\phi^\sharp$ (where $!_c : c \rightarrow 1_{\mathcal{C}_\mathcal{E}}$ is the terminal map of $c$ in $\mathcal{C}_\mathcal{E}$) through the finitely presented map $n$, and $ s : c_1 \rightarrow c_2$ to postcomposition of $a$ with $ E(s)$.
\end{division}

\begin{lemma}
For each $ \phi$ as above, the functor $ x_\phi^*$ lands in $\mathcal{E}$.
\end{lemma}

\begin{proof}
We must prove that for each $n$ in $\mathcal{V}_B$, the presheaf $x_\phi^*(n)$ is a sheaf for $ J_\mathcal{E}$. But this is a consequence of $E$ being a sheaf: for a family $ (s_i : c_i \rightarrow c)_{i \in I}$ in $J_\mathcal{E}$, we have $ E(c) $ is the limit of the $E(c_i)$ for the descent diagram for $(s_i)_{i \in I}$: hence, any matching family $ (a_i)_{i \in I}$ with $ a_i : n \rightarrow E(s_i)E(!_c)\phi^\sharp$ is in particular a family of arrows $ a_i : \cod(n) \rightarrow E(c_i)$ satisfying the commutations of the descent diagram, and hence induces uniquely a map $ (a_i)_{i \in I} : n \rightarrow E(c)$. 
\end{proof}

The functor constructed above is actually involved in the universal factorization of $\phi$. 

\begin{lemma}
For any $ \phi $ as above, $ x^*_\phi$ defines a geometric morphism $ \mathcal{E} \rightarrow \widehat{\mathcal{V}_B^{\op}}$.
\end{lemma}

\begin{proof}
This amounts to proving that $x_\phi^*$ is lex. Recall that $ \mathcal{V}_B$ is closed under finite colimits in $ B \downarrow \mathbb{T}[\mathcal{S}]$. Moreover, as (finite) limits in sheaf topoi are pointwise, it suffices to prove that for each $c$ in $\mathcal{C}_\mathcal{E}$ the functor $x_\phi^*(-)(c) : \mathcal{V}^{\op}_B \rightarrow \Set $ is lex: but in fact $x_\phi^*(-)(c) $ is nothing but the composite of the functor $ B \downarrow \mathbb{T}[\mathcal{S}][-, E(!_c)\phi^\sharp]$, which, as a representable, turns colimits into limits, along the inclusion of $\mathcal{V}_{B}$ into the coslice $ B \downarrow \mathbb{T}[\mathcal{S}]$. Hence $x_\phi^*$ is lex.\\
\end{proof}

In fact, this is because $ x^*_\phi$ ``points" to the \emph{etale part} of $\phi$ as classified by the presheaf topos $\widehat{\mathcal{V}_B^{\op}} $. Let us precise this intuition -- for which the geometric morphism $x_\phi$ will be called the \emph{classifying morphism} of $\phi$ in the following.

\begin{theorem}\label{universal factorization of a map}
For any morphism of the form $ \phi^\flat : \; !_{\mathcal{E}}^*B \rightarrow E $ in a Grothendieck topos $\mathcal{E}$ we have a universal factorization as below
\[\begin{tikzcd}
	{!_{\mathcal{E}}^*B} && E \\
	& {x_\phi^*\cod}
	\arrow["{\phi^\flat}", from=1-1, to=1-3]
	\arrow["{x_\phi^*(\nu^\flat_B)}"', from=1-1, to=2-2]
	\arrow["{u_{\phi^\flat}}"', from=2-2, to=1-3]
\end{tikzcd}\]
with $ u_{\phi^\flat}$ a local map. 
\end{theorem}

\begin{proof}
Our strategy is to construct first a pointwise factorization for the underlying presheaves and prove it coincides with inverse image -- forgetting first that $x_\phi^*$ land in the sheaf topos $\mathcal{E}$.
We have at any object $ c$ of $\mathcal{C}_\mathcal{E}$ a lex functor 
\[\begin{tikzcd}
	{\mathcal{V}^{\op}_B} & {\mathcal{S}}
	\arrow["{x_\phi^*(-)(c)}", from=1-1, to=1-2]
\end{tikzcd}\]
Its category of elements $ \int x_\phi^*(-)(c)$ is hence cofiltered and is also the category of all factorizations of the composite $ E(!_c)x^\sharp$ through an etale arrow on the left: but recall this is the category indexing the filtered colimit from which we constructed the etale part of the etale-local factorization ! In other words, we have 
\[  n_{E(!_c)\phi^\sharp} \simeq \underset{(n,a) \in (\int x_\phi^*(-)(c))^{\op}}{\colim} \, n  \]

But the latter coincides also with the expression of the inverse image (in the presheaf topos) $ x_\phi^*\cod$, which is the left Kan extension
\[\begin{tikzcd}
	{\mathcal{V}_B} & {\mathbb{T}[\mathcal{S}]} \\
	{(\widehat{\mathcal{C}_\mathcal{E}})^{\op}}
	\arrow["{(x_\phi^*)^{\op}}"', from=1-1, to=2-1]
	\arrow[""{name=0, anchor=center, inner sep=0}, "{\lan_{(x_\phi^*)^{\op}} \cod}"', from=2-1, to=1-2]
	\arrow["\cod", from=1-1, to=1-2]
	\arrow["q"', shorten >=2pt, Rightarrow, from=1-1, to=0]
\end{tikzcd}\]
whose computation at each $c$ returns
\begin{align*}
    \lan_{(x_\phi^*)^{\op}} \cod (c)  &\simeq \underset{(n,a) \in (\hirayo_c \downarrow x_\phi^*)^{\op}}{\colim} \, \cod(n) \\
    &\simeq  \cod \underset{(n,a) \in (\int x_\phi^*(-)(c))^{\op}}{\colim} \, n \\
    &\simeq \cod(n_{E(!_c)\phi^\sharp} )
\end{align*}
In other words, the following square
\[\begin{tikzcd}[column sep=6em]
	B & {\Gamma E} \\
	{\lan_{(x_\phi^*)^{\op}} \cod (c) } & {E(c)}
	\arrow["{E(!_c)}", from=1-2, to=2-2]
	\arrow["{\phi^\sharp}", from=1-1, to=1-2]
	\arrow["{\underset{(n,a) \in (\int x_\phi^*(c))^{\op}}{\colim} \, n }"', from=1-1, to=2-1]
	\arrow["{\langle a\rangle_{(n,a) \in (\int x_\phi^*(c))^{\op}}}"', from=2-1, to=2-2]
\end{tikzcd}\]
coincides with the etale-locale factorization of $E(!_c)\phi^\sharp$, whose local part is the induced map $ u_{E(!_c)\phi^\sharp} = \langle a\rangle_{(n,a) \in (\int x_\phi^*(c))^{\op}} $.\\

Since the factorization of presheaves morphisms is pointwise, this exactly says that in $\mathbb{T}[\widehat{\mathcal{C}_\mathcal{E}}]$ the etale-locale factorization is given as 
\[\begin{tikzcd}
	{!_{\mathcal{E}}^*B} && E \\
	& {\lan_{(x_\phi^*)^{\op}} \cod }
	\arrow["{\phi^\flat}", from=1-1, to=1-3]
	\arrow["{\lan_{(x_\phi^*)^{\op}} \nu_B^\flat }"', from=1-1, to=2-2]
	\arrow["{u_{\phi^\flat}}"', from=2-2, to=1-3]
\end{tikzcd}\]
Then the desired factorization in $ \mathbb{T}[\mathcal{S}]$ is obtained by the sheafification $\mathfrak{a}_{J_\mathcal{E}}$, which preserves the etale-local factorization, and turns the left Kan extension into the inverse image functor $ x_\phi^* = \mathfrak{a}_{J_\mathcal{E}} \lan_{(x^*)^{\op}}$. 
\end{proof}

\begin{proposition}
Let $ \phi$ be as above, and such that moreover $ E$ is a local object in $\mathbb{T}[\mathcal{S}]$. Then $x_\phi^*$ defines a geometric morphism $ x_\phi : \mathcal{E} \rightarrow \Spec(B)$.
\end{proposition}

\begin{proof}
$J_B$-continuity of $ x^*_\phi$ results from localness of $E$: let $ (m_i : n \rightarrow n_i)_{i \in I}$ be a $J_B$-cover in $\mathcal{V}_B$ induced from some $J$-cover $ (k_i : K \rightarrow K_i)_{i \in I} $ along some map $b : K \rightarrow \cod(n)$. We have to prove that the following morphism in $\mathcal{E}$
\[\begin{tikzcd}
	{\underset{i \in I}{\displaystyle{\coprod}}x_\phi^*n_i} & {x_\phi^*n}
	\arrow["{\langle x_\phi^*m_i \rangle_{i \in I}}", from=1-1, to=1-2]
\end{tikzcd}\]
is an epimorphism in $\mathcal{E}$, that is, a local surjection for $J_\mathcal{E}$. Let $c$ be in $\mathcal{C}_\mathcal{E}$ and $ a $ in $x_\phi^*n(c)$: then the composite $ ab : K \rightarrow E(c) $ is also an object of $E(K)(c)$ for $E$ seen as a $J$-continuous functor $ \mathcal{C}_\mathbb{T} \rightarrow \mathcal{E}$, so that $\langle E(m_i) \rangle_{i \in I} $ is itself a local surjection: hence there is a $J_\mathcal{E}$-cover $ (s_j : c_j \rightarrow c)_{i \in I'}$ such that for any $j \in I'$, $E(s_j)(K)(b)$ comes from some $E(c_j)(K_i)$ for some $i \in I$: but this exactly says that for each $j$ there is a $i$ together with factorization $d$ as below  
\[\begin{tikzcd}
	&&& B && {\Gamma wE} \\
	& K && {\cod(n)} && {E(c)} \\
	{K'} && {\cod(n_i)} && {E(s_j)}
	\arrow["{E(!_c)}", from=1-6, to=2-6]
	\arrow["n"{description}, from=1-4, to=2-4]
	\arrow["{\phi^\sharp}", from=1-4, to=1-6]
	\arrow["a"{description, pos=0.3}, from=2-4, to=2-6]
	\arrow["b"{pos=0.3}, from=2-2, to=2-4]
	\arrow["{E(s_j)}", from=2-6, to=3-5]
	\arrow["{m_i}"{description}, from=2-4, to=3-3]
	\arrow["{k_i}"', from=2-2, to=3-1]
	\arrow[from=3-1, to=3-3]
	\arrow["d"', dashed, from=3-3, to=3-5]
	\arrow["{n_i}"'{pos=0.4}, curve={height=18pt}, from=1-4, to=3-3, crossing over]
	\arrow["\lrcorner"{anchor=center, pos=0.125, rotate=180}, draw=none, from=3-3, to=2-2]
	\arrow["{E(!_{c_i})}"{description, pos=0.3}, curve={height=12pt}, from=1-6, to=3-5, crossing over]
\end{tikzcd}\]
and such a $d$ is in particular an element of $ x_\phi^*(n_i)(c_j)$ such that $x_\phi^*(n)(s_i)(a) = x_\phi^*(m_i)(c_j)(b) $. Whence continuity of $x_\phi^*$. 
\end{proof}

\begin{corollary}\label{universal factorization and admissibility}
For $ \phi$ as above with $E$ a local object, the universal factorization obtained at \cref{universal factorization of a map} coincides with
\[\begin{tikzcd}
	{!_{\mathcal{E}}^*B} && E \\
	& {x_\phi^*w\widetilde{B}}
	\arrow["{\phi^\flat}", from=1-1, to=1-3]
	\arrow["{x_\phi^*(\eta^\flat_B)}"', from=1-1, to=2-2]
	\arrow["{u_{\phi^\flat}}"', from=2-2, to=1-3]
\end{tikzcd}\]
\end{corollary}

\begin{proof}
This is because the inverse image part $x_\phi^*$ factorizes through the sheaf topos $ \Spec(B) \simeq \Sh(\mathcal{V}^{\op}, J_B)$ via the sheafification functor $ \mathfrak{a}_{J_B}$: hence doing the inverse image along $x_\phi^*$ localizes the sheafification map $ \gamma_B$, whence\begin{align*}
    x^*_\phi \cod &\simeq x^*_\phi \mathfrak{a}_{J_B} \cod \\
    &\simeq x^*_\phi w\widetilde{B}  
\end{align*} and the same for the generic etale map.
\end{proof}

\begin{remark}
This corollary is in fact a manifestation of admissibility: it says that the universal factorization of a map into a local object has its middle term local itself, which, in this case, says that the universal factorization above identifies the codomain functor with the structure sheaf, or that in other terms, the classifying morphism $ x^*_\phi$ ``sees" the codomain functor as a local object.
\end{remark}

Now \cref{points of the spectrum are local units} generalizes to the following more universal form, where points are replaced with arbitrary geometric morphisms and local units are considered up to inverse images:

\begin{theorem}\label{the spectrum classifies local units}
For $ B$ in $\mathbb{T}[\mathcal{S}]$ and any Grothendieck topos, geometric morphisms $ \mathcal{E} \rightarrow \Spec(B)$ correspond to etale maps $ !_\mathcal{E}^*B \rightarrow E$ with $E$ a local object in $\mathbb{T}[\mathcal{E}]$. 
\end{theorem}

\begin{proof}
We saw above how one can toggle between geometric morphisms into the spectrum and local units. We must prove that this correspondence is actually an equivalence of categories. \\

In one direction, consider a local unit $ \phi^\flat : \; !_\mathcal{E}^*B \rightarrow E$ with $E$ local. Take its classifying morphism $ x_\phi : \mathcal{E} \rightarrow \Spec(B)$ and then the inverse image of the generic etale map: this is exactly the left part of the factorization of $\phi^\flat$: but hence, if $\phi^\flat$ was already etale, then its etale part is an isomorphism, exhibiting 
\[  \phi^\flat \simeq  x_\phi^*(\eta^\flat_B) \hskip1cm E \simeq x_\phi^*w\widetilde{B} \]

In the other direction, take a geometric morphism $ x : \mathcal{E} \rightarrow \Spec(B)$, then the induced local unit $ x^*(\eta^\flat_B)$, and then back the classifying morphism of the later $ x^*_{x^*(\eta^\flat_B)} : \mathcal{V}^\op_B \rightarrow \mathcal{E}$. We must prove that for any $n$ in $\mathcal{V}^\op_B$ we have an isomorphism of sheaves $x^*_{x^*(\eta^\flat_B)}(n) \simeq x^*(n)$, which amounts to natural bijections $x^*_{x^*(\eta^\flat_B)}(n)(c) \simeq x^*(n)(c)$ for each $c$ of $\mathcal{C}_\mathcal{E}$. Recall that, on one hand, $x^*_{x^*(\eta^\flat_B)}(n)(c)$ is the set of all factorizations 
\[\begin{tikzcd}
	B & {\Gamma x^*w\widetilde{B}} \\
	{\cod(n)} & {x^*w\widetilde{B}(c)}
	\arrow["{x^*w\widetilde{B}(!_c)}", from=1-2, to=2-2]
	\arrow["{(x^*(\eta^\flat_B))^\sharp}", from=1-1, to=1-2]
	\arrow["n"', from=1-1, to=2-1]
	\arrow["a"', from=2-1, to=2-2]
\end{tikzcd}\]
while on the other hand $ x^*(n)(c)$ is by Yoneda lemma the set of all arrows $\hirayo_c \rightarrow x^*(n) $ in $\mathcal{E}$.
Since $x$ defined a morphism in $\Spec(B)$, its inverse image $x^*$ factorizes through the sheafification $ \mathfrak{a}_{J_B}$, so that we had a natural isomorphism between inverse images $ x^* w\widetilde{B} \simeq x^*\cod $. The latter is the $J_\mathcal{E}$-sheafification of the left Kan extension $ \lan_{(x^*)^{\op}} \cod $, whose value at $c$ is computed as the filtered colimit $\lan_{(x_\phi^*)^{\op}} \cod (c) \simeq  {\colim}_{(n,a) \in \hirayo_c \downarrow x_\phi^*}\cod(n) $: but the indexing set of this colimit is precisely $x^*(n)(c)$. Hence any factorization as above exactly corresponds to its colimit inclusion, whence the bijection.  \end{proof}


\subsection{$ \Spec\dashv \Gamma$-adjunction for set-valued models and sheaf representation}

\begin{division}Now recall that $\mathcal{S}$ is terminal amongst Grothendieck topoi, and for any Grothendieck topos $\mathcal{E}$, the terminal geometric morphism $ !_\mathcal{E} : \mathcal{E} \rightarrow \mathcal{S}$ has for direct image part $\Gamma = \mathcal{E}(1, -)$. Now as $\mathbb{T}$ is a finite limit theory, it is stable under direct image, so that $ \Gamma$ induces a functor 
\[ \mathbb{T}[\mathcal{E}] \stackrel{\Gamma}{\longrightarrow} \mathbb{T}[\mathcal{S}] \]
In particular for any locally $\mathbb{T}_J$-modelled topos $(\mathcal{E}, E)$, we can apply $ {!_\mathcal{E}}_* $ to $ wE$ to get a set-valued $ \mathbb{T}$-model $ \Gamma E$, and for a morphism of locally $\mathbb{T}_J$-modelled topoi $(f,\phi) : (\mathcal{E},E) \rightarrow (\mathcal{F},F)$, we have a morphism $ \Gamma \phi^\sharp : \Gamma E \rightarrow \Gamma F  $, as $ \Gamma f_*F = \Gamma F$ since direct images commute with global sections. Moreover, for a 2-cell $\alpha : (f,\phi) \Rightarrow (g,\psi)$, the equality $ \psi^\flat F*\alpha^\flat= \phi^\flat  $ corresponds to an equality $ \phi^\sharp = \alpha^\sharp* F \psi^\sharp$ with $ \alpha^\sharp : g_* \Rightarrow f_*$; but $ \Gamma$ sends $ F*\alpha^\sharp$ into an equality, so that $\alpha$ is collapsed into the equality of the morphism $ \Gamma \phi^\sharp = \Gamma \psi^\sharp$ in $\mathbb{T}[\mathcal{S}]$.
This defines a 2-functor 
\[ {\mathbb{T}_J^{\Loc}\hy\GTop} \stackrel{\Gamma}{\longrightarrow} 	{\mathbb{T}[\mathcal{S}]} \]
\end{division}

\begin{theorem}\label{biadjunction for set-models}
We have an adjunction 
\[\begin{tikzcd}
	{\mathbb{T}[\mathcal{S}]} \quad && {\mathbb{T}_J^{\Loc}\hy\GTop}
	\arrow["{\Spec}"{name=0}, from=1-1, to=1-3, curve={height=-20pt}, start anchor=40, end anchor=120]
	\arrow["{\Gamma}"{name=1}, from=1-3, to=1-1, curve={height=-20pt}, end anchor=320, start anchor=240]
	\arrow["\dashv"{rotate=-90}, from=0, to=1, phantom]
\end{tikzcd}\]
\end{theorem}

\begin{proof}
Let $(f,\phi) : (\Spec(B),\widetilde{B}) \rightarrow (\mathcal{E},E)$ be a morphism of locally modelled topoi, that is the data of a geometric morphism $ f: \mathcal{E} \rightarrow \Spec(B)$ and a local map $ \phi^\flat : f^*\widetilde{B} \rightarrow E$. Then from \cref{the spectrum classifies local units}, $f$ defines uniquely a local unit $ f^*(\eta^\flat_B) : \; !_{\mathcal{E}}^*B \rightarrow f^*\widetilde{B}$ in $\mathbb{T}[\mathcal{E}]$, which we can compose with the local map $ \phi^\flat$ to get a morphism into $E$
\[\begin{tikzcd}
	{!_{\mathcal{E}}^*B } && E \\
	& {f^*\widetilde{B}}
	\arrow["{f^*(\eta^\flat_B)}"', from=1-1, to=2-2]
	\arrow["{\phi^\flat}"', from=2-2, to=1-3]
	\arrow[dashed, from=1-1, to=1-3]
\end{tikzcd}\]
which corresponds uniquely to a morphism 
\[\begin{tikzcd}
	{B } && {\Gamma E}
	\arrow["{(\phi^\flat f^*(\eta^\flat_B))^\sharp}", from=1-1, to=1-3]
\end{tikzcd}\]

For the converse, any $f : B \rightarrow \Gamma E$ with $(\mathcal{E},E)$ a locally modelled topos defines uniquely a morphism $ f^\flat : \; !_{\mathcal{E}}^*B \rightarrow E$ in $\mathcal{E}$, whose etale-local factorization -- given in \cref{universal factorization and admissibility} defines uniquely a morphism of modelled topoi 
\[\begin{tikzcd}
	{(\Spec(B),\widetilde{B})} & {(\mathcal{E},E)}
	\arrow["{(x_{n_{f^\flat}}, u_{f^\flat})}", from=1-1, to=1-2]
\end{tikzcd}\]
where $x_{n_{f^\flat}} : \mathcal{E} \rightarrow \Spec(B) $ is the classifying morphism of the etale part $ n_{\phi^\flat} \simeq x_{n_{f^\flat}}^*(\eta^\flat_B) $ and $u_{f^\flat} $ is the local part of the factorization. 
\end{proof}

This process can be explicitly generalized for $ \mathbb{T}[\mathcal{S}]$-objects in arbitrary Grothendieck topoi. But this will be better understood in the light of the concepts of \emph{fibered sites} and \emph{fibered topoi}. The next section is devoted to some prerequisites on this notion, but also contains a new variation of it and some results allowing to adapt it to our situation.

However before turning to the general case and our auxiliary results about fibered topoi, we should end this section with a \emph{representability criterion}, explaining when the adjunction above has invertible units, that is, when a set valued model is representable as the global section object of its own structure sheaf. 

\begin{division}
Recall that a Grothendieck topology is \emph{subcanonical} if representable presheaves are sheaves, equivalently if the sheafified Yoneda embedding still is fully faithful. This also amounts that, for any $J$-cover $(u_i : c_i \rightarrow c)_{i \in I}$, the nerve diagram
\[\begin{tikzcd}
	{\displaystyle\underset{i,j \in I}{\coprod}\hirayo_{c_i \times_c c_j}} & {\displaystyle\underset{i \in I}{\coprod}\hirayo_{c_i }} & {\hirayo_c}
	\arrow[shift left=1, from=1-1, to=1-2]
	\arrow[shift right=1, from=1-1, to=1-2]
	\arrow["{\langle \hirayo_{u_i} \rangle_{i \in I}}", from=1-2, to=1-3]
\end{tikzcd}\]
exhibits $ \hirayo_c$ as the coequalizer of the left parallel pair, not only after sheafification but already in in the presheaf category. 
\end{division}

Recall we defined at \cref{generalized covers} the notion of extended $J$-covers as those $ (n_i : B \rightarrow C_i)_{i \in I}$ obtained as pushouts of $J$-cover under an object $B$. In particular $J$-covers are extended $J$-covers. Now for an extended $J$-cover we can define its nerve
\[\begin{tikzcd}
	B & {\displaystyle \prod_{i \in I} C_i} & {\displaystyle \prod_{i,j \in I} C_i \times_B C_j}
	\arrow["{\langle n_i\rangle_{i \in I}}", from=1-1, to=1-2]
	\arrow[shift left=1, from=1-2, to=1-3]
	\arrow[shift right=1, from=1-2, to=1-3]
\end{tikzcd}\]

\begin{theorem}\label{sheaf representation}
Let $(\mathbb{T},\mathcal{V},J)$ be a geometry, with $ \Spec \dashv \Gamma$ the associated adjunction; then the following are equivalents:\begin{itemize}
    \item any extended $J$-cover $ (n_i : B \rightarrow C_i)_{i \in I}$ exhibits $ B$ as the limit of its nerve;
    \item for any $B$ in $\mathbb{T}[\mathcal{S}]$ the structural presheaf $\cod$ is a sheaf for $J_B$;
    \item for any $B $ in $\mathbb{T}[\mathcal{S}]$ the unit $ \eta^\sharp_B : B \rightarrow \Gamma \widetilde{B} $ is invertible;
    \item $\Spec$ is full and faithful.
\end{itemize}
\end{theorem}

\begin{proof}
Observe that for any $J_B$ cover 
\[\begin{tikzcd}
	B & {\cod(n)} \\
	& {\cod(n_i)}
	\arrow["{n_i}"', from=1-1, to=2-2]
	\arrow["n", from=1-1, to=1-2]
	\arrow["{m_i}", from=1-2, to=2-2]
\end{tikzcd}\]
the family $ (m_i : \cod(n) \rightarrow \cod(n_i))_{i \in I}$ is an extended $J$-cover by the very definition of $J_B$. Hence if the first condition is satisfied, we have 
\[  \cod(n) \simeq \lim \big( \prod_{i \in I} \cod(n_i) \rightrightarrows \prod_{i,j \in I} \cod(n_i +_n n_j) \big)  \]
which exactly says that $ \cod$ is a sheaf for $J_B$. \\

As a consequence, the sheafification of the codomain functor provides with an isomorphism $ \cod \simeq \mathfrak{a}_{J_B} \cod = \widetilde{B}$ in $\mathbb{T}_J\big[[\mathcal{V}_B, \mathcal{S}]\big]$, so that the structure sheaf coincides with the codomain functor. In particular we have 
\[  \Gamma\widetilde{B} = \cod(1_B) \simeq B \]
which exhibits the unit as an isomorphism. The last item is equivalent by generality on adjunctions. 
\end{proof}

\begin{definition}
We say that $(\mathbb{T}, \mathcal{V}, J) $ has \emph{sheaf representation} if it satisfies one, hence all of the equivalent conditions of \cref{sheaf representation}.
\end{definition}

\begin{proposition}
If $ (\mathbb{T},\mathcal{V},J)$ has sheaf representation, then the topology $J$ is subcanonical.
\end{proposition}

\begin{proof}
Take a $J$-cover $ (\theta_i : \{\overline{x}_i, \phi_i \}  \rightarrow \{ \overline{x}, \phi \})_{i \in I}$ in $\mathcal{C}_\mathbb{T}$, with $ (n_{\theta_i} : K_{\phi} \rightarrow K_{\phi_i})_{i \in I} $ the corresponding family in $\mathbb{T}[\mathcal{S}]_\omega$. Then for any $\{ \overline{y}, \psi \}$ one has $ \hirayo_{\{ \overline{y}, \psi \} }(\{ \overline{x}, \phi \}) \simeq \mathbb{T}[\mathcal{S}](K_\psi, K_\phi) $. Then assuming that extended $J$-covers are limiting, we know that 
\[ K_\phi \simeq \lim \big( \prod_{i \in I} K_{\phi_i} \rightrightarrows \prod_{i,j \in I} K_{\phi_i} +_{K_\phi} K_{\phi_j} \big) \]
Then, since the covariant representable preserve limits, we have 
\[  \hirayo_{\{ \overline{y}, \psi \} }(\{ \overline{x}, \phi \}) \simeq \lim \big( \prod_{i \in I}  \hirayo_{\{ \overline{y}, \psi \} }(\{ \overline{x}_i, \phi_i \}) \rightrightarrows \prod_{i,j \in I}  \hirayo_{\{ \overline{y}, \psi \} }(\{ \overline{x}, \phi \})(\{ \overline{x}_i, \phi_i \} \times_{\{ \overline{x}, \phi \}}\{ \overline{x}_j, \phi_j \}) \big)  \]
This exactly says that the representable $ \hirayo_{\{ \overline{y}, \psi \} }$ is a sheaf for $J$, hence that $J$ is subcanonical. 
\end{proof}

\begin{division}Subcanonicity per se seems to be unsufficient in all generality to ensure sheaf representation. However, the situation is better when the topology consists of finite families -- which, from the point of view of logic, means that $\mathbb{T}_J$ is a special case of geometric extension of $\mathbb{T}$, namely a \emph{coherent extension}. Actually all examples we will consider fall under this situation, which make this assumption fairly reasonable. This is part of \cite{Coste}[Theorem 4.5.1], we give here a more detailed proof:
\end{division}

\begin{proposition}
    Let $ (\mathbb{T}, \mathcal{V}, J)$ be a geometry with $J$ consisting of finite families. Then one has sheaf representation if and only if $ J$ is subcanonical.
\end{proposition}

\begin{proof}
    If $J$ is subcanonical, then its cover corresponds to actual colimits in $\mathcal{V}^{\op}$: that is, each cover $ (m_i : K \rightarrow K_i)_{i \in I}$ in $J$ exhibit $ K = \lim_{i \in I} K_i$ of the corresponding nerve diagram. From the first item of \cref{sheaf representation}, it suffices to prove that extended covers exhibits their domain as the limit of the corresponding nerve to ensure sheaf representation. \\
    As $J$-covers are finite, so are extended $J$-covers; take some $B$ and $ (n_i : B \rightarrow B_i)_{i \in I} $ an extended $J$-cover. Define the category $ J \downarrow (n_i)_{i \in I}$ whose objects are all triples $(K, a, (m_i)_{i \in I})$ consisting of a $J$-covers $ (m_i: K \rightarrow K_i)_{i \in I} $ inducing $(n_i)_{i \in I}$ and morphisms $(m_i)_{i \in I} \rightarrow (m'_i)_{i \in I}$ are morphisms between lifts exhibiting the codomain lift as a pushout :
\[\begin{tikzcd}
	& K && {K_i} \\
	{K'} && {K'_i} \\
	& B && {B_i}
	\arrow["a"{description, pos=0.3}, from=1-2, to=3-2]
	\arrow["{a'}"', from=2-1, to=3-2]
	\arrow["{m_i}", from=1-2, to=1-4]
	\arrow["{m_{i*}u}"{description}, from=1-4, to=2-3]
	\arrow["u"', from=1-2, to=2-1]
	\arrow["{n_i}"', from=3-2, to=3-4]
	\arrow["{m'_{i*}a'}"{description}, from=2-3, to=3-4]
	\arrow["{m_{i*}a}"{description}, from=1-4, to=3-4]
	\arrow[""{name=0, anchor=center, inner sep=0}, crossing over, "{m'_i}"{description, pos=0.7}, from=2-1, to=2-3]
	\arrow["\lrcorner"{anchor=center, pos=0.125, rotate=180}, draw=none, from=2-3, to=1-2]
	\arrow["\lrcorner"{anchor=center, pos=0.125, rotate=180}, draw=none, from=3-4, to=0]
\end{tikzcd}\]
The functor $ J \downarrow (n_i)_{i \in I}$ sending a lift $ (K, a, (m_i)_{i \in I})$ to the underlying lift $ a : K \rightarrow B$ is clearly cofinal in the canonical cocone $\mathbb{T}_[\mathcal{S}]_\omega \downarrow B$. Similarly, for each $i$, the functor sending $ (K, a, (m_i)_{i \in I})$ to the corresponding pushout projection $ m_{i*}a : K_i \rightarrow B_i$ is cofinal in $ \mathbb{T}_[\mathcal{S}]_\omega \downarrow B_i $. Now commutativity of finite limits and filtered colimits ensures that 
\[ \underset{J \downarrow (n_i)_{i \in I}}{\colim} \, \underset{i \in I}{\lim} \, K_i \simeq \underset{i \in I}{\lim} \, \underset{J \downarrow (n_i)_{i \in I}}{\colim} \, K_i \]
But on one hand one has  
\begin{align*}
    \underset{J \downarrow (n_i)_{i \in I}}{\colim} \, \underset{i \in I}{\lim} \, K_i &\simeq   \underset{J \downarrow (n_i)_{i \in I}}{\colim} \, K \simeq B
\end{align*}
by cofinality in the canonical cocone, and similarly on the other hand
\[ \underset{i \in I}{\lim} \, \underset{J \downarrow (n_i)_{i \in I}}{\colim} \, K_i \simeq \underset{i \in I}{\lim} \, B_i \]
This proves that extended $J$-covers exhibit their domain as the limit of their nerve.\end{proof}

\section{Fibered sites and fibered topoi}

Before describing the spectral site of a general modelled topos -- generalizing the case of a mere set-valued model -- we chose to devote some time to introduce a special enhancement of \cite{sga4}[VI, part 7] notion of \emph{fibered site} and \emph{fibered topos}. Fibered sites are begotten when applying the Grothendieck construction to indexed categories returning sites as fibers; it is shown there that the topos of sheaves over a fibered site coincides with the category of cartesian sections of a corresponding direct fibration in Grothendieck topoi and direct image functors. In this source are considered fibered sites and topoi above base category without topology; we shall propose in this section an adaptation for the case of a topology on the base category, as it shall be used in the next section: we will recognize indeed the spectral site of a modelled topos as a fibered site over a definition site for the underlying topos. Though this section is not strictly required to understand the construction and can be skipped by the hurry reader, we think it is worth of interest as being close to the construction though which one externalize an internal site in a topos. 

\subsection{Fibered sites as oplax colimits}

\begin{definition}
A \emph{fibered lex site}\index{fibered!lex site} on a small category $ \mathcal{C}$ is an indexed category $ V : \mathcal{C}^{\op} \rightarrow \Cat$ such that in each $ c \in \mathcal{C}$, $V_c$ is lex and equipped with a Grothendieck topology $ J_c$ such that for each $s : c_1 \rightarrow c_2$ the corresponding transition functor $ V_s : V_{c_2} \rightarrow V_{c_1}$ is a morphism of lex site. 
\end{definition}

For a fibered site on small, lex category $\mathcal{C}$, one can consider the Grothendieck construction
\[ \int V \stackrel{p_V}{\longrightarrow} \mathcal{C} \]
at the indexed category $ V$, which is the oplax colimit of $V$, equipped with the canonical oplax cocone

\[\begin{tikzcd}
	{V_{c_1}} \\
	&& {\int V} \\
	{V_{c_2}}
	\arrow["{\iota_{c_1}}", ""{name=0, swap}, from=1-1, to=2-3, hook]
	\arrow["{\iota_{c_2}}"', ""{name=1}, from=3-1, to=2-3, hook]
	\arrow["{V_s}", from=3-1, to=1-1]
	\arrow[Rightarrow, "{\phi_s}"', from=0, to=1, shorten <=4pt, shorten >=4pt]
\end{tikzcd}\]

\begin{lemma}\label{Oplax colimit is lex}
For a fibered lex site $ V: \mathcal{C}^{\op} \rightarrow \Cat$ over a lex category $\mathcal{C}$, the oplax colimit $ \int V $ is lex, as well as the fibration $ p_V$.
\end{lemma}
\begin{proof}
The finite limit of a finite diagram $ F : I \rightarrow \int V$ is constructed as follows: first take the limit $ \lim_{i \in I} p_V(F(i)) $ in $\mathcal{C}$; then, we have a pseudococone diagram in $\mathcal{C}$
\[\begin{tikzcd}
	&& {V_{p_V(F(i))}} \\
	{} &&&& {V_{\underset{i \in I}{\lim} \, p_V(F(i))}} \\
	&& {V_{p_V(F(j))}}
	\arrow["{V_{p_V(F(d))}}", from=3-3, to=1-3]
	\arrow["{V_{p_i}}", ""{name=0, pos=0.3}, from=1-3, end anchor=170, to=2-5]
	\arrow["{V_{p_j}}"', ""{name=1, pos=0.3, swap}, from=3-3, to=2-5]
	\arrow[Rightarrow, "{\simeq}"', from=0, to=1, shorten <=3pt, shorten >=3pt, phantom, no head]
\end{tikzcd}\]
producing a finite diagram $(V_{p_i}(F(i)))_{i \in I}$ in $ V_{\underset{i \in I}{\lim} \, p_V(F(i))}$, where one can take the limit $ \lim_{i \in I} V_{p_i}(F(i))$. Then we have \[ \lim_I \, F \simeq (\lim_{i \in I} p_V(F(i)), \, \lim_{i \in I} V_{p_i}(F(i))) \] \end{proof}
 
\begin{remark}
In the case of the oplax colimits we need the indexing category to have itself finite limits; when considering pseudocolimits, it will be sufficient to require it to be cofiltered, as we wont need the underlying cone to be limiting anymore thanks to localizations along cartesian morphisms.
\end{remark}
 
\begin{division}Now we can equip $ \int V$ with a coarsest topology, the \emph{fibered topology}\index{fibered!topology}, making the inclusions $ \iota_c$ lex continuous: define the topology
\[ J_V = \langle \bigcup_{c \in \mathcal{C}} \iota_c(J_c) \rangle \]
as generated by the inclusion of all fiber topologies, then trivially each $ \iota_c : (V_c, J_c) \hookrightarrow (\int V, J_V)$ is a morphism of lex sites; the covers in the $ \iota_c(J_c)$ are called \emph{horizontal families}\index{cover!vertical}. \end{division}

\subsection{Fibered topos and topos of sections}

Now one can ask for a fibered site as seen above what is corresponding notion of sheaf topos. This is the purpose of the following notion

\begin{definition}
A \emph{fibered topos}\index{fibered!topos} on a category $\mathcal{C}$ is an indexed category $ \mathcal{E}_{(-)} : \mathcal{C}^{\op} \rightarrow \Cat$ such that for any $c \in \mathcal{C}$ the fiber $ \mathcal{E}_c$ is a Grothendieck topos, and for each $ s : c_1 \rightarrow c_2$ the transition functor $ \mathcal{E}_s : \mathcal{E}_{c_2} \rightarrow \mathcal{E}_{c_1}$ is the inverse image part of a geometric functor $ f_s : \mathcal{E}_{c_1} \rightarrow \mathcal{E}_{c_2}$. 
\end{definition}

Then we can also consider the Grothendieck construction at a fibered topos and define the fibration
\[ \int \mathcal{E}_{(-)} \stackrel{p_\mathcal{E}}{\longrightarrow} \mathcal{C} \]

\begin{remark}
Observe that we consider here inverse image part, so that this Grothendieck construction is in particular an oplax colimit \emph{in $\Cat$} of the underlying categories of the fiber topoi.
\end{remark}

\begin{division}
Now, to a fibered site $V$, we can canonically associate a fibered topos whose fiber at $ c$ is the sheaf topos $ \Sh(V_c, J_c) $ and the transition at $s : c_1 \rightarrow c_2$ is the inverse image of the geometric morphism $ \Sh(V_s) : \Sh(V_{c_1}, J_{c_1}) \rightarrow \Sh(V_{c_2},J_{c_2})$ induced by $ V_s$.
Then we can consider the Grothendieck construction associated to the fibered topos
\[ \int \Sh(V_{(-)}, J_{(-)}) \stackrel{\overline{p_V}}{\longrightarrow} \mathcal{C} \]
\end{division}

\begin{definition}
For a Grothendieck fibration $ p: \mathcal{M} \rightarrow \mathcal{C}$, we denote as $ \Gamma(p)$ the category whose objects are \emph{sections}\index{section!of a fibration}  
\[\begin{tikzcd}
	& {\mathcal{M}} \\
	{\mathcal{C}} && {\mathcal{C}}
	\arrow["{x}", from=2-1, to=1-2]
	\arrow["{p}", from=1-2, to=2-3]
	\arrow[Rightarrow, from=2-1, to=2-3, no head]
\end{tikzcd}\]
and whose arrows are natural transformations between those sections. In particular, \emph{cartesian sections}\index{section!cartesian} are sections $ x$ sending any arrow $ s : c_1 \rightarrow c_2$ to a cartesian morphism $ x(s) : x(c_1) \rightarrow x(c_2)$. 
\end{definition}

\begin{division}
Moreover, observe that any fibered topos $ \mathcal{E}_{(-)} : \mathcal{C} \rightarrow \Cat$ defines also a Grothendieck fibration on $\mathcal{C}^{\op}$ thanks to the adjunctions $ f_s^* \dashv {f_s}_*$, where the fiber at $ c$ is still $ \mathcal{E}_c$ but the transition morphism at $ s : c_1 \rightarrow c_2$ is now the direct image functor $ {f_s}_* : \mathcal{E}_{c_1} \rightarrow \mathcal{E}_{c_2}$. We denote as 
\[\begin{tikzcd}
	{\displaystyle{\int} \mathcal{E}'_{(-)}} & {\mathcal{C}^{\op}}
	\arrow["{p'_\mathcal{E}}", from=1-1, to=1-2]
\end{tikzcd}\] the associated fibration. In the following we call this fibration the \emph{direct fibration}\index{direct fibration} of the fibered topos $ \mathcal{E}_{(-)}$.\\

Then in particular a section $ X : \mathcal{C} \rightarrow \int \mathcal{E}_{(-)}'$ of the direct fibration of a fibered topos returns at each object $ c$ an object $X_c$ of the topos $ \mathcal{E}_c$ and at an arrow $ s : c_1 \rightarrow c_2$ a morphism $ (s, X_s) : (c_2, X_{c_2}) \rightarrow (c_1, X_{c_1})$ with $ X_s : X_{c_2} \rightarrow {f_s}_*X_{c_1}$ with $ {f_s}_*$ the direct image of the transition geometric morphism $ \mathcal{E}_s$. 
\end{division}

The following, which is \cite{sga4}[Proposition 7.4.7], says that the sheaf topos over the indexed site is the category of sections of the direct fibration associated to the fibered topos constructed by sheafification of the fibers. 

\begin{proposition}
Let $ V : \mathcal{C} \rightarrow \Cat$ a fibered lex site over a small lex category; then we have an equivalence of categories
\[ \Sh(\int V, J_V) \simeq \Gamma(\overline{p_V}') \]
\end{proposition}

\begin{division}Moreover, sheafification turns the oplax cone made of the inclusions $ (\iota_c)_{c \in \mathcal{C}}$ into an oplax cocone of direct image part of geometric morphisms 
\[\begin{tikzcd}
	{\Sh(V_{c_1},J_{c_1})} \\
	&& {\Sh(\int V,J_{V})} \\
	{\Sh(V_{c_2},J_{c_2})}
	\arrow["{\Sh(i_{c_1})}"', ""{name=0}, from=2-3, to=1-1]
	\arrow["{\Sh(i_{c_2})}", ""{name=1, swap}, from=2-3, to=3-1]
	\arrow["{\Sh(V_s)}"', from=1-1, to=3-1]
	\arrow[Rightarrow, "{{\phi_s}_*}", from=1, to=0]
\end{tikzcd}\]
where $ {\phi_s}_*$ is the mate of the transformation $ \phi_s^* : \Sh(\iota_{c_1})^* \Sh(V_s)^* \Rightarrow  \Sh(\iota_{c_2})^* $ induced from the cocone $ \phi_s$. In the bicategory of Grothendieck topoi, where 2-cells direction is the one of their inverse image part, this becomes a lax cone. Then we also have the following:\end{division}

\begin{proposition}
We have in the bicategory of Grothendieck topoi that
\[ \Sh(\int V, J_V) \simeq \underset{c \in \mathcal{C}}{\laxlim} \,\Sh(V_{c},J_{c})  \]
\end{proposition}

\begin{proof}
It would actually be expectable for oplaxcolimits of lex sites to be turned into oplaxlimits of topoi, as well as finite colimits of sites are turned into finite limits and filtered colimits into cofiltered limits. To see this, observe that a lax cone 
\[\begin{tikzcd}
	{\Sh(V_{c_1},J_{c_1})} \\
	&& {\mathcal{E}} \\
	{\Sh(V_{c_2},J_{c_2})}
	\arrow["{f_{c_1}}"', ""{name=0}, from=2-3, to=1-1]
	\arrow["{f_{c_2}}", ""{name=1, pos=0.51, swap}, from=2-3, to=3-1]
	\arrow["{\Sh(V_s)}"', from=1-1, to=3-1]
	\arrow[Rightarrow, "{{\psi_s}}"', from=0, to=1]
\end{tikzcd}\]
is the same as an lax cocone of lex-continuous functors 
\[\begin{tikzcd}
	{(V_{c_1},J_{c_1})} \\
	&& {\mathcal{E}} \\
	{(V_{c_2},J_{c_2})}
	\arrow["{f_{c_1}^*}", ""{name=0, swap}, from=1-1, to=2-3]
	\arrow["{f_{c_2}^*}"', ""{name=1, pos=0.49}, from=3-1, to=2-3]
	\arrow["{V_s}", from=3-1, to=1-1]
	\arrow[Rightarrow, "{{\psi_s}^*}"', from=0, to=1]
\end{tikzcd}\]
which factorizes uniquely in $\Cat$ through the oplax colimit $ f^* : \int V \rightarrow \mathcal{E}$, and this functor is both lex and $J_V$ continuous since all its restrictions at fibers are lex continuous, so that it defines a geometric morphism $ f : \mathcal{E} \rightarrow \Sh(\int V, J_V)$. 
\end{proof}

\begin{division}
From the construction above of the oplax limit, we can construct the pseudolimit in $\Cat$ of a fibered topos. It can be shown that the cartesian morphisms in the oplax colimit $ \int V$ have a left calculus of fractions in the sense of Gabriel-Zisman, as explained in \cite{sga4}[Proposition 6.4], so that we can consider the localization of the oplax colimit $ \int V$ at the cartesian morphisms; we know that this localization is the pseudocolimit of the pseudofunctor $V$ in $\Cat$, that is,
\[ \pscolim \; V \simeq \int V[ \Sigma_V^{-1}] \]
where $ \Sigma_V$ denotes the class of cartesian morphisms. Moreover, the topology $ J_V$ is transferred to the localization -- we still denote the induced topology as $J_V$. This provides the following expression of pseudolimits:\end{division}

\begin{proposition}\label{cofiltered pseudolimit of topos}
If $V $ is a fibered site on a small lex category $\mathcal{C}$, we have in the bicategory of Grothendieck topoi that \[ \Sh( \underset{\mathcal{C}}{\pscolim}\;  V, J_V) \simeq \underset{c \in \mathcal{C}}{\pslim} \; \Sh(V_c, J_c)   \]
Moreover the latter is also equivalent to the category $\Gamma_{\textbf{Cart}}(p'_V)$ of cartesian sections of the direct fibration.
\end{proposition}

\begin{remark}
In \cite{sga4}[Theorem 8.2.3] the underlying category is just supposed to be cofiltered for the equivalence above to hold; in our case this condition is automatically satisfied as we supposed $ \mathcal{C}$ to be lex. Beware that the results above are not necessarily true for an arbitrary small category $\mathcal{C}$. In fact, it is not known whether the bicategory of Grothendieck topoi has arbitrary small pseudolimits. 
\end{remark}

As this is the content of \cite{sga4}[Sections 6 and 8, in particular 8.2.3] (and has also its bilimit version for bifiltered diagrams at \cite{dubuc2011construction}[Theorem 2.4]) we do not prove it again. 

\subsection{Topos of continuous sections}

We now focus rather on the following adaptation in the case where the underlying category is endowed with a Grothendieck topology, asking for a way to make the fibration in a fibered site a comorphism of sites.\\

\begin{division}

First observe that any fibered lex site has a terminal section $ 1_{(-)} : \mathcal{C} \rightarrow \int V$ associating to each $ c$ the terminal object $1_{V_c}$ of $ V_c$ and to each $ s: c_1 \rightarrow c_2$. In particular this is a cartesian section as the transitions morphisms $ V_s$ for each $ s: c_1 \rightarrow c_2$ are lex so that $ V_s(1_{V_{c_2}}) = 1_{V_{c_1}}$, so the value of this section at $ s: c_1 \rightarrow c_2$ is the cartesian lift $ \overline{s} : (c_1, V_s(1_{V_{c_2}}) \rightarrow (c_2; 1_{V_{c_2}})$. Now for a topology $J$ on $ \mathcal{C}$, and $ \mathcal{V} : \mathcal{C}^{\op} \rightarrow \Cat$ a fibered site, observe that any covering family $ (c_i \rightarrow c)_{i \in I}$ can be lifted to a family $ (\overline{s}_i : (c_i, 1_{V_{c_i}})  \rightarrow (c, 1_{V_c} ))_{i \in I}$: we call such families \emph{horizontal}\index{cover!horizontal}. Then consider the topology
\[ J_{V,J} = \langle J_V \cup 1_{V_{(-)}}(J) \rangle \]
where $ 1_{(-)}(J)$ consists of all families of the form $ (\overline{s}_i : (c_i, 1_{V_{c_i}})  \rightarrow (c, 1_{V_c} ))_{i \in I}$. 
\end{division}

\begin{remark}
In the context of geometries, we consider pretopologies on the fibers, and want actually to manipulate covers of a spectral pretopology rather than sieves of a topology: then it is worth precising that the result above can be rephrased in term of pretopology: if the fibers $V_c$ are equipped with a Grothendieck pretopology $J'_c$, and if one chooses a basis $ J_0$ for the topology $J$ on the indexing category $ \mathcal{C}$ (for instance the maximal basis associated to $J$), then we can first generate a pretopology $ J'_{V,J}$ on the oplaxcolimit from the image of those basis $\bigcup_{c \in \mathcal{C}} \iota_c(J'_c) $ and $ 1_{V_{(-)}}(J_0)$, which are the vertical and horizontal covers. Then observe that \begin{itemize}
    \item the pretopology jointly generated by the fiber basis is itself a basis for the topology on the oplax colimit, or even more directly:
    \[ J_V = \langle  \bigcup_{c \in \mathcal{C}} \iota_c(J'_c) \rangle \]
    \item the pretopology jointly generated by the images $ \iota_c(J'_c) $ and $ 1_{V_{(-)}}(J_0)$ is a basis for $J_{V,J}$:
    \[ J_{V,J} = \langle J'_{V,J} \rangle = \langle \bigcup_{c \in \mathcal{C}} \iota_c(J'_c) \cup 1_{V_{(-)}}(J_0) \rangle \]
\end{itemize}
\end{remark}

We used the terminal element to canonically lift $J$-covers in $\mathcal{C}$ to covers in $ \int V$; but once the topology is generated from those data, we get actually horizontal covers by lifting $ J$-covers at any object in a fiber. To see this, use the following general lemma expressing that cartesian lifts of an arrow form altogether a cartesian transformation:

\begin{lemma}\label{lifts as pullback}
Let $V : \mathcal{C}^{\op} \rightarrow \Cat$ be a lex fibration. Then for any $ s : c_1 \rightarrow c_2$ the following square is a pullback
\[\begin{tikzcd}
	{(c_1, V_s(a_1))} & {(c_2, a)} \\
	{(c_1, 1_{V_{c_1}})} & {(c_2, 1_{V_{c_2}})}
	\arrow["{(s, 1_{V_s(a_2)})}"', from=2-1, to=2-2]
	\arrow["{(1_{c_2}, !_{V_s(a)})}"', from=1-1, to=2-1]
	\arrow["{(1_{c_2}, !_{a})}", from=1-2, to=2-2]
	\arrow["{(s, 1_{V_s(a_1)})}", from=1-1, to=1-2]
	\arrow["\lrcorner"{very near start, rotate=0}, from=1-1, to=2-2, phantom]
\end{tikzcd}\]
\end{lemma}

\begin{proof}
In any other square 
\[\begin{tikzcd}[column sep=large]
	{(c_3, a')} & {(c_2, a)} \\
	{(c_1, 1_{V_{c_1}})} & {(c_2, 1_{V_{c_2}})}
	\arrow["{(s, 1_{V_s(a_2)})}"', from=2-1, to=2-2]
	\arrow["{(u, !_{a'})}"', from=1-1, to=2-1]
	\arrow["{(1_{c_2}, !_{a})}", from=1-2, to=2-2]
	\arrow["{(t, f)}", from=1-1, to=1-2]
	\end{tikzcd}\]
the vertical component of the left map is forced to be the terminal map $ !_{a'}$ as $V_u$ preserves the terminal element. But then, such a square is the same as a situation testing the cartesianness of the lift $ (s, 1_{V_s(a)})$ which always produces a unique map as the dashed arrow below
\[\begin{tikzcd}[row sep=small]
	& {(c_1,V_s(a))} & {(c_2, a)} \\
	{(c_3,a_3)} \\
	& {c_1} & {c_2} \\
	{c_3}
	\arrow["{u}", from=4-1, to=3-2]
	\arrow["{s}", from=3-2, to=3-3]
	\arrow["{t}"', from=4-1, to=3-3, bend right=15]
	\arrow["{(s, 1_{V_s(a)})}", from=1-2, to=1-3]
	\arrow["{(t,f)}"', from=2-1, to=1-3, bend right=15]
	\arrow[from=2-1, to=1-2, dashed]
\end{tikzcd}\]
which provides in particular the desired factorization.
\end{proof}

As a consequence, from we know that Grothendieck topologies are closed under pullback of covering families, it appears the following:

\begin{corollary}
Let $ V$ be a fibered site on a lex category $\mathcal{C}$ and $J$ a Grothendieck topology on $\mathcal{C}$. Then for any $J$-cover $(s_i: c_i \rightarrow c)_{i \in I}$ and any $ a \in V_c$, the family $ ((s_i, 1_{V_{s_i}(a)}) : (c_i, V_{s_i}(a)) \rightarrow (c, a))_{i \in I}$ is a covering family in $ J_{V,J}$.
\end{corollary}

As defined above, the topology $ J_{V,J}$ is the simplest that allows to induce a geometric morphism toward $ \Sh(C,J)$. In fact:

\begin{proposition}
The topology $ J_{V,J}$ is the coarsest topology such that we have simultaneously the two following conditions: \begin{itemize}
    \item for each $ c $ in $\mathcal{C}$, the inclusion $ \iota_c : (V_{c_1}, J_{c_1}) \rightarrow (\int V, J_{V,J})$ is a morphism of sites;
    \item the fibration $ p_V : (\int V, J_{V,J}) \rightarrow (\mathcal{C}, J)$ is a comorphism of sites.
\end{itemize}
\end{proposition}

Moreover, as this topology refines the topology $ J_V$, there is a corresponding inclusion of topos
\[ \Sh(\int V, J_{V, J}) \hookrightarrow \Sh(\int V, J_V) \]
exhibiting objects of $\Sh(\int V, J_{V, J})$ as a specific kind of sections, because the right topos was the topos of sections. The intuition is that some notion of continuity relative to the base topology is involved. We introduce the following notion to this end:

\begin{definition}
Let $ \mathcal{E}_{(-)}$ be a fibered topos on a lex site $ (\mathcal{C},J)$. Then by a \emph{continuous section of }\index{section!continuous} $\mathcal{E}_{(-)}$ we mean a section of the associated direct fibration 
\[\begin{tikzcd}
	& {\int \mathcal{E}_{(-)}'} \\
	{\mathcal{C}^{\op}} && {\mathcal{C}^{\op}}
	\arrow["{X}", from=2-1, to=1-2]
	\arrow["{p_{\mathcal{E}}'}", from=1-2, to=2-3]
	\arrow[Rightarrow, from=2-1, to=2-3, no head]
\end{tikzcd}\]
such that for any $J$-covering family $ (s_i :c_i \rightarrow c)_{i \in I}$ the lifting $ (X_{s_i} : X_{c} \rightarrow {f_{s_i}}_*X_{c_i} )_{i \in I}$ exhibits $X_c$ a limit in the fiber $\mathcal{E}_c$ of the diagram 
\[ X_c = \underset{i \in I}{\lim} \, \Big{(} \prod_{i \in I} {f_{s_i}}_*X_{c_i} \rightrightarrows \prod_{i,j \in I} {f_{s_i}}_*{f_{s^i_{ij}}}_*X_{c_{ij}} \Big{)}   \]
where the double arrow is induced from the transitions $({f_{s_i}}_*(X_{s_{ij}^i}) : {f_{s_i}}_*X_{c_i} \rightarrow {f_{s_i}}_*{f_{s^i_{ij}}}_*X_{c_{ij}} )_{i, j \in I}$ over the nerve of the cover $(s_i : c_i \rightarrow c)_{i \in I}$.
We denote as $\Gamma_J(\overline{p_E}')$ the category of continuous sections\index{section!category of continuous} and natural transformations between them.
\end{definition}

\begin{theorem}\label{Topos of continuous sections}
Let $ (\mathcal{C},J)$ be a small lex site, $ V : \mathcal{C}^{\op} \rightarrow \Cat$ a fibered lex site on $\mathcal{C}$ with $ \overline{p_V} : \int \Sh(V_{(-)}, J_{(-)}) \rightarrow \mathcal{C}$ the associated fibered topos and $ \overline{p_V}' : \int \Sh(V_{(-)}, J_{(-)})' \rightarrow \mathcal{C}^{\op}$ the corresponding direct fibration. Then one has an equivalence of categories
\[ \Sh(\int V, J_{V,J}) = \Gamma_J(\overline{p_V}') \]
\end{theorem}

\begin{proof}
For one direction, observe that any sheaf $X$ in $\Sh(\int V, J_{V,J})$ can be composed with the duals of the fiber inclusions 
\[\begin{tikzcd}
	{V_{c_1}^{\op}} \\
	&& {\int V^{\op}} & {\mathcal{S}} \\
	{V_{c_2}^{\op}}
	\arrow["{\iota_{c_1}^{\op}}", ""{name=0, swap}, from=1-1, to=2-3]
	\arrow["{\iota_{c_2}^{\op}}"', ""{name=1}, from=3-1, to=2-3]
	\arrow["{V_s^{\op}}", from=3-1, to=1-1]
	\arrow["{X}", from=2-3, to=2-4]
	\arrow[Rightarrow, "{\phi_s^{\op}}", from=1, to=0]
\end{tikzcd}\]
and as the topology of each fiber is part of the topology $ J_{V,J}$ the restriction $ X \iota_{c}^{\op}$ is a sheaf for the topology $ J_{c}$, hence is an object of the sheaf topos $ \Sh(V_{c},J_c)$; moreover, for each $ s : c_1 \rightarrow c_2$, the diagram below provides us by whiskering with a transformation 
\[ X* \phi_s^{\op} : X \iota_{c_2}^{\op} \rightarrow X \iota_{c_1}^{\op} V_s^{\op} \]
But precomposition with $V_s$ is what the direct image $ \Sh(V_s)_*$ consists in. So we can associate to $X$ the section $ \widehat{ X} : \mathcal{C}^{\op} \rightarrow \int \mathcal{E}$ with $ \widehat{X}_{c} = X \iota^{\op}_c = \Sh(\iota_c)_*X$ and $ \widehat{X}_s = X *\phi_s $. Now we have to check that $ X$ is a continuous section. But we know that $ X$ is a sheaf for the horizontal topology, and also relative to any horizontal cover $ ((s_i, 1_{V_{s_i}(a)}) : (c_i, V_{s_i}(a)) \rightarrow (c, a))_{i \in I}$, which hence is sent on a limit diagram 
\[ \widehat{X}_c(a) = \underset{i \in I}{\lim} \, \Big{(} \prod_{i \in I} \widehat{X}_{c_i}(V_{s_i}(a)) \rightrightarrows \prod_{i,j \in I} \widehat{X}_{c_{ij}}( V_{s_{ij}^i}V_{s_i}(a) ) \Big{)}  \]
where \[V_{s_{ij}^i}V_{s_i}(a) = V_{s_{ij}^j}V_{s_j}(a) \] for each $ i,j \in I$: but from the fact that limits are pointwise in categories of sheaves, this is sufficient to ensure that $ \widehat{X}$ is a continuous section. \\

For the converse, suppose we have a continuous section $ \widehat{X} : \mathcal{C}^{\op} \rightarrow \int \mathcal{E}_{(-)}' $. Then each $X_c$ is a sheaf on $ (V_c, J_c)$; by the property of the oplaxcolimit, the data of the cocone 
\[\begin{tikzcd}
	{V_{c_1}} \\
	&& {\mathcal{S}^{\op}} \\
	{V_{c_2}}
	\arrow["{V_s}", from=3-1, to=1-1]
	\arrow["{X_{c_1}^{\op}}", ""{name=0, swap}, from=1-1, to=2-3]
	\arrow["{X_{c_2}^{\op}}"', ""{name=1}, from=3-1, to=2-3]
	\arrow[Rightarrow, "{X^{\op}_s}"', from=0, to=1]
\end{tikzcd}\]
induces a unique functor $X : \int V^{\op} \rightarrow \mathcal{S} $ whose restriction at each fiber coincides with $ X_c$ and whose whiskering with some $\phi_s^{\op}$ is $ X_s$. Hence $ X$ is a sheaf for the topology $ \langle \bigcup_{c \in \mathcal{C}} J_c\rangle$, with values $ X(c,a) = X_c(a)$ for $ a \in V_c$, and $ X(s,u)$ is obtained as the composite 
\[\begin{tikzcd}
	{X_{c_2}(a_2)} \\
	{X_{c_1}(V_s(a_2))} & {X_{c_1}(a_1)}
	\arrow["{X(s,u)}", from=1-1, to=2-2]
	\arrow["{(X_s)_{a_2}}"', from=1-1, to=2-1]
	\arrow["{X_{c_1}(u)}"', from=2-1, to=2-2]
\end{tikzcd}\]
where $ (X_s)_{a_2}$ is the component at $ a_2$ of the natural transformation $ X_{c_2} \rightarrow \Sh(V_s)_*X_{c_1}$ and $ X_{c_1}(u)$ is the restriction map of $ X_{c_1}$ at $ u : a_1 \rightarrow V_s(a_2)$. But now, being continuous as a section, for any $J$-cover $ (s_i : c_i \rightarrow c)_{i \in I}$ the transitions $ (X_{s_i} : X_c \rightarrow \Sh(V_s)_*X_{c_i})_{i \in I}$ form a limit diagram in $ \Sh(V_c, J_c)$, and as limits are pointwise in categories of sheaves, this means in particular that in $ 1_{V_c}$, the induced horizontal cover $ ((s_i, \overline{s_i}) : (c_i, 1_{V_{c_i}}) \rightarrow ( c, 1_{V_c})$ where $ 1_{V_{c_i}} = V_{s_i}(1_{V_c}))_{i \in I}$ is sent to a limit diagram
\[ X(c,1_{V_c}) = \underset{i \in I}{\lim} \, \Big{(} \prod_{i \in I} X(c_i, 1_{V_{c_i}}) \rightrightarrows \prod_{i,j \in I} X(c_{ij}, 1_{V_{c_{ij}}} ) \Big{)}  \]
which exactly means that $ X$ is a sheaf for the topology $ 1_{V_{(-)}}(J)$. \\

Now it is clear that those two processes are mutually inverse from the fact that $\int V$ is the oplaxcolimit, $X$ and $ \widehat{X} $ are mutually determined as functors. 
\end{proof}

\section{Spectral site of a modelled topos}

In this section we describe the spectral site of an arbitrary $ \mathbb{T}$-modelled topos for a fixed geometry $ (\mathbb{T}, \mathcal{V}, J)$. While this definition of the spectral site was given in \cite{Coste}, we think it is worth giving new details about the geometric and fibrational aspects involved in this construction. As we shall see, the spectral site of a sheaf of $\mathbb{T}[\mathcal{S}]$-models will be constructed as a fibered site made of the spectral site of its values at basic opens of the topos it lives in. Then from the results of the previous section, the spectrum will be exhibited as a topos of global sections of the direct fibration of the fibered topos constituted of the spectra of the local values. We shall also describe the structure sheaf in this context. \\

Previously we proved that, in the case of set-valued models, the functor returning the topos of sheaves over the spectral site together with its structure sheaf was left adjoint to a certain global section functor. In this section, we shall see that the 2-functor associating to an arbitrary modelled topos the topos constructed from the spectral site  together with its structure sheaf also defines a left adjoint, this time to the inclusion of locally modelled topoi. Equivalently, we prove that the spectral topos classifies local units modulo inverse images. 

\subsection{The spectral site as a fibered site}

\begin{division}\label{The spectral site defines an indexed category }
For a Grothendieck topos $ \mathcal{F}$ with a lex site of presentation $ \Sh(\mathcal{C}_\mathcal{F}, J_\mathcal{F})$ and $ F$ in $ \mathbb{T}[\mathcal{F}]$, each local value $ F(c)$ at basic open $c$ of $ \mathcal{C}_\mathcal{F}$ is a set valued $ \mathbb{T}$-model, while for a morphism $s : c_1 \rightarrow c_2$ we have a morphism of $\mathbb{T}$-models $ F(s) : F(c_2) \rightarrow F(c_1)$. This defines a pseudofunctor
\[ \begin{array}{rcl}
    \mathcal{C}^{\op} & \stackrel{\mathcal{V}_{F(-)}^{\op}}{\rightarrow} & \Lex  \\
    c  & \mapsto & \mathcal{V}_{F(c)}^{\op}\\
    c_1 \stackrel{s}{\rightarrow} c_2 & \mapsto & \mathcal{V}_{F(c_2)}^{\op} \stackrel{(s_*)^{\op}}{\rightarrow} \mathcal{V}_{F(c_1)}^{\op}
\end{array} \]
At an arrow $ s : c_1 \rightarrow c_2$ in $\mathcal{C}_\mathcal{F}$, one has a pushout functor 
\[ \mathcal{V}_{F(c_2)} \stackrel{F(s)_*}{\rightarrow} \mathcal{V}
_{F(c_1)}  \]
relating the fiber inclusions through a natural transformation
\[\begin{tikzcd}
	{\mathcal{V}^{\op}_{F(c_2)}} \\
	&& {\mathcal{V}^{\op}_F } \\
	{\mathcal{V}^{\op}_{F(c_1)}}
	\arrow[""{name=0, pos=0.375, inner sep=0}, "{\iota_{c_2}^{\op}}", from=1-1, to=2-3]
	\arrow[""{name=1, pos=0.4, inner sep=0}, "{\iota_{c_1}^{\op}}"', from=3-1, to=2-3]
	\arrow["{F(s)_*^{\op}}"', from=1-1, to=3-1]
	\arrow["{ \phi^s}", shorten <=4pt, shorten >=4pt, Rightarrow, from=1, to=0]
\end{tikzcd}\]
whose component $ \phi^s_n$ at an object $ n : F(c_2) \rightarrow C$ is given by the opposite of the pushout map $  (n_*F(s))^{\op}$ as below

\[ 
\begin{tikzcd}
F(c_2) \arrow[d, "n"'] \arrow[r, "F(s)"] \arrow[rd, "\ulcorner", phantom, near end] & F(c_1) \arrow[d, "F(s)_*n"] \\
C \arrow[r, "n_*F(s)"']                                                   & F(s)_*C                    
\end{tikzcd} \]
The data of the fibers inclusions $ \iota_c : \mathcal{V}^{\op}_{F(c)} \hookrightarrow \mathcal{V}^{\op}_F$ together with those transformations define an oplax cocone exhibiting $ \mathcal{V}^{\op}_F$ as the oplax colimit of the indexed category $ \mathcal{V}^{\op}_{F(-)} : \mathcal{C}_{\mathcal{F}}^{\op} \rightarrow \Cat$. Beware that the cocone made of the etale generator themselves $ \mathcal{V}_{F(-)} : \mathcal{C}^{\op}_{\mathcal{F}} \rightarrow \Cat$ is a \emph{lax cocone}. \end{division}

\begin{remark}
Beware here that the indexed site has as value $\mathcal{V}_{F(c)}^{\op}$ and not $ \mathcal{V}_{F(c)}$ itself. The applications of the previous section will hence require carefull substitution of $V$ with $ \mathcal{V}^{\op}$...
\end{remark}

\begin{definition}
Define the category $ \mathcal{V}_F$ as the following lax colimit in $\Lex$
\[ \mathcal{V}_F = \underset{c \in \mathcal{C}_\mathcal{F}^{\op}}{\laxcolim} \; \mathcal{V}_{F(c)} \]
\end{definition}

\begin{remark}
Observe we have then an oplax colimit
\[ \mathcal{V}_F^{\op} = \underset{c \in \mathcal{C}_\mathcal{F}^{\op}}{\oplaxcolim} \; \mathcal{V}_{F(c)}^{\op} \]
which will be the underlying category of the spectral site. \\

More explicitely, $\mathcal{V}_F$ has as objects the pairs
$ (c,n)$ with $ c \in \mathcal{C}_\mathcal{F}$ and $n$ a morphism in $ \mathcal{V}_{F(c)}$, that is some $ n : F(c) \rightarrow \cod(n) $ gotten as a pushout of map in $ \mathcal{V}$. As morphisms, it has the pairs $ (s,h) : (c_1,n_1) \rightarrow (c_2,n_2)$ with $ s : c_2 \in c_1$ and $ f : \cod(n_1) \rightarrow \cod(n_2)$ related by a commutative square 
\[\begin{tikzcd}
	{F(c_1) } & {F(c_2) } \\
	{\cod(n_1)} & {\cod(n_2)}
	\arrow["{l_1}"', from=1-1, to=2-1]
	\arrow["{F(s)}", from=1-1, to=1-2]
	\arrow["{n_2}", from=1-2, to=2-2]
	\arrow["f"', from=2-1, to=2-2]
\end{tikzcd}\]
\end{remark}

Now recall that each $ \mathcal{V}_{F(c)}$ is part of the spectral site $ (\mathcal{V}_{F(c)}^{\op}, J_{F(c)})$ of $F(c)$. Moreover we have the following:

\begin{proposition}
For each $ s : c_1 \rightarrow c_2$ we have a morphism of lex sites 
\[\begin{tikzcd}
	{(\mathcal{V}_{F(c_2)}^{\op}, J_{F(c_2)})} & {(\mathcal{V}_{F(c_1)}^{\op}, J_{F(c_1)})}
	\arrow["{F(s)_*}", from=1-1, to=1-2]
\end{tikzcd}\]
\end{proposition}
\begin{proof}
For a family in $ \mathcal{V}_{F(c')}$ obtained as a pushout along some $a$ of a covering family $ (n_i : K \rightarrow K_i)$ of finite presentation under $ \cod(n)$, then composition of pushouts ensures that this family is transferred into a covering family of $ F(s)_*n$ in $\mathcal{V}_{F(c_2)}$ as visualised below
\[\begin{tikzcd}
	& {F(c_1) } && {F(c_2) } \\
	K & {\cod(n)} && {F(s)_*\cod(n)} \\
	& {K_i} & {\cod(n_i)} && {F(s)_*\cod(l_i)}
	\arrow[""{name=0, anchor=center, inner sep=0}, "{F(s)}", from=1-2, to=1-4]
	\arrow["{F(s)_*n_i}", from=1-4, to=3-5]
	\arrow["{l_{i*}F(s)}"', from=3-3, to=3-5]
	\arrow["n"', from=1-2, to=2-2]
	\arrow["{a_*m_i}"', from=2-2, to=3-3]
	\arrow["{l_{*}F(s)}"{description}, from=2-2, to=2-4]
	\arrow["{F(s)_*n}"', from=1-4, to=2-4]
	\arrow[from=2-4, to=3-5]
	\arrow[""{name=1, anchor=center, inner sep=0}, "a", from=2-1, to=2-2]
	\arrow[from=3-2, to=3-3]
	\arrow["{m_i}"', from=2-1, to=3-2]
	\arrow["{n_i}"{description, pos=0.3}, from=1-2, to=3-3, crossing over]
	\arrow["\lrcorner"{anchor=center, pos=0.125, rotate=180}, draw=none, from=3-5, to=1-2]
	\arrow["\lrcorner"{anchor=center, pos=0.125, rotate=180}, draw=none, from=2-4, to=0]
	\arrow["\lrcorner"{anchor=center, pos=0.125, rotate=180}, draw=none, from=3-3, to=1]
\end{tikzcd}\]
\end{proof}

\begin{division}
From \cref{Oplax colimit is lex} we know that the oplax colimit $ \mathcal{V}^{\op}_{F(-)} $ inherits finite limits of the fibers $ \mathcal{V}_{F(c)}^{\op}$ and the basis $ \mathcal{C}$: for a finite diagram $((c_i, n_i))_{i \in I} $ one first computes the finite limit $ (p_i : \lim_{i \in I} c_i \rightarrow c_i)_{i \in I}$ in $\mathcal{C}_\mathcal{F}$, which is sent to a cocone $((F(p_i) : F(c_i) \rightarrow F(\lim_{i \in I} c_i))_{i \in I} $. Then pushing the $ n_i$ along the transition functor $ F(p_i) $ defines a finite diagram $ (F(p_i)_*n_i)_{i \in I}$ in $ F(\lim_{i \in I})$, and since $ \mathcal{V}_{F(\lim_{i \in I})}$ has finite colimits, we have
\[  \underset{i \in I}{\lim} \, (c_i, n_i) \simeq ( \underset{i \in I}{\lim} \, c_i, \underset{i \in I}{\colim} \, F(p_i)_*n_i) \]
In particular, pullbacks are computed as follows:
\[\begin{tikzcd}
	{(c_1 \times_c c_2, F(p_1)_*F(u_1)_*n_1 +_{F(p)_*n} F(p_2)_*F(u_2)_*n_2)} & {(c_2,n_2)} \\
	{(c_1,n_1)} & {(c,n)}
	\arrow["{(u_2,f_2)}", from=1-2, to=2-2]
	\arrow[""{name=0, anchor=center, inner sep=0}, "{(u_1,f_1)}"', from=2-1, to=2-2]
	\arrow[from=1-1, to=2-1]
	\arrow[from=1-1, to=1-2]
	\arrow["\lrcorner"{anchor=center, pos=0.125}, draw=none, from=1-1, to=0]
\end{tikzcd}\]
where $F(p_1)_*F(u_1)_*n_1 +_{F(p)_*n} F(p_2)_*F(u_2)_*n_2)$ is the pushout in $ \mathcal{V}_{F(c_1 \times_c c_2)}$ of the span 
\[\begin{tikzcd}
	{F(p)_*n} & {F(p_2)_*F(u_2)_*n_2} \\
	{F(p_1)_*F(u_1)_*n_1}
	\arrow["{F(p)_*\langle n_1, f_1\rangle}"', from=1-1, to=2-1]
	\arrow["{F(p)_*\langle n_2, f_2\rangle}", from=1-1, to=1-2]
\end{tikzcd}\]
\end{division}

\subsection{The spectrum from the fibered spectral site}

\begin{division}
Observe that the fibration $ p_F : \mathcal{V}^{\op}_F \rightarrow \mathcal{C}_{\mathcal{F}}$ has a left adjoint 
\[\begin{tikzcd}
	{\mathcal{C}_\mathcal{F} } & {\mathcal{V}_F^{\op}}
	\arrow["{\iota_{F}}", from=1-1, to=1-2, hook]
\end{tikzcd}\]
sending an object $ c $ to the pair $ (c, 1_{F(c)})$ with the identity arrow $ 1_{F(c)} : F(c) \rightarrow F(c)$, which is etale and finitely presented (and is the terminal object of $ \mathcal{V}^{\op}_{F(c)}$). 
\end{division}

\begin{division}
Applying what was done in the previous section, if we fix $ J_\mathcal{F}$ as a subcanonical pretopology such that $ \mathcal{F} \simeq \Sh(\mathcal{C}_\mathcal{F}, J_\mathcal{F})$, we can equip the oplax colimit $ \mathcal{V}_F$ with the pretopology induced jointly from the spectral pretopologies of the local values and the underlying pretopology, that is,
\[  J_F = \langle \iota_{{F}}(J_\mathcal{F}) \cup \bigcup_{c \in \mathcal{C}_\mathcal{F}} \iota_c(J_{F(c)}) \rangle   \]
In other words this pretopology is the pretopology jointly generated from \emph{horizontal} families of the form
    \[ ((c,1_c) \stackrel{(s_i,F(s_i))}{\longrightarrow} (c_i,1_{c_i}))_{i \in I} \; \textrm{ with } (c \stackrel{s_i}{\rightarrow} c_i)_{i \in I} \in J^{\op}_\mathcal{F}(c)  \]
and \emph{vertical families} of the form
    \[ (( c, n) \stackrel{(1_c,m_i)}{\longrightarrow} (c, n_i))_{ i \in I} \textrm{ with } \Bigg{(} \begin{tikzcd}[row sep=small]
F(c) \arrow[]{r}{n} \arrow[]{rd}[swap]{n_i} & \cod(n) \arrow[]{d}{m_i} \\ & \cod(l_i)  
\end{tikzcd} \Bigg{)}_{ i \in I} \in J^{\op}_{F(c)}(n) \]
\end{division}

\begin{remark}
In general, it can be quite difficult to give an explicit description of a topology generated this way; however the pretopology generated is more easy to handle. In this context we can make the following remarks. Recall that in a fibered category, any map factorizes uniquely as a vertical arrow followed by an horizontal, cartesian arrow: here this corresponds to
\[\begin{tikzcd}
	{(d,m)} && {(c,n)} \\
	& {(d,F(s)_*n)}
	\arrow["{(s,f)}", from=1-1, to=1-3]
	\arrow["{(1_d, \langle f,m\rangle)}"', from=1-1, to=2-2]
	\arrow["{(s,1_{F(s)_*n})}"', from=2-2, to=1-3]
\end{tikzcd}\]
where $ \langle f,m \rangle $ is the map induced from the pushout property.\\

By the pullback axiom of pretopology, observe that if $ (s_i : c_i \rightarrow c)_{i \in I}$ is covering in $J_\mathcal{F}$ and $ n$ is in $\mathcal{V}_{F(c)}$ then the family of cartesian lifts
\[\begin{tikzcd}[sep=large]
	{(c_i, F(s_i)_*n)} & {(c,n)}
	\arrow["{(s_i, 1_{F(s_i)_*n})}", from=1-1, to=1-2]
\end{tikzcd}\]
is a cover for $ J_F$ as being the pullback of the family $ \iota_F((s_i)_{i \in I} $ along the morphism $ (c, n) \rightarrow (c,1_{F(c)})$ by \cref{lifts as pullback}. \\

For $ s : d \rightarrow c$ in $ \mathcal{C}_\mathcal{F}$ and $ n $ in $\mathcal{V}_{F(c)}$, then the identity arrow $ 1_{F(s)_*n}$ is a $ J_{F(d)}$ cover of $ F(s)_*n$ as it is induced by pushing out the identity map of $K$ for any $ a : K \rightarrow F(c)$. Then a cartesian morphism is a covering singleton as soon as $s$ itself is a covering singleton in $J_\mathcal{F}$. \\

Now, for any $ (s_i :c_i \rightarrow c)_{i \in I}$ in $J_{\mathcal{F}}$ and $ (m_j : n \rightarrow n_j)_{j \in I'}$ a cover in $ J_{F(c)}$, then for each $i \in I$ the pushout family $ (F(s_i)_*m_j : F(s_i)_*n \rightarrow F(s_i)_*n_j)_{j \in I'}$ is a cover in $ J_{F(c_i)}$.\\

But now, by transitivity axiom, we have a particular kind of covering families in $J_F$ constructed as follows. Take a covering family $ (s_i : c_i \rightarrow c)_{i \in I}$ in $J_\mathcal{F}(c)$, an object $ n $ in $\mathcal{V}_{F(c)}$, and for each $i \in I$ a covering family $ (m_{ij} : F(s_i)_*n \rightarrow n_{ij})_{i \in I_i}$ in $J_{F(c_i)}$: then the induced family 
\[ ( 
\begin{tikzcd}
	{(c_i, n_{ij})} && {(c,n)}
	\arrow["{(s_i, m_{ij}n_*F(s))}", from=1-1, to=1-3]
\end{tikzcd}  )_{i \in I, j \in I_i}\] 
must be covering in $J_F$. \\

However beware that it may happen that there will be families $ (s_i, d_i)_{i \in I}$ in $ J_F$ whose underlying family $ (s_i)_{i \in I}$ is not in $J_\mathcal{F}$. Otherwise the projection $ p_F$ would be a morphism of site, which is not true in general. 
\end{remark}

\begin{definition}
Define the \emph{spectral site}\index{spectral!site} of the $\mathbb{T}$-modelled topos $(\mathcal{F},F)$ as the site $ (\mathcal{V}_F^{\op}, J_F)$. Then define the \emph{(Coste) spectrum}\index{(Coste) spectrum} as 
\[ \Spec (\mathcal{F}, F) = \Sh(\mathcal{V}_F^{\op}, J_F) \]
\end{definition}

In particular, this construction generalizes the construction of the spectrum of a set valued $\mathbb{T}$-model, as those ones just are the modelled topos of the form $ ( *, B)$ with $ B$ seen as the constant sheaf of $\mathbb{T}$-models on the point.

\begin{division}\label{oplax cone of topos}
It is worth understanding how the spectrum $ \Spec(F)$ and its structure sheaf $ \widetilde{F}$ compare to the spectra $ \Spec(F(c))$ and structure sheaves $ \widetilde{F(c)}$ of the local values. Recall that the oplax colimit $ \mathcal{V}_F^{\op} \simeq \oplaxcolim_{c \in \mathcal{C}_{\mathcal{F}}^{\op}} \mathcal{V}_{F(c)}^{\op}$ is canonically equipped with an induced pretopology 
\[J_{\mathcal{V}^{\op}_{F(-)}} = \langle
\bigcup_{c \in \mathcal{C}_\mathcal{F}} \iota_c(J_{F(c)}) \rangle \]
such that we have a geometric equivalence in the bicategory of Grothendieck topoi
\[  \Sh(\mathcal{V}_F, J_{\mathcal{V}^{\op}_{F(-)}}) \simeq \underset{c \in \mathcal{C}_\mathcal{F}}{\laxlim} \, \Spec(F(c)) \]
Now from the topology $ J_F$ is generated from $ J_{\mathcal{V}^{\op}_{F(-)}}$ together with $ \iota_F(J_\mathcal{F})$ we have a geometric inclusion
\[\begin{tikzcd}
	{\Spec(F)} & {\underset{c \in\mathcal{C}_\mathcal{F}}{\laxlim} \, \Spec(F(c))}
	\arrow["{\mathfrak{i}_F}", hook, from=1-1, to=1-2]
\end{tikzcd}\]
whose inverse image part is the sheafification functor $ \mathfrak{a}_{\iota_F(\mathcal{F})}$: intuitively, this functor ``corrects" sheaves for the generated pretopology $J_{\mathcal{V}^{\op}_{F(-)}}  $ into sheaves for the spectral pretopology $J_F$. \\

Precomposing $ \mathfrak{i}_F$ together with the limiting projections
\[\begin{tikzcd}
	{\underset{c \in \mathcal{C}_\mathcal{F}}{\laxlim} \, \Spec(F(c))} & { \Spec(F(c)) }
	\arrow["{p_c }", from=1-1, to=1-2]
\end{tikzcd}\]
produces a lax cone 
\[\begin{tikzcd}
	({\Spec(F)} & {\Spec(F(c))})_{c \in \mathcal{C}_\mathcal{F}}
	\arrow["{p_c \mathfrak{i}_F}", hook, from=1-1, to=1-2]
\end{tikzcd}\]
where $ p_c\mathfrak{i}_F = \Sh(\iota_c)$. But now recall that $ \Sh(F(s)_*) = \Spec(F(s)) $, seeing $ F(s) : F(c_2) \rightarrow F(c_1)$ as a morphism in $\mathbb{T}[\mathcal{S}]$: hence the transition morphisms of this cone are given as
\[\begin{tikzcd}
	& {\Spec(F(c_1))} \\
	{\Spec(F)} \\
	& {\Spec(F(c_2))}
	\arrow[""{name=0, anchor=center, inner sep=0}, "{\Sh(\iota_{c_1})}", from=2-1, to=1-2]
	\arrow["{\Spec(F(s))}", from=1-2, to=3-2]
	\arrow[""{name=1, anchor=center, inner sep=0}, "{\Sh(\iota_{c_2})}"', from=2-1, to=3-2]
	\arrow["{\phi_s}", shift right=2, shorten <=4pt, shorten >=4pt, Rightarrow, shift left=2, from=0, to=1]
\end{tikzcd}\]
\end{division}

\subsection{The spectrum as a topos of continuous sections and the canonical fibration}

\begin{division}
Now we turn to the relation between the spectrum and its base topos. We have a fibered lex site $ \mathcal{V}_{F(-)}^{\op} : \mathcal{C}^{\op} \rightarrow \Lex $. 
Now recall that each of the sites $ (\mathcal{V}^{\op}_{F(c)}, J_{F(c)})$ is the spectral site of $F(c)$, that is, $ \Spec(F(c)) = \Sh(\mathcal{V}^{\op}_{F(c)}, J_{F(c)})$. Then the associated fibered topos has the spectrum of the $F(c)$ 
as fiber, that is, we have a fibered topos
\[\begin{tikzcd}
	{\displaystyle\int \Spec(F(-))} & { \mathcal{C}}
	\arrow["{\overline{p_F}}", from=1-1, to=1-2]
\end{tikzcd}\]
whose transition morphisms are the inverse image $ \Spec(F(s))^* : \Spec(F(c_2)) \rightarrow \Spec(F(c_1))$ for $ s : c_1 \rightarrow c_2$. We can consider its associated direct fibration 
\[\begin{tikzcd}
	{\displaystyle\int \Spec(F(-))'} & { \mathcal{C}^{\op}}
	\arrow["{\overline{p_F}'}", from=1-1, to=1-2]
\end{tikzcd}\]
\end{division}

Recognize now in the definition of the topology $J_F$ that it is generated from the horizontal and vertical families as done in the previous section. Applying \cref{Topos of continuous sections} we have the following:

\begin{theorem}\label{The spectrum as the topos of continuous sections}
The spectrum of $F$ is the topos of continuous sections of the direct fibration associated to the fibered spectral topos:
\[  \Spec(F) \simeq \Gamma_J(\overline{p_F}') \]
\end{theorem}

\begin{division}
We saw above that the fibration of site $ p_F$ is equipped here with a left adjoint $\iota_F$; moreover from the way $J_F$ was defined, we know trivially $ \iota_F$ to define a morphism of site. Conversely the projection $ p_F$, though not being a morphism of site, is a comorphism of site. Hence we are in a situation
\[ \begin{tikzcd}
                           (\mathcal{V}^{\op}_F, J_F) \arrow[rr, phantom, "\perp"] \arrow[rr, "p_F", bend left=20] && (\mathcal{C}_\mathcal{F}, J_\mathcal{F}) \arrow[ll, "\iota_F", bend left=20]
\end{tikzcd}   \]
with a comorphism of sites right adjoint to a morphism of sites. Hence from \cite{caramello2020denseness}[Proposition 3.14] or also \cite{maclane&moerdijk}[Theorem VII.10.5] they both induce a same  
geometric morphism 
\[\begin{tikzcd}
	{\Spec(F)} & {\mathcal{F}}
	\arrow["{h_{F}}", from=1-1, to=1-2]
\end{tikzcd}\] 
which moreover lifts to an adjoint functor between categories of $ \mathbb{T}$-models
\[\begin{tikzcd}
	{\mathbb{T}[\Spec(F)]} && {\mathbb{T}[\mathcal{F}]}
	\arrow[""{name=0, anchor=center, inner sep=0}, "{h_{F*}}"', curve={height=12pt}, from=1-1, to=1-3, start anchor=350, end anchor=200]
	\arrow[""{name=1, anchor=center, inner sep=0}, "{h^*_F}"', curve={height=12pt}, from=1-3, to=1-1, start anchor=160, end anchor=10]
	\arrow["\dashv"{anchor=center, rotate=-90}, draw=none, from=1, to=0]
\end{tikzcd}\]
with the inverse and direct part respectively given by
\[ h_{F}^* = \mathfrak{a}_{J_F}((-) \circ {p_F^{\op}}) = \lan_{\iota_F^{\op}} (-) \] 
\[h_{F*} = \mathfrak{a}_{J_\mathcal{F}}( \mathcal{C}_\mathcal{F}(p_F^{\op}, -)) = (-) \circ \iota_F^{\op} \] 
Beware that in the computation of the inverse image $ h_F^*$, precomposing a $J_\mathcal{F}$-sheaf with the comorphism $p_F^{\op}$ does not return a $J_F$ sheaf (since the vertical families have no reason to be taken into account by this process), hence the necessity of the further sheafification. Observe that we must op the functors $ \iota_F$ and $ p_F$ as the involved sheaves of $ \mathbb{T}$-objects have $ \mathcal{V}_F$ and $ \mathcal{C}^{\op}$ as domains. 
\end{division}

\subsection{The structure sheaf}

Again, the structure sheaf is obtained as the sheafification of the codomain functor: observe this latter is the functor
\[\begin{tikzcd}
	{\mathcal{V}_{F}} & {\mathbb{T}[\mathcal{S}]}
	\arrow["\cod_F", from=1-1, to=1-2]
\end{tikzcd}\]
defined as acting fiberwisely as the codomain functor, that is, sending $ (c,n)$ on $\cod_c(n)$, seeing $n$ as an object of $\mathcal{V}_{F(c)}$, and $(s,f) : (c_1,n_1) \rightarrow (c_2,n_2)$ as the underlying map $ f : \cod(n_1) \rightarrow \cod(n_2)$. In other word $\cod_F$ is the functor induced by the laxcolimit property of $\mathcal{V}_F$ from the laxcocone 
\[\begin{tikzcd}
	{\mathcal{V}_{F(c_2)}}  \\
	&& {\mathbb{T}[\mathcal{S}]}  \\
	{\mathcal{V}_{F(c_1)}} 
	\arrow["{F(s)_*}"', from=1-1, to=3-1]
	\arrow[""{name=0, anchor=center, inner sep=0}, "{\cod_{c_2}}", from=1-1, to=2-3]
	\arrow[""{name=1, anchor=center, inner sep=0}, "{\cod_{c_1}}"', from=3-1, to=2-3]
	\arrow["{{\phi_s}}"', shorten <=4pt, shorten >=4pt, Rightarrow, shift right=2, from=0, to=1]
\end{tikzcd}\]

\begin{definition}
The \emph{structure sheaf} of the modelled topos $ (\mathcal{F},F)$ is the sheaf
\[ \widetilde{F} = \mathfrak{a}_{J_F}\cod_F \]
obtained as the sheafification of the presheaf of $ \mathbb{T}$-objects defined on $ \mathcal{V}_F$ as the codomain functor 
\[ \overline{F} : (c,n) \mapsto \cod(n) \]
where we denote as $\gamma : \cod \rightarrow \mathfrak{a}_{J_F}\cod$ the unit of the sheafification at $\cod$.  
\end{definition}

\begin{remark}
Hence this structure sheaf is obtained as applying successively the sheafification of vertical, then horizontal families. 
\end{remark}

\begin{division}
Now at the level of the structure sheaves, observe that the data of all the $ (\widetilde{F(c)})_{c \in \mathcal{C}_\mathcal{F}} $ induce a lax cocone 
\[\begin{tikzcd}
	{(\mathcal{V}_{F(c)}} & {\mathbb{T}[\mathcal{S}])_{c \in \mathcal{C}_\mathcal{F}}}
	\arrow["{\widetilde{F(c)}}", from=1-1, to=1-2]
\end{tikzcd}\]
whose transition 2-cells are given by
\[\begin{tikzcd}
	{\mathcal{V}_{F(c_2)}}  \\
	&& {\mathbb{T}[\mathcal{S}]}  \\
	{\mathcal{V}_{F(c_1)}} 
	\arrow["{F(s)_*}"', from=1-1, to=3-1]
	\arrow[""{name=0, anchor=center, inner sep=0}, "{\widetilde{F(c_2)}}", from=1-1, to=2-3]
	\arrow[""{name=1, anchor=center, inner sep=0}, "{\widetilde{F(c_1)}}"', from=3-1, to=2-3]
	\arrow["{\overline{\phi_s}}"', shorten <=4pt, shorten >=4pt, Rightarrow, shift right=2, from=0, to=1]
\end{tikzcd}\]
where $ \overline{\phi_s}$ is the morphism of sheaves induced from the comparison functor obtained from the pushout maps $ \phi_s = (n_*F(s))_{n \in \mathcal{V}_{F(c_2)}}$ after sheafification 
\[\begin{tikzcd}
	{\cod_{c_2}} & {F_{s*}\cod_{c_1} \simeq \cod_{c_1}F_{s*}} \\
	{\widetilde{F(c_2)}} & {\Spec(F_s)_*\widetilde{F(c_2)}}
	\arrow["{\phi_s}", Rightarrow, from=1-1, to=1-2]
	\arrow["{\gamma_{c_2}}"', Rightarrow, from=1-1, to=2-1]
	\arrow["{\gamma_{c_1}}", Rightarrow, from=1-2, to=2-2]
	\arrow["{\overline{\phi_s}}"', Rightarrow, from=2-1, to=2-2]
\end{tikzcd}\]
Then by the universal property of the oplax colimit in $\Cat$ this induces uniquely a canonical functor
\[\begin{tikzcd}
	{\mathcal{V}_{F(c)}} \\
	{\mathcal{V}_F} && {\mathbb{T}[\mathcal{S}]}
	\arrow[""{name=0, anchor=center, inner sep=0}, "{\widetilde{F(c)}}", from=1-1, to=2-3]
	\arrow["{\iota_c}"', hook, from=1-1, to=2-1]
	\arrow["{\langle \widetilde{F(c)} \rangle_{c \in \mathcal{C}_\mathcal{F}}}"', dashed, from=2-1, to=2-3]
	\arrow["\simeq"{description}, Rightarrow, draw=none, from=0, to=2-1]
\end{tikzcd}\]
Or in other terms, we know that \[ {p_c}_*\langle \widetilde{F(c)} \rangle_{c \in \mathcal{C}_\mathcal{F}} \simeq \widetilde{F(c)}\]
Moreover, for each $ \widetilde{F}(c)$ is a $J_{F(c)}$ sheaf of $\mathbb{T}[\mathcal{S}]$-objects, the induced $\langle \widetilde{F(c)} \rangle_{c \in \mathcal{C}_\mathcal{F}} $ is a $J_{\mathcal{V}^{\op}_{F(-)}}$-sheaf. \\

Recalling that the lax limit topos is also the category of all sections of the direct fibration $ \Gamma(\overline{p_F}')$, this object can also be described as the functor 
\[\begin{tikzcd}[sep=large]
	{\mathcal{C}_\mathbb{T}} & {\underset{c \in \mathcal{C}_\mathcal{F}}{\laxlim} \, \Spec(F(c))}
	\arrow["{\langle \widetilde{F(c)} \rangle_{c \in \mathcal{C}_\mathcal{F}}}", from=1-1, to=1-2]
\end{tikzcd}\]
sending each $ \{ \overline{x}, \phi \}$ to the section of the direct fibration 
\[\begin{tikzcd}[sep=large]
	{\mathcal{C}_\mathcal{F}} & {\displaystyle{\int}\Spec(F(-))'}
	\arrow["{\widetilde{F(-)}(\{ \overline{x}, \phi \}) }", from=1-1, to=1-2]
\end{tikzcd}\] sending $c$ to $ \widetilde{F(c)} (\{ \overline{x}, \phi \}) $ in $\Spec(F(c))$. \\

Beware however that, at this step, this induced sheaf is not yet a sheaf for $ J_F$ as the $\iota_F(J_\mathcal{F})$-families are not considered in $J_{\mathcal{V}^{\op}_{F(-)}}$.  
\end{division}

\begin{lemma}\label{two step sheafification}
The structure sheaf is the sheafification for the horizontal topology of the induced presheaf
\[  \widetilde{F} \simeq \mathfrak{a}_{\iota_\mathcal{F}}(\langle \widetilde{F(c)} \rangle_{c \in \mathcal{C}_\mathcal{F}}) \]
\end{lemma}

\begin{proof}
Each $\widetilde{F(c)}$ is the sheafification of the codomain functor for the spectral topology at $c$; but the topology in the lax limit topos is exactly the topology jointly generated by the fiberwise spectral topology, while the transition functors are continuous: hence being a sheaf for this topology amounts to being locally a sheaf for the fiberwise topologies, and we have
\begin{align*}
    \langle \widetilde{F(c)} \rangle_{c \in \mathcal{C}_\mathcal{F}} &= \langle \mathfrak{a}_{J_{F(c)}} \cod_c \rangle_{c \in \mathcal{C}_\mathcal{F}} \\
    &\simeq \mathfrak{a}_{\underset{c \in \mathcal{C}_\mathcal{F}}{\bigcup}\iota_c(J_{F(c)})} \langle \cod_c \rangle_{c \in \mathcal{C}_\mathcal{F}}
\end{align*}
but then it suffice to apply the sheafification for the horizontal topology to get the structure sheaf. 
\end{proof}

\begin{division}
In particular, there is a canonical way to compare the global structure sheaf to the structure sheaves associated to local values. At any $ c$ the codomain functor over $ \mathcal{V}_{F(c)}$ and the restriction along $\iota_c$ of the codomain functor over $ \mathcal{V}_F$ coincide. Yet however, this is not sufficient for the structure sheaves to coincide after restriction. As a direct image, $ \Sh(\iota_c)_*$ does not commute with sheafification: hence $ \Sh(\iota_c)_*\widetilde{F} = \Sh(\iota_c)_* \mathfrak{a}_{J_F}\cod  $ needs not be the sheafification of $ \cod $ for $ J_{F(c)}$, though it is a sheaf for $ J_{F(c)}$ because $ \iota_c$ is a morphism of site. Hence by the universal property of the sheafification, we have a factorization in $\mathbb{T}[\Spec(F(c))]$
\[\begin{tikzcd}
	\cod & {{\iota_c}_*\cod} \\
	{\mathfrak{i}_{F(c)}\mathfrak{a}_{J_{F(c)}} \cod} & {{\iota_c}_*\mathfrak{i}_{F}\mathfrak{a}_{J_F} \cod}
	\arrow["{\gamma_{\cod}}"', from=1-1, to=2-1]
	\arrow["{{\iota_c}_*(\gamma_{\cod})}", from=1-2, to=2-2]
	\arrow[Rightarrow, no head, from=1-1, to=1-2]
	\arrow[dashed, from=2-1, to=2-2]
\end{tikzcd}\]
where $ {\iota_c}_*\mathfrak{i}_{F}= \mathfrak{i}_F \Sh(\iota_c)_* $ and $ \mathfrak{i}_{F(c)} $ is the geometric inclusion $ \Spec(F(c)) \hookrightarrow [\mathcal{V}_{F(c)}, \mathcal{S}]$. We denote this morphism as 
\[\begin{tikzcd}
	{\widetilde{F(c)}} & {\Sh(\iota_c)_*\widetilde{F}}
	\arrow["{\zeta_c^\sharp}", from=1-1, to=1-2]
\end{tikzcd}\]
and its mate through the $ \Sh(\iota_c)$ adjunction as
\[\begin{tikzcd}
	{\Sh(\iota_c)^*\widetilde{F(c)}} & {\widetilde{F}}
	\arrow["{\zeta_c^\flat}", from=1-1, to=1-2]
\end{tikzcd}\]
\end{division}

\begin{theorem}
The structure sheaf $ \widetilde{F}$ is a local object in $\mathbb{T}[\Spec(F)]$. 
\end{theorem}

\begin{proof}
In fact we are going to prove that the structure sheaf is already a local object in the lax limit topos even before sheafification, that is, we prove that $\langle \widetilde{F(c)} \rangle_{c \in \mathcal{C}_\mathcal{F}}$ is already a local object in $\laxlim_{c \in \mathcal{C}_\mathcal{F}} \Spec(F(c))$, 
We claim this functor to be a local object. For each $ (n_i : K_\phi \rightarrow K_{\phi_i})_{i \in I}$ in $J$, we must display an epimorphism in the category of sections
\[\begin{tikzcd}[sep=large]
	{\underset{i \in I}{\displaystyle{\coprod}} \langle \widetilde{F(c)} \rangle_{c \in \mathcal{C}_\mathcal{F}}(\{ \overline{x}_i, \phi_i \}) } && { \langle \widetilde{F(c)} \rangle_{c \in \mathcal{C}_\mathcal{F}}(\{ \overline{x}, \phi\}) }
	\arrow["{\langle \langle \widetilde{F(c)} \rangle_{c \in \mathcal{C}_\mathcal{F}}(\theta_{n_i}) \rangle_{i \in I} }", from=1-1, to=1-3]
\end{tikzcd}\]
As in the set-valued case, this is obtained by proving that this morphism comes from a local epimorphism in $\widehat{\mathcal{V}^{\op}_F}$ given by 
\[\begin{tikzcd}[sep=large]
	{\underset{i \in I}{\displaystyle{\coprod}} \mathbb{T}[\mathcal{S}][(K_{\phi_i}, \cod_F]) } && { \mathbb{T}[\mathcal{S}][(K_{\phi}, \cod_F]) }
	\arrow["{\langle \mathbb{T}[\mathcal{S}][(n_i, \cod_F]) \rangle_{i \in I} }", from=1-1, to=1-3]
\end{tikzcd}\]

For each $(c,n)$, we know that the restriction $\langle \mathbb{T}[\mathcal{S}][(n_i, \cod_c]) \rangle_{i \in I} $ at $c$ is a local epimorphism, as seen in \cref{structure sheaf is local for sets} since, for each $ a : K_\phi \rightarrow \cod(n)$, the cover induced as the pushout $ (a_*n_i)_{i \in I}$ produced a family of antecedents of $a$ ensuring local surjectivity: but then the family $ ((c, n) \rightarrow (c, a_*n_i n))_{i \in I}  $ is itself convenient in $ \widehat{V^{\op}_F} $, and being a family of for a vertical cover, this proves $\langle \mathbb{T}[\mathcal{S}][(n_i, \cod_c]) \rangle_{i \in I} $ to be a local epimorphism relative to the vertical topology $\underset{c \in \mathcal{C}_\mathcal{F}}{\bigcup}\iota_c(J_{F(c)})$. \\

Since $ \langle \widetilde{F(c)} \rangle_{c \in \mathcal{C}_\mathcal{F}}$ is the sheafification of $ \cod_F$ for the vertical topology, this proves the induced map $\langle \langle \widetilde{F(c)} \rangle_{c \in \mathcal{C}_\mathcal{F}}(\theta_{n_i}) \rangle_{i \in I}$ to be an epimorphism in the lax limit topos. But now, from \cref{two step sheafification}, for the structure sheaf is the sheafification of $\langle \widetilde{F(c)} \rangle_{c \in \mathcal{C}_\mathcal{F}}$ along the horizontal topology, and sheafification preserves epimorphisms, we deduce $ \widetilde{F}$ to be a local object in $\Spec(F)$  
\end{proof}

\begin{division}
Now we turn to the canonical morphisms of sheaves associated with the structure sheaf. At an object $ (c,n)$ of the spectral site $ \mathcal{V}^{\op}_F$ we are provided with a canonical arrow given by the composite of $n $ with the sheafification
\[\begin{tikzcd}
	{F(c)} && {(\mathfrak{a}_{J_F}\cod)(n)} \\
	& {\cod(n)}
	\arrow["n"', from=1-1, to=2-2]
	\arrow["{\gamma_{(c,n)}}"', from=2-2, to=1-3]
	\arrow["{(g_F)_{(c,n)}}", from=1-1, to=1-3]
\end{tikzcd}\]
This map can be shown to be natural in $(c,n)$, and having in mind that $ F(c) = F(p_F^{\op}(c,n))$, this defines a natural transformation {in the category of presheaves} $ [\mathcal{V}_F, \mathcal{S}]$:
\[\begin{tikzcd}
	F\circ p_F^{\op} && {\mathfrak{i}_{F*}(w\widetilde{F})}
	\arrow["{g_F}", from=1-1, to=1-3]
\end{tikzcd}\]
(where we also denote $ \mathfrak{i}_F : \Spec(F) \hookrightarrow [\mathcal{V}_F, \mathcal{S}]$), which in turns factorizes uniquely through the sheafification in $ \Spec(F)$
\[\begin{tikzcd}
	{h_{F}^*F} && {\widetilde{F}}
	\arrow["{\eta_{F}^\flat}", from=1-1, to=1-3]
\end{tikzcd}\]
Moreover, its mate 
\[\begin{tikzcd}
	F && {h_{F*}w\widetilde{F}}
	\arrow["{\eta_{F}^\sharp}", from=1-1, to=1-3]
\end{tikzcd}\]
along the adjunction $ h^*_F \dashv h_{F*}$ indexes the values of $ \eta^\flat_F $ at the top element of the fibers, as we have $h_{F*}w\widetilde{F}(c) = w\widetilde{F}(\iota_F(c)) =  w\widetilde{F}(c,1_{F(c)}) $, so that we have
\[ (\eta^\sharp_F)_c = (\eta^\flat_F)_{(c,1_{F(c)})} \]
\end{division}

\begin{proposition}\label{unit is a generic etale map}
The natural transformation $ \eta_{F}^\flat $ is an etale map in $ \mathbb{T}[\Spec(F)]$, while $\eta_{F}^\sharp $ is etale in $ \mathbb{T}[\mathcal{F}]$.
\end{proposition}

\begin{proof}
By naturality of the unit of the sheafification $ \gamma : \cod \Rightarrow \widetilde{F}$, the triangles above are natural in $(c,n)$ and exhibit a factorization of $ \eta_{F} $ in the category of $\mathbb{T}$-models in the presheaf topos $ [\mathcal{V}_F,\mathcal{S}]$
\[\begin{tikzcd}
	F \circ p_F^{\op} && {\mathfrak{i}_{F*}w\widetilde{F}} \\
	& \cod
	\arrow["{\nu_F}"', from=1-1, to=2-2]
	\arrow["{\gamma_{\cod}}"', from=2-2, to=1-3]
	\arrow["{g_F}", from=1-1, to=1-3]
\end{tikzcd}\]
where $ \nu_F$ is the canonical map induced from all the $ n : F(c) \rightarrow \cod(n)$ for $ (c,n)$ in $\mathcal{V}^{\op}_{F}$, which is etale as it is pointwise etale. But now, as $ {\mathfrak{i}_F}_*$ is full and faithful with left adjoint $\mathfrak{a}_{J_F}$, we have a natural isomorphism $\mathfrak{a}_{J_F}\mathfrak{i}_{F*} \simeq 1 $, so that the aforementioned triangle is sent after sheafification to the following triangle in $ \mathbb{T}[ \Spec(F)]$
\[\begin{tikzcd}
	{h_{F}^*F} && {w\widetilde{F}} \\
	& {\mathfrak{a}_{J_F}\cod}
	\arrow["{\mathfrak{a}_{J_F}\nu_F}"', from=1-1, to=2-2]
	\arrow["{\mathfrak{a}_{J_F}\gamma_{\cod}}"', from=2-2, to=1-3]
	\arrow["{\eta_{F}^\flat}", from=1-1, to=1-3]
\end{tikzcd}\]
where $ {\mathfrak{a}_{J_F}\gamma_{\cod}}$ is an isomorphism, while $ \mathfrak{a}_{J_F}\nu_F$ still is etale as etale maps are stable under inverse images. Hence, factorizing it as an etale map followed by an isomorphism, we have proven $\eta_{F}^\flat $ to be an etale arrow in $\mathbb{T}[\Spec(F)]$. From the equality $(\eta^\sharp_F)_c = (\eta^\flat_F)_{(c,1_{F(c)})}$, we have that $ \eta^\sharp_F$ is pointwise etale, hence is etale in $ \mathbb{T}[\mathcal{F}]$. 
\end{proof}

\begin{division}
These data define a morphism of modelled topoi 
\[\begin{tikzcd}
	{(\mathcal{F},F)} && {(\Spec(F),w\widetilde{F})}
	\arrow["{(h_{F},\eta_{F}^\flat)}", from=1-1, to=1-3]
\end{tikzcd}\]
\end{division}

\section{Functoriality of $\Spec$}

Let us examine functoriality of $\Spec$. As we shall see, this is far from a trivial fact and is obtained through several steps; those complications were left totally implicit in the other sources on the topic, as they gave no precision on the functoriality: we think it justifies to devote some effort to clarify this aspect.\\

 We saw in \cref{factorization of morphisms of modelled topoi} that $ \mathbb{T}\hy\GTop$ admitted two factorizations systems (Vertical, Cartesian) and (Cocartesian, Vertical), where the vertical morphisms could be seen as ``algebraic morphisms" with trivial geometric part, while the other were given respectively from the counit and unit of the codomain and domain sheaf along the direct and inverse images part of a geometric morphism; in some sense, algebraic data in the cartesian and cocartesian morphisms were trivial, so they could be seen as two sorts of ``geometric morphisms" of modelled topoi. In the following we shall construct separately the spectrum of vertical, cartesian and cocartesian morphisms of modelled topoi, and define the spectrum of an arbitrary morphism as the composite of the spectrum of its vertical and cartesian part -- or equivalently, its cocartesian and vertical part.

\subsection{Functoriality relative to vertical morphisms}

Recall the vertical morphisms are those of the form $(1_\mathcal{F}, \phi): (\mathcal{F}, F_1) \rightarrow (\mathcal{F}, F_2)$, that is, whose underlying geometric morphism part is an identity -- which makes the inverse and direct image part $\phi^\flat$ and $ \phi^\sharp$ coincide. Vertical morphisms over $\mathcal{F}$ correspond exactly to morphisms in $\mathbb{T}[\mathcal{F}]$, and we shall see the construction of the construction of their spectrum generalizes the set valued case which corresponded to vertical morphisms over $ \mathcal{S}$.

\begin{division}\label{spectra of vertical morphisms}
For $ \mathcal{F}$ a Grothendieck topos, let $ \phi : F_1 \rightarrow F_2$ be a morphism in $\mathbb{T}[\mathcal{F}]$. Then we are provided with a natural transformation of pseudofunctors 
\[\begin{tikzcd}
	{\mathcal{C}_\mathcal{F}^{\op}} && \Cat
	\arrow[""{name=0, anchor=center, inner sep=0}, "{\mathcal{V}_{F_1(-)}}", curve={height=-12pt}, from=1-1, to=1-3]
	\arrow[""{name=1, anchor=center, inner sep=0}, "{\mathcal{V}_{F_2(-)}}"', curve={height=12pt}, from=1-1, to=1-3]
	\arrow["{\phi_*}", shorten <=3pt, shorten >=3pt, Rightarrow, from=0, to=1]
\end{tikzcd}\]
defined at $ c$ in $\mathcal{C}_\mathcal{F}$ as the pushout functor
\[\begin{tikzcd}[row sep=tiny]
	{\mathcal{V}_{F_1(c)}} & {\mathcal{V}_{F_2(c)}} \\
	n & {(\phi_c)_*n}
	\arrow["{(\phi_c)*}", from=1-1, to=1-2]
	\arrow[shorten <=5pt, shorten >=5pt, maps to, from=2-1, to=2-2]
\end{tikzcd}\]
whose naturality square at $ u : c_1 \rightarrow c_2$ 
\[\begin{tikzcd}
	{\mathcal{V}_{F_1(c_2)}} & {\mathcal{V}_{F_2(c_2)}} \\
	{\mathcal{V}_{F_1(c_1)}} & {\mathcal{V}_{F_2(c_1)}}
	\arrow["{(\phi_{c_2})*}", from=1-1, to=1-2]
	\arrow["{F_1(s)_*}"', from=1-1, to=2-1]
	\arrow["{F_2(s)_*}", from=1-2, to=2-2]
	\arrow["{(\phi_{c_1})*}"', from=2-1, to=2-2]
	\arrow["\simeq"{description}, draw=none, from=1-1, to=2-2]
\end{tikzcd}\]
is given by the commutation of pushouts functors along the two sides of the naturality square
\[\begin{tikzcd}
	{{F_1(c_2)}} & {{F_2(c_2)}} \\
	{{F_1(c_1)}} & {{F_2(c_1)}}
	\arrow["{\phi_{c_2}}", from=1-1, to=1-2]
	\arrow["{F_1(s)}"', from=1-1, to=2-1]
	\arrow["{F_2(s)}", from=1-2, to=2-2]
	\arrow["{\phi_{c_1}}"', from=2-1, to=2-2]
\end{tikzcd}\]
By functoriality of the Grothendieck construction, that is, by naturality of oplaxcolimits, the natural transformation $ \phi_*$ induces a morphism of fibrations
\[\begin{tikzcd}
	{\mathcal{V}^{\op}_{F_1} } && {\mathcal{V}^{\op}_{F_2} } \\
	& {\mathcal{C}_\mathcal{F}}
	\arrow["{\int \phi_*}", from=1-1, to=1-3]
	\arrow["{p_{F_1}}"', from=1-1, to=2-2]
	\arrow["{p_{F_2}}", from=1-3, to=2-2]
\end{tikzcd}\]
where $ \int \phi_* (c,n) = (c, (\phi_c)_*n)$ for $ (c,n)$ in $\mathcal{V}_{F_1}$. Moreover, by composition of pushouts, each functor $ (\phi_c)_*$ defines a morphism of lex site
\[\begin{tikzcd}
	{(\mathcal{V}_{F_1(c)}^{\op} ,J_{F_1(c)})} && {(\mathcal{V}_{F_2(c)} ^{\op} ,J_{F_2(c)}) }
	\arrow["{(\phi_c)_*}", from=1-1, to=1-3]
\end{tikzcd}\]
so that the induced functor $ \int \phi_*$ is a morphism of lex site. Hence $ \int \phi_*$ is a lex functor and is continuous for the jointly generated pretopologies $ \langle \bigcup_{c \in \mathcal{C}_\mathcal{F}} \iota_{c}(J_{F_1(c)}) \rangle $ and $ \langle \bigcup_{c \in \mathcal{C}_\mathcal{F}} \iota_{c}(J_{F_2(c)}) \rangle $. Moreover, since identities are preserved by pushouts, we also have the following commutation
\[\begin{tikzcd}
	{\mathcal{V}_{F_1(c)}^{\op} } && {\mathcal{V}_{F_2(c)} ^{\op}} \\
	& {\mathcal{C}_\mathcal{F}}
	\arrow["{(\phi_c)_*}", from=1-1, to=1-3]
	\arrow["{\iota_{F_1}}", from=2-2, to=1-1]
	\arrow["{\iota_{F_2}}"', from=2-2, to=1-3]
\end{tikzcd}\]
Hence $ \int \phi_*$ sends $ \iota_{F_1}(J_{\mathcal{F}})$-families to  $ \iota_{F_2}(J_{\mathcal{F}})$-families. This provides us with a morphism of lex sites
\[\begin{tikzcd}
	{(\mathcal{V}_{F_1}^{\op} ,J_{F_1})} && {(\mathcal{V}_{F_2} ^{\op} ,J_{F_2}) }
	\arrow["{\int \phi_*}", from=1-1, to=1-3]
\end{tikzcd}\]
and consequently with a geometric morphism 
\[\begin{tikzcd}
	{\Spec(F_2)} & {\Spec(F_1)}
	\arrow["{\Spec(\phi)}", from=1-1, to=1-2]
\end{tikzcd}\]
\end{division}

\begin{division}
Now, the associated morphisms of sheaves are defined as follows. For the direct image part, the data at each $(c, n)$ of the map $ n_*\phi_c$ as obtained in the following pushout
\[\begin{tikzcd}
	{F_1(c)} & {F_2(c)} \\
	{\cod(n)} & {\cod((\phi_c)_*n)}
	\arrow["{(\phi_c)_*n}", from=1-2, to=2-2]
	\arrow["n"', from=1-1, to=2-1]
	\arrow["{\phi_c}", from=1-1, to=1-2]
	\arrow["{n_*\phi_c}"', from=2-1, to=2-2]
	\arrow["\lrcorner"{anchor=center, pos=0.125, rotate=180}, draw=none, from=2-2, to=1-1]
\end{tikzcd}\]
define altogether a natural transformation 
\[\begin{tikzcd}
	{\cod_1} & {\Spec(\phi)_*\cod_2 }
	\arrow["{\nu_\phi}", from=1-1, to=1-2]
\end{tikzcd}\]
which is sent after sheafification to a morphism in $\mathbb{T}[\Spec(F_1)]$
\[\begin{tikzcd}
	{\widetilde{F_1}} & {\Spec(\phi)_*\widetilde{F_2}}
	\arrow["{\mathfrak{a}_{J_{F_1}}(\nu_\phi)}", from=1-1, to=1-2]
\end{tikzcd}\]
Hence we have to define 
\[ \widetilde{\phi}^\sharp = \mathfrak{a}_{J_{F_1}}(\nu_\phi) \]
while the existence of its mate
\[\begin{tikzcd}
	{\Spec(\phi)^*\widetilde{F_1}} & {\widetilde{F_2}}
	\arrow["{\widetilde{\phi}^\flat}", from=1-1, to=1-2]
\end{tikzcd}\]
is ensured through the adjunction defining $ \Spec(\phi)$. However, it admits a more concrete construction, which is related to the (etale, locale) factorization in the sense that it is induced from taking the local part after precomposition with $ \phi_c$, in the same vein as \cref{f flat is local set case}:
\end{division}

\begin{proposition}
The morphism $ \widetilde{\phi}^\flat$ is in $ \Loc[\Spec(F_2)]$
\end{proposition}

\begin{proof}
Recall that inverse images and sheafification commutes, so we can first compute the inverse image of the codomain functor over $ \mathcal{V}_{F_1}$ along $ \Spec(\phi)$ before sheafifying into $ \Spec(\phi)^*\widetilde{F_1}$. Recall that the inverse image of the codomain functor (as a presheaf) is obtained as the left Kan extension 
\[\begin{tikzcd}
	{\mathcal{V}_{F_1}} & {\mathbb{T}[\mathcal{S}]} \\
	{\mathcal{V}_{F_2}}
	\arrow["{\int \phi_*}"', from=1-1, to=2-1]
	\arrow["\cod", from=1-1, to=1-2]
	\arrow[""{name=0, anchor=center, inner sep=0}, "{\lan_{\int \phi_*}\cod}"', from=2-1, to=1-2]
	\arrow["\zeta"', shorten >=2pt, Rightarrow, from=1-1, to=0]
\end{tikzcd}\]
For each $(c,n)$ in $ \mathcal{V}_{F_2}$, this left Kan extension is computed as the filtered colimit 
\[ (\lan_{\int \phi_*}\cod)(c,n) \simeq \underset{ \int \phi_* \downarrow (c,n)}{\colim} \cod(m) \]
ranging over all $ (u,f) : (d, (\phi_d)*m) \rightarrow (c,n) $ where $ (d,(\phi_d)_*m) = \int \phi_*(d,m) $. Beware that we use there the native orientation of $ \mathcal{V}_{F_2}$, so that this corresponds to an underlying morphism $ u : c \rightarrow d$ in $ \mathcal{C}_\mathcal{F}$. But observe now that the subcategory of $ \int \phi_* \downarrow (c,n)$ admits a cofinal subcategory consisting of all the morphisms of the form $ (1_c, f) : (c, (\phi_c)_*m) \rightarrow (c,n)$, since any object in $\int \phi_* \downarrow (c,n) $ factorizes in $ \mathcal{V}_{F_2}$ as 
\[\begin{tikzcd}
	{(d,(\phi_c)_*m)} && {(c,n)} \\
	& {(c, F_2(u)_*(\phi_c)_*m)}
	\arrow["{(u,f)}", from=1-1, to=1-3]
	\arrow["{(f,1_{ F_2(u)_*(\phi_c)_*m})}"', from=1-1, to=2-2]
	\arrow["{(1_c, \langle n,f\rangle)}"', from=2-2, to=1-3]
\end{tikzcd}\]
while any two parallel factorizations can be merged by a pushout.
Hence the colimit reduces to a colimit over those $ (1_c, f)$, so we can restrict in some sense at a computation in the fiber at $c$. But alike what was said at \cref{f flat is local set case}, remark that $ (\phi_c)^*$ is left adjoint to the precomposition functor $ (\phi_c)^! : \mathcal{V}_{F_2(c)} \rightarrow F_1(c) \downarrow \mathbb{T}[\mathcal{S}]$, so that the cofinal category we exhibited above is equivalent to $ \mathcal{V}_{F_1(c)} \downarrow n \phi_c$. Hence we have
\[ \lan_{\int \phi_*}\cod(c,n) \simeq \underset{ \mathcal{V}_{F_1(c)} \downarrow n \phi_c}{\colim} \cod(m)  \]
From \cref{factorization from saturated class}, this exhibits $ \lan_{\int \phi_*}\cod(c,n) $ as the middle term in the (etale, locale) factorization of $ n(\phi_c)_*$, whose local part shall be denoted $ (u_\phi)_c : \lan_{\int \phi_*}\cod(c,n) \rightarrow \cod(c,n)$. Now by naturality this can be gathered into a local map in the category of $ \mathbb{T}$-models in the presheaf topos $ [\mathcal{V}_{F_2},\mathcal{S}]$
\[\begin{tikzcd}
	{\Spec(\phi)^*\cod} & \cod
	\arrow["{u_f}", from=1-1, to=1-2]
\end{tikzcd}\]
whose sheafification returns $ \widetilde{\phi}^\flat$. 
\end{proof}

\begin{remark}
Observe that we exploited morally the fact that (etale, local) factorizations in arbitrary topoi are done pointwisely (modulo sheafification) as observed in \cref{factorization in topoi}, since we constructed the inverse image part from factorizations in the fibers before gathering them into a morphism of sheaf. From the triangle of geometric morphisms 
\[\begin{tikzcd}
	{\Spec(F_2)} && {\Spec(F_1)} \\
	& {\mathcal{F}}
	\arrow["{\Spec(\phi)}", from=1-1, to=1-3]
	\arrow["{h_{F_2}}"', from=1-1, to=2-2]
	\arrow["{h_{F_1}}", from=1-3, to=2-2]
\end{tikzcd}\]
we got a square in $ \Spec(F_2)$
\[\begin{tikzcd}
	{\Spec(\phi)^*h_{F_1}^*F_1} && {\Spec(\phi)^*\widetilde{wF_1}} \\
	{h_{F_2}^*F_1} & {h_{F_2}^*F_2} & {w\widetilde{F_2}}
	\arrow["{\Spec(\phi)^*\eta_{F_1}^\flat}", from=1-1, to=1-3]
	\arrow["{w\widetilde{\phi}^\flat}", from=1-3, to=2-3]
	\arrow["{\eta_{F_2}^\flat}"', from=2-2, to=2-3]
	\arrow[Rightarrow, no head, from=1-1, to=2-1]
	\arrow["{h_{F_2}^*\phi}"', from=2-1, to=2-2]
\end{tikzcd}\]
In fact, we can see this square encodes the (etale, locale) factorization of the composite 
\[\begin{tikzcd}[sep=large]
	{h_{F_2}^*F_1} & {w\widetilde{F_2} }
	\arrow["{ h_{F_2}^*\phi\eta_{F_2}^\flat }", from=1-1, to=1-2]
\end{tikzcd}\]
through a local unit, as $\widetilde{F_2} $ is a local object, $\widetilde{\phi}^\flat $ is a local map, and $ \eta_{F_1}^\flat$ is the etale map coding the universal etale form under $F_1$. 
\end{remark}

\subsection{Functoriality relative to horizontal morphisms}

Recall that there are actually two kinds of horizontal morphisms, respectively the cartesian obtained through direct images and the cocartesian obtained through inverse images. We shall see it is sufficient to have only the cartesian together with the vertical ones to compute the spectrum of any morphism.

We turn to the cartesian morphisms and the construction of the spectrum of an arbitrary morphism. In fact the main difficulty is that one needs to consider the underlying geometric morphisms of a morphism of modelled topos as induced in a manner or another from a morphism of site in order to perform the pushout of basic etales map along it. Of course this is not possible in general to find simultaneously two \emph{small} sites of presentation for the domain and codomain topos inducing a geometric morphism between them. But we can always induce it from a site morphism landing in the domain topos, which, together with its canonical topology, can be seen as a \emph{small generated site}, for which there exist also far enough theory for our needs, as in \cite{caramello2020denseness}. We shall show we can extend canonically a $\mathbb{T}$-model seen as a sheaf of $\mathbb{T}$-objects over a presentation site to a sheaf over the whole topos with the canonical topology.

\begin{division}
Recall, if $ \mathcal{F}$ admits a lex subcanonical site of presentation $(\mathcal{C}_{\mathcal{F}}, J_{\mathcal{F}}) $ then we have a geometric equivalence 
\[ \Sh(\mathcal{C}_{\mathcal{F}}, J_{\mathcal{F}}) \simeq \Sh(\Sh(\mathcal{C}_{\mathcal{F}}, J_{\mathcal{F}}), J_{\can}) \]
sending a sheaf $ X $ over $ \mathcal{C}_{\mathcal{F}}$ to its left Kan extension
\[\begin{tikzcd}
	{\mathcal{C}_{\mathcal{F}}^{\op}} & {\mathcal{S}} \\
	{\mathcal{F}^{\op}}
	\arrow["X", from=1-1, to=1-2]
	\arrow["{\hirayo^{\op}}"', hook, from=1-1, to=2-1]
	\arrow[""{name=0, anchor=center, inner sep=0}, "{\overline{X} =\lan_{\hirayo^{\op}}X}"', from=2-1, to=1-2]
	\arrow["\simeq"{description}, Rightarrow, draw=none, from=1-1, to=0]
\end{tikzcd}\]
where $ \hirayo$ is full and faithful and lands in the category of sheaves as $J_\mathcal{F}$ is subcanonical; in the other direction, a sheaf for the canonical topology is sent to its restriction along ${\hirayo}^{\op} $. Moreover, this construction extends to categories of sheaves in locally finitely presentable categories, in particular for $\mathbb{T}[\mathcal{S}]$, so we have an equivalence of categories
\[ \mathbb{T}[ \mathcal{F}] \simeq \mathbb{T}[\Sh(\mathcal{F}, J_{\can})] \]
\end{division}

\begin{division}
Hence from the construction above, any $ \mathbb{T}$-modelled topos $ (\mathcal{F}, F)$ defines canonically another $ \mathbb{T}$-modelled topos
$ (\Sh(\mathcal{F},J_{\can}), \overline{F})$, which will be called its \emph{extended modelled topos}\index{modelled topos!extended}, which is actually equipped with an invertible morphism of modelled topoi for $ \mathcal{F} \simeq \Sh(\mathcal{F},J_{\can}) $ while $ F \simeq \lan_{\hirayo} F \circ \hirayo$ (by subcanonicity) and $ \overline{ F} = \lan_\hirayo F$. Hence we expect those two modelled topoi to have the same spectrum, but from the construction of Coste sprectral site, this is not obvious -- especially since we have not yet achieved functoriality. We shall in fact need this result in the following to prove functoriality for horizontal morphisms.\\

In fact, we can see that the spectral sites themselves of $ F$ and $\overline{F}$ are not equivalent. At each $c$ in $\mathcal{C}$, we have an isomorphism $ F(c) \simeq \overline{F}(\hirayo_c)$, and hence an equivalence between the fibers at object of $ \mathcal{C}$
\[ \mathcal{V}_{F(c)} \simeq \mathcal{V}_{\overline{F}(\hirayo_c)} \]
This induces a canonical inclusion between the oplax colimits 
\[\begin{tikzcd}
	{\mathcal{V}_F \simeq \underset{c \in \mathcal{C}_{\mathcal{F}}}{\oplaxcolim} \, \mathcal{V}_{{F}(c)}} & {\underset{X\in \mathcal{F}}{\oplaxcolim} \, \mathcal{V}_{\overline{F}(X)} \simeq \mathcal{V}_{\overline{F}}}
	\arrow["{i}_F", hook, from=1-1, to=1-2]
\end{tikzcd}\]
sending $ (c,n)$ to $(\hirayo_c, n)$. However, though we have a dense inclusion $ \hirayo : \mathcal{C}_\mathcal{F} \hookrightarrow \mathcal{F}$, 
the inclusion $i_F$ is by no mean an equivalence itself, as the oplax colimits cannot be contracted -- as a pseudocolimit may have been with an argument of cofinality. However, the following notion will help us to fix this at the level of induced topoi. Recall the following definition:
\end{division}

\begin{definition}
A morphism of site $ f : (\mathcal{C},J) \rightarrow (\mathcal{D},K)$ is said to be \emph{$K$-dense}\index{morphism of site!dense} if it satisfies the following conditions:\begin{itemize}
    \item $f $ creates covers, that is, a family $ (u_i : c_i \rightarrow c)_{i \in I}$ is $J$-covering if and only if its image $ (f(u_i))_{i \in I}$ is a $K$-covering;
    \item for any object $d$ in $\mathcal{D}$ there exists a $K$-cover of the form $ (f_i :f(c_i) \rightarrow d)_{i \in I}$ with $c_i $ in $ \mathcal{C}$;
    \item for any $c_1, c_2$ in $\mathcal{C}$ and $ g : f(c_1) \rightarrow f(c_2)$, there exists a $J$-cover $ (f_i : c_i \rightarrow c)_{i \in I}$ and a family of arrows $ (g_i : c_i \rightarrow c_2)_{i \in I}$ such that $ g f(f_i) = f(g_i)$.
\end{itemize}
\end{definition}

Then it is known from \cite{shulman2012exact}[Theorem 11.14](see also \cite{caramello2020denseness}[Remark 5.2]) that if $ f : (\mathcal{C},J) \rightarrow (\mathcal{D},K)$ is a $ K$-dense morphism of site, then the induced functor $ \Sh(f)$ is a geometric equivalence $ \Sh(\mathcal{D},K) \simeq \Sh(\mathcal{C},J)$. Observe that in particular that, if $ J$ is subcanonical, then the dense inclusion $ \hirayo : (\mathcal{C}, J) \rightarrow (\Sh(\mathcal{C},J),J_\can)$ is a dense morphism of site, and moreover that restricting back the canonical topology to $\mathcal{C}$ defines a dense morphism of sites $(\mathcal{C}, J) \rightarrow (\mathcal{C}, J_{\can}\mid_\mathcal{C}) $ and also a dense morphism of site
$ (\mathcal{C}, J) \rightarrow (\hirayo(\mathcal{C}), J_{\can}\mid_{\hirayo(\mathcal{C})}) $.

\begin{lemma}\label{extended spectrum}
We have a geometric equivalence $ \Spec(F) \simeq \Spec(\overline{F})$.
\end{lemma}

\begin{proof}
We prove that the functor ${i}_F$ above defines a $J_{\overline{F}}$-dense morphism of site 
\[\begin{tikzcd}
	{(\mathcal{V}_F^{\op}, J_F)} & {(\mathcal{V}^{\op}_{\overline{F}}, J_{\overline{F}})}
	\arrow["{i}_F", hook, from=1-1, to=1-2]
\end{tikzcd}\]
For the first condition, it suffices to prove separately that $ {i}_F$ creates vertical and horizontal families: \begin{itemize}
    \item for vertical families, observe that in each $c$ in $ \mathcal{C}_\mathcal{F}$, the fibers are equivalent as well as the topologies $ J_{F(c)}$ and $ J_{\overline{F}(\hirayo_c)}$: whence the creation of vertical families;
    \item for horizontal families, this is a consequence of the fact that $ (\mathcal{C}_\mathcal{F}, J_\mathcal{F}) \rightarrow (\hirayo(\mathcal{C}_\mathcal{F}), J_{\can}\mid_{\hirayo(\mathcal{C}_\mathcal{F})}) $ is a dense morphism of sites, so that a family $ (\hirayo_{u_i} : \hirayo_{c_i} \rightarrow \hirayo_c)_{i \in I}$ is in $J_{\can}$ if and only if $ (u_i : c_i \rightarrow c)_{i \in I}$ is in $J_\mathcal{F}$. Hence a family $ (\hirayo_{u_i}, f) : (\hirayo_{c_i}, n_i) \rightarrow (\hirayo_c, n))_{i \in I}$ is $J_{\overline{F}}$-covering if and only if $ ({u_i}, f) : ({c_i}, n_i) \rightarrow (c, n))_{i \in I}$ is $J_{F}$-covering.
\end{itemize}

For the second condition, recall that $ \hirayo$ is a dense functor so that for any $X$ in $ \mathcal{F}$ we have a colimit $ X \simeq \hirayo \downarrow X$ (where beware the colimit is computed in the category of sheaves). Hence $ \hirayo \downarrow X$ is never empty, and in fact, from the definition of colimits in topoi and their relation with the canonical topology, the family $(a : \hirayo_c \rightarrow X)_{\hirayo \downarrow X}$ is a cover in $J_\can$. Hence one can take for each object $(X,n)$ of $ \mathcal{V}_{\overline{F}}$ with $ n : \overline{F}(X) \rightarrow \cod(n)$ the horizontal $J_{\overline{F}}$-cover 
\[\begin{tikzcd}[sep=large]
	{(\hirayo_c, \overline{F}(a)_*n)} & {(X,n)}
	\arrow["{(a, n_*\overline{F}(a))}", from=1-1, to=1-2]
\end{tikzcd}\]
where $(\hirayo_c, \overline{F}(a)_*n) \simeq {i}_F (c, \overline{F}(a)_*n)$.\\

For the third item, we use that, by subcanonicity of $J_\mathcal{F}$, $\hirayo$ is full and faithful, so that for any $ (c_1, n_1) $, $(c_2, n_2)$ and $(f,g): (\hirayo_{c_1}, n_1) \rightarrow (\hirayo_{c_2}, n_2) $, the arrow $ f$ must come from some $ f = \hirayo_u$ with $ u : c_1 \rightarrow c_2$ and hence we had already $ (u,g) : (c_1, n_1) \rightarrow (c_2, n_2)  $.
\end{proof}

\begin{division}\label{extended sheaf has same spectrum}
We saw that $ (\mathcal{F}, F) $ is actually equivalent as a modelled topos to its extended form $(\Sh(\mathcal{F}_2, J_{can}), \overline{F})$. Consequently, we expect their respective spectrum to be equivalent not only at the level of the underlying topoi but at the level of the structure sheaves. This can be seen as follows: at each $ (c,n)$ of $ \mathcal{V}_{F}$ we have $ i_{F}(c,n) = (\hirayo_c, n)$ as $ \overline{F}(\hirayo_c) = F(c)$ -- so that $n$ is both an object of $ \mathcal{V}_{F(c)}$ and $ \mathcal{V}_{\overline{F}(\hirayo_c)}$; then we have a pointwise equality $ i_{F*}\cod (c,n) = \cod(\hirayo_c,n) = n $, so that we get an isomorphism after sheafification relative to $ J_{\overline{F}}$ 
\[  \widetilde{\overline{F}} \simeq \Sh(i_F)_* \widetilde{F} \]
whose mate between inverse images is also an isomorphism as the unit and counit of the adjoint equivalence $ \Sh(i_F)^*\dashv \Sh(i_F)_*$ are so.
\end{division}

\begin{division}
Now for a morphism $(f, \phi) : (\mathcal{F}_1, F_1) \rightarrow (\mathcal{F}_2, F_2)$, even though $ f$ does not necessarily arise form a morphism of site $ (\mathcal{C}_{\mathcal{F}_1}, J_{\mathcal{F}_1}) \rightarrow (\mathcal{C}_{\mathcal{F}_2}, J_{\mathcal{F}_2}) $, $f$ is still induced as a morphism of sites $ f^*: \mathcal{C}_{\mathcal{F}_1} \rightarrow \mathcal{F}_2$ sending $ J_{\mathcal{F}_1}$-covers to covers for the canonical topology $J_{\can}$ on $\mathcal{F}_2$ seen as a small-generated standard site, and the direct image can be computed explicitly at $ c$ in $\mathcal{C}_1$ as 
\[ f_*F_2(c) \simeq \overline{F_2}(f^*(c))  \]
For this later is an ordinary set-based $\mathbb{T}$-model, we can compute its spectral site $ \mathcal{V}_{\overline{F_2}(f^*(c))}$, and for each $ c$ we have a transition functor 
\[\begin{tikzcd}
	{\mathcal{V}_{F_1(c)}} & {\mathcal{V}_{\overline{F_2}(f^*(c))}}
	\arrow["{(\phi^\sharp_c)_*}", from=1-1, to=1-2]
\end{tikzcd}\]
which induces by naturality of the oplax colimit a functor
\[\begin{tikzcd}
	{\mathcal{V}_{F_1}} & {\underset{c \in \mathcal{C}_{\mathcal{F}_1}}{\oplaxcolim} \, \mathcal{V}_{\overline{F_2}(f^*(c))}}
	\arrow["{\int \phi^\sharp_*}", from=1-1, to=1-2]
\end{tikzcd}\]
where the later oplax colimit is $ \mathcal{V}_{f_*F_2}$. Then we can apply \cref{spectra of vertical morphisms} to the vertical morphism 
\[\begin{tikzcd}[sep=huge]
	{ \mathcal(\mathcal{F}_1, F_1) } & {(\mathcal{F}_1, f_*F_2)}
	\arrow["{(1_{\mathcal{F}_1}, (\phi^\sharp, \phi^\sharp))}", from=1-1, to=1-2]
\end{tikzcd}\]
so we get in particular a geometric morphism 
\[\begin{tikzcd}[sep=large]
	{\Spec(f_*F_2)} & {\Spec(F_1)}
	\arrow["{\Spec(\phi^\sharp)}", from=1-1, to=1-2]
\end{tikzcd}\]
\end{division}
\begin{division}
On the other hand, for each $ f^*(c)$ is in particular an object of the sheaf topos $ \mathcal{F}_{2}$, reindexing along $ f^*: \mathcal{C}_1 \rightarrow \mathcal{F}_2$ induces a canonical inclusion between the oplax colimits 
\[\begin{tikzcd}
	{{\underset{c \in \mathcal{C}_{\mathcal{F}_1}}{\oplaxcolim} \, \mathcal{V}_{\overline{F_2}(f^*(c))}}} & {{\underset{X \in \mathcal{F}_{2}}{\oplaxcolim} \, \mathcal{V}_{\overline{F_2}(X)}}}
	\arrow[" q_{(f,\phi)} ", hook, from=1-1, to=1-2]
\end{tikzcd}\]
which is moreover trivially a morphism of sites. Hence it induces a geometric morphism 
\[\begin{tikzcd}[sep=large]
	{\Spec(\overline{F}_2)} & {\Spec(f_*F_2)}
	\arrow["{\Sh(q_{(f,\phi)})}", from=1-1, to=1-2]
\end{tikzcd}\]
Hence from \cref{extended spectrum}, we can finally compose all those data into a geometric morphism $ \Spec(\phi)$ as below
\[\begin{tikzcd}[column sep=large]
	& {\Spec(F_2)} & {\Spec(F_1)} \\
	{} & {\Spec(\overline{F}_2)} & {\Spec(f_*F_2)}
	\arrow["{\Sh(q_{(f,\phi)})}"', from=2-2, to=2-3]
	\arrow["{\Sh({i}_{F_2}) \atop \simeq}"', from=1-2, to=2-2]
	\arrow["{\Spec(\phi^\sharp)}"', from=2-3, to=1-3]
	\arrow["{\Spec(\phi)}", dashed, from=1-2, to=1-3]
\end{tikzcd}\]

\end{division}

\begin{division}
We conclude with the computation of the sheaf data associated to $ \Spec(\phi)$. This will be done by combining easy to compute data associated to each part of the decomposition $ \Spec(\phi) = \Spec(\phi^\sharp) \Sh(q_{(f,\phi)}) \Sh(i_{F_2})$. \\

From \cref{extended sheaf has same spectrum} we know that $ \widetilde{\overline{F_2}} \simeq \Sh(i_{F_2})_*\widetilde{F_2}$. On the other side, the vertical morphism $ (1, \phi^\sharp) : (\mathcal{F}_1, F_1) \rightarrow (\mathcal{F}_1, f_*F_2)$ is sent to a morphism of locally modelled topos 
\[\begin{tikzcd}
 {(\Spec(F_1), \widetilde{F_1})}&& 	{(\Spec(f_*F_2), \widetilde{f_*F_2})}
	\arrow["{(\Spec(\phi), \widetilde{\phi})}", from=1-1, to=1-3]
\end{tikzcd}\]
We are going to see that this part contains actually all the sheaf information of the composite above.\\

Indeed, the intermediate part acts like restriction, that is, 
\[ \widetilde{f_*F_2} \simeq \Sh(q_{f,\phi*})\overline{F_2} \]
This is because $ \mathcal{V}_{f_*F_2}$ is a subcategory of $ \mathcal{V}_{\overline{F}_2}$ with the $ \mathcal{V}_{F_2(f^*(c))}$ as fibers, and again, one can apply sheafification to the equality between codomain functors. This returns a cartesian morphism
\[\begin{tikzcd}[column sep=large]
	{(\Spec(f_*F_2), \widetilde{f_*F_2})} && {(\Spec(\overline{F_2}), \widetilde{\overline{F_2}})}
	\arrow["{(\Sh(q_{(f,\phi)}, 1_{\widetilde{f_*F_2}})}", from=1-1, to=1-3]
\end{tikzcd}\]

Hence we can define the direct image part $\widetilde{\phi}^\sharp$ of $ \widetilde{\phi}$ as the composite in $ \Spec(F_1)$ 
\[\begin{tikzcd}
	{\widetilde{F_1}} & {\widetilde{f_*F_2} \simeq \Sh(q_{(f,\phi)})_*\widetilde{\overline{F_2}} \simeq\Spec(\phi)_*F_2}
	\arrow["{\widetilde{\phi^\sharp}}", from=1-1, to=1-2]
\end{tikzcd}\]
Summing up those considerations, we have a decomposition in $\mathbb{T}_J\hy\GTop^\Loc$
\[\begin{tikzcd}
	{(\Spec(F_1), \widetilde{F_1})} && {(\Spec(F_2), \widetilde{F_2})} \\
	{{(\Spec(f_*F_2), \widetilde{f_*F_2})}} && {{(\Spec(\overline{F_2}), \widetilde{\overline{F_2}})}}
	\arrow["{(\Sh(q_{(f,\phi)}, 1_{\widetilde{f_*F_2}})}"', from=2-1, to=2-3]
	\arrow["{(\Spec(\phi), \widetilde{\phi})}"', from=1-1, to=2-1]
	\arrow["{(\Spec(\phi), \widetilde{\phi})}", from=1-1, to=1-3]
	\arrow["{(\Sh(i_{F_2}), 1_{\Sh(i_{F_2})_*\widetilde{F_2}})}"', from=2-3, to=1-3]
\end{tikzcd}\]
\end{division}

\subsection{Spectrum of 2-cells}

For the sake of exhaustiveness, we give here to the treatment of 2-cells. 

\begin{division}
Take a 2-cell in $\mathbb{T}\hy\GTop$
\[\begin{tikzcd}
	{(\mathcal{F}_1, F_1)} && {(\mathcal{F}_2,F_2)}
	\arrow[""{name=0, anchor=center, inner sep=0}, "{(f,\phi)}", bend left=20, from=1-1, to=1-3]
	\arrow[""{name=1, anchor=center, inner sep=0}, "{(g,\gamma)}"', bend right=20, from=1-1, to=1-3]
	\arrow["\sigma", shorten <=3pt, shorten >=3pt, Rightarrow, from=0, to=1]
\end{tikzcd}\]
that is, a 2-cell in $\GTop$
\[\begin{tikzcd}
	{\mathcal{F}_2} && {\mathcal{F}_1}
	\arrow[""{name=0, anchor=center, inner sep=0}, "f", curve={height=-12pt}, from=1-1, to=1-3]
	\arrow[""{name=1, anchor=center, inner sep=0}, "g"', curve={height=12pt}, from=1-1, to=1-3]
	\arrow["\sigma", shorten <=3pt, shorten >=3pt, Rightarrow, from=0, to=1]
\end{tikzcd}\]
such that we have one, hence both of the following commutations between inverse and direct images in $ \mathcal{F}_2$ and $\mathcal{F}_1$ respectively respectively
\[\begin{tikzcd}
	{f^*F_1} && {g^*F_1} \\
	& {F_2}
	\arrow["{\phi^\flat}"', from=1-1, to=2-2]
	\arrow["{\sigma^\flat_F}", from=1-1, to=1-3]
	\arrow["{\gamma^\flat}", from=1-3, to=2-2]
\end{tikzcd} \hskip1cm \begin{tikzcd}
	& {F_1} \\
	{g_*F_2} && {f_*F_2}
	\arrow["{\gamma^\sharp}"', from=1-2, to=2-1]
	\arrow["{\phi^\sharp}", from=1-2, to=2-3]
	\arrow["{\sigma^\sharp_{F_2}}"', from=2-1, to=2-3]
\end{tikzcd} \]
The later, as a triangle of vertical morphisms of modelled topoi, induces an invertible 2-cell between the corresponding spectra
\[\begin{tikzcd}
	& {\Spec(F_1)} \\
	{\Spec(g_*F_2)} && {\Spec(f_*F_1)}
	\arrow["{\Spec(\gamma^\sharp)}", from=2-1, to=1-2]
	\arrow["{\Spec(\phi^\sharp)}"', from=2-3, to=1-2]
	\arrow[""{name=0, anchor=center, inner sep=0}, "{\Spec(\sigma^\sharp_{F_2})}"', from=2-1, to=2-3]
	\arrow["\simeq"{description}, Rightarrow, draw=none, from=1-2, to=0]
\end{tikzcd}\]
which induces in $\Spec(F_1)$ a triangle toward direct images
\[\begin{tikzcd}
	{F_1} & {\Spec(\gamma^\sharp)_*\widetilde{g_*F_2}} \\
	& {\Spec(\phi^\sharp)_*\widetilde{f_*F_2}}
	\arrow["{\widetilde{\phi^\sharp}}"', from=1-1, to=2-2]
	\arrow["{\widetilde{\gamma^\sharp}}", from=1-1, to=1-2]
	\arrow["{\Spec(\phi^\sharp)(\widetilde{\sigma^\sharp_{F_2}}^\sharp})", from=1-2, to=2-2]
\end{tikzcd}\]
where the middle arrow comes from the direct image part of $ \widetilde{\sigma^\sharp_{F_2}}$
\[\begin{tikzcd}
	{\widetilde{g_*F_2}} & {\Spec(\sigma^\sharp_{F_2})_*(\widetilde{f_*F_2)}}
	\arrow["{\widetilde{\sigma^\sharp_{F_2}}^\sharp}", from=1-1, to=1-2]
\end{tikzcd}\]
\end{division}
\begin{division}
Now recall that, though neither $f$ nor $g$ are supposed to be induced from a morphism of site, we can use the extended sheaf $ \overline{F_2}$ associated to $ F_2$ which is at each $c $ of $\mathcal{C}_{\mathcal{F}_1}$ respectively $ g_*F_2(c) = \overline{F_2}(g^*(c))$ and $ f_*F_2(c) = \overline{F_2}(f^*(c))$. Then the component at $c$ of the direct image part $ \sigma^\sharp_{F_2} $ is also the image along $ \overline{F_2}$ of the component of $ \sigma^\sharp : g_* \Rightarrow f_*$ at $c$ by Yoneda lemma, that is 
\[  (\sigma^\sharp_{F_2})_c = \overline{F_2}(\sigma^\sharp_c)  \]
along which we can consider the pushout functor
\[\begin{tikzcd}
	{\overline{F_2}(g^*(c))} & {\overline{F_2}(f^*(c))} \\
	{\cod(n)} & {\cod(\overline{F_2}(\sigma^\sharp_c)_*n)}
	\arrow["n"', from=1-1, to=2-1]
	\arrow["{\overline{F_2}(\sigma^\sharp_c)}", from=1-1, to=1-2]
	\arrow["{\overline{F_2}(\sigma^\sharp_c)_*n}", from=1-2, to=2-2]
	\arrow["{n_*\overline{F_2}(\sigma^\sharp_c)}"', from=2-1, to=2-2]
	\arrow["\lrcorner"{anchor=center, pos=0.125, rotate=180}, draw=none, from=2-2, to=1-1]
\end{tikzcd}\]
This defines at each $c$ a 2-cell
\[\begin{tikzcd}
	{\mathcal{V}^{\op}_{\overline{F_2}(g^*(c))}} \\
	&& {\mathcal{V}^{\op}_{\overline{F_2}}} \\
	{\mathcal{V}^{\op}_{\overline{F_2}(f^*(c))}}
	\arrow[""{name=0, pos=0.45, inner sep=0}, "{\iota^{\op}_{g^*(c)}}", from=1-1, to=2-3]
	\arrow[""{name=1, anchor=center, inner sep=0}, "{\iota^{\op}_{f^*(c)}}"', from=3-1, to=2-3]
	\arrow["{\overline{F_2}(\sigma^\sharp_c)_*}"', from=1-1, to=3-1]
	\arrow["{(\overline{\sigma^\sharp_c})^{\op}}", shift left=3, shorten <=4pt, shorten >=4pt, Rightarrow, from=1, to=0]
\end{tikzcd}\]
Those data induce a morphism of fibrations between the fibered sites of the inverse images 
\[\begin{tikzcd}
	{\mathcal{V}^{\op}_{g_*\overline{F_2}}} \\
	&& {\mathcal{V}^{\op}_{\overline{F_2}}} \\
	{\mathcal{V}^{\op}_{f_*\overline{F_2}}}
	\arrow[""{name=0, anchor=center, inner sep=0}, "{q_{(g,\gamma)}}", from=1-1, to=2-3]
	\arrow[""{name=1, anchor=center, inner sep=0}, "{q_{(f,\phi)}}"', from=3-1, to=2-3]
	\arrow["{\int\overline{F_2}(\sigma^\sharp)}"', from=1-1, to=3-1]
	\arrow["{(\overline{\sigma^\sharp})^{\op}}", shift left=3, shorten <=4pt, shorten >=4pt, Rightarrow, from=1, to=0]
\end{tikzcd}\]
\end{division}

\begin{division}
Then to get the desired 2-cell in $\mathbb{T}_J \hy\GTop^\Loc$, one has to paste the associated 2-cell between the spectra of inverse images together with the spectrum of the triangle of vertical morphism 
\[\begin{tikzcd}
	&& {\Spec(g_*F_2)} \\
	{\Spec(F_2) } & {\Spec(\overline{F_2})} && {\Spec(F_1)} \\
	&& {\Spec(f_*F_1)}
	\arrow["{\Spec(\gamma^\sharp)}", from=1-3, to=2-4]
	\arrow["{\Spec(\phi^\sharp)}"', from=3-3, to=2-4]
	\arrow[""{name=0, anchor=center, inner sep=0}, "{\Spec(\sigma^\sharp_{F_2})}"{description}, from=3-3, to=1-3]
	\arrow[""{name=1, pos=0.49, inner sep=0}, "{\Sh(q_{(g,\gamma)})}", from=2-2, to=1-3]
	\arrow[""{name=2, anchor=center, inner sep=0}, "{\Sh(q_{(f,\phi)})}"', from=2-2, to=3-3]
	\arrow["{\Sh(i_{F_2}) \atop \simeq}", from=2-1, to=2-2]
	\arrow["\simeq"{description, pos=0.4}, Rightarrow, draw=none, from=2-4, to=0]
	\arrow["{\Spec(\sigma)}"{description}, shorten <=4pt, shorten >=4pt, Rightarrow, from=2, to=1]
\end{tikzcd}\]

We would have to check that this 2-cell satisfies the compatibility condition of morphisms of modelled topos. But this comes from the very definition of $ \Spec(\sigma)$ as induced from $ \overline{\sigma^\sharp}$ so that its component at $ \Sh(i_{F_2})_*\widetilde{F_2}$ satisfies
\[  \Spec(\sigma)_{\Sh(i_{F_2})_*\widetilde{F_2}} =  \widetilde{\sigma^\sharp_{F_2}}^\sharp   \]

This achieves to produce a 2-cell of locally modelled topoi
\[\begin{tikzcd}
	{(\Spec(F_1), \widetilde{F_1})} && {(\Spec(F_2), \widetilde{F_2})}
	\arrow[""{name=0, anchor=center, inner sep=0}, "{(\Spec(\phi), \widetilde{\phi})}", bend left=20, from=1-1, to=1-3]
	\arrow[""{name=1, anchor=center, inner sep=0}, "{(\Spec(\gamma), \widetilde{\gamma})}"', bend right=20, from=1-1, to=1-3]
	\arrow["{\Spec(\sigma)}"{description}, shorten <=5pt, shorten >=5pt, Rightarrow, from=0, to=1]
\end{tikzcd}\]
\end{division}

\begin{division}
To sum up this section, we have shown how to make the spectrum into a pseudofunctor 
\[\begin{tikzcd}
	{\mathbb{T}_J\hy\GTop^\Loc} & {\mathbb{T}\hy\GTop}
	\arrow["\Spec", from=1-1, to=1-2]
\end{tikzcd}\]
The next section will now prove it to be the desired left adjoint of the inclusion $ \iota_{J,\Loc}$.
\end{division}

\section{Site-theoretic version of the spectral adjunction}

In this section we shall prove that the notion of spectrum as constructed in this chapter has the desired universal property announced in the previous chapter and is part of the desired spectral adjunction. In \cite{Coste}, there was already such a proof, yet we shall follow a different strategy based on the in-depth description of the spectral site and some reflections on locally modelled topoi.\\

An important preliminary result of this section is the localness of the spectrum of a local object over its base topos, generalizing \cref{Spec of local object is local}. We also describe how the spectrum operates relative to local units and stalks of its input sheaf of $\mathbb{T}$-models.

\subsection{Spectra of locally modelled topoi are local over their base}\label{Spectra of locally modelled topoi are local}

\begin{division}\label{local sheaves are locally local}
Consider now a $\mathbb{T}_J$-locally modelled topos $ (\mathcal{E},E)$ with $ \mathcal{E} = \Sh(\mathcal{C}_\mathcal{E}, J_\mathcal{E})$. Recall that the condition that $ E$ is a $J$-local objects in $ \mathbb{T}[\mathcal{E}]$ amounts to saying that the corresponding $J$-continuous functor 
\[\begin{tikzcd}
	{\mathcal{C}_\mathbb{T}} & {\mathcal{E}}
	\arrow["{E^*}", from=1-1, to=1-2]
\end{tikzcd}\]
sends any $J$-family $ ([\theta_j]_\mathbb{T} : \{ \overline{x_j}, \phi_j\} \rightarrow \{ \overline{x}, \phi\})_{j \in I}$ to an epimorphism 
\[\begin{tikzcd}
	{\displaystyle{\coprod_{j \in I}}E(\{ \overline{x_j}, \phi_j\})} && {E(\{ \overline{x}, \phi\})}
	\arrow["{\langle E([\theta_i]_\mathbb{T}) \rangle_{j \in I}}", from=1-1, to=1-3, two heads]
\end{tikzcd}\]
But the property of being an epimorphism in the category of sheaves over $ (\mathcal{C}_\mathcal{E}, J_\mathcal{E})$ says that for any $ c \in \mathcal{C}_\mathcal{E}$ and any $ a \in E(\{ \overline{x}, \phi\})(c)$, there is a $ J_\mathcal{E}$-cover $ (u_i : c_i \rightarrow c)_{i \in I'}$ such that the image $ E(\{ \overline{x}, \phi\})(u_i)(a) $ is in the range of $\langle E([\theta_j]_\mathbb{T}) \rangle_{j \in I}(c_i)$. But each $E(c)$ is a $ \mathbb{T}[\mathcal{S}]$-model, and we have 
\[ E(\{ \overline{x}, \phi\})(c) \simeq \mathbb{T}[\mathcal{S}](K_\phi, E(c)) \]
so $ a$ is some $ a : K_\phi \rightarrow E(c)$. Then this means that each of the composite $ E(u_i)a$ factorizes through some $ f_{\theta_j} : K_{\phi} \rightarrow K_{\phi_j}$ for some $j \in J$, and then we have a factorization of the identiy of $E(c_i)$ through the corresponding pushout as follows:
\[\begin{tikzcd}[row sep=large]
	{K_\phi} & {E(c)} & {E(c_i)} & {E(c_i)} \\
	{K_{\phi_j}} & {a_*K_{\phi_j}} & {E(u_i)_*a_*K_{\phi_j}}
	\arrow["{E(u_i)}", from=1-2, to=1-3]
	\arrow["a", from=1-1, to=1-2]
	\arrow["{f_{\theta_j}}"', from=1-1, to=2-1]
	\arrow[from=2-1, to=2-2]
	\arrow["{a_*f_{\theta_j}}"{description}, from=1-2, to=2-2]
	\arrow[Rightarrow, no head, from=1-3, to=1-4]
	\arrow["{E(u_i)_*a_*f_{\theta_j}}"{description}, from=1-3, to=2-3]
	\arrow[from=2-2, to=2-3]
	\arrow["b"', dashed, from=2-3, to=1-4]
	\arrow["\lrcorner"{anchor=center, pos=0.125, rotate=180}, draw=none, from=2-2, to=1-1]
	\arrow["\lrcorner"{anchor=center, pos=0.125, rotate=180}, draw=none, from=2-3, to=1-2]
\end{tikzcd}\]
This leads to the following:
\end{division}

\begin{proposition}\label{geometric part of the counit}
If $ E$ is a $J$-local object in $\mathbb{T}[\mathcal{E}]$, then $\iota_{wE} : (\mathcal{C}_\mathcal{E}, J_\mathcal{E}) \hookrightarrow (\mathcal{V}_{wE}^{\op}, J_{wE}) $ is also a comorphism of site, and defines a geometric morphism $e_E= \Sh(\iota_{wE})_*\dashv \Sh(\iota_{wE})^!$ which is moreover a section of $ h_E$
\[\begin{tikzcd}
	& {\Spec(E)} \\
	{\mathcal{E}} && {\mathcal{E}}
	\arrow["{e_E}", from=2-1, to=1-2]
	\arrow["{h_E}", from=1-2, to=2-3]
	\arrow[Rightarrow, no head, from=2-1, to=2-3]
\end{tikzcd}\]
\end{proposition}

\begin{proof}
For any $c$ in $\mathcal{C}_\mathcal{E}$, a $ \iota_c( J_{F(c)})$-cover $ (1_c,m_j):  (c, n_j) \rightarrow (c, 1_{F(c)}))_{j \in J}$ coming from $ (m_j : E(c) \rightarrow \cod(n_j))_{j \in I} $ is obtained as
\[\begin{tikzcd}
	{K_\phi} & {E(c)} \\
	{K_{\phi_j}} & {\cod(n_i)}
	\arrow["a", from=1-1, to=1-2]
	\arrow["{f_{\theta_j}}"', from=1-1, to=2-1]
	\arrow[from=2-1, to=2-2]
	\arrow["{n_j}", from=1-2, to=2-2]
	\arrow["\lrcorner"{anchor=center, pos=0.125, rotate=180}, draw=none, from=2-2, to=1-1]
\end{tikzcd}\]
with $ ([\theta_j]_\mathbb{T} : \{ \overline{x_j}, \phi_j\} \rightarrow \{ \overline{x}, \phi\})_{j \in I}  $ a $J$-family. But from what precedes, as $ a$ is an element of $ E(\{ \overline{x}, \phi\})(c)$, there is some $ (u_i : c_i \rightarrow c)_{i \in I'} $ such that for each $ i \in I'$, there is $ j \in I$ and a factorization 
\[\begin{tikzcd}
	{K_\phi} & {E(c)} & {E(c_i)} \\
	{K_{\phi_j}} & {\cod(n_j)} & {E(c_i)}
	\arrow["a", from=1-1, to=1-2]
	\arrow["{f_{\theta_j}}"', from=1-1, to=2-1]
	\arrow[from=2-1, to=2-2]
	\arrow["{n_j}"{description}, from=1-2, to=2-2]
	\arrow["\lrcorner"{anchor=center, pos=0.125, rotate=180}, draw=none, from=2-2, to=1-1]
	\arrow["{E(u_i)}", from=1-2, to=1-3]
	\arrow["b"', dashed, from=2-2, to=2-3]
	\arrow[Rightarrow, no head, from=1-3, to=2-3]
\end{tikzcd}\]
In other word, for any $i \in I'$ there is some $ j \in I$ and a factorization in $ \mathcal{V}_{wE}^{\op }$ 
\[\begin{tikzcd}
	& {(c, n_j)} \\
	{(c_i, 1_{E(c_i)})} & {(c,1_{F(c)})}
	\arrow["{(1_c, n_i)}", from=1-2, to=2-2]
	\arrow["{(u_i, E(u_i))}"', from=2-1, to=2-2]
	\arrow["{(u_i, b)}"', dashed, from=1-2, to=2-1]
\end{tikzcd}\]
But in the diagram above, $(u_i, E(u_i))_{i \in I'}$ is an horizontal family in $ \iota_{wE}(J_\mathcal{E})$, while $(1_c, n_i) $ is a vertical family in $ \iota_c(J_{F(c)})$ refining it; in particular, the vertical family must hence be contained in the sieve generated by the horizontal family, and so is the sieve its generates. Hence any $ J_{\mathcal{V}_{E(-)}}$-covering sieve is in particular a $ \iota_{wE}(J_\mathcal{E})$-covering sieve: the topology $ J_{\mathcal{V}_{E(-)}}$ is coarser than $ \iota_{wE}(J_\mathcal{E})$, and hence the inclusion $ \iota_{wE}$ becomes a comorphism of sites. 
\end{proof}

Now recall that in the case of set-based models, the spectrum of a local object was a local topos: this result will generalize as follows. To this end we invoke \cite{caramello2020denseness}[Theorem 7.20] characterization of local geometric morphisms: a continuous comorphism of site $ F :(\mathcal{C},J) \rightarrow (\mathcal{D},K)$ which is moreover a morphism of site induces a triple of adjoints $ \Sh(F)^* \dashv \Sh(F)_* \dashv \Sh(F)^!$, which can be proven to be moreover local if $ F$ is full and faithful. 

\begin{remark}
Actually, \cite{caramello2020denseness}[Theorem 7.20, (iii)] requires the functor $F$ to be \emph{$J$-full and $J$-faithful} for the condition to hold. Though we cannot give the definition of those notions which involves a notion of local equality of morphisms in the sheaf topos, we have in \cite{caramello2020denseness}[Proposition 5.16] the following characterization: $F$ is $J$-full and $J$-faithul if and only if for each $c$ in $\mathcal{C}$, we have an isomorphism
\[ \mathfrak{a}_J\hirayo_c \simeq \mathfrak{a}_J(\mathcal{D}[F(-), F(c)]) \]
But then observe that if $F$ is full and faithful in the usual sense, then for each $d $ we have $ \mathcal{C}[d,c] \simeq \mathcal{D}[F(d), F(c)]$, and hence there is a natural isomorphism $ \hirayo_c \simeq \mathcal{D}[F(-), F(c)] $ which still is an isomorphism after sheafification, so that $F$ is also $J$-full and $J$-faithful.
\end{remark}

\begin{corollary}\label{spectrum of a locally modelled topoi is local over its base}
If $ (\mathcal{E},E)$ is a $\mathbb{T}_J$-locally modelled topos, then $ h_{wE} : \Spec(E) \rightarrow \mathcal{E}$ is a local geometric morphism, with center $e_E$. In particular, whenever $\mathcal{E}$ was itself a local topos with center $ x$, then $\Spec(E)$ is a local topos with center $e_Ex$.
\end{corollary}

\begin{proof}
We saw in the previous part that $ \iota_{wE}$ becomes in this case a continuous comorphism, while still being also a morphism of site: hence it induces a triple of adjoints functors
\[\begin{tikzcd}
	{\Spec(wE)} & {} & {\mathcal{E}}
	\arrow["{\Sh(\iota_{wE})^*}"', bend right=30, from=1-3, to=1-1, start anchor=140, end anchor=10]
	\arrow["{\Sh(\iota_{wE})_*}"{description}, from=1-1, to=1-3]
	\arrow["{\Sh(\iota_{wE})^!}",  bend left=30, from=1-3, to=1-1, start anchor=220, end anchor=350]
\end{tikzcd}\]
Moreover, as $ \iota_E$ is full and faithful, then by the remark above, this triple of adjoints functors defines a local geometric morphism. Observe that we have the equalities
\[ \underset{h^*_{wE}}{\underbrace{\Sh(\iota_{wE})^* = \Sh(p_{wE})_!}} \dashv \underset{h_{wE*}=e^*_E}{\underbrace{\Sh(\iota_{wE})_* = \Sh(p_{wE})^* }} \dashv \underset{e_{E*} }{ \underbrace{\Sh(\iota_{wE})^! =\Sh(p_{wE})_*}}  \]
where both $ h^*_E$ and $ e_{E*}$ are full and faithful. In particular $ h_E$ is a connected geometric morphism. 
\end{proof}

\begin{division}\label{sheaf part of unit of local object is invertible}
Now our goal is to complete the geometric morphism $ e_E$ with a local map $ \epsilon^\flat_E : e_E^*\widetilde{wE} \rightarrow E$ in $ \mathbb{T}[\mathcal{E}]$ defining a retraction 
\[\begin{tikzcd}
	& {e_E^*\widetilde{wE}} \\
	wE && wE
	\arrow[Rightarrow, no head, from=2-1, to=2-3]
	\arrow["{we^*_E\eta^\flat_{wE}}", from=2-1, to=1-2]
	\arrow["{\epsilon^\flat_E}", from=1-2, to=2-3]
\end{tikzcd}\]
(where we used the fact that $E \simeq  e^*_Eh^*_{wE} E \simeq \Sh(\iota_{wE})_* \Sh(\iota_{wE})^* E$ from the fact that $\Sh(\iota_{wE})^* $ is full and faithful, hence has invertible unit). But we can already guess that actually such a map $ \epsilon^\flat_{E}$ has to be an isomorphism: indeed, we know from \cref{unit is a generic etale map} that $ \eta^\flat_{wE}$ has to be an etale map, as well as its inverse image $ e^*_E(\eta^\flat_{wE})$, while $ \epsilon^\flat_E$ must be defined in such a way it is local: but as we hence have a factorization of the identity of $E$, those maps form an (etale, local) factorization of the identity, hence both are isomorphisms, and we then use that $w$ is conservative.  \\

Actually, we can see directly which is this isomorphism: recall that sheafification commutes with inverse images, so that we have 
\begin{align*}
    e^*_E\widetilde{wE} &\simeq e^*_E\mathfrak{a}_J \cod\\
    &\simeq \mathfrak{a}_J \cod \circ \iota_E
\end{align*}
But we have at each $ c$ of $\mathcal{C}_\mathcal{E}$ that $ \cod(\iota_E(c)) = \cod(c, 1_{E(c)}) = E(c)$: so we have a pointwise equality $ \cod \circ \iota_E = E $ which is sent to an isomorphism after sheafification $ \epsilon^\sharp_E : we^*_E \widetilde{wE} \simeq E$ for $ \mathfrak{a}_J\iota_{\mathcal{E}} E \simeq E$ as $E$ is already a sheaf. \\  

Now the mate $ \epsilon^\sharp_E$ coincides up to a canonical iso with the unit of the $e^*_E \dashv e_{E*}$ adjunction as one has
\[\begin{tikzcd}
	{\widetilde{wE}} & {e_{E*}E} \\
	{e_{E*}e^*_E\widetilde{wE}}
	\arrow[from=1-1, to=2-1]
	\arrow["{\epsilon^\sharp_E}", from=1-1, to=1-2]
	\arrow["{e^*(\epsilon^\flat_E) \atop \simeq}"', from=2-1, to=1-2]
\end{tikzcd}\]
But as $ h_{wE}$ is a local geometric geometric morphism with $ h_{wE}^! \simeq e_{E*}$, $e_{E*}$ is full and faithful, so its unit is an isomorphism, and as a consequence $ \epsilon^\sharp_E$ is itself an isomorphism. \\

This exhibits the counit as an horizontal morphism 
\[\begin{tikzcd}[sep=huge]
	{(\Spec(wE), \widetilde{wE})} & {(\mathcal{E},E)}
	\arrow["{(e_E,(1_E, 1_{\widetilde{wE}}))}", from=1-1, to=1-2]
\end{tikzcd}\]
\end{division}

\begin{remark}
Beware that $ e_{E*}$ is not just precomposition with the projection $p_{wE}$ as one needs further sheafification over the vertical families. 
\end{remark}

\begin{remark}
When $ E$ lives in $\mathcal{S}$, then $ (\mathcal{S}, E)$ being a locally modelled topos means for $E$ to be a set-valued local object. Then $ e_E$ defines an initial point of $ \Spec(wE) $, and as $ e^*(*)=1_{\Spec(E)} $ one has 
\begin{align*}
    \Gamma \widetilde{wE} &\simeq \widetilde{wE}(1_E) \\ &\simeq h_{wE*}E \\
    &\simeq e^*_E \widetilde{wE}
\end{align*}
Therefor $ e_E$ corresponds to the local form $ 1_E : E \rightarrow E$ and we have a representation
\begin{align*}
    E &\simeq e_E^*\widetilde{wE} \\
    &\simeq \Gamma\widetilde{wE}
\end{align*}

In the case of a set based local object, it would just express the fact that global sections are determined by the stalk at the focal point of the spectrum of the local object. The general case means the same thing but in term of $\mathcal{E}$-indexed point, and enforces that the inverse image of the structure sheaf along $ e^*_E$ -- equivalently, the direct image along $ h_{E*}$ -- returns the original sheaf. Hence local objects enjoy sheaf representation ``for free" -- which however will require additional assumption for arbitrary objects.
\end{remark}

\subsection{The spectral bi-adjunction}

In this subsection we prove the morphisms of modelled topos $ (h_F, \eta_F)$ and the morphisms of locally modelled topos $ (e_E, \epsilon_E)$ are respectively units and counits of a bi-adjunction $ \Spec \dashv \iota_{J,\Loc}$. Our proof here is totally different from \cite{Coste} -- as it exploits the concrete site presentation of the spectrum as well as the localness result established above. 

\begin{theorem}\label{Coste general adjunction}
We have a biadjunction 
\[\begin{tikzcd}
	{\mathbb{T}_J\hy\GTop^{\Loc}} && {\mathbb{T}\hy\GTop}
	\arrow[""{name=0, anchor=center, inner sep=0}, "{\iota_{J,\Loc}}"', curve={height=20pt}, hook, from=1-1, to=1-3]
	\arrow[""{name=1, anchor=center, inner sep=0}, "\Spec"', curve={height=20pt}, from=1-3, to=1-1]
	\arrow["\dashv"{anchor=center, rotate=-90}, draw=none, from=1, to=0]
\end{tikzcd}\]
\end{theorem}

\begin{proof}
Let $ (\mathcal{F}, F)$ be a $ \mathbb{T}$-modelled topos and $ (\mathcal{E}, E)$ a $ \mathbb{T}_J$-locally modelled topos. We must construct an equivalence of homcategories 
\[ {\mathbb{T}\hy\GTop} \big[ (\mathcal{F }, F), (\mathcal{E}, wE) \big] \simeq {\mathbb{T}_J\hy\GTop^{\Loc}}\big[ (\Spec(F), \widetilde{F}), (\mathcal{E}, E) \big]  \]

Let $ (f,\phi) : (\mathcal{F},F) \rightarrow (\mathcal{E},E)$ be in ${\mathbb{T}\hy\GTop}$. Then by pseudonaturality of $ (h,\eta)  : 1 \Rightarrow \iota_{J,\Loc} \Spec$ we have a pseudocommutative square in $ \GTop$
\[\begin{tikzcd}
	{\Spec(wE)} & {\Spec(F)} \\
	{\mathcal{E}} & {\mathcal{F}}
	\arrow["{h_{wE}}"', from=1-1, to=2-1]
	\arrow["{\Spec(\phi)}", from=1-1, to=1-2]
	\arrow["{h_F}", from=1-2, to=2-2]
	\arrow["f"', from=2-1, to=2-2]
	\arrow["{\eta_f \atop \simeq}"{description}, draw=none, from=1-1, to=2-2]
\end{tikzcd}\]
and in particular a natural isomorphism between inverse images
\[ \Spec(\phi)^*h_F^* \stackrel{\eta_f}{\simeq} h_{wE}^*f^*  \]
Hence we have a commutative square between inverse images in $\Spec(wE)$
\[\begin{tikzcd}
	{\Spec(\phi)^*h_F^*F} & {h_{wE}^*f^*F} & {h^*_{wE}wE} \\
	{\Spec(\phi)^*w\widetilde{F}} && {w\widetilde{E}}
	\arrow["{\eta_{wE}^\flat}", from=1-3, to=2-3]
	\arrow["{w\widetilde{\phi}^\flat}"', from=2-1, to=2-3]
	\arrow["{h_{wE}^*\phi^\flat}", from=1-2, to=1-3]
	\arrow["{\Spec(\phi)^*\eta_{F}^\flat}"', from=1-1, to=2-1]
	\arrow["{(\eta_f)_F \atop\simeq}", from=1-1, to=1-2]
\end{tikzcd}\]
Pasting the diagram of geometric morphisms above with the retraction we obtained at \cref{geometric part of the counit}, we get 
\[\begin{tikzcd}
	{\mathcal{E}} & {\Spec(wE)} & {\Spec(F)} \\
	& {\mathcal{E}} & {\mathcal{F}}
	\arrow["{h_{wE}}"{description}, from=1-2, to=2-2]
	\arrow["{\Spec(\phi)}", from=1-2, to=1-3]
	\arrow["{h_F}", from=1-3, to=2-3]
	\arrow["f"', from=2-2, to=2-3]
	\arrow["{\eta_f \atop \simeq}"{description}, draw=none, from=1-2, to=2-3]
	\arrow[""{name=0, anchor=center, inner sep=0}, Rightarrow, no head, from=1-1, to=2-2]
	\arrow["{e_E}", from=1-1, to=1-2]
	\arrow["\simeq"{description}, Rightarrow, draw=none, from=0, to=1-2]
\end{tikzcd}\]
so that, from $ e^*_E h_{wE}^* \simeq 1$, the commutative square of inverse images above is sent to 
\[\begin{tikzcd}
	{f^*F} && wE \\
	{e_E^*\Spec(\phi)^*w\widetilde{F}} && {e_E^*w\widetilde{E}}
	\arrow["{e^*_E\eta_{wE}^\flat}", from=1-3, to=2-3]
	\arrow["{e_E^*w\widetilde{\phi}^\flat}"', from=2-1, to=2-3]
	\arrow["{e_E^*\Spec(\phi)^*\eta_{F}^\flat}"', from=1-1, to=2-1]
	\arrow["{\phi^\flat}", from=1-1, to=1-3]
\end{tikzcd}\]
But from \cref{sheaf part of unit of local object is invertible} we know that $ e^*_E\eta_{wE}^\flat$ is an isomorphism with inverse $ \epsilon^\flat_E$ as $ E$ is local, so this actually defines a morphism in $\mathbb{T}[\mathcal{E}]$ between local objects
\[\begin{tikzcd}
	{e_E^*\Spec(\phi)^*w\widetilde{F}} && wE
	\arrow["{\epsilon^\flat_Ee_E^*w\widetilde{\phi}^\flat}", from=1-1, to=1-3]
\end{tikzcd}\]
But as $ \widetilde{\phi}^\flat$ is local as well as the isomorphism $ \epsilon^\flat_E$ this morphism is local, and comes uniquely from a morphism in $\mathbb{T}_J[\mathcal{E}]^\Loc$
\[\begin{tikzcd}
	{e_E^*\Spec(\phi)^*\widetilde{F}} && E
	\arrow["{\epsilon^\flat_Ee_E^*\widetilde{\phi}^\flat}", from=1-1, to=1-3]
\end{tikzcd}\]
This provides us with a morphism of locally modelled topoi which satisfies by its very construction the following pseudocommutation
\[\begin{tikzcd}
	{(\mathcal{F},F)} && {(\mathcal{E},E)} \\
	{(\Spec(F),\widetilde{F})}
	\arrow[""{name=0, anchor=center, inner sep=0}, "{(\Spec(\phi)e_E, {\epsilon^\flat_Ee_E^*\widetilde{\phi}^\flat})}"', from=2-1, to=1-3]
	\arrow["{(h_F,\eta_F)}"', from=1-1, to=2-1]
	\arrow["{(f,\phi)}", from=1-1, to=1-3]
	\arrow["\simeq"{description}, Rightarrow, draw=none, from=1-1, to=0]
\end{tikzcd}\]

The converse direction follow similar argument. A morphism of locally modelled topoi
\[\begin{tikzcd}
	{(\Spec(F), \widetilde{F})} & {(\mathcal{E},E)}
	\arrow["{(g,\gamma)}", from=1-1, to=1-2]
\end{tikzcd}\]
can have its image pasted with with the unit of $ (\mathcal{F},F) $ to get a morphism as desired; this will return in $\mathcal{E}$ a composite
\[\begin{tikzcd}
	{g^*h_F^*F} & {g^*w\widetilde{F}} \\
	& wE
	\arrow["{g^*(\eta_F^\flat)}", from=1-1, to=1-2]
	\arrow["{\gamma^\flat}", from=1-2, to=2-2]
	\arrow["{\gamma^\flat(\eta^\flat_F)}"', from=1-1, to=2-2]
\end{tikzcd}\]
which is actually uniquely determined from $ \eta_F^\flat$ and $ \gamma^\flat$ as the first one is etale -- as well as its inverse image -- while the second one is local. Hence no other morphism $ g^*F \rightarrow E$ induces the same composite as it would provide two distinct local parts for a same map, which is impossible by the uniqueness of the factorization. Proving that $(g,\gamma)(h_F,\eta_F)$ induces back the same $(g,\gamma) $ after applying the spectrum and pasting it with the canonical retraction of $(\mathcal{E},E) $ is routine.
\end{proof}

\begin{remark}
Observe that in both directions, we actually compute the etale-local factorization of the inverse image of the morphism of sheaves, where the etale part is indexed by the unit $ \eta^\flat_F$ of $ F$ in a universal way, in the sense that $ f^*\eta^\flat_F$ still indexes the etale parts of morphisms from $f^*F$ toward local objects in $\mathcal{E}$. 
\end{remark}

\begin{division}
To conclude this section, we shall use the result above to describe points of $ \Spec (F)$ and how they are related to points of the underlying topos $\mathcal{F}$. Observe that each point $ x : \mathcal{S} \rightarrow \mathcal{F}$ induces a stalk $ x^*F$ which is in $\mathbb{T}[\mathcal{S}]$. From the general theory of sheaves, we know this stalk to be expressed as the filtered colimit \[ x^*F \simeq \underset{(c,a) \in (\int x^*)^{\op}}{\colim } \, F(c) \]
where the colimit inclusions are the restrictions functors $ \rho^{(c,a)}_x : F(c) \rightarrow x^*F$ -- this colimit being filtered for $ x^* : \mathcal{C}_\mathcal{F} \rightarrow \mathcal{S}$ is flat. Moreover, $ x$ determines a morphism of modelled topoi
\[\begin{tikzcd}[sep=large]
	{(\mathcal{F}, F)} & {(\mathcal{S}, x^*F)}
	\arrow["{(x, 1_{x^*F})}", from=1-1, to=1-2]
\end{tikzcd}\]
which is sent to a morphism of locally modelled topoi
\[\begin{tikzcd}[sep=large]
	{(\Spec(F), \widetilde{ F})} && {(\Spec(x^*F), \widetilde{x^*F})}
	\arrow["{( \Spec(1_{x^*F}), \widetilde{ 1_{x^*F}})}", from=1-1, to=1-3]
\end{tikzcd}\]
\end{division}

\begin{theorem}\label{pseudolimit decomposition of stalks}
For any $x : \mathcal{S} \rightarrow \mathcal{F}$ we have a pseudolimit decomposition 
\[ \Spec( x^*F) \simeq \underset{(c,a) \in \int x^*}{\bilim} \, \Spec(F(c)) \]
\end{theorem}

\begin{proof}
From the bi-adjunction obtained in \cref{biadjunction for set-models}, we know that $ \Spec$ preserves bicolimits of set-valued models. Hence we have 
\[  \Spec( x^*F, \widetilde{x^*F}) \simeq \underset{(c,a) \in \int x^*}{\bicolim} \, (\Spec(F(c)), \widetilde{F(c)})  \]
Now from the expression of colimits of modelled topoi, and the fact that $ \iota_{J, \Loc}$ is a morphism of fibration, we know that the underlying Grothendieck topos is a cofiltered bilimit as desired.
\end{proof}

\subsection{Points of the spectrum}

\begin{division}
Now we would like to describe points of $ \Spec (F)$ and how they are related to points of the underlying topos $\mathcal{F}$. Observe that each point $ x : \mathcal{S} \rightarrow \mathcal{F}$ induces a stalk $ x^*F$ which is in $\mathbb{T}[\mathcal{S}]$. From the general theory of sheaves, we know this stalk to be expressed as the filtered colimit \[ x^*F \simeq \underset{(c,a) \in (\int x^*)^{\op}}{\colim } \, F(c) \]
where the colimit inclusions are the restrictions functors $ \rho^{(c,a)}_x : F(c) \rightarrow x^*F$ -- this colimit being filtered as $ x^* : \mathcal{C}_\mathcal{F} \rightarrow \mathcal{S}$ is flat. Moreover, $ x$ determine a morphism of modelled topos
\[\begin{tikzcd}[sep=large]
	{(\mathcal{F}, F)} & {(\mathcal{S}, x^*F)}
	\arrow["{(x, 1_{x^*F})}", from=1-1, to=1-2]
\end{tikzcd}\]
which is sent to a morphism of locally modelled topoi
\[\begin{tikzcd}[sep=large]
	{(\Spec(F), \widetilde{ F})} && {(\Spec(x^*F), \widetilde{x^*F})}
	\arrow["{( \Spec(1_{x^*F}), \widetilde{ 1_{x^*F}})}", from=1-1, to=1-3]
\end{tikzcd}\]
\end{division}

\begin{theorem}\label{pseudolimit decomposition of stalks}
For any $x : \mathcal{S} \rightarrow \mathcal{F}$ we have a pseudolimit decomposition 
\[ \Spec( x^*F) \simeq \underset{(c,a) \in \int x^*}{\lim} \, \Spec(F(c)) \]
\end{theorem}

\begin{proof}
To see this we are going to prove the pseudocolimit decomposition
\[  \mathcal{V}_{ x^*F} \simeq \underset{(c,a) \in (\int x^*)^{\op}}{\colim} \,  \mathcal{V}_{F(c)}  \]
Observe that we have a functor
\[\begin{tikzcd}
	{\underset{(c,a) \in \int x^*}{\oplaxcolim} \, \mathcal{V}_{F(c)}} & {\mathcal{V}_{x^*F}}
	\arrow[from=1-1, to=1-2]
\end{tikzcd}\]
sending an object $ ((c, a), n)$ with $ n : F(c) \rightarrow \cod(n)$ and $ a \in x^*(c)$ to the pushout 
\[\begin{tikzcd}
	{F(c)} & {x^*F} \\
	{\cod(n)} & {{\rho^{(c,a)}_x}_*\cod(n)}
	\arrow["n"', from=1-1, to=2-1]
	\arrow["{\rho^{(c,a)}_x}", from=1-1, to=1-2]
	\arrow["{{\rho^{(c,a)}_x}_*n}", from=1-2, to=2-2]
	\arrow[from=2-1, to=2-2]
	\arrow["\lrcorner"{anchor=center, pos=0.125, rotate=180}, draw=none, from=2-2, to=1-1]
\end{tikzcd}\]
and a morphism $ (u,f) :((c_1, a_1),n_1) \rightarrow
((c_2, a_2),n_2) $ of $ \oplaxcolim_{(c,a) \in (\int x^*)^{\op}} \mathcal{V}_{F(c)} $ to the morphism induced from the universal property of the pushout as depicted below
\[\begin{tikzcd}
	& {F(c_2)} && {F(c_1)} \\
	{B_2} &&&& {B_1} \\
	&& {x^*F} \\
	& {{\rho^{(c_2,a_2)}_x}_*B_2} && {{\rho^{(c_1,a_1)}_x}_*B_1}
	\arrow["{n_2}"', from=1-2, to=2-1]
	\arrow["{F(u)}", from=1-2, to=1-4]
	\arrow["{n_1}", from=1-4, to=2-5]
	\arrow[from=2-1, to=4-2]
	\arrow["{\rho^{(c_1,a_1)}_x}"{description, pos=0.2}, from=1-2, to=3-3]
	\arrow["{\rho^{(c_2,a_2)}_x}"{description, pos=0.2}, from=1-4, to=3-3]
	\arrow["{{\rho^{(c_2,a_2)}_x}_*n_2}"{description}, from=3-3, to=4-2]
	\arrow["{{\rho^{(c_1,a_1)}_x}_*n_1}"{description}, from=3-3, to=4-4]
	\arrow[from=2-5, to=4-4]
	\arrow[""{name=0, anchor=center, inner sep=0}, "f"{description}, from=2-1, to=2-5, crossing over]
	\arrow[dashed, from=4-2, to=4-4]
	\arrow["\lrcorner"{anchor=center, pos=0.125, rotate=90}, shift left=2, draw=none, from=4-2, to=0]
	\arrow["\lrcorner"{anchor=center, pos=0.125, rotate=180}, shift right=1, draw=none, from=4-4, to=0]
\end{tikzcd}\]
This functor is essentially surjective: indeed, for any pushout square as below in $\mathcal{V}_{x^*F}$
\[\begin{tikzcd}
	K & {x^*F} \\
	{K'} & {b^*K'}
	\arrow["m"', from=1-1, to=2-1]
	\arrow["a", from=1-1, to=1-2]
	\arrow["{b_*m}", from=1-2, to=2-2]
	\arrow[from=2-1, to=2-2]
	\arrow["\lrcorner"{anchor=center, pos=0.125, rotate=180}, draw=none, from=2-2, to=1-1]
\end{tikzcd}\]
then because of the colimit decomposition of $x^*F = \colim_{(c,a) \in \int x^*} F(c)$ we have a lift of $ b$ through some $\rho^{(c,a)}_x $ which induces a factorization of pushouts a follows
\[\begin{tikzcd}
	& {F(c)} \\
	K && {x^*F} \\
	& {\overline{b}_*K'} \\
	{K'} && {b^*K'}
	\arrow[""{name=0, anchor=center, inner sep=0}, "m"', from=2-1, to=4-1]
	\arrow["{a_*m}", from=2-3, to=4-3]
	\arrow[from=4-1, to=4-3]
	\arrow[""{name=1, anchor=center, inner sep=0}, "\lrcorner"{anchor=center, pos=0.125, rotate=180}, shift right=3, draw=none, from=4-3, to=2-1]
	\arrow["{\overline{b}}", from=2-1, to=1-2]
	\arrow["{\rho^{(c,a)}_x}", from=1-2, to=2-3]
	\arrow[from=4-1, to=3-2]
	\arrow[from=1-2, to=3-2]
	\arrow["b"{description, pos=0.4}, from=2-1, to=2-3, crossing over]
	\arrow[from=3-2, to=4-3, crossing over]
	\arrow["\lrcorner"{anchor=center, pos=0.125, rotate=180}, draw=none, from=3-2, to=2-1]
\end{tikzcd}\]
Now recall that a morphism $ (u,f) :((c_1, a_1),n_1) \rightarrow
((c_2, a_2),n_2) $ is cartesian if $f $ if and only if the corresponding square 
\[\begin{tikzcd}
	{F(c_2)} & {F(c_1)} \\
	{B_2} & {B_1}
	\arrow["{n_2}"', from=1-1, to=2-1]
	\arrow["{F(u)}", from=1-1, to=1-2]
	\arrow["{n_1}", from=1-2, to=2-2]
	\arrow["f"', from=2-1, to=2-2]
\end{tikzcd}\]
is a pushout. Then any cartesian morphism $ (u,f) :((c_1, a_1),n_1) \rightarrow
((c_2, a_2),n_2) $ is obviously sent to an isomorphism in $ \mathcal{V}_{x^*F}$ by composition and uniqueness of pushouts. Hence the functor from the oplax colimit localizes cartesian morphisms, and hence factorizes through the pseudocolimit as an essentially surjective functor
\[  \underset{(c,a) \in (\int x^*)^{\op}}{\colim} \,  \mathcal{V}_{F(c)} 
\rightarrow \mathcal{V}_{ x^*F} \]
Now we must prove that this functor is full and faithful. For faithfulness, observe that for any parallel pair in the oplax colimit which are send to the same morphism in $ \mathcal{V}_{F(c)}$ as depicted below
\[\begin{tikzcd}
	& {F(c_2)} && {F(c_1)} \\
	{B_2} &&&& {B_1} \\
	&& {x^*F} \\
	& {{\rho^{(c_2,a_2)}_x}_*B_2} && {{\rho^{(c_1,a_1)}_x}_*B_1}
	\arrow["{n_2}"', from=1-2, to=2-1]
	\arrow["{F(u)}", shift left=1, from=1-2, to=1-4]
	\arrow["{n_1}"{description}, from=1-4, to=2-5]
	\arrow[from=2-1, to=4-2]
	\arrow["{\rho^{(c_1,a_1)}_x}"{description, pos=0.2}, from=1-2, to=3-3]
	\arrow["{\rho^{(c_2,a_2)}_x}"{description, pos=0.2}, from=1-4, to=3-3]
	\arrow["{{\rho^{(c_2,a_2)}_x}_*n_2}"{description}, from=3-3, to=4-2]
	\arrow["{{\rho^{(c_1,a_1)}_x}_*n_1}"{description}, from=3-3, to=4-4]
	\arrow[from=2-5, to=4-4]
	\arrow[""{name=0, anchor=center, inner sep=0}, "f"{description}, shift left=1, from=2-1, to=2-5, crossing over]
	\arrow[dashed, from=4-2, to=4-4]
	\arrow["{F(v)}"', shift right=1, from=1-2, to=1-4]
	\arrow["g"{description}, shift right=2, from=2-1, to=2-5, crossing over]
	\arrow["\lrcorner"{anchor=center, pos=0.125, rotate=90}, shift left=2, draw=none, from=4-2, to=0]
	\arrow["\lrcorner"{anchor=center, pos=0.125, rotate=180}, shift right=1, draw=none, from=4-4, to=0]
\end{tikzcd}\]
there exists by filteredness of $ (\int x^*)^{\op}$ some $(c,a)$ and a morphism $ w : c \rightarrow c_1 $ equalizing $ u$, $v$ and such that we have equalization
\[\begin{tikzcd}
	{F(c_2)} & {F(c_1)} & {F(c)} \\
	& {x^*F}
	\arrow["{F(u)}", shift left=1, from=1-1, to=1-2]
	\arrow["{F(v)}"', from=1-1, to=1-2]
	\arrow["{\rho^{(c_1,a_1)}_x}"'{pos=0.6}, from=1-1, to=2-2]
	\arrow["{\rho^{(c_2,a_2)}_x}"{description}, from=1-2, to=2-2]
	\arrow["{F(w)}", from=1-2, to=1-3]
	\arrow["{\rho^{(c,a)}_x}", from=1-3, to=2-2]
\end{tikzcd}\]
Observe moreover that $ f$, $g$ are uniquely determined by the induced maps $ \langle f, n_1 \rangle $ and $ \langle g,n_1 \rangle$ respectively as depicted below
\[\begin{tikzcd}
	& {F(c_1)} \\
	{F(u)_*B_2} && {F(v)_*B_2} \\
	& {B_1}
	\arrow["{F(u)_*n_2}"', from=1-2, to=2-1]
	\arrow["{F(v)_*n_2}", from=1-2, to=2-3]
	\arrow["{n_1}"{description}, from=1-2, to=3-2]
	\arrow["{\langle f, n_1 \rangle }"', from=2-1, to=3-2]
	\arrow["{\langle g, n_1 \rangle }", from=2-3, to=3-2]
\end{tikzcd}\]
But we have $ F(w)_*(F(u)_*n_2) =F(w)_*(F(v)_*n_2) $, so that we get a parallel pair 
\[\begin{tikzcd}
	& {F(c)} \\
	{F(w)_*F(u)_*B_2} && {F(w)_*B_1}
	\arrow["{F(w)_*F(u)_*n_2 \atop =F(w)_*F(v)_*n_2 }"', from=1-2, to=2-1]
	\arrow["{F(w)_*n_1}", from=1-2, to=2-3]
	\arrow["{\langle f, F(w)_*n_1 \rangle }", shift left=1, from=2-1, to=2-3]
	\arrow["{\langle g, F(w)_*n_1 \rangle }"', shift right=1, from=2-1, to=2-3]
\end{tikzcd}\]
But those two arrows are induced actually from a same arrow \[{n_1}_*F(w) \langle f, F(w)_*n_1 \rangle = {n_1}_*F(w) \langle g, F(w)_*n_1 \rangle  \]
and hence are equal. Hence, being equalized by a cartesian morphism $ (w,{n_1}_*F(w)) $, $ (u,f)$ and $(v,g)$ are already identified in the pseudocolimit: hence the faithfulness. \\

For fullness, consider $((c_1, a_1),n_1)$ and $ ((c_2, a_2),n_2)$ and a morphism $f$ between their image as in the situation below
\[\begin{tikzcd}
	& {F(c_2)} && {F(c_1)} \\
	{B_2} && {x^*F} && {B_1} \\
	& {{\rho^{(c_2,a_2)}_x}_*B_2} && {{\rho^{(c_1,a_1)}_x}_*B_1}
	\arrow["{\rho^{(c_2,a_2)}_x}"{description}, from=1-2, to=2-3]
	\arrow["{\rho^{(c_1,a_1)}_x}"{description}, from=1-4, to=2-3]
	\arrow["{{\rho^{(c_2,a_2)}_x}_*n_2}"{description}, from=2-3, to=3-2]
	\arrow["{n_2}"', from=1-2, to=2-1]
	\arrow[from=2-1, to=3-2]
	\arrow["{{\rho^{(c_1,a_1)}_x}_*n_1}"{description}, from=2-3, to=3-4]
	\arrow["{n_1}", from=1-4, to=2-5]
	\arrow[from=2-5, to=3-4]
	\arrow["f"', from=3-2, to=3-4]
\end{tikzcd}\]
Then by filteredness of $(\int x^*)^{\op}$ there is some $ u_1 : c\rightarrow c_1 $, $u_2 : c \rightarrow c_2$ such that $ F(u_1)(a_1) = F(u_2)(c_2) = a $, so that we have a factorization
\[\begin{tikzcd}
	{F(c_1)} & {F(c)} & {F(c_2)} \\
	& {x^*F}
	\arrow["{\rho^{(c_1,a_1)}_x}"', from=1-1, to=2-2]
	\arrow["{\rho^{(c_2,a_2)}_x}", from=1-3, to=2-2]
	\arrow["{F(u_1)}", from=1-1, to=1-2]
	\arrow["{F(u_2)}"', from=1-3, to=1-2]
	\arrow["{\rho^{(c,a)}_x}"{description}, from=1-2, to=2-2]
\end{tikzcd}\]
Moreover, choosing some $ (b,m)$ inducing $ n_2$ with $ m$ in $\mathcal{V}$, this $(c,a)$ can be chosen to be also a solution provided by \cref{Anel lemma} applied to the two following situations
\[\begin{tikzcd}
	{K'} & {F(c_1)} \\
	{B_2} & {x^*F} & {B_1} \\
	{{\rho^{(c_2,a_2)}_x}_*B_2} && {{\rho^{(c_1,a_1)}_x}_*B_1}
	\arrow["{{\rho^{(c_1,a_1)}_x}_*n_1}"{description}, from=2-2, to=3-3]
	\arrow["f"', from=3-1, to=3-3]
	\arrow["{\rho^{(c_1,a_1)}_x}"', from=1-2, to=2-2]
	\arrow["{n_1}", from=1-2, to=2-3]
	\arrow[from=2-3, to=3-3]
	\arrow["{m_*b}"', from=1-1, to=2-1]
	\arrow["{{n_2}_*{\rho^{(c_2,a_2)}_x}}"', from=2-1, to=3-1]
\end{tikzcd} \hskip 1cm \begin{tikzcd}
	K \\
	{K'} & {F(c_1)} \\
	{B_2} & {x^*F} & {B_1} \\
	{{\rho^{(c_2,a_2)}_x}_*B_2} && {{\rho^{(c_1,a_1)}_x}_*B_1}
	\arrow["{{\rho^{(c_1,a_1)}_x}_*n_1}"{description}, from=3-2, to=4-3]
	\arrow["f"', from=4-1, to=4-3]
	\arrow["{\rho^{(c_1,a_1)}_x}"', from=2-2, to=3-2]
	\arrow["{n_1}", from=2-2, to=3-3]
	\arrow[from=3-3, to=4-3]
	\arrow["{m_*b}"', from=2-1, to=3-1]
	\arrow["{{n_2}_*{\rho^{(c_2,a_2)}_x}}"', from=3-1, to=4-1]
	\arrow["m"', from=1-1, to=2-1]
\end{tikzcd}\]
so that we can infere the existence of some arrow $g$ as below
\[\begin{tikzcd}
	& K \\
	{K'} & {F(c_2)} && {F(c_1)} \\
	{B_2} && {F(c)} && {B_1} \\
	& {F(u_2)_*B_2} & {x^*F} & {F(u_1)_*B_1} \\
	& {{\rho^{(c_2,a_2)}_x}_*B_2} && {{\rho^{(c_1,a_1)}_x}_*B_1}
	\arrow["{n_1}", from=2-4, to=3-5]
	\arrow["{n_2}"{description}, from=2-2, to=3-1]
	\arrow["f"{description}, from=5-2, to=5-4]
	\arrow["{{\rho^{(c_2,a_2)}_x}_*n_2}"{description}, from=4-3, to=5-2]
	\arrow["{{\rho^{(c_1,a_1)}_x}_*n_1}"{description}, from=4-3, to=5-4]
	\arrow[""{name=0, anchor=center, inner sep=0}, "{F(u_2)_*n_2}"{description}, from=3-3, to=4-2]
	\arrow[""{name=1, anchor=center, inner sep=0}, "{F(u_1)_*n_1}", from=3-3, to=4-4]
	\arrow[""{name=2, anchor=center, inner sep=0}, "{F(u_2)}"{description}, from=2-2, to=3-3]
	\arrow[from=3-1, to=4-2]
	\arrow[""{name=3, anchor=center, inner sep=0}, "{F(u_1)}"{description}, from=2-4, to=3-3]
	\arrow[from=3-5, to=4-4]
	\arrow[from=4-4, to=5-4]
	\arrow[from=4-2, to=5-2]
	\arrow[from=3-3, to=4-3]
	\arrow[from=1-2, to=2-2]
	\arrow[from=2-1, to=3-1]
	\arrow["m"', from=1-2, to=2-1]
	\arrow["\lrcorner"{anchor=center, pos=0.125, rotate=90}, draw=none, from=3-1, to=1-2]
	\arrow["g"{description}, dashed, from=2-1, to=4-4, crossing over]
	\arrow["\lrcorner"{anchor=center, pos=0.125, rotate=90}, draw=none, from=4-2, to=2]
	\arrow["\lrcorner"{anchor=center, pos=0.125, rotate=180}, draw=none, from=4-4, to=3]
	\arrow["\lrcorner"{anchor=center, pos=0.125, rotate=90}, shift right=1, draw=none, from=5-2, to=0]
	\arrow["\lrcorner"{anchor=center, pos=0.125, rotate=180}, shift left=1, draw=none, from=5-4, to=1]
\end{tikzcd}\]
inducing an arrow 
\[ {\rho^{(c_2,a_2)}_x}_*n_2 \rightarrow {\rho^{(c_1,a_1)}_x}_*n_1 \] 
in $\mathcal{V}_{F(c)}$ inducing $f$ in $ \mathcal{V}_{x^*F}$. Hence the fullness.\\

Now the pseudocolimit is equipped with an induced topology
\[  \langle \bigcup_{(c,a) \in \int x^*} J_{F(c)} \rangle \]
and one would have to prove that any $J_{ x^*F}$-covering family is a covering family of this induced topology and conversely: but the strategy is exactly the same as in \cref{pro-etale spectrum}, so we dispense ourselves from this redundant work. We end up with an equivalence of sites, and hence, with an equivalence of the induced topoi as desired.
\end{proof}

\section{Naturality of the spectral adjunction}

For the sake of completeness, we treat also here the functoriality of the spectral construction relative to transformations of geometries as described at \cref{morphisms of geometries}. In fact most of the geometries presented in the examples section are related to each other through such transformations, which induce such comparisons functor -- some of them where individually investigated, as the so called \emph{Belluce functor} we recover as a comparison map at \cref{Belluce}. 

\subsection{Comparison of spectral sites}

\begin{division}\label{mate for transformations of geometries}
 For a transformation of geometries $ \Phi$ we want to describe explicitly the mate 
\[\begin{tikzcd}
	{\displaystyle\int \Phi \Spec_2} & {\displaystyle\Spec_1 \int \Phi }
	\arrow["{\sigma_\Phi}", Rightarrow, from=1-1, to=1-2]
\end{tikzcd}\]
 where $ \int \Phi$ denotes the morphism of opfibrations between the bicategories of modelled topoi for the respective geometries and also its restriction to the bicategories of locally modelled topoi.\\
 
 For $ (\mathcal{F}, F)$ a $ \mathbb{T}_2$-modelled topos, recall first that we have in each $c$ in $\mathcal{C}_\mathcal{F}$ an equality $ \Phi[\mathcal{F}]_*F(c) = \Phi[\mathcal{S}]_*(F(c))$. Then observe that for $ (c, n)$ in the site $ \mathcal{V}^2_{\Phi F}$, any choice of pushout square for $n$ in $ \mathbb{T}_2[\mathcal{S}]$
\[\begin{tikzcd}
	K & {K'} \\
	{\Phi[S]_* (F(c))} & B
	\arrow["a"', from=1-1, to=2-1]
	\arrow["n"', from=2-1, to=2-2]
	\arrow["m", from=1-1, to=1-2]
	\arrow[from=1-2, to=2-2]
	\arrow["\lrcorner"{anchor=center, pos=0.125, rotate=180}, draw=none, from=2-2, to=1-1]
\end{tikzcd}\]
corresponds uniquely to a pushout square in $ \mathbb{T}_1[\mathcal{S}]  $ 
\[\begin{tikzcd}
	{\Phi[S]^*( K)} & {\Phi[S]^*(K')} \\
	{ F(c)} & {\overline{a}_*\Phi[S]^*(K')}
	\arrow["{\overline{a}}"', from=1-1, to=2-1]
	\arrow["{\overline{a}_*\Phi[S]^*(m)}"', from=2-1, to=2-2]
	\arrow["{\Phi[S]^*(m)}", from=1-1, to=1-2]
	\arrow[from=1-2, to=2-2]
	\arrow["\lrcorner"{anchor=center, pos=0.125, rotate=180}, draw=none, from=2-2, to=1-1]
\end{tikzcd}\]
But recall that restricting back $ \Phi[\mathcal{S}]^*$ to finitely presented objects of $ \mathbb{T}_1[\mathcal{S}]$ returns finitely presented objects in $\mathbb{T}_2[\mathcal{S}]$ $ \Phi[S]^*( K)$, and sends finitely presented $ \mathcal{V}_2$ etale maps to finitely presented $ \mathcal{V}_1$-etale maps, so that $ \Phi^*[\mathcal{S}](m)$ is in $ \mathcal{V}_2$, and $ \overline{ a} _*\Phi^*[\mathcal{S}](m)$ is in $ \mathcal{V}^2_F$. This process defines hence a functor
\[\begin{tikzcd}[row sep=tiny]
	{\mathcal{V}^1_{\Phi[\mathcal{F}]_*F}} && {\mathcal{V}^2_{F}} \\
	{(c,n)} && {(c,\overline{a}_*\Phi[S]^*m)}
	\arrow["{\Phi_F}", from=1-1, to=1-3]
	\arrow[shorten <=14pt, shorten >=14pt, maps to, from=2-1, to=2-3]
\end{tikzcd}\]
\end{division}

\begin{lemma}
The functor $ \Phi_F$ defines a morphism of site $ ({\mathcal{V}^1_{\Phi[\mathcal{F}]_*F}}^{\op}, {J^1_{\Phi[\mathcal{F}]_*F}}) \rightarrow ({\mathcal{V}^2_{F}}^{\op}, {J^2_{F}})$ between the spectral sites. 
\end{lemma}

\begin{proof}
We can process by proving that $ \Phi_F$ sends separately horizontal families and vertical families of ${J^1_{\Phi[\mathcal{F}]_*F}}$ to horizontal and vertical families in ${J^2_{F}}$. For horizontal families, this is just a consequence of the definition of $ \Phi[\mathcal{F}]_*F$. For vertical families, let $c$ be in $ \mathcal{C}_\mathcal{F}$ and $ (l_i : n \rightarrow n_i)_{i \in I}$ in $\mathcal{V}^1_{\Phi[\mathcal{S}]_* (F(c))}$; then we can find some $a : K \rightarrow \Phi[\mathcal{S}]_* (F(c))$ and a family $(k_i : m \rightarrow m_i$ in $ J_1(K)$ such that all the squares in the diagram below are pushouts
\[\begin{tikzcd}
	K & {K'} \\
	{\Phi[\mathcal{S}]_*(F(c))} & B & {K_i} \\
	&& {B_i}
	\arrow["m", from=1-1, to=1-2]
	\arrow["{k_i}", from=1-2, to=2-3]
	\arrow["a"', from=1-1, to=2-1]
	\arrow[from=1-2, to=2-2]
	\arrow[from=2-3, to=3-3]
	\arrow["{l_i}"{description}, from=2-2, to=3-3]
	\arrow["{n_i}"', from=2-1, to=3-3]
	\arrow["n"{description}, from=2-1, to=2-2]
	\arrow["{m_i}"{description, pos=0.4}, from=1-1, to=2-3, crossing over]
	\arrow["\lrcorner"{anchor=center, pos=0.125, rotate=180}, draw=none, from=2-2, to=1-1]
	\arrow["\lrcorner"{anchor=center, pos=0.125, rotate=180}, shift right=2, draw=none, from=3-3, to=2-2]
\end{tikzcd}\]
Now again from the adjunction we can consider the following pushouts
\[\begin{tikzcd}[row sep=large]
	{\Phi[\mathcal{S}]^*(K)} && {\Phi[\mathcal{S}]^*(K')} \\
	{F(c)} && B & {\Phi[\mathcal{S}]^*(K_i)} \\
	&&& {B_i}
	\arrow[""{name=0, anchor=center, inner sep=0}, "{\Phi[\mathcal{S}]^*(m)}", from=1-1, to=1-3]
	\arrow["{\Phi[\mathcal{S}]^*(k_i)}", from=1-3, to=2-4]
	\arrow["{\overline{a}}"', from=1-1, to=2-1]
	\arrow[from=1-3, to=2-3]
	\arrow[from=2-4, to=3-4]
	\arrow["{\overline{a}_*\Phi[\mathcal{S}]^*(k_i)}"{description}, from=2-3, to=3-4]
	\arrow["{\overline{a}_*\Phi[\mathcal{S}]^*(m_i)}"', from=2-1, to=3-4]
	\arrow["{\overline{a}_*\Phi[\mathcal{S}]^*(m)}"{description}, from=2-1, to=2-3]
	\arrow["{\Phi[\mathcal{S}]^*(m_i)}"{description, pos=0.4}, from=1-1, to=2-4, crossing over]
	\arrow["\lrcorner"{anchor=center, pos=0.125, rotate=180}, shift right=2, draw=none, from=3-4, to=2-3]
	\arrow["\lrcorner"{anchor=center, pos=0.125, rotate=180}, draw=none, from=2-3, to=0]
\end{tikzcd}\]
But from the definition of a transformation of geometry, we know that $ \Phi$ sends $ J_1$-covering families to $ J_2$-covering families, so that the triangle above $ \Phi[\mathcal{S}]^*(k_i) : \Phi[\mathcal{S}]^*(K') \rightarrow \Phi[\mathcal{S}]^*(K_i)  $ is a $ J_2$-cover, exhibiting $ \overline{a}_*\Phi[\mathcal{S}]^*(k_i)$ as a $J^2_{\Phi[\mathcal{S}]^*(F(c))}$-cover of $ \Phi_F(c, n)$. Hence $ \Phi_F$ sends covers to covers. 
 \end{proof}

All of this provides us with a geometric morphism 
\[\begin{tikzcd}
	{\Spec_2(F)} & {\Spec_1(\Phi[\mathcal{F}]_*(F))}
	\arrow["{\Sh(\Phi_F)}", from=1-1, to=1-2]
\end{tikzcd}\]

\subsection{Comparison of structure sheaves}

\begin{division}
Now let us describe the induced morphism of sheaves between the corresponding structure sheaves. Observe that for any $ (c,n)$ and any pair $ (a,m)$ inducing $n$, recall that the $ \overline{a}$ produced by the adjunction $ \Phi[\mathcal{S}]^* \dashv \Phi[\mathcal{S}]_* $ is obtained as the composite
\[\begin{tikzcd}
	& {\Phi[\mathcal{S}]^*(K)} \\
	{\Phi[\mathcal{S}]^*\Phi[\mathcal{S}]_*(F(c))} \\
	& {F(c)}
	\arrow["{\Phi[\mathcal{S}]^*(a)}"', from=1-2, to=2-1]
	\arrow["{\epsilon_{F(c)}}"', from=2-1, to=3-2]
	\arrow["{\overline{a}}", from=1-2, to=3-2]
\end{tikzcd}\]
and conversely $ a$ can be retrieved from $ \overline{ a}$ as the composite 
\[\begin{tikzcd}
	& K \\
	{\Phi[\mathcal{S}]_*\Phi[\mathcal{S}]^*(K)} \\
	& {\Phi[\mathcal{S}]_*(F(c))}
	\arrow["{\eta_K}"', from=1-2, to=2-1]
	\arrow["a", from=1-2, to=3-2]
	\arrow["{\Phi[\mathcal{S}]_*(\overline{a})}"', from=2-1, to=3-2]
\end{tikzcd}\]
Now, while $ \Phi[\mathcal{S}]_*$ does not preserve pushouts, we can consider the image of the pushout $ \Phi_F(n) = \overline{a}_*\Phi[\mathcal{S}]^*(m)$ along $ \Phi[\mathcal{S}]_*$, and we see that the expression of $n$ as a pushout induces a factorization as depicted below
\[\begin{tikzcd}
	& K &&& {K'} \\
	{\Phi[\mathcal{S}]_*\Phi[\mathcal{S}]^*(K)} \\
	& {\Phi[\mathcal{S}]_*\Phi[\mathcal{S}]^*(K')} \\
	{\Phi[\mathcal{S}]_*(F(c))} && B \\
	& {\Phi[\mathcal{S}]_*(\overline{a}_*\Phi[\mathcal{S}]^*(K'))}
	\arrow["m", from=1-2, to=1-5]
	\arrow["{\eta_K}"{description}, from=1-2, to=2-1]
	\arrow["{\eta_{K'}}"{description}, from=1-5, to=3-2]
	\arrow["a"{description, pos=0.7}, from=1-2, to=4-1]
	\arrow[from=1-5, to=4-3]
	\arrow["n"{description, pos=0.3}, from=4-1, to=4-3]
	\arrow["{\Phi[\mathcal{S}]_*(\overline{a})}"{description}, from=2-1, to=4-1]
	\arrow["{\Phi[\mathcal{S}]_*(\overline{a}_*\Phi[\mathcal{S}]^*(m))}"', from=4-1, to=5-2]
	\arrow[from=3-2, to=5-2, crossing over]
	\arrow["\lrcorner"{anchor=center, pos=0.125, rotate=180}, draw=none, from=4-3, to=1-2]
	\arrow["{\Phi[\mathcal{S}]_*\Phi[\mathcal{S}]^*(m)}"{description}, from=2-1, to=3-2]
	\arrow[dashed, from=4-3, to=5-2]
\end{tikzcd}\]
This dashed arrow $ B \rightarrow \Phi[\mathcal{S}]_*(\overline{a}_*\Phi[\mathcal{S}]^*(K'))$ can be shown to be natural in $n$. Moreover, the codomain functor $ \cod$ lives in the presheaf topos $\widehat{{\mathcal{V}^2_{F}}^{\op}}$, and $ \Phi[\mathcal{S}]_*(\overline{a}_*\Phi[\mathcal{S}]^*(K'))$ coincides with $ \cod \, \Phi_F (c,n) $ which is the value at $ (c,n)$ of the direct image $ \Sh(\Phi_F)_* \cod$. Moreover we have 
\[ \Phi[\mathcal{S}]_*(\overline{a}_*\Phi[\mathcal{S}]^*(m)) = (\Phi[\widehat{{\mathcal{V}^2_{F}}^{\op}}]_*\cod \, \Phi_F)(c,n) \]
Hence we are provided with a morphism of presheaves \[ \cod \Rightarrow \Phi[\widehat{{\mathcal{V}^2_{F}}^{\op}}]_*\cod \, \Phi_F \]
But now, naturality of $ \Phi[-]$ at the inclusion $ \Spec_2(F) \hookrightarrow \widehat{{\mathcal{V}^2_{F}}^{\op}}$ gives the commutation 
\[\begin{tikzcd}
	&& {{\mathbb{T}_{J_2}}[\mathcal{\Spec}_2]} & {{\mathbb{T}_{J_2}}[\widehat{{\mathcal{V}^2_{F}}^{\op}}]} \\
	{} && {{\mathbb{T}_{J_1}}[\mathcal{\Spec}_2]} & {{\mathbb{T}_{J_1}}[\widehat{{\mathcal{V}^2_{F}}^{\op}}]}
	\arrow["{\Phi[\mathcal{\Spec}_2]_*}"', from=1-3, to=2-3]
	\arrow["{\mathfrak{a}_{J^2_F}}", from=1-3, to=1-4]
	\arrow["{\Phi[\widehat{{\mathcal{V}^2_{F}}^{\op}}]_*}", from=1-4, to=2-4]
	\arrow["{\mathfrak{a}_{J^2_F}}"', from=2-3, to=2-4]
\end{tikzcd}\]
Therefore we have an isomorphism of sheaves
\begin{align*}
     \mathfrak {a}_{J^2_F}\Phi[\widehat{{\mathcal{V}^2_{F}}^{\op}}]_*\cod \,\Phi_F
     &\simeq  \Phi[\mathcal{\Spec}_2]_* \mathfrak {a}_{J^2_F} \cod \,\Phi_F \\
     &\simeq  \Phi[\mathcal{\Spec}_2]_* {\widetilde{ F}}^2 \,\Phi_F \\
     &\simeq \Sh(\Phi_F)_* \Phi[\mathcal{\Spec}_2]_* {\widetilde{ F}}^2
\end{align*}
Hence the morphism of presheaves above provides us with a canonical morphism of sheaves in $ \Spec_1(\Phi[\mathcal{F}]_*F)$
\[\begin{tikzcd}
	{\widetilde{\Phi[\mathcal{F}]_*F}^1} & {\Sh(\Phi_F)_* \Phi[\mathcal{\Spec}_2]_* {\widetilde{ F}}^2}
	\arrow["{\sigma_\Phi^\sharp}", from=1-1, to=1-2, Rightarrow]
\end{tikzcd}\]
Moreover by adjunction we know this corresponds uniquely to a morphism of sheaves in $\Spec_2(F)$ 
\[ \begin{tikzcd}
	{\Sh(\Phi_F)^*\widetilde{\Phi[\mathcal{F}]_*F}^1} & { \Phi[\mathcal{\Spec}_2]_* {\widetilde{ F}}^2}
	\arrow["{\sigma_\Phi^\flat}", from=1-1, to=1-2, Rightarrow]
\end{tikzcd} \]
whose component at a point $ (x, \xi)$ of $ \Spec_2(F)$ is the local part of the $(\Et_1, \Loc_1)$-factorization 
\[\begin{tikzcd}
	{\Phi[\mathcal{S}]_*(x^* F) } && {\Phi[\mathcal{S}]_*(A_\xi) } \\
	& {A_{\Phi[\mathcal{S}]_*(\xi) }}
	\arrow["{\Phi[\mathcal{S}]_*(\xi) }", from=1-1, to=1-3]
	\arrow["{n^1_{\Phi[\mathcal{S}]_*(\xi) }}"', from=1-1, to=2-2]
	\arrow["{u^1_{\Phi[\mathcal{S}]_*(\xi) }}"', from=2-2, to=1-3]
\end{tikzcd}\]
\end{division} 

\section{Examples of geometries and spectra}

We give here several examples from literature, most of them are known at least in a form or another. Better known are examples from ring theory, as the well known Zariski geometry for commutative rings involving local rings, ring localizations and conservative morphisms. But this is not the only geometry for commutative rings: other examples includes \begin{itemize}
    \item the \emph{étale geometry} with strictly henselian local rings, etale morphisms as etale maps (whence the name) and henselian morphisms as local maps;
    \item the \emph{Pierce geometry} with connected rings, localization at idempotent as etale maps and connected morphisms as local maps; 
    \item the \emph{integral domain geometry} with integral domains as local objects, épimorphisms as étales maps and monomorphisms as local map (the only one without a sheaf-representation theorem). 
\end{itemize}   
However, those examples are fairly known and well treated in accessible sources as \cite{Anel}, or also \cite{johnstone1982stone}[Chapter V] and \cite{Coste}, and were the primary motive of investigation initiating the topic in \cite{hakim}. We hence prefer to focus on another class of examples arising from Stone-like dualities, which are lesser known under this form (in particular, the treatment of Jipsen-Moshier duality as a geometry is new to our knowledge. \\

In general, Stone duality is considered from the point of view of concrete dualities. However it is possible to provide a spectral account of it, or at least to reconstruct the Stone dual of a distributive lattice as the spectra of a certain geometry. But in this process the stone dual is also endowed with a structure sheaf which is not considered in classical Stone duality, and the spectrum will be adjoint to the global section functor rather than the ``compact open" functor, which, in Stone duality, returns the basis of compact open set of Stone spaces to which it is hence restricted.\\

Those Stone-like duality, altogether with those for commutative rings (except the integral domain geometry) are subsumed conjointly with the geometry for \emph{commutatives rigs} described at \cref{CMZB geometry}. We also include the geometry behind the construction of \cite{DububPoveda}.\\

While the Stone-like dualities are usually considered for set-valued ordered structures corresponding to Lindenbaum-Tarski algebras of propositional theories, the spectral construction extends beyond set valued model and allows to construct spectra of sheaves of lattices, boolean algebras, or Heyting algebras: this allows in particular a new way to construct the classifying topoi of theories in various fragments of first-order logics by applying the construction to such sheaves living in the classifying topos of the theory of objects.\\

Finally we apply the construction to special strains of hyperdoctrines satisfying the sheaf condition, and in particular to the subobject hyperdoctrine. 

\subsection{Jipsen-Moshier geometry for semilattices}\label{Jipsen-Moshier}

Jipsen-Moshier duality is an example of an geometry without specification of local objects, so that any object is actually local. It plays however a central role amongst propositional geometries, which all are in fact refinement it this one. 

\begin{division}
In a space $X$ with specialization order $ \sqsubseteq$, a compact open filter is an upset $ F$ for $ \sqsubseteq$ which is both open and compact. For $ X$ denote $ \textbf{KOF}(X)$ its set of compact open filters. A point $x$ is basic compact open if $ \uparrow$ is a compact open filter. \emph{Hofmann-Mislove-Stralka spaces} -- for short, \emph{HMS spaces} -- are sober spaces $X$ such that $\textbf{KOF}(X)$ is a basis closed under finite intersection. \\

Denote $ \textbf{HMS}$ the category of HMS spaces with continuous maps $f : X \rightarrow Y $ such that $ f^{-1}$ restricts to $ \textbf{KOF}(Y)  \rightarrow \textbf{KOF}(X) $. Any compact open filter of a HMS space has a focal point. Moreover, in a HMS space, any point is a directed join of basic compact open points. The specialization order makes $ (X, \sqsubseteq)$ a complete lattice, and there are simultaneously an initial point and terminal point in such a $X$.
\end{division}

\begin{division}
Then, one can recover the spectra of Jispen and Moshier duality for $\wedge$-semilattices with unit: 
              \[ \MSLat^{\op} \simeq \textbf{HMS}  \]
              
Define $\Spec(S) = (\mathcal{F}_S, \downarrow S)$. For $X$ HMS, $ \textbf{KOF}(X)$ is a $ \wedge$-semilattice. Moreover, for a semillatice one has $ S \simeq \textbf{KOF}(Spec(S))$, while for a HMS space, one has $ X = \Spec(\textbf{KOF}(X))$. \\
 
If $ S$ is a $\wedge$-semilattice, $ \mathcal{F}_S \simeq (\mathcal{I}^{prime}_S)^{\op}$ is a complete lattice. Moreover any filter of a $\wedge$-semilattice is trivially prime. This says that $ \Spec(S) = \MSLat[S, 2]$.
\end{division}

\begin{division}We recall here the admissibility structure for this geometry. Our ambient locally finitely presentable category is $ \MSLat$. Recall that $\wedge$-semilattices are not 1-regular, that is, for any given filter there are several congruences whose class in 1 is this filter; as the poset of congruence of a semilattice is complete and cocomplete, for any filter $F$, there is a minimal congruence $ \theta^{\min}_F$ such that $ [1]_{\theta^{\min}_F } = F$ -- in particular, for each principal filter $ \uparrow a$, there is a smallest congruence identifying $a$ with 1, denoted $ \theta_{(a,1)}$.

Then for etale maps one can choose \emph{1-minimal quotients}: those are morphisms $ A \twoheadrightarrow A/\theta$ with $ \theta$ minimal amongst congruences whose class in 1 is $ [ 1]_\theta$. One can easily prove this class is closed by composition and colimits, contains iso and is left-cancellative. For a semilattice $ S$, finitely presented 1-minimal quotients are of the form $ S \rightarrow S/\theta_{(a,1)}$.

Then it can be shown that one has a factorization system $(1\hy\textbf{Quo}, 1\hy\textbf{Cons})$ on $ \MSLat$, where $1\hy\textbf{Cons}$ is the class of maps that reflect the top element, that is those $ f: S \rightarrow S'$ such that $ f^{-1}(1) = \{ 1 \}$. The factorization is obtained from the quotient throuhg the minimal congruence associated to the pre-image of the top element (which is a filter):
\[\begin{tikzcd}
	S && T \\
	& {S/\theta^{\min}_{f^{-1}(1)}}
	\arrow["f", from=1-1, to=1-3]
	\arrow["{q_{\theta^{\min}_{f^{-1}(1)}}}"', from=1-1, to=2-2]
	\arrow["{r_f}"', from=2-2, to=1-3]
\end{tikzcd}\]\end{division}
 
\begin{division}
Then the geometry associated to Jispen and Moshier duality can be constructed as follows: take $(1\hy\textbf{MinQuo}, \; 1\hy\textbf{Cons}) $ as factorization system on $ \MSLat$, and no topology. Finitely presented etale maps under a $ \wedge$-semilattice $S$ are principal 1-minimal quotients and they always define a basic compact open point. Conversely for any filter $F$ one has a minimal quotient \[S \twoheadrightarrow S/\theta^{\min}_F = \underset{a \in F}{\colim}\; S/\theta(a,1) \]
where $\theta^{\min}_F$ is the congruence in $F$ given as $ \theta^{\min}_F = \bigcap \{ \theta \mid [1]_\theta = F \}$. This defines a point of $\Spec(S)$, and any saturated compact actually has a focal point. Moreover observe that $S$ is itself the spectral site as $ \mathcal{V}(S) = S^{\op}$ as consisting of all the principal filters. 

Finally, sheaf representation is trivial as the topology here is the trivial topology, for which the codomain functor already is a sheaf. 
\end{division}

\subsection{Stone geometry for distributive lattices}\label{Stone}

As for Stone duality, this construction can be done in two manners, equipping the Stone dual with either the \emph{Zariski} or the \emph{coZariski} topology: this depends on the way we define the admissibility structure. Such construction was investigated in \cite{brezuleanu1969duale}. \\

\begin{division}We recall here the admissibility structure for Zariski. Our ambient locally finitely presentable category is $ \textbf{DLat}$, the category of bounded distributive lattices. For etale maps one can choose again {1-minimal quotients}. Similarly as for the case of $\wedge$-semilattices, finitely presented 1-minimal quotients of a lattice $ D$ are of the form $ D \rightarrow D/\theta_{(a,1)}$ and hence are in one-one correspondence with the dual lattice of $D$. Moreover, one can check that $\DLat$ inherits the $(1\hy\textbf{Quo}, 1\hy\textbf{Cons})$ factorization system.\\

Then define the category $ 1\hy\textbf{LocDLat}^{1\hy\textbf{Cons}}  $ having:\begin{itemize}
        \item as objects \emph{local distributive lattices}, where $\{ 1 \} $ is prime filter
        \item as morphisms 1\hy\textbf{Cons}ervative morphisms $f$.
    \end{itemize}
    Then $ 1\hy\textbf{LocDLat}^{1\hy\textbf{Cons}} \hookrightarrow \textbf{DLat} $ is a multireflection. But we can also axiomatize the category of local lattices as follows:
    define $ J_{Zar}$ on $ \textbf{DLat}_{fp}^{op}$ generated by $(f_i : D \twoheadrightarrow D/\theta_i)  $  such that $ \bigcap_{i \in I} \theta_i = \Delta_D  $. Now observe that a distributive lattice $D$ is $ J_{Zar}$-local if an only if $\{ 1 \}$ is a prime ideal, that is, $D $ has a minimal point $ L \rightarrow 2$ sending any $a \neq 1$ on 0. Local lattices are the points of the topos $  \Sh(\textbf{DLat}_{fp}^{\op}, J_{Zar})$.
  \end{division} 
  
  \begin{division}
The associated Diers spectrum for $D$ is 
     \[ (\Spec(D) = (\mathcal{F}^{Prime}_D, \tau^{Zariski}_D), \widetilde{D}) \]
     with $\widetilde{D} $ defined on the basis as $ \widetilde{D}(U^{coZar}_a) = D/\theta_{(a,0)} $ for any $a \in D$. Then we have an adjunction 
     \[ \begin{tikzcd}
\textbf{DLat}^{op} \arrow[bend right = 20]{rr}[swap]{\Spec^{Zar}} & \quad \quad  \perp & \quad \textbf{DLat}^*\hy\textbf{Spaces} \arrow[bend right = 20 ]{ll}[swap]{\Gamma}
\end{tikzcd} \]
Then one recovers the Stone spaces as the underlying spaces of affine $\textbf{DLat}$-spaces.\\
    
The spectral site of a distributive lattice $D$ is $(Zar_D^{op},J_{Zar}(D))  $ where $ Zar_D$ consists of finitely presented 1-minimal quotients of $ D$; in particular for a filter $ F$ of $D$, a factorization as below 
      $ D \twoheadrightarrow D/\theta_{(a,1)}$, and a factorization
      \[ 
\begin{tikzcd}[row sep=small]
D \arrow[d, "q_a"', two heads] \arrow[r, "q_F", two heads]   & D/\theta_F \\
{D/\theta_{(a,1)}} \arrow[ru, two heads, dashed] &           
\end{tikzcd} \]
expresses the fact that $ a $ lies in $F$. \\

Now, at a distributive lattice $D$, the induced topology $J_{Zar, D}$ on $D$ consists of finite families $ (a_i \leq a)_{i \in I}  $ with $ \bigvee a_i = 1 $, or seen as a poset of quotient under $D$. Being made of epi, $Zar_D$ is a poset and $ Zar_D \simeq D^{\op}$ and we have $ D \hookrightarrow \tau_{Zar}$. $ J_{Zar}(D)$ coincides with the coherent topology on $D$. The spectrum is spatial and is equipped with the Zariski topology which is the frame of filters $ \mathcal{F}_D$.\\

Opens of Zariski topology form the frame $ \tau_{Zar} = \Sh(Zar_D^{op}, J_{Zar}(D)) = \Sh(I_D)$: Zariski opens correspond to ideals of $ D$ and $ D \hookrightarrow I_D$ is a base of compact open of Zariski topology. On the other side, $D \hookrightarrow (\mathcal{F}_D)^{op}$, but a filter $F$ of $D$ defines a filtered diagram whose colimit is the 1-minimal quotient at $F$
      \[S \twoheadrightarrow S/\theta^{\min}_F = \underset{a \in F}{\colim}\; S/\theta(a,1) \]
Those filters are \emph{saturated compact } of Zariski topology. A prime filter $ x$ corresponds to the 1-quotient $ D \twoheadrightarrow D/\theta_x$, which is the saturated compact $ \uparrow x$, the focal component in $x$.
\end{division}

\begin{division}
However Zariski geometry is not the only way to retrieve Stone duality. One could have either defined the factorization system $ (0\hy\textbf{MinQuo}, 0\hy\textbf{Cons})$, whith $ 0\hy\textbf{MinQuo}$ as the minimal quotient with a fixed ideal, and $0\hy\textbf{Cons}$ the morphisms $f$ such that $f^{-1}(0)=\{ 0 \}$, and could have taken as local objects those distributive lattices  with $ \{0\}$ prime. \\
      
 The CoZariski site would have been $(coZar_D^{op}, J_{coZar}(D)) $ with $coZar_D$ made of the minimal quotients $ D \twoheadrightarrow D/\theta_{(a,0)} $ and $J_{coZar}(D)$ defined by $ (a_i)_{i \in I}$ such that $\bigwedge_{i \in I} a_i = 0$. \\
  
Then $ D \simeq coZar_D$, so that $ D^{op} \hookrightarrow \tau_{coZar} = \Sh(coZar_D^{op}, J_{coZar}(D)) \simeq (\mathcal{F}_D)^{op} $. Then filters are the closed subsets of coZariski topology. On their sides ideals $ I_D$ define filtered colimits of finitely presented minimal quotients maps in $coZar_D$, hence correspond to saturated compacts. Observe that the existence of both a Zariski and a coZariski topology for a distributive lattice, which is called \emph{Hochster duality}, is in fact nothing but an instance of Isbell duality. 
\end{division}


      

\subsection{Geometry for boolean algebras}\label{boolean stone}
  
The geometry of boolean algebras is a restriction of both Zariski and coZariski geometries, which happens to coincide for a boolean algebra. 

\begin{division}
For an element $a$ in a distributive lattice, a \emph{complement} is an element, unique whenever it exists, $ \neg a$ such that $ a \vee \neg a= 1$, and $ a \wedge \neg a =0$. Observe that a prime ideal $x$ of a distributive lattice $D$ contains always either $a$ or $\neg a$ for any complemented element $a$ in $D$ from $ a \wedge \neg a=0$, while an ideal never contains simultaneously an element and its complement unless it is trivial; the corresponding statements are true for filters and prime filters. Any morphism of distributive lattices preserves complement. \\

A \emph{boolean algebra} is a distributive lattice where any element is complemented. Boolean algebras form a full, reflective subcategory $ \textbf{Bool} \hookrightarrow \textbf{DLat}$. Boolean algebras have better exactness properties than distributive lattices. First, they form a regular categor: any ideal $ I$ (resp. a filter $F$) in a boolean algebra, there is exactly one congruence $ \theta_I$ such that $ [0]_{\theta_I} = I$ (resp. one congruence $ \theta_F$ such that $ [1]_{\theta_F} = F$). In particular, there is no more distinction of 1-minimal or 0-minimal quotient amongst quotient. Moreover, any epimorphism of boolean algebra is a regular quotient, and the category of boolean algebras is balanced. Finally, recall that any prime ideal (or filter) of a boolean algebra is maximal. 
\end{division} 

\begin{division}
Now the only 1-local (or 0-local) boolean algebra is the two elements algebra 2: indeed, if $\{ 1 \}$ is a prime filter in $A$, then $ x \vee y = 1$ implies $ x=1$ or $ y=1$; but then for $ x \vee \neg x=1$, any $x$ in $A$ must actually be 0 or 1. \\

Then the Diers context for both the Zariski and coZariski geometries of distributive lattices reduces on the inclusion 
\[\begin{tikzcd}
	{\{ 2 \}} & {\textbf{Bool}} \\
	{1\hy\textbf{LocDLat}} & {\textbf{DLat}}
	\arrow[hook, from=2-1, to=2-2]
	\arrow[hook, from=1-2, to=2-2]
	\arrow[hook, from=1-1, to=1-2]
	\arrow[hook, from=1-1, to=2-1]
\end{tikzcd}\]
\end{division}

\subsection{Esakia geometry for Heyting algebras}\label{Esakia}   
     
Esakia duality for Heyting algebras, as other Stone-like dualities, was introduced from a concrete duality point of view and without reference to structure sheaves. However in \cite{di2002grothendieck} was introduced a new presentation including sheaf theoretic consideration, and in particular a sheaf representation theorem. This geometry is also a restriction of Zariski geometry to Heyting algebras, being intermediate between Zariski and the geometry of boolean algebras.\\

However, factorization data simplify as they do in boolean algebras. Heyting algebras are 1-regular, so that for any filter $ F$ there is a unique congruence $\theta_F$ such that $ [1]_{\theta_F} = F$. Hence 1-minimal quotient are not distinguished from other quotient, as they are determined from the preimage of $1$. \\

Local maps can still be defined as conservative morphisms, but in fact they coincide then with monomorphisms. Local units are quotient at primes filters.

\begin{remark}Beware that coZariski is not suited for Heyting algebras, as they are not 0-regular; however one could dually define a \emph{coEsakia geometry} for coHeyting algebras, which are 0-regular, and use ideal instead of filters. 
\end{remark}

  \subsection{Castiglioni-Menni-Zuluaga-Botero geometry for rigs}\label{CMZB geometry}

Here we sum up the results of \cite{MenniRig}, which develop a geometry for integral rigs subsuming both Zariski geometry of commutative rings and Stone geometry of distributive lattices.\\

\begin{division}
Recall that a \emph{rig} (aka ``ring without negative") is the data of $ (A,+,\cdot, 0,1)$ such that\begin{itemize}
    \item $ (A_,+,0)$ and $ (A, \cdot, 1)$ are commutative monoids
    \item both of the generic equations hold $ x \cdot 0 = 0$ and $ (x + y) \cdot z = x \cdot z + y \cdot z$.
\end{itemize} 

Any rig can be equipped with a pre-order $ \leq$ where $ x \leq y$ if there exists some $ z$ such that $ x+z=y$. In general this relation is only a pre-order. A rig is said to be \emph{integral} if it satisfies the generic equation $ x + 1 = 1$; then one has $ x \leq 1$ for any $x$. As a consequence, the operation $+$ is idempotent in an integral rig, as
\[ x + x = (1 + 1) \cdot x = 1 \cdot x = x \]
so that the monoid $ (A,+, 0)$ is then a $\vee$-semilattice.
In particular a distributive lattice is an integral rig where $ \cdot$ is also idempotent. \\

Similarly, one can define a commutative ring in the signature of rigs as a rig satisfying the cartesian sequent $ \vdash_{x} \exists y x+y=0 $. By uniqueness of such a solution, the theory of commutative rings is a finite limit extension of the theory of rigs.  
\end{division}

\begin{division}
An element $a$ of a rig $A$ is said to be \emph{invertible} if there exists $ b $ in $A$ such that $ a \cdot b = 1$. Denote as $ \textup{Inv}(A)$ the set of pairs $(a,b)$ of $A$ such that $ a\cdot b =1$: this is the pullback 
\[\begin{tikzcd}
	{\textup{Inv}(A)} & 1 \\
	{A \times A} & A
	\arrow[from=1-1, to=1-2]
	\arrow[from=1-1, to=2-1]
	\arrow["{\cdot }"', from=2-1, to=2-2]
	\arrow["{\ulcorner 1 \urcorner}", from=1-2, to=2-2]
	\arrow["\lrcorner"{anchor=center, pos=0.125}, draw=none, from=1-1, to=2-2]
\end{tikzcd}\]
Moreover we have two canonical projections $ \textup{Inv}(A) \rightarrow A$ which are monomorphisms and code for the same subobject. \\

Then, similarly to the case of rings, one can consider \emph{conservative maps} of rigs (they are directly qualified of local in \cite{MenniRig}) as maps that reflect invertibility, that are those morphisms of rigs $f : A \rightarrow B$ such that one has a pullback
\[\begin{tikzcd}
	{\textup{Inv}(A)} & {\textup{Inv}(B)} \\
	A & B
	\arrow[hook, from=1-1, to=2-1]
	\arrow[from=1-1, to=1-2]
	\arrow[hook, from=1-2, to=2-2]
	\arrow["\lrcorner"{anchor=center, pos=0.125}, draw=none, from=1-1, to=2-2]
	\arrow["f"', from=2-1, to=2-2]
\end{tikzcd}\]
\end{division}

\begin{division}
On the other hand, a \emph{multiplicative submonoid} is a $F \subseteq A$ which is a submonoid of $ (A,\cdot, 1)$. At a multiplicative submonoid $F$ one can consider the \emph{localization} $ n_F : A \rightarrow A[F^{-1}]$ forcing each element of $F$ to become invertible; in particular, one can define for each $a$ in $A$ the \emph{principal localization at $a$} $ n_a : A \rightarrow A[a^{-1}] $ which corresponds to the localization at the multiplicative submonoid $ \{ a^n \mid n \in \mathbb{N} \}$. A \emph{filter} of a rig is a multiplicative submonoid $ F$ such that $ a \in F$ implies that $a+b \in F$ for all $b$ in $A$. An \emph{ideal of rigs} is a $ I \subseteq A$ which is a submonoid of $(A, +, 0)$ and such that for each $a$ in $ A$, $ ax \in I $ whenever $x \in I$. For each multiplicative submonoid $F$ the complement $ A \setminus F$ is an ideal, and conversely. A filter $F $ of rig is prime if $ a + b \in F$ implies that either $ a \in F$ or $b \in F$ ; similarly an ideal of rig is {prime} if $ ab \in I$ implies that either $ a \in I$ or $b \in I$.  \\

When $ A$ is an integral rig, then $ \textup{Inv}(A) \simeq 1$, and localizing at an element amounts to forcing it to become 1, as $a \leq 1 $ implies that $ 1 \leq a^{-1}$ if the latter exists. Then the localization $ A \rightarrow A[a^{-1}]$ is exhibited as the quotient $ q_a : A \rightarrow A/\theta_a$ where $ \theta_a = \{ (x,y) \mid ax = ay \}$; in particular $(a,1) \in \theta_a$. More generally for a mutiplicative submonoid $ F$ one can define the congruence $ \theta_F = \{ (x,y) \mid \exists z \in F, \, xz = yz \}$, and the localization $  A \rightarrow A[F^{-1}] $ coincides with the quotient $ q_F : A \twoheadrightarrow A/\theta_{F}$ contracting $F$ on 1. In particular localization of integral rigs are regular epimorphisms; observe however that, contrarily to ring localizations, they are not monomorphisms. \\

One has a factorization system $ (\textbf{Loc}, \textbf{Cons})$ in the category $\textbf{IRig}$ of integral local rigs, where the factorization is obtained as the quotient 
\[\begin{tikzcd}
	A && B \\
	& {A/f^{-1}(1)}
	\arrow["f", from=1-1, to=1-3]
	\arrow["{q_f}"', two heads, from=1-1, to=2-2]
	\arrow["{u_f}"', from=2-2, to=1-3]
\end{tikzcd}\]
\end{division}

\begin{division}
A rig is \emph{really local} if it satisfies 
\[  0 = 1 \vdash \perp  \hskip1cm \textrm{and} \hskip1cm  x + y =1 \vdash_{x,y} x = 1 \vee y=1 \]
Now observe that really local integral rigs have the gliding property relatively to conservative maps: indeed for a conservative morphism $ u : A \rightarrow B$ with $B$ really local, if one has $ a + b = 1$ in $A$ then either $ a $ or $b$ is in $ u^{-1}(1)$, which is $\{1 \}$. \\

An integral rig is really local if and only if $ \{ 1 \}$ is a prime filter. Observe that in this case $ A\setminus \{ 1 \}$ is an ideal, for it becomes closed under $+$, and it is then automatically prime. In fact from this, a really local rig is a rig where $ A\setminus \{ 1 \}$ defines a unique maximal ideal.\\

A filter $F$ is prime if and only if $ A[F^{-1}] \simeq A/\theta_{F}$ is really local: indeed, if $a + b$ is in $F$, then $ [a + b]_F= [a]_F + [b]_F = 1$, so $A$ is really local if and only if $F$ is prime. \\

Hence we have a stable inclusion
\[\begin{tikzcd}
	{\textbf{ILocRig}^{\textbf{Cons}}} & {\textbf{IRig}}
	\arrow[hook, from=1-1, to=1-2]
\end{tikzcd}\]
where the local units of a rig $A$ are exactly its localization at prime filters. \\ 

Observe that we have a morphism of Diers contexts comparing Stone geometry to the geometry for Rigs
\[\begin{tikzcd}
	{1\hy\textbf{LocLat}^{1\hy\textbf{Cons}}} & {\textbf{DLat}} \\
	{\textbf{ILocRig}^{\textbf{Cons}}} & {\textbf{IRig}}
	\arrow[hook, from=1-1, to=1-2]
	\arrow[hook, from=1-2, to=2-2]
	\arrow[hook, from=1-1, to=2-1]
	\arrow[hook, from=2-1, to=2-2]
\end{tikzcd}\]

As well as a morphism of Diers contexts comparing the geometry of commutative rings to the geometry of rigs
\[\begin{tikzcd}
	{\textbf{LocCRing}^{\textbf{Cons}}} & {\textbf{CRing}} \\
	{\textbf{LocRig}^{\textbf{Cons}}} & {\textbf{Rig}}
	\arrow[hook, from=1-1, to=1-2]
	\arrow[hook, from=1-2, to=2-2]
	\arrow[hook, from=1-1, to=2-1]
	\arrow[hook, from=2-1, to=2-2]
\end{tikzcd}\]

\end{division}

\begin{division}
Finally, it is proved at \cite{MenniRig}[Lemma 4.2] that for any $ a, b$ one has a pullback-pushout square
\[\begin{tikzcd}
	{A[(a+b)^{-1}]} & {A[b^{-1}]} \\
	{A[a^{-1}]} & {A[(ab)^{-1}]}
	\arrow[from=1-1, to=2-1]
	\arrow[from=1-1, to=1-2]
	\arrow[from=1-2, to=2-2]
	\arrow[from=2-1, to=2-2]
	\arrow["\lrcorner"{anchor=center, pos=0.125, rotate=180}, draw=none, from=2-2, to=1-1]
	\arrow["\lrcorner"{anchor=center, pos=0.125}, draw=none, from=1-1, to=2-2]
\end{tikzcd}\]
Hence we have a sheaf representation for integral rig. 
\end{division}

\subsection{Dubuc-Poveda geometry for MV-algebras}

Spectra of MV-Algebras were introduced in \cite{DububPoveda} and later studied in further works on residuated lattices as \cite{GEHRKE2014290}. We list here the main results of this work. 

\begin{division}
MV-algebra can be defined in various choices of signatures, either emphasizing their residuated lattice nature, or from their specific operations, those data being mutually determined. An MV-algebra is the data of $(A, \oplus, 0, \neg) $ such that \begin{itemize}
    \item $ (A,\oplus, 0) $ is a commutative monoid 
    \item $ \neg$ is an involution $ \neg\neg x = x$
    \item  and we have the generic identities 
 \begin{align*}
     x \oplus \neg 0 &= \neg  0\\
     \neg (\neg x \oplus y) \oplus y &= \neg (\neg y \oplus x) \oplus x
 \end{align*}
 \end{itemize}
It is also natural to define the following other operations 
\begin{align*}
    x \odot y &= \neg(\neg x \oplus \neg y)\\
    x \ominus y &= x \odot \neg y \\
    1 &= \neg 0 
\end{align*}
Then we end up with the following identities:
\begin{align*}
    x \oplus y &= \neg(\neg x \odot \neg y)\\
    x \oplus \neg x &= 1 \\
    \neg 1 &= 0 
\end{align*}

Any morphism of MV-algebra preserves automatically those operations and constants, for they are defined from the ones in the signature. Denote as $\textbf{MV}$ the category of MV-algebras and their morphisms. \\

Moreover we can also equip $A$ with an order $ \leq $ with $ x \leq y$ if there exists $ z $ such that $ x \oplus z = y$, which can be chosen as $ y \ominus x$. Then the negation $ \neg$ is order reversing, while the operations are monotonic. Then this order defines a structure of distributive lattice where 
\begin{align*}
    x \vee y &= (x \odot \neg y) \oplus y = (x \ominus y) \oplus y\\
    x \wedge y &= x \odot (\neg x \oplus y)
\end{align*}
In particular one always have $ x \vee y \leq x \oplus y$ and $ x \odot y \leq x \wedge y $. Any morphism of MV-algebras is also a morphism of distributive lattices, and is hence order preserving. The category of MV-algebras has hence in particular a faithful functor $ \iota : \textbf{MV} \rightarrow \textbf{DLat} $ sending a MV-algebra on its undelying distributive lattice. The functor $ \iota$ is moreover a morphism of locally finitely presentable category, for both filtered colimits and limits of MV-algebras and distributive lattices are computed from the underlying set; the left adjoint of $ \iota$ is the functor $ \iota_*$ sending a distributive lattice $ D$ with a generator and relations presentation $ D = \textbf{DLat}[X]/\theta$ to the MV-algebra presented as  
\[ \iota^*D = \textbf{MV}[X]/\overline{\theta} \]
where $ \theta $ is the free MV-congruence generated by $\theta$ in $\textbf{MV}[1]$. We see this functor sends finitely presented distributive lattices on finitely presented algebras. \\
\end{division}

\begin{division}
We can define the \emph{distance} between two elements in a MV-algebra as
\[ d(x,y) = (x \ominus y) \oplus (y \ominus x) \]
In particular one has $ d(x,y) = 0 $ iff $ x = y$. The following identity is always satisfied in an MV-algebra:
\[ (x \ominus y) \wedge (y \ominus x) = 0 \]
Observe that $ x \ominus y = 0 $ iff $x \leq y$, hence $ d(x,y) = y \ominus x$ if $ x \leq y$. In general not all elements are pairwise comparable; a totally ordered MV-algebra is called an \emph{MV-chain}.\\

An \emph{ideal of MV-algebra} of $A$ is a lattice ideal of $(A, \wedge, \vee, 0, 1)$ which is moreover closed under the $\oplus$ operation. For $ S \subseteq A$ a subset, the ideal $ ( S ]$ generated by $ S$ is its the closure under $ \oplus$ of the downset $ \downarrow S$. In particular, an ideal is said to be \emph{principal} if it is of the form $ (a ]$ for some $a $ in $A$. For any pair $ a,b$ we have 
\begin{align*}
    (a, b ] &= (a \oplus b ] = (a \vee b ]\\
    (a \wedge b ] &= (a] \cap (b] 
\end{align*}
Hence any finitely generated ideal is principal. If $ f : A \rightarrow B$ is a morphism of MV-algebras, then $ \ker(f) = f ^{-1}(0)$ is an ideal in $A$. A morphism of MV-algebras is a monomorphism if and only if $\ker(f)=\{0 \}$. A congruence $\theta$ of MV-algebra defines an ideal $ [0]_\theta$; conversely, for any $ I$, define the congruence $ \theta_I$ as 
\[ \theta_I = \{ (a,b) \in A^2 \mid d(a,b) \in I \} \]
In particular MV-algebras are 0-regular, as any congruence is determined by the class of $0$. We denote as $\mathcal{I}^{\textbf{MV}}_A$ the set of MV-ideals of $A$. Moreover, $\textbf{MV}$ is a regular category, and has its (Epi,Mono) factorization obtained as
\[\begin{tikzcd}
	A && B \\
	& {A/\ker(f)}
	\arrow["f", from=1-1, to=1-3]
	\arrow["{e_f}"', two heads, from=1-1, to=2-2]
	\arrow["{m_f}"', hook, from=2-2, to=1-3]
\end{tikzcd}\]
Moreover, we have that for any MV-ideal $ I$ in $A$, $I$ is directed as a subset for it is closed under $\oplus$, and we have a filtered colimit
\[ A/p \simeq \underset{a \in p}{\colim}\, A/(a] \]

Beware that not all lattice ideals are MV-ideals, and not all lattice congruences are MV-congruences. However MV-algebras have reticulation: we can define the relation $ a \simeq b $ if $ (a ] = (b]$, which is a lattice congruence on the underlying lattice $ \iota(A)$, and the quotient $ \beta(A) =  A/\simeq$, called the \emph{Belluce lattice} of $A$, satisfies 
\[ \mathcal{I}^{\textbf{MV}}_A \simeq \mathcal{I}^{\textbf{DLat}}_{\beta(A)} \]
Moreover this construction is functorial as any morphism of MV-algebras induces a unique factorization in $\textbf{DLat}$
\[\begin{tikzcd}
	{\iota(A)} & {\iota(B)} \\
	{\beta(A)} & {\beta(B)}
	\arrow["{\iota(f)}", from=1-1, to=1-2]
	\arrow["{\sigma_B}", two heads, from=1-2, to=2-2]
	\arrow["{\sigma_A}"', two heads, from=1-1, to=2-1]
	\arrow["{\beta(f)}"', dashed, from=2-1, to=2-2]
\end{tikzcd}\]
This defines a functor $ \beta : \textbf{MV} \rightarrow \textbf{DLat}$.
\end{division}

 \begin{division}
 A non trivial MV-ideal $p$ is \emph{prime} if for any $x,y$ either $ x \ominus y$ or $ y \ominus x$ is in $p$. In particular an MV-ideal is prime if and only if the underlying lattice ideal is prime. Moreover, we also have that an MV-ideal is prime if and only if the quotient $ A/I$ is an MV-chain. Conversely, one can prove that $A$ is an MV-chain if and only if $ \{ 0 \}$ is a prime ideal. Prime MV-ideals are stable under inverse image: if $ p$ is prime in $B$ and $ f : A \rightarrow B$ is a morphism of MV-algebra, then $ f^{-1}(p) $ is prime. Any ideal that contains a prime ideal is prime, and for each prime ideal $p$ the set $ \{ I \in \mathcal{I}^{\textbf{MV}}_A \mid p \subseteq I \}$ is totally ordered. Hence the set of prime ideals $Z_A$ of $A$ is a \emph{root system} for the inclusion, that is, satisfies that for any $ p $ in $X_A$, the upset in $p$ is totally ordered.  \\

Hence an MV-chain have the gliding property along monomorphisms of MV-algebras: if one has a monomorphism $ m : A \hookrightarrow L $ with $ L$ an MV-chain, then $A$ itself is an MV-chain. Hence we have a local right adjoint
\[\begin{tikzcd}
	{\textbf{MVC}^{\textbf{Mono}}} & {\textbf{MV}}
	\arrow[hook, from=1-1, to=1-2]
\end{tikzcd}\]
where $ \textbf{MVC}^{\textbf{Mono}}$ is the category of MV-chains and monomorphisms. From the fact that the (epi, mono) factorization is left generated, we know this defines a Diers context.
 \end{division}

 \begin{division}
 Moreover, observe that MV-chains are local as a distributive lattice, while mono are conservative, so that we have a morphism of Diers Contexts
\[\begin{tikzcd}
	{	{\textbf{MVC}^{\textbf{Mono}}}} & {{\textbf{MV}}} \\
	{{0 \hy \textbf{LocDLat}}^{0\hy\textbf{Cons}}} & {{\textbf{DLat}}}
	\arrow[hook, from=1-1, to=1-2]
	\arrow["\iota", hook, from=1-2, to=2-2]
	\arrow[hook, from=1-1, to=2-1]
	\arrow[hook, from=2-1, to=2-2]
\end{tikzcd}\]
 \end{division}

\begin{division}
A local unit under a MV-algebra $ A$ is a quotient at a prime ideal $n_p : A \twoheadrightarrow A/p$, and $A/p$ is an MV-chain. One can construct the Diers space of $A$ pointwisely by defining $ Z_A$ as the set of prime MV-ideals of A; for any $a$, one has $ a \in p$ if and only if one has the factorization
\[\begin{tikzcd}
	A & {A/p} \\
	{A/(a]}
	\arrow["{q_a}"', two heads, from=1-1, to=2-1]
	\arrow["{q_p}", two heads, from=1-1, to=1-2]
	\arrow[two heads, from=2-1, to=1-2]
\end{tikzcd}\]
One can define coZariski basic open 
\[ W_a = \{ p \in Z_A \mid a \in p \}  \]
We have in particular 
\begin{align*}
    W_{a \oplus b} &= W_a \cap W_b \\
    W_0 &= Z_A 
\end{align*}
so that we have a basis $ (W_a)_{a \in A} $ for a topology $ \tau_A$ on $Z_A$; moreover observe that $W_a = W_b$ if and only if $ a \simeq b$, so that the basis is actually $ \beta(A)$. We can extend the definition to arbitrary open by defining $ W_I = \{ p \in Z_A \mid I \subseteq p \}$, and we have 
\[ W_I = \bigcap_{a \in I} W_a \]
so that $ \mathcal{I}_A^{\textbf{MV}} $ indices the set of saturated compact of the topology on $Z_A$. Observe that $ (Z_A, \tau_A)$ coincides with the Stone dual of the Belluce lattice $ (X_{\beta(A)}, \tau^{coZar}_{\beta(A)})$ together with its coZariski topology.
\end{division}

\begin{division}
The structure sheaf can be constructed explicitely as follows. 
Consider the bundle $ p_A : E_A \rightarrow Z_A$ where 
\[ E_A = \coprod_{p \in Z_A} A/p \]
Observe that elements of $A$ defines families of sections
\[\begin{tikzcd}
	A & {\displaystyle{\prod_{p \in Z_A} A/p}} \\
	& {A/p}
	\arrow["{\eta_A}", from=1-1, to=1-2]
	\arrow[two heads, from=1-2, to=2-2]
	\arrow[two heads, from=1-1, to=2-2]
\end{tikzcd}\]
where $ \eta_A$ sends $ a$ to the section $ \widehat{a} = ([a]_p)_{p \in A}$ of $p_A$. Then one equip $E_A$ with the etale topology $ \langle \widehat{a }(W_b) \rangle_{a,b \in A}$, and one just has to define $ \widetilde{A}= \Gamma(p_A)$.\\

Moreover it is known that MV-algebras satisfy the pullback-pushout lemma:
\[\begin{tikzcd}
	{A/(a \wedge b ]} & {A/(a  ]} \\
	{A/( b ]} & {A/(a \oplus b ]}
	\arrow[two heads, from=1-1, to=2-1]
	\arrow[two heads, from=1-1, to=1-2]
	\arrow[two heads, from=1-2, to=2-2]
	\arrow[two heads, from=2-1, to=2-2]
	\arrow["\lrcorner"{anchor=center, pos=0.125}, draw=none, from=1-1, to=2-2]
	\arrow["\lrcorner"{anchor=center, pos=0.125, rotate=180}, draw=none, from=2-2, to=1-1]
\end{tikzcd}\]
Hence the structure sheaf $ \widetilde{A}$ is flabby on the basis and we have a sheaf representation theorem $ A \simeq \Gamma \widetilde{A} $; moreover this means that any MV-algebras is a subdirect product of MV-chains. \\

One can also process in a pointfree way as follows. The distributive lattice $ \beta(A)$ indexing the basis can be equipped with its coherent topology to get a site $ (\beta(A), J_{Zar})$, and we can set $ \Spec(A) = \Sh(\beta(A), J_{Zar})$; for the topology $J_{Zar}$ is finitary, we know $ \Spec(A)$ to have enough points, which are the prime lattice ideals of $\beta(A)$, we know to be the prime MV-ideals of $A$.
\end{division}

\begin{remark}\label{Belluce}
In fact, one can recover the Belluce quotient of $A$ as the comparison transformation
\[\begin{tikzcd}
	{\textbf{MV}} & {\mathbb{T}_{\textbf{MV}, J_{\textbf{MVC}, \mathcal{M}}}\hy\GTop} \\
	{\textbf{DLat}} & {\mathbb{T}_{\textbf{DLat}, J_{\textbf{Zar}, \mathcal{V_{0\hy\textbf{Cons}}}}}\hy\GTop}
	\arrow["{\int \iota}", hook, from=1-2, to=2-2]
	\arrow[""{name=0, anchor=center, inner sep=0}, "\iota"', hook, from=1-1, to=2-1]
	\arrow["{\Spec^{\textbf{MV}}}", from=1-1, to=1-2]
	\arrow[""{name=1, anchor=center, inner sep=0}, "{\Gamma^{\textbf{DLat}}}", from=2-2, to=2-1]
	\arrow["\sigma", curve={height=-6pt}, shorten <=4pt, shorten >=4pt, Rightarrow, from=0, to=1]
\end{tikzcd}\]
\end{remark}

\subsection{Classifying topoi as spectra}

In this subsection we recover the construction of the classifying topos of theory of various fragment of first order logics as spectra of some internal Lindenbaum-Tarksi algebra in the classifying topos of the theory of object -- or its multi-sorted versions. This construction is infact very akin to the externalization process through which one turns the internal locale coding a first order geometric theory into a localic morphism, or also the constructions described in the first sections of \cite{wrigley2023geometric}.  

\begin{division}Let be $ \Sigma$ a signature: then one can define the category $ \Ctx_\Sigma$ of \emph{context} whose objects are finite string of sorted variable $ \overline{x} = x_1^{A_1}, ..., x_n^{A_n}$ and morphisms are relabelling $ \sigma : \overline{x} \rightarrow \overline{y}$, that are maps between string of variables preserving sorts. This is a small category {with finite colimits} (as for instance is the category of finite sets for the one-sorted case), which one can see as the syntactic category for the \emph{theory of $\Sigma$-structures}, whose classifying topos is the presheaf topos $\mathcal{S}[\mathbb{O}_\Sigma] \simeq \widehat{ \Ctx^{\op}}$; a $\Sigma$-structure in a topos $\mathcal{E}$ hence just is a lex functor $ \Ctx_\Sigma^{\op} \rightarrow \mathcal{E} $. For instance, for any one sorted signature, the category of contexts is equivalent to the category of finite sets $ \mathcal{S}_{\omega}$ as a context is essentially a finite set of variables, and one recover the usual classifier of the theory of objects $ \mathcal{S}[\mathbb{O}]$. 
\end{division}

\begin{division}
Then for a finite limit (resp. regular, coherent, first order intuitionistic, resp. first order classical) theory $ \mathbb{T}$ on a signature $ \Sigma$, there is a primary (resp. existential, resp. coherent, resp. first order, resp. boolean) doctrine $ P_\mathbb{T} : \Ctx_\Sigma \rightarrow \MSLat $ (resp $ P_\mathbb{T} : \Ctx_\Sigma \rightarrow \MSLat $ for a regular theory, resp. $ P_\mathbb{T} : \Ctx_\Sigma \rightarrow \DLat $ for a coherent theory, resp. $ P_\mathbb{T} : \Ctx_\Sigma \rightarrow \Heyt $ for a first order intuitionistic theory, resp. $ P_\mathbb{T} : \Ctx_\Sigma \rightarrow \Bool $ for a first order classical theory) sending a context $ \overline{x}$ to the poset of cartesian (resp. regular, resp. coherent resp. intuitionistic, resp. classical) formulas in $\Sigma$ over the context $ \overline{x}$ ordered by $\mathbb{T}$-provable entailment, and a relabelling to the corresponding substitution functor.\\

This poset exactly coincides with the poset of subobjects $ \Sub_{\mathcal{C}_\mathbb{T}}(\{ \overline{x}, \top \}$ in the cartesian (resp. regular, coherent...) syntactic category $ \mathcal{C}^{cart}_\mathbb{T}$ (resp. $ \mathcal{C}^{reg}_\mathbb{T}$, resp. $ \mathcal{C}^{coh}_\mathbb{T}$, $ \mathcal{C}^{f.o.}_\mathbb{T}$). On the other hand we have the corresponding fibration $ \int \Sub_{\mathcal{C}_\mathbb{T}} \rightarrow \Ctx_\Sigma$ sending $ \{ \overline{ x}, \phi \}$ to the underlying context $ \overline{x}$, with section sending $ \overline{x}$ to the top formula $ \{ \overline{x}, \top \}$. In fact we have an equivalence
\[  \mathcal{C}_\mathbb{T} \simeq \int P_\mathbb{T}  \]
But seeing $ P_\mathbb{T}$ as an internal $\wedge$-semilattice (resp. distributive lattice for the coherent case, resp. Hyeting algebra for the intuitionistic case and boolean algebra for the classical case)  in the presheaf topos $ \mathcal{S}[\mathbb{O}_\Sigma]$, we can apply the spectral construction for Jipsen-Moshier (resp. Zariski, resp. Esakia, resp. Stone) geometry to the modelled topos $ (\mathcal{S}[\mathbb{O}_\Sigma], P_\mathbb{T})$: as the spectral site of each $P_\mathbb{T}(\overline{x})$ is $ P_\mathbb{T}(\overline{x}) $ itself (as consisting of the dual poset of the principal filters on it), the category$ \int P_\mathbb{T}$ exactly coincides with the spectral site $ \mathcal{V}_{P_\mathbb{T}}^{\op}$.
\end{division}

\begin{division}
In the case of a finite-limit theory, there is no topology on the syntactic site $ \mathcal{C}^{cart}_\mathbb{T}$, nor on the spectral site $ \int P_\mathbb{T}$: indeed there are not specified horizontal covers for $ \Ctx_\Sigma$ bears no non-trivial topology, while there is no topology for the Jipsen-Moshier geometry. Hence taking the presheaf topos over those two categories returns the same topos and exhibit, as stated below, the classifying topos of $\mathbb{T}$ as the Jipsen-Moshier spectrum. \\


For the coherent case, the coherent syntactic category $ \mathcal{C}^{coh}_\mathbb{T}$ comes equipped with its coherent topology generated from finite families $ [\theta_i]_\mathbb{T} : \{ \overline{y_i}, \psi_i \} \rightarrow \{ \overline{x_i}, \phi \}$ such that $ \phi \vdash  \bigvee_{i \in I} \exists\overline{y_i} (\psi_i(y_i) \wedge \theta_i(x,y_i))$, which exactly means that one has $ 1_{ \{\overline{x}, \phi \} } = \bigvee_{i \in I} \exists\overline{y_i}(\psi_i(y_i) \wedge \theta_i(x,y_i))$ in $ \Sub_{\mathcal{C}^{coh}_\mathbb{T}}(\{ \overline{x}, \phi \})$. But this latter condition exactly says that the corresponding family of subobjects $(\{ \overline{x}, \exists \overline{y_i} (\psi_i(y_i) \theta_i \} \rightarrowtail \{\overline{x}, \phi \})_{i \in I}$ is a cover for the Zariski topology $ J_{Zar}$ in the corresponding fiber $ P_\mathbb{T}(\overline{x}) = \Sub_{\mathcal{C}^{coh}_\mathbb{T}} ( \{ \overline{ x}, \top \})$, and as each cover for the coherent topology is in particular included in the sieve generated from its own image factorization which lives in the fibers of its codomain, the coherent topology $ J_{coh}$ and the spectral topology correspond through the equivalence above and define the same sheaf topos 
\[ \Sh(\mathcal{C}_\mathbb{T}^{coh}, J_{coh}) \simeq \Sh(\int P_\mathbb{T}, J_{Zar, P_\mathbb{T}}) \]

The same argument applies for the case of first order intuitionistic and classical theories, where the coherent topology is also used on the syntactic site and induces an equivalence, while the localizing topology of the corresponding Esakia and Stone geometries are just restrictions of Zariski geometry. We hence have the following, which is a mere observation at this point yet underpins the expressive power of the spectral construction:
\end{division}

\begin{theorem}
For a finite-limit (resp. coherent, resp. first order intuitionistic, resp. classical) theory $\mathbb{T}$, we have an equivalence 
\[ \mathcal{S}[\mathbb{T}] \simeq \Spec(P_\mathbb{T}) \]
for the Jipsen-Moshier spectrum (resp. Zariski, resp. Esakia, resp. Stone) spectrum.
\end{theorem}

\begin{remark}
    Observe that this process is not able to recover the regular case. We conjecture there is a Grothendieck topology on the opposite category of finitely presented $\wedge$-semilattices, for wich all $\wedge$-semilattices are models in set, but which is not trivial in other topoi than $\mathcal{S}$, and which can be chosen to induce a refinement of Jipsen-Mosheir geometry such that the corresponding notion of spectrum applied to the regular case begets the classifying topos of a regular theory. We still do not know what is this topology: it should express somewhat that the top element is projective, which of course trivialises in set-valued $\wedge$-semilattice... 
\end{remark}

\subsection{Spectrum of hyperdoctrines}

We conclude this paper by applying the spectral construction to the subobject hyperdoctrine of syntactic sites for theories of fragment of logics. 
Though hyperdoctrine coding for the different fragments of first order logic are not in general explicitly required to satisfy the sheaf condition relative to a specified topology on their base, they often happen to do so: we will see for instance than the subobject hyperdoctrine of a regular, resp. a coherent category, defines a sheaf of $\wedge$-semilattices for the regular topology, resp. a sheaf of distributive lattice for the coherent topology. Then one can ask what begets to apply the spectral construction in such case: we prove at \cref{localness over classifying topos} that the spectrum is local over the classifying topos. \\


\begin{division}A primary doctrine on a small lex category $\mathcal{C}$ is a functor $ P : \mathcal{C}^{\op} \rightarrow \MSLat$. No left adjoint to the transition morphisms are yet required. If $\mathcal{C}$ is a small lex category, then the functor $ \Sub_\mathcal{C}$ sending an object $ c$ to the $\wedge$-semilattice of its subobject is a primary doctrine on $\mathcal{C}$, for mono are stable under pullback and subobject form (necessarily small) $\wedge$-semilattices.

Let be $ \mathbb{T}$ a finite-limit theory, with $ \mathcal{C}_\mathbb{T}$ its cartesian syntactic category. Then $ \Sub_{\mathcal{C}_\mathbb{T}} : \mathcal{C}_\mathbb{T}^{\op} \rightarrow \MSLat$ is a primary doctrine and furthermore a $ \wedge$-semilattice in the classifying topos $ \mathcal{S}[\mathbb{T}] \simeq \widehat{\mathcal{C}_\mathbb{T}} $. Hence we can apply the construction of the spectrum for Jipsen-Moshier geometry to $ \Sub_{\mathcal{C}_\mathbb{T}}$: we get a fibered spectral site $ \int \Sub_\mathcal{C}$ whose objects are pairs $(c, m)$ with $m$ a subobject of $c$ and arrows are $ u : (c_1,m_1) \rightarrow (c_2, m_2)$ such that $ m_1 \leq u^*m_2$ (or equivalently, $ \exists_u m_1 \leq m_2$); there is not topology at all to consider, for neither the base topos nor the Jipsen-Moshier spectral site admit non-trivial cover. 

Clearly this construction does not provide an equivalence of categories between the syntactic category of $\mathbb{T}$ and the spectral site for $\Sub_{\mathcal{C}_\mathbb{T}}$, for any monomorphism begets two copies of its domain $ m : (d,1_d) \rightarrow (c,m)$ and $ m : (c,m) \rightarrow (c,1_c)$, which cannot be identified. However we will see later what can be said of the relation between the spectrum and the classifying topos. 
\end{division}

\begin{division}
 
Now we turn to the case of existential (aka regular) hyperdoctrines: those are those $ P : \mathcal{C}^{\op} \rightarrow \MSLat $ with $\mathcal{C}$ lex such that each transition morphism $ P(u) P(c_2) \rightarrow P(c_1)$ at $ u : c_1 \rightarrow c_2$ has a left adjoint $ \exists_u $ satisfying:\begin{itemize}
    \item \emph{Frobenius condition}: for any $x_1$ in $P(c_1)$, $x_2$ in $P(c_2)$, one has 
    \[ \exists_u(c_1 \wedge P(u)(c_2)) = \exists_u(c_1) \wedge c_2 \]
    \item \emph{Beck-Chevalley condition}: for any pair of morphisms $ u_1 : c_1 \rightarrow c$, $u_2 : c_2 \rightarrow c$ in $\mathcal{C}$, the following square over their pullback commutes:
\[\begin{tikzcd}
	{P(c_1\times_c c_2)} & {P(c_2)} \\
	{P(c_1)} & {P(c)}
	\arrow["{P(u_1^*u_2)}", from=2-1, to=1-1]
	\arrow["{\exists_{u_1}}"', from=2-1, to=2-2]
	\arrow["{P(u_2)}"', from=2-2, to=1-2]
	\arrow["{\exists_{u_2^*u_1}}", from=1-1, to=1-2]
\end{tikzcd}\]
\end{itemize}

In particular, for any small regular category, the subobject functor $ \Sub_\mathcal{C} : \mathcal{C}^{\op} \rightarrow \MSLat $ defines an existential hyperdoctrine thanks to the existence of the image factorization. Now recall that any regular category is naturally equiped with the \emph{regular topology} $J_{reg}$ whose cover consists of all singleton families consisting of a regular epimorphism $ \{u : d \twoheadrightarrow c \}$. Though the formalism of hyperdoctrines generaly ignores topological data attached to the indexing category, it is natural to ask whether the subobject hyperdoctrine of a regular category is a special case of internal semilattice in the sheaf topos $ \Sh(\mathcal{C}, J_{reg})$. We first prove an auxiliary lemma:
   
\end{division}

\begin{lemma}
    For any regular epimorphism $u : d \rightarrow c$, we have a retraction 
\[\begin{tikzcd}
	{\Sub_\mathcal{C}(c)} && {\Sub_\mathcal{C}(c)} \\
	& {\Sub_\mathcal{C}(d)}
	\arrow["{\exists_u}"', from=2-2, to=1-3]
	\arrow[Rightarrow, no head, from=1-3, to=1-1]
	\arrow["{u_* }"', from=1-1, to=2-2]
\end{tikzcd}\]
Moreover, $ \exists_u$ preserves the top element.  
\end{lemma}

\begin{proof}
    This is just a consequence of the uniqueness of the epi-mono factorization combined with stability of epimorphisms. This can also be seen as applying Frobenius $ \exists_u (1_d \wedge u^*m) = \exists_u(1_d) \wedge m$ knowing that $ \exists_u(1_d)=1_c$. 
\end{proof}

\begin{lemma}
    For any small regular category $ \mathcal{C}$, $ \Sub_{\mathcal{C}}$ is a sheaf of $\wedge$-semilattices for the regular topology.
\end{lemma}

\begin{proof}
Recall that any regular epimorphism is the coequalizer of its own kernel pair: hence the descent condition of a sheaf for the regular topology reduce to prove that the following diagram is an equalizer 
\[\begin{tikzcd}
	{\Sub_\mathcal{C}(c)} & {\Sub_\mathcal{C}(d)} & {\Sub_\mathcal{C}(d \times_cd) }
	\arrow["{u^*}", from=1-1, to=1-2]
	\arrow["{\pi_1^*}", shift left=1, from=1-2, to=1-3]
	\arrow["{\pi_2^*}"', shift right=1, from=1-2, to=1-3]
\end{tikzcd}\]
First it is clear that one has an injection $ \Sub_\mathcal{C}(c)$ into the equalizer of this parallel pair by commutation of pullback along projections. In the converse direction, take a subobject $ m : e \rightarrowtail d$ such that $ \pi_1^*m = \pi_2^*m$: it defines a subobject $ \exists_u m $ of $c$, and we must show that pulling it back again along $u$ yeild again $m$. The Beck-Chevalley condition at the kernel pair of $u$ says that the following square commutes
\[\begin{tikzcd}
	{\Sub_\mathcal{C}(c)} & {\Sub_\mathcal{C}(d)} \\
	{\Sub_\mathcal{C}(d)} & {\Sub_\mathcal{C}(d \times_cd) }
	\arrow["{u^*}"', from=1-1, to=2-1]
	\arrow["{\exists_{\pi_2}}", shift left=1, from=2-2, to=2-1]
	\arrow["{\pi_1^*}", shift left=1, from=1-2, to=2-2]
	\arrow["{\exists_u}"', from=1-2, to=1-1]
\end{tikzcd}\]
But, observing that the projections $ \pi_1, \pi_2$ are regular epimorphisms by pullback stability, the previous lemma ensures that $ \exists_{\pi_2}$ is a retraction of $ \pi_2$, and applying those in the case of $m$ yields
\begin{align*}
    u^* \exists_u m &= \exists_{\pi_2} \pi_1^*m \\ 
    &= \exists_{\pi_2} \pi_2^*m \\
    &= m
\end{align*}
This ensures that $ \Sub_\mathcal{C}(c)$ is isomorphic to the equalizer of $\pi_1^*, \pi_2 ^*$.
\end{proof}


Seeing now $ \Sub_\mathcal{C}$ as an $\wedge$-semilattice object in the sheaf topos $ \Sh(\mathcal{C},J_{reg})$, we can apply the spectral construction relative to Jipsen-Moshier geometry. At each $ c$ of $\mathcal{C}$, the Jipsen-Moshier spectral site simply is the $\wedge$-semilattice of subobjects $ \Sub_{\mathcal{C}} (c)$ itself, with the trivial topology. Now the fibered site is the category $ \int \Sub_\mathcal{C}$ whose objects are pair $(c, m)$ with $m$ a subobject of $c$ and arrows are $ u : (c_1,m_1) \rightarrow (c_2, m_2)$ such that $ m_1 \leq u^*m_2$ (or equivalently, $ \exists_u m_1 \leq m_2$), equipped with a spectral topology generated only from horizontal families $ \{ u : (d,1_d) \rightarrow (c, 1_c) \}$ obtained by lifting a regular epimorphism $ u : d \twoheadrightarrow c$. Then the spectrum for Jipsen-Moshier geometry is the sheaf topos
\[ \Spec(\Sub_\mathcal{C}) = \Sh(\int \Sub_{\mathcal{C}},J_{\Sub_\mathcal{C}, J_{reg}}) \]

Actually the spectral site here is somewhat too big for the spectrum to coincide with the classifying topos: indeed each subobject $ m : d \rightarrow c$ is duplicated in $ \int {\Sub_{\mathcal{C}}}$ as both the object $(c, m)$ and the arrow $m : (d, 1_d) \rightarrow (c,m)$, and one cannot expect a sheaf for the spectral topology to satisfy the requirement that $ X (c,m) = X(d, 1_d)$ as the arrow $m : (d, 1_d) \rightarrow (c,m)$ fails to be a cover for the lifted topology, as $ m$ is not a regular epimorphism (unless being an isomorphism). One can easily define a retraction  
\[\begin{tikzcd}
	{\Sh(\int \Sub_{\mathcal{C}},J_{\Sub_\mathcal{C}, J_{reg}})} & {\Sh(\mathcal{C},J_{reg})}
	\arrow[from=1-1, to=1-2]
\end{tikzcd}\]
sending a sheaf $ X $ over $ \int \Sub_{\mathcal{C}}$ to the functor $ \widetilde{X} : \mathcal{C}^{\op} \rightarrow \mathcal{S}$ defined as $ \widetilde{X}(c) = X(c, 1_c)$, with a section sending a sheaf $F$ on $\mathcal{C}$ for the regular topology to the functor $ \widetilde{F} : \int \Sub_{\mathcal{C}}^{\op} \rightarrow \mathcal{S}$ sending $ (x,m)$ with $ m : d \rightarrowtail c$ to the value $F(d)$ of $F$ at the domaine of $m$. Thoses assignement are respectively the inverse and direct image part of a retraction: however we cannot improve this into an equivalence between the spectrum and the classifying topos. However we still can tell a bit more, see below \cref{localness over classifying topos} for a general statement that will encompass not only the regular case but also the other cases.

\begin{division}
We turn now to the case of coherent hyperdoctrines. For $ P: \mathcal{C}^{\op} \rightarrow \DLat$ a coherent hyperdoctrine which is moreover a sheaf for $J$ on $\mathcal{C}$, we can consider the spectrum of the underlying distributive lattice object $P$ in $\Sh(\mathcal{C},J)$, which is constructed as 
\[  \Spec(P) = \Sh(\int P, J_{P,J}) \]
where recall that $ \int P$ is the category whose objects are pairs $(c,a)$ with $a \in P(c)$ (recall that the spectral site at each $c$ is $ Zar_P(c)^{\op} \simeq P(c)$), morphisms are $u : (c_1,a_1) \rightarrow (c_2,a_2)$ with $ u : c_1 \rightarrow c_2$ and $c_1 \leq P(u)(a_2)$ in $P(c_1)$, and the topology $J_{P,J}$ is jointly generated from the Zariski topologies $ (J^{Zar}_{c})_{c \in \mathcal{C}}$ and the horizontal families $(u_i, 1_{1_i}) : (c_i, 1_{c_i}) \rightarrow (c,1_c))_{i \in I}$ for $ (u_i)_{i \in I}$ in $J$.


Now consider the case of the coherent hyperdoctrine $ \Sub_{\mathcal{C}} : \mathcal{C}^{\op} \rightarrow \DLat $ associated to a small coherent category $\mathcal{C}$, with $ J_{coh}$ the coherent topology on $\mathcal{C}$, which, recall is the topology generated from all finite families $(u_i : c_i \rightarrow c)_{i \in I}$ such that $ \coprod_{i \in I} c_i \rightarrow c$ is a regular epimorphism (so that we have a cover of the top element $ 1_c = \bigvee_{i \in I} \exists_{u_i} 1_{c_i}$ in $\Sub_{\mathcal{C}}(c)$). Though this is somewhat folklore, we think it is worth emphasizing that the coherent hyperdoctrine of subobject of a small coherent category is an sheaf of distributive lattices for the coherent topology, hence is an internal lattice in the associated sheaf topos. The following proof use essentially the same argument as \cite{coumans2012generalising}[Proposition 28] -- though the later provides a corresponding statement only for the \emph{canonical extension} of the subobject hyperdoctrine).    
\end{division} 

\begin{lemma}
For $\mathcal{C}$ a small coherent category, $\Sub_\mathcal{C} : \mathcal{C}^{\op} \rightarrow \DLat$ is a sheaf for $J_{coh}$. 
\end{lemma}

\begin{proof}
    Let be $(u_i : c_i \rightarrow c)_{i \in I}$ a (finite) covering family for $J_{coh}$: we must exhibit a limit decomposition 
    \[ \Sub_{\mathcal{C}} (c) = \lim \big{(} \prod_{i \in I} \Sub_{\mathcal{C}}(c_i) \rightrightarrows \prod_{i,j \in I} \Sub_{\mathcal{C}}(c_{ij}) \big{)} \]
    which will be provided by sending a subobject $ m : d \rightarrowtail c$ to the family $(u_i^*m : u_i^*d \rightarrowtail c_i)_{i \in I}$ and a matching family $ (m_i : d_i \rightarrowtail c_i)_{i \in I}$ to the finite join of their direct images $ \bigvee_{i \in I} \exists_{u_i}m_i$. We must prove this process to be a bijection. For each $i,j$ we have $ u_{ij}^*u_i^*m = u_{ji}^*u_j^*m $ as a subobject of $ c_{ij} = c_i \times_c c_j$; now take the coproduct $ u_I : \coprod_{i \in I} \rightarrow c$, which is a regular epi, and then the coproduct of the fibers $ (u_i^*m : u_i^*d \rightarrow c_i)_{i \in I}$ over $ \coprod_{i \in I} c_i $. Then the image this coproduct map is identified with the pullback $ u_I^*m$ as depicted below
\[\begin{tikzcd}
	& {\displaystyle\coprod_{i \in I}u_i^*d} \\
	{u_i^*d} & {u_I^*d} & d \\
	& {\displaystyle\coprod_{i \in I} c_i} & c \\
	{c_i}
	\arrow["{u_I}", two heads, from=3-2, to=3-3]
	\arrow["m", tail, from=2-3, to=3-3]
	\arrow["{\langle m^*u_i \rangle_{i \in I}}", two heads, from=1-2, to=2-3]
	\arrow["{q_i}", from=4-1, to=3-2]
	\arrow["{u_i}"', from=4-1, to=3-3]
	\arrow["{u_i^*m}"', tail, from=2-1, to=4-1]
	\arrow[from=2-1, to=1-2]
	\arrow[two heads, from=1-2, to=2-2]
	\arrow["{u_I^*m}"', tail, from=2-2, to=3-2]
	\arrow["{m^*u_I}"', two heads, from=2-2, to=2-3]
	\arrow["\lrcorner"{anchor=center, pos=0.125}, draw=none, from=2-2, to=3-3]
\end{tikzcd}\]
and moreover $ \im(\coprod_{i \in I} u_i^*m) = \bigvee_{i \in I} \exists_{q_i} u_i^*m $ in $\Sub_\mathcal{C}(\coprod_{i \in I} c_i)$ while uniqueness of the image factorization combined with stability of regular epimorphisms ensures that $ m = \exists_{u_I} u_I^*m $, and hence we have 
\begin{align*}  m = \exists u_I u_I^*m &=  \exists_{u_I} (\bigvee_{i \in I} \exists_{q_i} u_i^*m) \\
  &= \bigvee_{i \in I} \exists_{u_I} \exists_{q_i} u_i^*m \\
  &= \bigvee_{i \in I} \exists_{u_i} u_i^*m 
  \end{align*}

  For the converse, a matching family is a family of subobjects $ (\exists_{q_i}m_i : \exists_{q_i}d_i \rightarrowtail \coprod_{i \in I} c_i)_{i \in I}$ such that $ u_{ij}^*m_i = u_{ji}^*m_j$, and one takes $ \bigvee_{i \in I} \exists_{u_i}m_i$ in $\Sub_{\mathcal{C}}(c)$ as the gluing: we must prove this restrict backs. First observe that $ m_i \leq u_i^*\exists_{u_i} m_i \leq u_i^*( \bigvee_{j \in I} \exists_{u_j}m_j)$. We must prove the second inequality. As finite joint are pullback stable, one has $ u_i ^*(\bigvee_{j \in I} \exists_{u_j}m_j) = \bigvee_{j \in I} u_i^*\exists_{u_j}m_j$ so it suffices to prove that for each $j $ one has $ u_i^*\exists_{u_j}m_j \leq m_i $. Applying Beck-Chevalley to the pullback square at $u_i, u_j$ yields a square
\[\begin{tikzcd}
	{\Sub_\mathcal{C}(c)} & {\Sub_\mathcal{C}(c_j)} \\
	{\Sub_\mathcal{C}(c_i)} & {\Sub_\mathcal{C}(c_i \times_c c_i) }
	\arrow["{\exists_{u_j}}"', from=1-2, to=1-1]
	\arrow["{u_i^*}"', from=1-1, to=2-1]
	\arrow["{\exists_{u_{ij}}}", from=2-2, to=2-1]
	\arrow["{u_{ji}^*}", from=1-2, to=2-2]
\end{tikzcd}\]
whose commutativity combined with the compatibility condition entails the desired inequality: 
\begin{align*}
    u_i^*\exists_{u_j}m_j &= \exists_{u_{ij}} u_{ji}^*m_j \\
    &= \exists_{u_{ij}} u_{ij}^*m_i \\
    &\leq m_i
\end{align*}
This proves we get back the $ (m_i)_{i \in I}$ by pulling back their gluing, which concludes to prove $\Sub_{\mathcal{C}}$ to be a sheaf for the coherent topology. 
  \end{proof}

As a consequence the hyperdoctrine of subobject of a small coherent category defines a modelled topos for the theory of distributive lattices $ (\Sh(\mathcal{C},J_{coh}), \Sub_{\mathcal{C}})$, and the spectrum of this hyperdoctrine is obtained as 

\[  \Spec(\Sub_\mathcal{C}) \simeq \Sh(\int \Sub_\mathcal{C},J_{\Sub_\mathcal{C},J_{coh}}) \]

where $J_{\Sub_\mathcal{C},J_{coh}}$ is defined as above as jointly generated from lifting of coherent families and fiber-wise Zariski-covering families.\\

But we know this construction will be somewhat redundant: in fact, while the spectrum is aimed at correcting an ambient object into a local object, here a distributive lattice into a local one, it happens that for a coherent category, the internal distributive lattice $\Sub_\mathcal{C}$ already is local:

\begin{proposition}
    For $\mathcal{C}$ a small coherent category, $ \Sub_\mathcal{C}$ is an internal 1-local distributive lattice.  
\end{proposition}

\begin{proof}
    The subobject functor can be seen as the functor $ \DLat_{\omega}^{\op} \rightarrow \Sh(\mathcal{C},J_{coh})$ sending a finitely presented distributive lattice $ K$ (that is, a finite distributive lattice) to the functor sending an object $c$ of $\mathcal{C}$ to the homset $ \DLat[k, \Sub_\mathcal{C}(c)] $. We must prove this functor to send $J_{Zar}$-covers to jointly epimorphic families. We saw that covers for $J_{Zar}$ were generated by families of the form
\[(\begin{tikzcd}
	{\frac{\langle x_1, \dots x_n \rangle}{\bigvee_{i =1, \dots, n} x_i = 1}} & {\frac{\langle x_1, \dots x_n \rangle}{\bigvee_{i =1, \dots, n} x_i = 1 , x_j =1}}
	\arrow["{q_j}", two heads, from=1-1, to=1-2]
\end{tikzcd})_{j = 1, \dots, n} \]
It suffices now to show that the following natural transformation 
\[\begin{tikzcd}[column sep=large]
	{\displaystyle\coprod_{i=1, \dots, n} \DLat[ \frac{\langle x_1, \dots x_n \rangle}{\bigvee_{i =1, \dots, n} x_i = 1}, \Sub_{\mathcal{C}}(-)]} && {\DLat[\frac{\langle x_1, \dots x_n \rangle}{\bigvee_{i =1, \dots, n} x_i = 1 , x_j =1}, \Sub_{\mathcal{C}}(-)]}
	\arrow["{\DLat[q_j, \, \Sub_{\mathcal{C}}(-)]}", two heads, from=1-1, to=1-3]
\end{tikzcd}\]
is a $J_{coh}$-local epimorphism. For any $c$, the homset $\DLat[\frac{\langle x_1, \dots x_n \rangle}{\bigvee_{i =1, \dots, n} x_i = 1 , x_j =1}, \Sub_{\mathcal{C}}(c)] $ contains exactly $n$-tuples of subobjects $ (m_i : d_i \rightarrowtail c)_{i \in I}$ such that $ \bigvee_{i= 1, \dots, n} m_i = 1_c$: but any such tuples exactly defines a cover for the coherent topology, and for each $i = 1, ..., n$ the corresponding restriction functor $ m_i^* : \Sub_{\mathcal{C}}(c) \rightarrow \Sub_{\mathcal{C}}(c_i)$ sends $ m_i$ to $ 1_{d_i}$ as $m_i$ is a monomorphism so the following square is a pullback
\[\begin{tikzcd}
	{d_i} & {d_i} \\
	{d_i} & c
	\arrow["{m_i}", tail, from=1-2, to=2-2]
	\arrow["{m_i}"', tail, from=2-1, to=2-2]
	\arrow[Rightarrow, no head, from=1-1, to=2-1]
	\arrow[Rightarrow, no head, from=1-1, to=1-2]
	\arrow["\lrcorner"{anchor=center, pos=0.125}, draw=none, from=1-1, to=2-2]
\end{tikzcd}\]
But this exactly means that for each $j = 1, ... , n$ we have a factorization 
\[\begin{tikzcd}[column sep=large]
	{\frac{\langle x_1, \dots x_n \rangle}{\bigvee_{i =1, \dots, n} x_i = 1}} && {\Sub_\mathcal{C}(c)} \\
	{\frac{\langle x_1, \dots x_n \rangle}{\bigvee_{i =1, \dots, n} x_i = 1 , x_j =1}} && {\Sub_\mathcal{C}(d_j)}
	\arrow["{\ulcorner m_1, \dots m_n \urcorner}", from=1-1, to=1-3]
	\arrow["{q_j}"', from=1-1, to=2-1]
	\arrow["{m_j^*}", from=1-3, to=2-3]
	\arrow["{\ulcorner u_j^*m_1, \dots u_j^*m_n \urcorner}"', from=2-1, to=2-3]
\end{tikzcd}\]
This ensures that the cover $ (m_i)_{i \in I}$ satisfies the condition that any restriction of the arrow $\ulcorner m_1, \dots, m_n \urcorner$ along a $ m_j^*$ has an antecedents along the corresponding $q_i$ since $m_j^* $ preserves finite joins so that one still has $ 1_{d_j} = \bigvee_{i=1, \dots , n} m_j^*m_i$ and by what was said above one has $ m_j^*m_j=1_{d_j}$. Hence the morphism above is a local epimorphism, and $ \Sub_{\mathcal{C}}$ is ensured to be a local lattice.  
\end{proof}

\begin{remark}\label{localness of pralouf}
    Now a local Heyting algebra is exactly a Heyting algebra whose underlying distributive lattice is local, while the only local boolean algebra is 2. As a consequence, for the subobject hyperdoctrine is an internal local lattice in $\Sh(\mathcal{C},J_{coh})$, and that any Heyting category, and any boolean coherent category are in particular coherent, their subobject hyperdoctrine are respectively an internal local Heyting algebra, and internally the 2-element lattice -- which does not means of course that it has to return 2 as values !
\end{remark}

In particular we can apply this construction to the coherent syntactic category $\mathcal{C}_\mathbb{T}$ of a coherent theory $\mathbb{T}$: 

\begin{corollary}\label{localness over classifying topos}

For any finite-limit, resp. regular, resp. coherent, resp. intuitionistic, resp. classical first order theory $ \mathbb{T}$, one has a local geometric morphism 
\[\begin{tikzcd}
	{\Spec(\Sub_{\mathcal{C}_\mathbb{T}})} & {\mathcal{S}[\mathbb{T}]}
	\arrow[from=1-1, to=1-2]
\end{tikzcd}\]
where $ \Spec$ denote respectively the Jipsen-Moshier spectrum for the finite limit and regular case, the Zariski spectrum for the coherent case, Esakia spectrum for the intuitionnistic case and Stone spectrum for the boolean case. 
\end{corollary}

\begin{proof}
    In each case, the subobject hyperdoctrine defines a local object in $\mathcal{S}[\mathbb{T}]$ for the corresponding geometry, which can also be seen in the property of the fibration. \\

    For the finite-limit case, the fibration $ \int \Sub_{\mathcal{C}_\mathbb{T}} \rightarrow \mathcal{C}_\mathbb{T}$ has a right adjoint section, hence by \cite{elephant}[example C3.6.3(b)] the induced geometric morphism between presheaf topoi is local, but it is exactly the underlying part of the unit of the spectral adjunction. \\

    For the regular case, observe that the covers of $J_{\Sub_{\mathcal{C}_\mathbb{T}}} $ consists only of lifts of regular covers $ \{ u : (d, 1_d) \rightarrow (c,1_c) \}$ for $u : d \twoheadrightarrow c $ a regular epimorphism in $\mathcal{C}_\mathbb{T}$, hence their projections still are cover for the regular cover, which makes the spectral fibration a continuous comorphism of site with fully faithful right adjoint. Whence the localness of the induced morphism. \\

    For the coherent and additional cases, we saw at \cref{localness of pralouf} that the subobject hyperdoctrine is an internal local distributive lattice, so that the underlying geometric morphism of the unit of the spectral adjunction is local as established in \cref{Spectra of locally modelled topoi are local}.     
\end{proof}

\newpage

\section*{Declarations}

\subsubsection*{Authors contribution}
Axel Osmond is the author of this paper.

\subsubsection*{Funding} The author acknowledges the support of the Grothendieck Institute through his post-doctoral fellowship inside the project \emph{Topos theory and its applications}. 

\subsubsection*{Availability of Data and Material Data} sharing not applicable to this article as no datasets were generated or
analysed during the current study.

\subsubsection*{Competing interests} The author declares that there are no competing interests.

\subsubsection*{Open Access} This article is licensed under a Creative Commons Attribution 4.0 International License, which
permits use, sharing, adaptation, distribution and reproduction in any medium or format, as long as you give
appropriate credit to the original author and the source, provide a link to the Creative Commons licence,
and indicate if changes were made.

\printbibliography

\vfill

 \begin{minipage}{0.49\textwidth}
 
 \textsc{Axel Osmond} 

\vspace{0.2cm}
{\small \textsc{Istituto Grothendieck,
		Corso Statuto 24, 12084 Mondovì, Italy.}\\
	\emph{E-mail address:} \texttt{axelosmond@orange.fr}}
 
\end{minipage}

\end{document}